\documentclass[10pt]{article}
 \usepackage{amsmath, amsfonts, amsthm, amssymb, amscd, enumerate, url, slashed, stmaryrd, upgreek, thmtools, mathdots}
 \usepackage{float}
 \usepackage{tikz} 
 \usetikzlibrary{matrix}
 
\begingroup
 \makeatletter
 \@for\theoremstyle:=definition,remark,plain\do{%
 \expandafter\g@addto@macro\csname th@\theoremstyle\endcsname{%
 \addtolength\thm@preskip\parskip
 }%
 }
 \endgroup
 
 \usepackage[titletoc,toc,title]{appendix}
 \usepackage{array}
 
 \usepackage[colorlinks,citecolor=red,linkcolor=blue,urlcolor=cyan,pagebackref]{hyperref}

\usepackage[alphabetic,backrefs]{amsrefs}
 
 \usepackage[parfill]{parskip}
 \usepackage{anysize}
 \marginsize{1in}{1in}{1in}{1in}

\newtheorem{thm}{Theorem}[section]
\newtheorem{prop}[thm]{Proposition}
\newtheorem{lem}[thm]{Lemma}
\newtheorem{cor}[thm]{Corollary}

\newenvironment{dfn}{\medskip\refstepcounter{thm}
\noindent{\bf Definition \thesection.\arabic{thm}.}}{\medskip}
\newenvironment{remark}{\medskip\refstepcounter{thm}
\noindent{\bf Remark \thesection.\arabic{thm}.}}{\medskip}
\newenvironment{remarks}{\medskip\refstepcounter{thm}
\noindent{\bf Remarks \thesection.\arabic{thm}.}}{\medskip}
\numberwithin{equation}{section}
\numberwithin{table}{section}
\numberwithin{figure}{section}

\def\eq#1{{\rm(\ref{#1})}}

\DeclareMathOperator\vol{vol}
\def\SS{\mathbb{S}}
\DeclareMathOperator\GG{G}

\DeclareMathOperator\GL{GL}
\DeclareMathOperator\SO{SO}
\DeclareMathOperator\U{U}
\DeclareMathOperator\SU{SU}
\DeclareMathOperator\Sp{Sp}

\DeclareMathOperator\Ric{Ric}
\DeclareMathOperator\Ree{Re}
\DeclareMathOperator\Imm{Im}

\def\d{\mathrm{d}}
\def\w{\wedge}

\def\C{\mathbb{C}}
\def\P{\mathbb{P}}
\def\R{\mathbb{R}}
\def\Z{\mathbb{Z}}
\def\H{\mathbb{H}}

\begin{document}

\title{Bryant--Salamon \texorpdfstring{$\GG_2$}{G2} manifolds and coassociative fibrations}

\author{Spiro Karigiannis\thanks{Department of Pure Mathematics, University of Waterloo, \tt{karigiannis@uwaterloo.ca}} \, and Jason D.\ Lotay\thanks{Mathematical Institute, University of Oxford, \tt{jason.lotay@maths.ox.ac.uk}}}

\date{}

\maketitle

\begin{abstract}
Bryant--Salamon constructed three 1-parameter families of complete manifolds with holonomy $\GG_2$ which are asymptotically conical to a holonomy $\GG_2$ cone. For each of these families, including their asymptotic cone, we construct a fibration by asymptotically conical and conically singular coassociative 4-folds. We show that these fibrations are natural generalizations of the following three well-known coassociative fibrations on $\R^7$: the trivial fibration by 4-planes, the product of the standard Lefschetz fibration of $\C^3$ with a line, and the Harvey--Lawson coassociative fibration. In particular, we describe coassociative fibrations of the bundle of anti-self-dual 2-forms over the 4-sphere $\mathcal{S}^4$, and the cone on $\C\P^3$, whose smooth fibres are $T^*\mathcal{S}^2$, and whose singular fibres are $\R^4/\{\pm 1\}$. We relate these fibrations to hypersymplectic geometry, Donaldson's work on Kovalev--Lefschetz fibrations, harmonic 1-forms and the Joyce--Karigiannis construction of holonomy $\GG_2$ manifolds, and we construct vanishing cycles and associative ``thimbles'' for these fibrations.

\end{abstract}

\tableofcontents

\section{Introduction} \label{introduction}

A key challenge in the study of manifolds with special holonomy is to construct, understand, and make use of calibrated fibrations. In the context of $\GG_2$ holonomy in dimension 7, the focus is on studying fibrations by coassociative 4-folds. Inspired by the SYZ conjecture~\cite{StromingerYauZaslow}, which relates certain special Lagrangian fibrations to mirror symmetry of Calabi--Yau 3-folds, one hopes to use coassociative fibrations of $\GG_2$ manifolds to understand analogous dualities. (See for example~\cite{GukovYauZaslow} and, more recently,~\cite{YangLi-MukaiDuality}.) In another direction, one aims to use coassociative fibrations as a means to understand and potentially construct new $\GG_2$ manifolds, cf.~\cite{Donaldson-KovalevLefschetz}, including the possibility of certain singular $\GG_2$ manifolds of interest in M-Theory (cf.~\cites{AcharyaWitten,AtiyahWitten}).

In this paper we exhibit explicit coassociative fibrations on the three complete non-compact $\GG_2$ manifolds discovered by Bryant--Salamon~\cite{BryantSalamon}. Each of these three examples comes in two versions.
\begin{itemize}
\item The ``smooth version'' is a 7-manifold $M$ equipped with a torsion-free $\GG_2$-structure $\varphi_c$ inducing an asymptotically conical holonomy $\GG_2$ metric $g_c$, depending on a parameter $c>0$.
\item The ``cone version'' is a Riemannian cone $(M_0, \varphi_0, g_0)$ with holonomy $\GG_2$ which is the asymptotic cone of the smooth version $(M, \varphi_c, g_c)$ and corresponds to the limit as $c\to0$ of $(M, \varphi_c, g_c)$.
\end{itemize}

The torsion-free $\GG_2$-structures on the Bryant--Salamon manifolds are \emph{cohomogeneity one}. In particular, their asymptotic cones are \emph{homogeneous} nearly K\"ahler 6-manifolds. (In fact, it was proved by the authors in~\cite[Corollary 6.1]{KarigiannisLotay} that an asymptotically conical $\GG_2$ manifold $(M, \varphi)$ whose asymptotic cone is the homogeneous nearly K\"ahler manifold $\mathcal{S}^3 \times \mathcal{S}^3$, $\C\P^3$, or $\SU(3)/T^2$, must necessarily be cohomogeneity one, from which it follows that $(M, \varphi)$ is a Bryant--Salamon manifold.)

For each of the three Bryant--Salamon manifolds, we consider the action of a Lie group $G$ that is embedded in a particular way in the cohomogeneity one symmetry group. We demonstrate that each is fibred by $G$-invariant coassociative submanifolds. Table~\ref{table:BS} gives the topology of the Bryant--Salamon manifolds (both the smooth version $M$ and the cone version $M_0$), their cohomogeneity one symmetry groups $H$, and the particular choices of embedded Lie subgroups $G$ that we choose to study in this paper.
\begin{table}[H]
\begin{center}
{\setlength{\extrarowheight}{4pt}
\begin{tabular}{|c|c|c|c|}
\hline
$M$ & $M_0 = M \setminus \{ \text{zero section} \}$ & $H$ & $G$ \\[2pt]
\hline
$\SS(\mathcal{S}^3)$ & $\R^+ \times (\mathcal{S}^3 \times \mathcal{S}^3)$ & $\Sp(1)^3$ & $\SU(2) \cong \Sp(1) \times \{ 1 \} \text{ acting on } \mathbb H \oplus \R^3$ \\[2pt]
\hline
$\Lambda^2_-(T^* \mathcal{S}^4)$ & $\R^+ \times \C\P^3$ & $\SO(5)$ & $\SO(3) \cong \SO(3) \times \{ 1 \} \text{ acting on } \R^3 \oplus \R^2$ \\[2pt]
\hline
$\Lambda^2_-(T^* \C\P^2)$ & $\R^+ \times (\SU(3) / T^2)$ & $\SU(3)$ & $\SU(2) \cong \SU(2) \times \{ 1 \} \text{ acting on } \C^2 \oplus \C$ \\[2pt]
\hline
\end{tabular}
}
\end{center}
\caption{For each Bryant--Salamon manifold $M$ and $M_0$, we list their cohomogeneity one symmetry group $H$, and the particular symmetry subgroup $G$ of the coassociative fibrations studied in this paper.} \label{table:BS}
\end{table}

\begin{remark} \label{which-groups-remark}
Each of the Bryant--Salamon manifolds admits several non-conjugate subgroups of their symmetry group which have generic orbits that are 3-dimensional, on which the $\GG_2$-structure vanishes, and hence could admit cohomogeneity one coassociative submanifolds. In particular, for $\Lambda^2_- (T^* \mathcal{S}^4)$ these subgroups are classified by Kawai~\cite{Kawai}. In the present paper, we have chosen to consider one particular subgroup of the symmetry group for each of the three Bryant--Salamon manifolds.
\begin{itemize}
\item The group $\Sp(1)^3$ acts on $\R^7 = \R^4 \oplus \R^3$ via the usual action $(p,q) \cdot \mathbf{a} = p \mathbf{a} \bar{q}$ of $\Sp(1)^2$ on $\R^4 = \mathbb H$ and the action $r \cdot \mathbf{x} = r \mathbf{x} \bar{r}$ of $\Sp(1)$ on $\R^3 = \mathrm{Im} \mathbb H$. In \S\ref{sec:SS3}, we study the action on $\SS(\mathcal{S}^3)$ of a particular $\SU(2) \cong \Sp(1)$ that is embedded in $\Sp(1)^3$ as $\Sp(1) \times \{ 1 \} \times \{ 1 \}$, thus acting by $p \cdot (\mathbf{a}, \mathbf{x}) = (p \mathbf{a}, \mathbf{x})$. The analysis for this particular $\SU(2)$ action on $\SS(\mathcal{S}^3)$ is essentially trivial, as explained in the next paragraph, but we include a detailed discussion for completeness, for comparison with the other two cases which are highly nontrivial, and because we can explicitly exhibit the hypersymplectic structure on the coassociative fibres.
\item In \S\ref{sec:ASDS4}, we study the action on $\Lambda^2_-(T^* \mathcal{S}^4)$ of a particular $\SO(3)$ that is embedded in $\SO(5)$ as $\SO(3) \times \{ 1 \} \subseteq \SO(\R^3 \oplus \R^2)$, for two reasons. First, because this particular $\SO(3)$ admits a commuting $\U(1)$ action, we are able to relate our coassociative fibration in a suitable limit to the classical Lefschetz fibration of $\C^3$, which we do in \S\ref{sec:flat.limit.SO3}. Second, the coassociative fibration that results from this particular $\SO(3)$ action on $\Lambda^2_-(T^* \mathcal{S}^4)$ can be directly related to Donaldson's theory~\cite{Donaldson-KovalevLefschetz} of Kovalev--Lefschetz fibrations, which we do in \S\ref{sec:VCTS4}. Nevertheless, it is certainly of interest to consider other symmetry groups of $\Lambda^2_-(T^* \mathcal{S}^4)$ which could admit cohomogeneity one coassociative fibrations, and the authors are currently studying this question.
\item In \S\ref{sec:ASDCP2}, we study the action on $\Lambda^2_-(T^* \C\P^2)$ of a particular $\SU(2)$ that is embedded in $\SU(3)$ as $\SU(2) \times \{ 1 \} \subseteq \SU(\C^2 \oplus \C)$, because it is the simplest nontrivial one. Already the results in the $\Lambda^2_-(T^* \C\P^2)$ case are much more complicated for this action.
\end{itemize}
\end{remark}

The Bryant--Salamon $\GG_2$ manifold $\SS(\mathcal{S}^3)$ is naturally a fibre bundle (in fact a vector bundle) over a 3-dimensional dimensional base with 4-dimensional fibres, and the fibres are coassociative submanifolds. Thus $\SS(\mathcal{S}^3)$ is exhibited as a coassociative fibration in a trivial way. However, the situation for the other two Bryant--Salamon $\GG_2$ manifolds $\Lambda^2_- (T^* \mathcal{S}^4)$ and $\Lambda^2_- (T^* \C\P^2)$ is very different. These manifolds are both naturally fibre bundles (in fact vector bundles) over a \emph{$4$-dimensional base with $3$-dimensional fibres}. Thus in order to exhibit these $\GG_2$ manifolds as coassociative fibrations, we need to ``flip them over'' and describe them as bundles over a $3$-dimensional base with $4$-dimensional (coassociative) fibres. Viewed in this way, the manifolds are not vector bundles. The topology and the geometry of the coassociative fibres is a very important aspect of our work in this paper.

Fibrations have also played an important role in complex and symplectic geometry, notably Lefschetz fibrations. We show that one of the family of fibrations we describe can be viewed as an analogue of a standard Lefschetz fibration in this $\GG_2$ setting. Moreover, we show that  natural topological objects of interest in Lefschetz fibrations, namely vanishing cycles and thimbles, have analogues here which moreover can be \emph{represented by calibrated submanifolds}. This discussion, in particular, fits well with Donaldson's theory~\cite{Donaldson-KovalevLefschetz} of Kovalev--Lefschetz fibrations.

\textbf{Summary of results and organization of the paper.} In \S\ref{sec:prelim} we first review some basic results on $\GG_2$ manifolds, Riemannian conifolds, and calibrated submanifolds. Then we discuss the \emph{multimoment maps} of Madsen--Swann~\cites{MadsenSwann1,MadsenSwann2}, some general facts about coassociative fibrations that we require, and introduce the notion of a \emph{hypersymplectic structure} due to Donaldson~\cite{Donaldson}.

In \S\ref{sec:BS-manifolds} we give an essentially self-contained, detailed account of the three Bryant--Salamon $\GG_2$ manifolds which first appeared in~\cite{BryantSalamon}. In particular, our treatment is presented using conventions and notation that are carefully chosen to be compatible with the construction of coassociative fibrations on these manifolds in the remainder of the paper. We also explain how, in the limit as the volume of the compact base goes to infinity, the Bryant--Salamon metrics all formally converge to the flat metric on $\R^7$, as we use this in later sections to show that our three coassociative fibrations limit to certain well-known calibrated fibrations of Euclidean space.

Section~\ref{sec:SS3} considers the case of $M = \SS(\mathcal{S}^3)$ and its cone $M_0 = M \setminus \mathcal{S}^3$. As mentioned above, this case is essentially trivial as it already naturally exhibited as a coassociative fibration with fibres that are topologically $\R^4$, respectively $\R^4 \setminus \{ 0 \}$. However, we also show that the induced Riemannian metric on the fibres is conformally flat and asymptotically conical (respectively, conical); that the induced hypersymplectic structure in both cases is the standard flat Euclidean hyperk\"ahler structure; and that the ``flat limit'' is the trivial coassociative $\R^4$ fibration of $\R^7 = \R^3 \oplus \R^4$ over $\R^3$.

The remaining two cases $\Lambda^2_- (T^* \mathcal{S}^4)$ and $\Lambda^2_- (T^* \C\P^2)$ (and their cone versions) are much more complicated and interesting, and their study takes up the bulk of the paper, in Sections~\ref{sec:ASDS4} and~\ref{sec:ASDCP2}, respectively. In both of these cases we find that there are both smooth and singular fibres, and we describe the induced Riemannian geometry on these coassociative fibres. Some of the singular fibres are exactly Riemannian cones, while other exhibit both conically singular and asymptotically conical behavour. The smooth fibres are asymptotically conical Riemannian manifolds that are total spaces of complex line bundles over $\C\P^1 = \mathcal{S}^2$. These smooth fibres are topologically $T^* \mathcal{S}^2 = \mathcal{O}_{\C\P^1}(-2)$ in the $\Lambda^2_- (T^* \mathcal{S}^4)$ case and are topologically $\mathcal{O}_{\C\P^1}(-1)$ in the $\Lambda^2_- (T^* \C\P^2)$ case.

The ``flat limit'' is obtained by letting the volume of the compact base go to infinity. In this limit, we show that our fibrations become, respectively, the product with $\R$ of the standard Lefschetz fibration of $\C^3$ by complex surfaces, and the Harvey--Lawson $\SU(2)$-invariant coassociative fibration of $\R^7$.

In the case of $\Lambda^2_- (T^* \mathcal{S}^4)$ (and its cone version) we are actually able to do much more. In \S\ref{sec:hypersymplectic.S4}, in addition to the induced Riemannian geometry, we also determine the induced hypersymplectic structure and verify that it is \emph{not} hyperk\"ahler. In \S\ref{sec:harmonicS4} we identify the Riemannian metric on the base of the fibration, which is a cone metric $k_0$ over a half-space in $\R^3$ for the cone version $M_0 = \R^+ \times \C\P^3$, and is a smooth metric $k_c$ which is asymptotically conical to $k_0$ with rate $-2$ for the smooth version $M = \Lambda^2_- (T^* \mathcal{S}^4)$. Inspired by analogy with the construction~\cite{JoyceKarigiannis}, in \S\ref{sec:harmonicS4} we also establish the existence of a \emph{harmonic 1-form} $\lambda$ on the base, which vanishes precisely at the points in the base corresponding to the singular coassociative fibres. In \S\ref{sec:circlequotientS4} we discuss the links between our work and a particular circle quotient construction studied by Acharya--Bryant--Salamon~\cite{AcharyaBryantSalamon} and Atiyah--Witten~\cite{AtiyahWitten}. Finally, in \S\ref{sec:VCTS4} we construct vanishing cycles and associative ``thimbles'' for this fibration, connecting with work of Donaldson~\cite{Donaldson-KovalevLefschetz}.

\textbf{Links to related work.}   As mentioned above, the coassociative submanifolds which appear in these fibrations are \emph{cones}, \emph{asymptotically conical}, or \emph{conically singular}, and coassociative 4-folds of this type have been studied in detail by the second author in~\cites{LotayCS,LotayDef,LotayStab}.  In particular, one can link the moduli space theory from~\cites{LotayCS,LotayDef,LotayStab} to the fibrations constructed in this paper. Most of the coassociative submanifolds considered in this paper are topologically either the total spaces of 2-plane bundles over a compact surface, or the same with the zero section removed. As such, they are  examples of \emph{$2$-ruled} coassociative submanifolds, which were extensively studied by the second author in~\cite{Lotay2Ruled} and~\cite{LotaySymm}, and 2-ruled coassociative cones were studied through an alternative perspective by Fox in~\cite{Fox}. Calibrated submanifolds in the Bryant--Salamon manifolds that are vector bundles over a surface in the base of the $\GG_2$ manifold were studied by the first author and Min-Oo in~\cite{KarigiannisMinOo}, and were later generalized by the first author with Leung in~\cite{KarigiannisNat}.  The examples constructed in~\cites{KarigiannisMinOo,KarigiannisNat} in the Bryant--Salamon $\GG_2$ manifolds arise as special fibres in the fibrations we construct here.  We also particularly note the study by Kawai~\cite{Kawai} of cohomogeneity one coassociative 4-folds in $\Lambda^2_-(T^*\mathcal{S}^4)$, which includes the examples appearing in the fibration we construct in this Bryant--Salamon $\GG_2$ manifold.

\textbf{Coassociative fibrations.} Here we clarify precisely what is meant by the term ``coassociative fibration'' in this paper.

\begin{dfn} \label{dfn:fibration}
Let $M$ be a $\GG_2$ manifold. We say that $M$ admits a \emph{coassociative fibration} if there is a $3$-dimensional space $\mathcal{B}$ parametrizing a family of (not necessarily disjoint, and possibly singular) coassociative submanifolds $N_b$ for $b \in \mathcal{B}$ of $M$, with the following two properties.
\begin{itemize}
\item The family $\{ N_b : b \in \mathcal{B} \}$ covers $M$ and there is a dense open subset $\mathcal{B}^\circ$ of $\mathcal{B}$ such that $N_b$ is smooth for all $b\in\mathcal{B}^{\circ}$. That is, every point $p \in M$ lies in \emph{at least one} coassociative submanifold $N_b$ in this family, and the generic member of the family is smooth.
\item On a dense open subset $M'$ of $M$, there is a genuine fibration of $M'$ onto a submanifold $\mathcal{B}'$ of $\mathcal{B}$, in the sense that there is a smooth map $\pi : M' \to \mathcal{B}'$ which is a locally trivial fibration, and such that $\pi^{-1} (b) = N_b \subset M$ for each $b \in \mathcal{B}$.
\end{itemize}
The set $\mathcal{B}\setminus\mathcal{B}^{\circ}$ parametrizes the singular fibres in the fibration, and the set $M \setminus M'$ consists of the points in $M$ where two coassociatives in the family $\{ N_b : b \in \mathcal{B} \}$ intersect.
\end{dfn}

In the particular case of the three Bryant--Salamon manifolds studied in this paper, the coassociative fibrations have the following qualitative features.
\begin{itemize}
\item For $M = \SS(\mathcal{S}^3)$ and the cone $M_0 = M \setminus \mathcal{S}^3$, the coassocative fibration is actually a fibre bundle, as all the fibres are diffeomorphic.
\item For $M = \Lambda^2_-(T^* \mathcal{S}^4)$ and the cone $M_0 = M \setminus \mathcal{S}^4$, the coassocative fibration has singular fibres, where $\mathcal{B}^{\circ}$ has codimension 2 and 3 in $\mathcal{B}$, respectively, but it is a ``genuine'' fibration in the sense that the fibres $\pi^{-1} (b)$ are all disjoint for distinct $b \in \mathcal{B}$.
\item For $M = \Lambda^2_-(T^* \C\P^2)$ and the cone $M_0 = M \setminus \C\P^2$, the coassociative fibration has singular fibres, where $\mathcal{B}^{\circ}$ has codimension 1 in $\mathcal{B}$, and there \emph{do exist} intersecting ``fibres''. This case is the reason that we introduce the weaker notion of ``fibration'' in Definition~\ref{dfn:fibration}. We find that $M \setminus M'$ and $M_0 \setminus M_0'$ are of \emph{codimension $4$} in $M$ and $M_0$, respectively. (See the discussion at the end of~\S\ref{subs:summary}.)
\end{itemize}

\textbf{Topology of the smooth fibres.} Here we clarify how we determine the  topology of the $\R^2$ bundles over $\mathcal{S}^2 \cong \C\P^1$ which are ``smooth fibres'' of our coassociative fibrations. Suppose that $N$ is the total space of an $\R^2$-bundle over $\C\P^1$ that arises as a coassociative submanifold of a $\GG_2$ manifold. Then $N$ is orientable, and since it is a bundle over an oriented base, it is an oriented bundle. Therefore, $N$ may be viewed as  a $\C $-bundle over $\C\P^1$, so $N$ is topologically isomorphic to a holomorphic line bundle $\mathcal{O}_{\C\P^1}(k)$ for some $k \in \Z$. (The cases $\pm k$ are of course isomorphic as topological vector bundles.)

It is well-known that $\mathcal{O}_{\C\P^1}(-1)$ is the \emph{tautological line bundle} and $\mathcal{O}_{\C\P^1}(-2) = T^* \C\P^1 = T^* \mathcal{S}^2$. For our purposes we can characterize them topologically as follows. For $k>0$, the space $\mathcal{O}_{\C\P^1}(-k)$ minus the zero section is   diffeomorphic to
\begin{equation*}
\mathcal{O}_{\C\P^1}(-k) \setminus \C\P^1 \cong \C^2 / \Z_k \cong \R^+ \times (\mathcal{S}^3 / \Z_k).
\end{equation*}
Recall also that $\mathcal{S}^3/\Z_2 \cong \SO(3) \cong \R\P^3$. Let $B$ denote the zero section of $N$, which is diffeomorphic to $\mathcal{S}^2 \cong \C\P^1$. Thus $N$ is topologically $T^* \mathcal{S}^2$ if and only if $N \setminus B \cong \R^+ \times \SO(3) \cong \R^+ \times \R\P^3$, and $N$ is topologically $\mathcal{O}_{\C\P^1}(-1)$ if and only if $N \setminus B \cong \R^+ \times \mathcal{S}^3$. In this paper we determine the topological type of $N$ in precisely this way by examining the topology of $N \setminus B$.

{\bf Acknowledgements.} The authors would like to thank Bobby Acharya, Robert Bryant, and Simon Salamon for useful discussions and for sharing with them some of their work in progress. Some of the writing of this paper was completed while the first author was a visiting scholar at the Center of Mathematical Sciences and Applications at Harvard University. The first author thanks the CMSA for their hospitality. The authors also thank the anonymous referee for useful suggestions that improved the clarity of our paper, and enabled us to compute the AC rate of convergence in Case 3 of the smooth case ($c>0$) in \S\ref{sec:hypersymplectic.CP2}.

The research of the first author was partially supported by NSERC Discovery Grant RGPIN-2019-03933. The research of the second author was partially supported by the Simons Collaboration on Special Holonomy in Geometry, Analysis, and Physics (\#724071 Jason Lotay).

\section{Preliminaries} \label{sec:prelim}

In this section we review various preliminary results on $\GG_2$ manifolds, Riemannian conifolds, calibrated submanifolds, multimoment maps, coassociative fibrations, and hypersymplectic structures.

\subsection{Overview and definitions}

First we introduce our principal objects of study, and we describe how these objects arise in the context of the Bryant--Salamon $\GG_2$ manifolds.

\subsubsection{\texorpdfstring{G\textsubscript{2}}{G2} manifolds, their calibrated submanifolds, and Riemannian conifolds}

We briefly recall the notion of a $\GG_2$ manifold in a manner which is convenient for our purposes. The local model is $\R^7$ where, if we decompose $\R^7=\R^3\oplus\R^4$ with coordinates $(x_1,x_2,x_3)$ on $\R^3$ and $(y_0,y_1,y_2,y_3)$ on $\R^4$, we define the 3-form
\begin{equation}\label{eq:varphi.R7}
\varphi_{\R^7}=\d x_1\w \d x_2\w \d x_3+\d x_1\wedge \omega_1+\d x_2\wedge \omega_2+\d x_3\wedge\omega_3,
\end{equation}
where
\begin{equation}\label{eq:omegas-R7}
\omega_1=\d y_0\w\d y_1-\d y_2\w\d y_3,\quad
\omega_2=\d y_0\w\d y_2-\d y_3\w\d y_1,\quad
\omega_3=\d y_0\w\d y_3-\d y_1\w\d y_2 
\end{equation}
form an orthogonal basis for the anti-self-dual 2-forms on $\R^4$. The key point is that the stabilizer of $\varphi_{\R^7}$ in $\GL(7,\R)$ is isomorphic to $\GG_2$.

\begin{dfn}\label{dfn:G2.structure}
A smooth 3-form $\varphi$ on a 7-manifold $M$ is a \emph{$\GG_2$-structure} if for all $x\in M$, there exists an isomorphism $\iota_x:\R^7\to T_xM$ so that $\iota_x^*\varphi=\varphi_{\R^7}$ where $\varphi_{\R^7}$ is given in~\eqref{eq:varphi.R7}. A $\GG_2$-structure is also sometimes called a \emph{definite} or a \emph{positive} $3$-form.
\end{dfn}

A $\GG_2$-structure $\varphi$ defines a metric $g_{\varphi}$ and orientation on $M$, given by a volume form $\vol_{\varphi}$, and thus a Hodge star operator $*_\varphi$ on $M$. In fact~\cite[Section 4.2]{K-intro} the metric $g_{\varphi}$ and the volume form $\vol_{\varphi}$ induced by the $\GG_2$-structure $\varphi$ can be extracted from the fundamental relation
\begin{equation} \label{eq:metric-from-form}
-6 g_{\varphi} (X, Y) \vol_{\varphi} = (X \lrcorner \varphi) \w (Y \lrcorner \varphi) \w \varphi.
\end{equation}
Moreover, we have that
$$\vol_{\varphi}=\textstyle\frac{1}{7}\varphi\w *_{\varphi}\varphi.$$
In $\R^7$, one sees explicitly that $\varphi_{\R^7}$ induces the flat metric on $\R^7$ and the standard volume form, and that the Hodge dual of $\varphi_{\R^7}$ is:
\begin{equation*}
*_{\varphi_{\R^7}}\varphi_{\R^7}=\d y_0\w\d y_1\w \d y_2\w \d y_3-\d x_2\w\d x_3\w\omega_1-\d x_3\w\d x_1\w\omega_2-\d x_1\w\d x_2\w\omega_3.
\end{equation*}

\begin{dfn}\label{dfn:G2.manifold}
We say that $(M^7,\varphi)$ is a \emph{$\GG_2$ manifold} if $\varphi$ is a $\GG_2$-structure on $M$ which is \emph{torsion-free}, which means that
\begin{equation*}
\d\varphi=0\quad\text{and}\quad \d*_{\varphi}\varphi=0.
\end{equation*}
It follows that the metric $g_{\varphi}$ has holonomy contained in $\GG_2$.
\end{dfn}

The Bryant--Salamon construction yields a 3-parameter family $\varphi_{c_0,c_1,\kappa}$ of torsion-free $\GG_2$-structures for $c_0,c_1,\kappa>0$ inducing complete holonomy $\GG_2$ metrics on certain total spaces of vector bundles $M^7$ over certain Riemannian manifolds $N$. However, Bryant--Salamon show that two of these parameters can be removed by rescaling and reparametrisation, which means that we can set $\kappa=1$, say, and obtain a true 1-parameter family $\varphi_c$ (for $c>0$) of torsion-free $\GG_2$-structures on $M$ inducing complete holonomy $\GG_2$ metrics $g_c$. We can say more about the properties of $g_c$, for which we need a definition.

\begin{dfn}\label{dfn:AC}
A Riemannian manifold $(M,g)$ is \emph{asymptotically conical} (with rate $\lambda<0$) if there exists a Riemannian cone $(M_0=\R^+\times \Sigma, g_0=\d r^2+r^2g_{\Sigma})$, where $r$ is the coordinate on $\R^+$ and $g_{\Sigma}$ is a Riemannian metric on $\Sigma$, and a diffeomorphism $\Psi:(R,\infty)\times\Sigma\to M\setminus K$, for some $R>0$ and compact subset $K\subseteq M$, such that
\[
|\nabla^j(\Psi^*g-g_0)|=O(r^{\lambda-j})\quad
\text{as $r\to\infty$ for all $j\in\mathbb{N}$.}
\]
We say that $(M_0,g_0)$ is the \emph{asymptotic cone} of $(M,g)$ at infinity (or asymptotic cones at infinity if $M_0$ has multiple components).
\end{dfn}

The Bryant--Salamon metrics on the vector bundle $M$ over the base $N$ are asymptotically conical with rate $\lambda=-\dim N$ and asymptotic cone $M_0=M\setminus N$, where we view $N$ as the zero section. Moreover, the conical metric $g_0$ on $M_0$ has holonomy $\GG_2$ and is induced by a conical torsion-free $\GG_2$-structure $\varphi_0$, which is simply the limit of the $\varphi_c$ on $M_0$ as $c\to 0$ (and the same is true for $g_0$ and $g_c$). 

In this paper we are interested in a distinguished class of 4-dimensional submanifolds of the Bryant--Salamon $\GG_2$ manifolds called \emph{coassociative 4-folds}, which we now define.

\begin{dfn}\label{dfn:coassociative}
An oriented 4-dimensional submanifold $N^4$ of a $\GG_2$ manifold $(M^7,\varphi)$ is \emph{coassociative} if $N$ is calibrated by $*\varphi$. This means that
$$*\varphi|_N=\vol_N.$$
Equivalently~\cite{HarveyLawson}, up to a choice of orientation $N^4$ is coassociative if and only if
$$\varphi|_N\equiv 0.$$
This latter definition is often more useful in practice.
\end{dfn}

When we describe our coassociative fibrations we will find that there are \emph{singular fibres}, where the singularities are of a special nature as follows.

\begin{dfn}\label{dfn:CS}
A Riemannian manifold $(N,g)$ is \emph{conically singular} (with rate $\mu>0$) if there exists a Riemannian cone $(N_0=\R^+\times \Sigma, g_0=\d r^2+r^2g_{\Sigma})$, where $r$ is the coordinate on $\R^+$ and $g_{\Sigma}$ is a Riemannian metric on $\Sigma$, and a diffeomorphism $\Psi:(0,\epsilon)\times\Sigma\to N\setminus K$, for some $\epsilon>0$ and compact subset $K\subseteq N$, such that
\[
|\nabla^j(\Psi^*g-g_0)|=O(r^{\mu-j})\quad
\text{as $r\to 0$ for all $j\in\mathbb{N}$.}
\]
We say that $(N_0,g_0)$ is the \emph{asymptotic cone} of $(N,g)$ at the singularity (or asymptotic cones at the singularities if $N_0$ has multiple components).
\end{dfn}

Note that a special case of a conically singular Riemannian manifold is a Riemannian cone itself, which is also a special case of an asymptotically conical Riemannian manifold. Given a conically singular Riemannian manifold $N$, if we add the singularity (the vertex of the asymptotic cone at the singularity) to $N$, then we obtain a topological space which is in general not a smooth manifold. Although $N$ itself, without the singular point, is a smooth manifold, we nevertheless denote such fibres as ``singular fibres'' because their Riemannian metrics typically degenerate as we approach the singular point.

Another important class of submanifolds that we encounter in this paper are the following.

\begin{dfn}
An oriented 3-dimensional submanifold $L$ of a $\GG_2$ manifold $(M,\varphi)$ is \emph{associative} if $L$ is calibrated by $\varphi$. This means that
$$\varphi|_L=\vol_L.$$
Note that the orthogonal complement of an associative 3-plane is a coassociative 4-plane, and vice versa.
\end{dfn}

\subsection{Multimoment maps} \label{sub:multimoment.intro}

Suppose we have a Lie group $G$ acting on a $\GG_2$ manifold $M$ preserving $\varphi$ and $*_{\varphi} \varphi$, which are both closed forms, since $\varphi$ is torsion-free. We can consider the \emph{multi-moment maps} developed by Madsen--Swann~\cites{MadsenSwann1,MadsenSwann2}, which are generalizations of the well-known theory of moment maps in symplectic geometry. In the case when $G$ is $3$-dimensional, the \emph{na\"ive definition} for the multimoment map would be as follows. Let $X_1,X_2,X_3$ denote generators for the action of $G$ on a $\GG_2$ manifold $(M,\varphi)$ preserving $\varphi$ (and thus also preserving $g_{\varphi}$ and $*_{\varphi} \varphi$). The multi-moment maps $\mu=(\mu_1,\mu_2,\mu_3):M\to\R^3$ and $\nu:M\to\R$ for the action on $\varphi$ and $*_{\varphi}\varphi$, respectively, should be defined (up to additive constants) by:
\begin{equation} \label{eq:mmm3}
\varphi(X_2,X_3,\cdot)=\d\mu_1,\quad \varphi(X_3,X_1,\cdot)=\d\mu_2,\quad \varphi(X_1,X_2,\cdot)=\d\mu_3,
\end{equation}
\begin{equation}
*_{\varphi}\varphi(X_1,X_2,X_3,\cdot)=\d\nu, \label{eq:mmm4}
\end{equation}
if they exist.

The above na\"ive definition works fine if $G$ is abelian. However, one can check using the Cartan formula for the Lie derivative, and the nondegeneracy of $\varphi$, that in the case where $G = \SO(3)$ or $\SU(2)$, the multimoment map $\mu = (\mu_1, \mu_2, \mu_3)$ for the $3$-form $\varphi$ in~\eqref{eq:mmm3} \emph{never exists}. It is precisely for this reason that Madsen--Swann introduce~\cite[Section 3.2]{MadsenSwann2} the \emph{Lie kernel} to give the correct definition of multimoment map in~\cite[Definition 3.5]{MadsenSwann2}. We do not state this correct definition here as we do not use it.

Unfortunately, however, when $G = \SO(3)$ or $\SU(2)$, when using the correct Madsen--Swann definition, the multimoment map $\mu = (\mu_1, \mu_2, \mu_3)$ for the $3$-form $\varphi$ is always \emph{necessarily trivial}. (See~\cite[Example 3.3]{MadsenSwann2}, which is for $\SU(2)$ actions, but the same reasoning works in the $\SO(3)$ case). However, in this case the multi-moment map $\nu$ for the 4-form $*_{\varphi} \varphi$ in the correct definition agrees with the na\"ive definition~\eqref{eq:mmm4}, and furthermore automatically exists by~\cite[Page 33]{MadsenSwann1} as our manifolds are simply-connected.

For each of the three situations we consider in this paper, we explicitly compute the multimoment map $\nu$ for the $4$-form $*_{\varphi} \varphi$ as in~\eqref{eq:mmm4}. 

\begin{remark} \label{rmk:mmm}
We also explicitly compute the left hand sides $\varphi(X_i, X_j, \cdot)$ of~\eqref{eq:mmm3} and verify that they are never closed, as must be true. We do this for two reasons. First, the forms $\varphi(X_i, X_j, \cdot)$ are needed to understand the coassociative fibration structure, via Lemma~\ref{lemma:adapted-frame-2} below. Second, the precise forms of $\varphi(X_i, X_j, \cdot)$ could offer clues on how one might be able to meaningfully extend the notion of multi-moment map in the Madsen--Swann theory to this particular non-abelian context. For example, in the $\SS(\mathcal{S}^3)$ case, we relate the left hand sides of~\eqref{eq:mmm3} to the hyperK\"ahler moment map in Remark~\ref{rmk:mmmS3}. See also Remark~\ref{rmk:mmmS4}.
\end{remark}

\subsection{Coassociative fibrations of \texorpdfstring{G\textsubscript{2}}{G2} manifolds}

In this section we describe how to construct a local coframe on a $\GG_2$ manifold with a coassociative fibration, which is adapted to the $\GG_2$-structure and the fibration, to facilitate calculations later.

Suppose that a $\GG_2$ manifold $(M, \varphi, g_{\varphi}, \ast_{\varphi} \varphi)$ can be described as a fibration by coassociative submanifolds. Then its tangent bundle $TM$ admits a \emph{vertical subbundle} $V$, which is the bundle of (coassocative) tangent subspaces of the coassociative fibres of $M$. Since coassociative subspaces come equipped with a preferred orientation, the bundle $V$ is oriented. A local \emph{vertical vector field} is a local section of $V$, and hence is everywhere tangent to the coassociative fibres.

Let $\alpha^{\sharp}$ be the metric dual vector field to the 1-form $\alpha$ with respect to the metric $g_{\varphi}$. Recall~\cite[\S3.4]{K-intro} that $\varphi(X, Y, Z) = g_{\varphi} (X \times Y, Z)$, where $\times$ is the \emph{cross product} induced by $\varphi$.

\begin{lem} \label{lemma:adapted-frame}
Let $\{ h_2, h_3, \varpi_0 \}$ be local \emph{orthogonal} 1-forms on $M$ satisfying the following conditions:
\begin{itemize}
\item The $h_k$ are \emph{horizontal}. That is, $h_2 (X) = h_3 (X) = 0$ for any vertical vector field $X$.
\item We have $\varphi( h_2^{\sharp}, h_3^{\sharp}, \varpi_0^{\sharp} ) = 0$. That is, $h_2^{\sharp} \times h_3^{\sharp}$ is orthogonal to $\varpi_0^{\sharp}$.
\end{itemize}
Using the metric $g_{\varphi}$ we have the orthonormal 1-forms $\{ \hat h_2, \hat h_3, \hat \varpi_0 \}$, where $\hat \alpha = \frac{1}{|\alpha|} \alpha$.

We can complete this to a local oriented orthonormal coframe
\begin{equation*}
\{ \hat h_1, \hat h_2, \hat h_3, \hat \varpi_0,  \hat \varpi_1,  \hat \varpi_2,  \hat \varpi_3 \}
\end{equation*}
which means that
\begin{equation} \label{eq:canonical-metric}
\begin{aligned}
g_{\varphi} & = (\hat h_1)^2 + (\hat h_2)^2 + (\hat h_3)^2 + (\hat \varpi_0)^2 + (\hat \varpi_1)^2 + (\hat \varpi_2)^2 + (\hat \varpi_3)^2, \\
\vol_{\varphi} & = \hat h_1 \w \hat h_2 \w \hat h_3 \w \hat \varpi_0 \w \hat \varpi_1 \w \hat \varpi_2 \w \hat \varpi_3,
\end{aligned}
\end{equation}
in such a way that, with respect to this coframe, the 3-form $\varphi$ and the 4-form $\ast_{\varphi} \varphi$ are given by
\begin{equation} \label{eq:canonical-3form}
\begin{aligned}
\varphi & = \hat h_1 \w \hat h_2 \w \hat h_3 + \hat h_1 \w (\hat \varpi_0 \w \hat \varpi_1 - \hat \varpi_2 \w \hat \varpi_3) \\
& \qquad {} + \hat h_2 \w (\hat \varpi_0 \w \hat \varpi_2 - \hat \varpi_3 \w \hat \varpi_1) + \hat h_3 \w (\hat \varpi_0 \w \hat \varpi_3 - \hat \varpi_1 \w \hat \varpi_2), 
\end{aligned}
\end{equation}
and
\begin{equation} \label{eq:canonical-4form}
\begin{aligned}
\ast_{\varphi} \varphi & = \hat \varpi_0 \w \hat \varpi_1 \w \hat \varpi_2 \w \hat \varpi_3 - \hat h_2 \w \hat h_3 \w (\hat \varpi_0 \w \hat \varpi_1 - \hat \varpi_2 \w \hat \varpi_3) \\
& \qquad {} - \hat h_3 \w \hat h_1 \w (\hat \varpi_0 \w \hat \varpi_2 - \hat \varpi_3 \w \hat \varpi_1) - \hat h_1 \w \hat h_2 \w (\hat \varpi_0 \w \hat \varpi_3 - \hat \varpi_1 \w \hat \varpi_2).
\end{aligned}
\end{equation}
\end{lem}
\begin{proof}
Define
\begin{equation*}
\hat h_1 = \hat h_3^{\sharp} \lrcorner ( \hat h_2^{\sharp} \lrcorner \varphi), \qquad \hat \varpi_k = \hat \varpi_0^{\sharp} \lrcorner ( \hat h_k^{\sharp} \lrcorner \varphi) \quad \text{for $k=1, 2, 3$.}
\end{equation*}
It then follows from the description of the standard $\GG_2$ package on $\R^7$ that $\varphi$ and $\ast_{\varphi} \varphi$ are given by~\eqref{eq:canonical-3form} and~\eqref{eq:canonical-4form}, respectively. (See~\cite{K-intro}, for example.)
\end{proof}

We say that such a coframe is $\GG_2$-adapted, because in this oriented orthonormal coframe the $3$-form and $4$-form agree with the standard versions in the Euclidean $\R^7$. Moreover, such a coframe is also adapted to the coassociative fibration structure, because the vector fields $\{ \varpi_0^{\sharp}, \varpi_1^{\sharp}, \varpi_2^{\sharp}, \varpi_3^{\sharp} \}$ are a local oriented orthonormal frame for the vertical subbundle $V$ of $TM$, at every point, and the vector fields $\{ \hat h_1^{\sharp}, \hat h_2^{\sharp}, \hat h_3^{\sharp} \}$ are an oriented orthonormal frame for the orthogonal complement $V^{\perp}$, which is a bundle of associative subspaces. In general the distribution corresponding to $V^{\perp}$ is \emph{not integrable}.

\begin{remark}
Note that in this entire discussion, we have \emph{not} chosen, nor did we need to choose, an (Ehresmann) connection on the fibre bundle $M$, which would be a choice of complement $H$ to the vertical subbundle $V$ of $TM$. A canonical choice would be to take $H = V^{\perp}$, but we do not need to consider a connection on the total space $M$ in this paper. This should be compared to the work of Donaldson~\cite{Donaldson-KovalevLefschetz} in which a choice of connection is an integral part of the data used to encode the $\GG_2$-structure admitting a coassociative fibration.
\end{remark}

\begin{lem} \label{lemma:adapted-frame-2}
Let $X_1, X_2, X_3$ be local linearly independent \emph{vertical vector fields} on $M$. Define local 1-forms $\tilde h_1, \tilde h_2, \tilde h_3, \tilde \varpi$ on $M$ by
\begin{gather*}
\tilde h_1 = \varphi(X_2,X_3,\cdot), \quad \tilde h_2 = \varphi(X_3,X_1,\cdot), \quad \tilde h_3 = \varphi(X_1,X_2,\cdot),\\
\tilde \varpi_0 = \ast_{\varphi} \varphi (X_1, X_2, X_3, \cdot).
\end{gather*}
Then $\{ \tilde h_1, \tilde h_2, \tilde h_3 \}$ are local linearly independent horizontal 1-forms, although not necessarily orthogonal. Moreover, $\tilde \varpi_0$ is non-vanishing and is orthogonal to $\tilde h_k$ for $k=1, \ldots, 3$.
\end{lem}
\begin{proof}
We observe that $\tilde h_1 = \varphi(X_2, X_3, \cdot) = (X_2 \times X_3)^{\sharp}$ and similarly for cyclic permutations of $1,2,3$. Thus it is enough to show that the set $\{ X_2 \times X_3, X_3 \times X_1, X_1 \times X_2 \}$ is linearly independent. From the fact~\cite[equation (3.71)]{K-intro} that
\begin{equation} \label{eq:cross.prod.g}
\langle X_i \times X_j, X_i \times X_k \rangle = |X_i|^2 \langle X_j, X_k \rangle - \langle X_i, X_j \rangle \langle X_i, X_k \rangle,
\end{equation}
it is straightforward to show that the linear independence of $\{ X_1, X_2, X_3 \}$ implies the linear independence of $\{ X_2 \times X_3, X_3 \times X_1, X_1 \times X_2 \}$. By completing to a local oriented frame $X_0, X_1, X_2, X_3$ for $V$, and using the fact that $\ast_{\varphi} \varphi$ calibrates the fibres of $V$, we deduce that $\tilde \varpi_0$ is non-vanishing.

Let $\{ \hat h_1, \hat h_2, \hat h_3, \hat \varpi_0,  \hat \varpi_1,  \hat \varpi_2,  \hat \varpi_3 \}$ be a $\GG_2$ adapted coframe as in Lemma~\ref{lemma:adapted-frame}. From equations~\eqref{eq:canonical-3form} and~\eqref{eq:canonical-4form} and the fact that the $X_i$ are vertical vector fields, it follows that the $\tilde h_i$ are in the span of $\{ \hat h_1, \hat h_2, \hat h_3 \}$ and that $\tilde \varpi$ is in the span of $\{ \hat \varpi_0, \hat \varpi_1, \hat \varpi_2, \hat \varpi_3\}$. Thus $\tilde \varpi$ is orthogonal to the $\tilde h_k$.
\end{proof}

\begin{remark}
Note from~\eqref{eq:cross.prod.g} that if $\{ X_1, X_2, X_3 \}$ are orthogonal then $\{ \tilde h_1, \tilde h_2, \tilde h_3 \}$ are also orthogonal. Also note that $\{ \tilde h_1, \tilde h_2, \tilde h_3, \tilde \varpi_0 \}$ are the exterior derivatives of the left hand sides of~\eqref{eq:mmm3} and~\eqref{eq:mmm4}.
\end{remark}

The results of this section are used later in the paper as follows. The coassociative fibres of $M$ are given as 1-parameter families of orbits of a group action by a 3-dimensional group $G$ which is either $\SO(3)$ or $\SU(2)$, where the generic orbits are 3-dimensional. Let $\{ X_1, X_2, X_3 \}$ be local vector fields on $M$ generating the group action. Since they are everywhere tangent to the orbits, they are vertical vector fields. From Lemma~\ref{lemma:adapted-frame-2} we construct the local linearly independent horizontal 1-forms $\{ \tilde h_1, \tilde h_2, \tilde h_3 \}$. In fact we only need two of these horizontal 1-forms. Using the metric $g_{\varphi}$ we can thus obtain an orthonormal set $\{ \hat h_1, \hat h_2 \}$ of horizontal 1-forms. Lemma~\ref{lemma:adapted-frame-2} also gives us a nonvanishing 1-form $\tilde \varpi$ that is orthogonal to the $\hat h_k$, and by using the metric $g_{\varphi}$ we get a unit length 1-form $\hat \varpi_0$ that is orthogonal to the $\hat h_k$. Then applying Lemma~\ref{lemma:adapted-frame} allows us to construct $\{ \hat h_3, \hat \varpi_1, \hat \varpi_2, \hat \varpi_3 \}$ so that the metric $g_{\varphi}$ and volume form $\vol_{\varphi}$ are given by~\eqref{eq:canonical-metric} and the 3-form $\varphi$ and dual 4-form $\ast_{\varphi} \varphi$ are given by~\eqref{eq:canonical-3form} and~\eqref{eq:canonical-4form}, respectively.

\subsection{Hypersymplectic geometry} \label{sub:hypersymplectic.intro}

The induced geometric structure on each of the coassociative fibres for the three Bryant--Salamon manifolds is a \emph{hypersymplectic} structure. This structure was first introduced by Donaldson~\cite{Donaldson}. Another good reference is~\cite{Fine-Yao}.

\begin{dfn}\label{dfn:hypersymplectic}
A \emph{hypersymplectic triple} on an orientable 4-manifold $N$ is a triple $(\omega_1,\omega_2,\omega_3)$ of 2-forms  on $N$ such that
$$\d\omega_i=0\;\text{for $i=1,2,3$}\quad\text{and}\quad \omega_i\wedge\omega_j=2 Q_{ij} \nu,$$
where $\nu$ is an arbitrary volume form and $(Q_{ij})$ is a positive definite symmetric matrix of functions. Such a hypersymplectic triple is a \emph{hyperk\"ahler triple} if and only if we can choose $\nu$ such that $Q_{ij}=\delta_{ij}$, up to constant rotations and scalings of $(\omega_1,\omega_2,\omega_3)$.
\end{dfn}

A torsion-free $\GG_2$ structure on $M$ induces a hypersymplectic triple on a coassociative submanifold using the following elementary result.

\begin{lem}\label{lem:hypersymplectic}
Let $(M,\varphi)$ be a $\GG_2$ manifold, let $N$ be a coassociative 4-fold in $(M,\varphi)$ and suppose that $n_1,n_2,n_3$ are a linearly independent triple of normal vector fields on $N$ such that $(\mathcal{L}_{n_k}\varphi)|_N=0$ for $k=1,2,3$. (Here we extend $n_1, n_2, n_3$ to vector fields on an open neighbourhood of $N$ in $M$ in order to compute $\mathcal{L}_{n_k} \varphi$, but the restriction $(\mathcal{L}_{n_k}\varphi)|_N$ is independent of this extension.) Then the triple $(\omega_1,\omega_2,\omega_3)$ given by $\omega_k=(n_k\lrcorner\varphi)|_N$ is a hypersymplectic triple on $N$.
\end{lem}
\begin{proof}
Since $\d$ commutes with restriction to $N$ and $\d \varphi = 0$, the hypotheses imply that $\omega_1, \omega_2, \omega_3$ are closed 2-forms on $N$. Thus we need only establish that given any volume form $\nu$ on $N$, the $3 \times 3$ symmetric matrix $Q$ defined by $\omega_i \wedge \omega_j = 2 Q_{ij} \nu$ is positive definite.

We use the notation of Lemma~\ref{lemma:adapted-frame}, which remains valid on the single coassociative submanifold $N$ of $M$. Let $\tilde{n}_1, \tilde{n}_2, \tilde{n}_2$ be obtained from $n_1, n_2, n_3$ by orthonormalization, where we can assume without loss of generality that they are also oriented, since the normal bundle of a coassociative submanifold $N$ in $M$ comes equipped with a preferred orientation. In terms of the canonical expression~\eqref{eq:canonical-3form} for $\varphi$ this means that $\{ \tilde{n}_1, \tilde{n}_2, \tilde{n}_3 \}$ is dual to the horizontal coframe $\{ \hat h_1, \hat h_2, \hat h_3 \}$. Then if we define $\tilde \omega_k = \tilde{n}_k \lrcorner \varphi$, it is clear that $\tilde \omega_i \w \tilde \omega_j = 2 \delta_{ij} \nu$ where $\nu = \hat{\varpi}_0 \w \hat{\varpi}_1 \w \hat{\varpi}_2 \w \hat{\varpi}_3$.

Using summation convention, there exists an invertible $3 \times 3$ matrix $A$ such that $n_i = A_{ik} \tilde{n}_k$. Thus we have $\omega_i = A_{ik} \tilde \omega_k$ and hence
\begin{equation*}
\omega_i \w \omega_j = A_{ik} A_{jl} \tilde \omega_k \w \tilde \omega_l = 2 A_{ik} A_{jk} \nu.
\end{equation*}
Thus $Q_{ij} = (AA^T)_{ij}$ which is clearly positive definite as $A$ is invertible.
\end{proof}

\begin{remark}
For Lemma~\ref{lem:hypersymplectic} to hold, it suffices for $\varphi$ to be a closed $\GG_2$ structure.
\end{remark}

We emphasize that the hypersymplectic triple on $N$ induced by $\varphi$ \emph{depends on the choice} $n_1, n_2, n_3$ of linearly independent normal vector fields satisfying the conditions $(\mathcal{L}_{n_k}\varphi)|_N=0$.

For the Bryant--Salamon manifold $\SS(\mathcal{S}^3)$ we apply Lemma~\ref{lem:hypersymplectic} directly in~\S\ref{sec:hypersymplectic.S3} to determine the induced hypersymplectic structure on the fibres. However, for the Bryant--Salamon manifold $\Lambda^2_- (T^* \mathcal{S}^4)$, the way we use Lemma~\ref{lem:hypersymplectic} in~\S\ref{sec:hypersymplectic.S4} is as follows. Let $\{ \hat h_1, \hat h_2, \hat h_3, \hat \varpi_0, \hat \varpi_1, \hat \varpi_2, \hat \varpi_3 \}$ be as in Lemma~\ref{lemma:adapted-frame}. Suppose that each $\hat h_k$ is \emph{conformally closed} in the sense that $\hat h_k$ can be written as $\hat h_k = |h_k|^{-1} h_k$ where $\d h_k = 0$. Then we can write~\eqref{eq:canonical-3form} as
\begin{equation} \label{eq:canonical-3form-variant}
\varphi = F h_1 \w h_2 \w h_3 + h_1 \w \beta_1 + h_2 \w \beta_2 + h_3 \w \beta_3,
\end{equation} 
where $F = (|h_1| \, |h_2| \, |h_3|)^{-1}$ and $\beta_k = |h_k|^{-1} (\hat \varpi_0 \w \hat \varpi_k - \hat \varpi_i \w \hat \varpi_j)$ for $i,j,k$ a cyclic permutation of $1,2,3$. Let $n_1, n_2, n_3$ be the frame of normal vector fields that is dual to the frame $h_1, h_2, h_3$ of horizontal 1-forms. That is, $h_i(n_j) = \delta_{ij}$. Then we have $(n_k \lrcorner \varphi)|_N = \beta_k|_N$ and
\begin{equation*}
(\mathcal{L}_{n_k} \varphi)|_N = \d (n_k \lrcorner \varphi)|_N = \d ( \beta_k |_N).
\end{equation*}
Using the fact that $\d \varphi = 0$ and $\d h_k = 0$, we have
\begin{equation*}
\d F \w h_1 \w h_2 \w h_3 - h_1 \w \d \beta_1 - h_2 \w \d \beta_2 - h_3 \w \d \beta_3 = 0.
\end{equation*}
Taking the interior product of the above with $n_k$ and restricting to $N$ gives $(\d \beta_k)|_N = \d (\beta_k|_N) = 0$. We summarize the above discussion as follows.

\begin{cor} \label{cor:hypersymplectic}
Suppose $\varphi$ is a closed $\GG_2$ structure on $M$ and $N$ is a coassociative submanifold of $M$, such that on $N$ we can write $\varphi$ in the form~\eqref{eq:canonical-3form-variant} where the $1$-forms $h_1, h_2, h_3$ are closed on $M$ and satisfy $h_k|_N = 0$. Then the triple $(\omega_1, \omega_2, \omega_3)$ defined by $\omega_k = \beta_k|_N$ is a hypersymplectic triple on $N$.
\end{cor}

\begin{remark}
A hypersymplectic structure $( \omega_1, \omega_2, \omega_3)$ on an orientable 4-manifold $N$ naturally determines a Riemannian metric on $N$. (See~\cites{Donaldson, Fine-Yao} for details.) However, when the hypersymplectic structure is induced as in Corollary~\ref{cor:hypersymplectic} on a coassociative submanifold $N$ of a $\GG_2$ manifold $(M, \varphi, g_{\varphi})$, then the hypersymplectic metric is in general \emph{not} the same as the induced Riemannian metric $g_{\varphi}|_N$ on $N$ obtained by restriction from $g_{\varphi}$ on $M$. In fact this difference is observed in all the cases considered in this paper.
\end{remark}

\section{Bryant--Salamon \texorpdfstring{G\textsubscript{2}}{G2} manifolds} \label{sec:BS-manifolds}

In this section we recall the construction of the Bryant--Salamon $\GG_2$ manifolds from~\cite{BryantSalamon}, as these are the central objects of our study, and it allows us to introduce some key notation.

\subsection{The round 3-sphere}

Let $\mathcal{S}^3$ be the 3-sphere endowed with a Riemannian metric with constant curvature $\kappa>0$. Since $\mathcal{S}^3$ is spin (as every oriented 3-manifold is spin), we may consider the spinor bundle $M^7=\mathbb{S}(\mathcal{S}^3)$, which is topologically trivial (that is, $M$ is diffeomorphic to the product $\R^4 \times \mathcal{S}^3$). It is on this bundle $M$ that Bryant--Salamon construct torsion-free $\GG_2$-structures, as well as on the cone $M_0=\R^+\times \mathcal{S}^3\times\mathcal{S}^3$, as we briefly explain in this section.

\begin{remark}
This construction also works for 3-dimensional space forms, but if the curvature $\kappa\leq 0$ we do not obtain smooth complete holonomy $\GG_2$ metrics, and for $\kappa>0$ we can reduce to the case of $\mathcal{S}^3$ by taking a finite cover.
\end{remark}

\subsubsection{Coframe, spin connection, and vertical 1-forms}

Let $\sigma_1,\sigma_2,\sigma_3$ define the standard left-invariant coframe on $\mathcal{S}^3\cong\SU(2)$ with
\begin{equation}\label{eq:d-sigmas-S3}
\d\left(\begin{array}{c} \sigma_1\\ \sigma_2\\ \sigma_3\end{array}\right)=\left(\begin{array}{c} \sigma_2\w\sigma_3\\ \sigma_3\w\sigma_1\\ \sigma_1\w\sigma_2\end{array}\right).
\end{equation}
We can rewrite~\eqref{eq:d-sigmas-S3} as 
\begin{equation}\label{eq:spin-connection}
\d\left(\begin{array}{c} \sigma_1\\ \sigma_2\\ \sigma_3\end{array}\right)
=-\left(\begin{array}{ccc} 0 & -2\rho_3 & 2\rho_2\\ 2\rho_3 & 0 & -2\rho_1 \\
-2\rho_2 & 2\rho_1 & 0\end{array} \right)\w \left(\begin{array}{c} \sigma_1\\ \sigma_2\\ \sigma_3\end{array}\right)
\end{equation}
where 
\begin{equation}\label{eq:rhos-S3}
\rho_j=-\frac{1}{4}\sigma_j\quad\text{for $j=1,2,3$}
\end{equation}
describes the spin connection on $\mathbb{S}(\mathcal{S}^3)$, which explains the factor of 2 in~\eqref{eq:spin-connection}. We see that
\begin{equation}\label{eq:d.rho-S3.1}
\d\left(\begin{array}{c} \rho_1\\ \rho_2\\ \rho_3\end{array}\right)+2\left(\begin{array}{c} \rho_2\wedge\rho_3\\ \rho_3\w\rho_1\\ \rho_1\w\rho_2\end{array}\right)=-\frac{1}{8}\left(\begin{array}{c} \sigma_2\w\sigma_3\\ \sigma_3\w\sigma_1\\ \sigma_1\w\sigma_2\end{array}\right).
\end{equation} 
Consider the metric
$$g_{\mathcal{S}^3 }=\frac{1}{4\kappa}(\sigma_1^2+\sigma_2^2+\sigma_3^2)$$
on $\mathcal{S}^3$. Then~\eqref{eq:d.rho-S3.1} is equivalent to the statement that this metric
has constant sectional curvature $\kappa$. We therefore set
\begin{equation}\label{eq:b-S3}
b_j=\frac{1}{2\sqrt{\kappa}}\sigma_j=-\frac{2}{\sqrt{\kappa}}\rho_j\quad\text{for $j=1,2,3$}
\end{equation}
to obtain an orthonormal coframe on $\mathcal{S}^3$, so that using~\eqref{eq:d-sigmas-S3} and~\eqref{eq:rhos-S3} we have
\begin{equation}\label{eq:d-b-S3}
\d\left(\begin{array}{c} b_1\\ b_2\\ b_3\end{array}\right)=2\left(\begin{array}{c}
b_3\wedge\rho_2-b_2\wedge \rho_3\\ b_1\w\rho_3- b_3\wedge \rho_1\\ b_2\wedge\rho_1- b_1\wedge \rho_2 \end{array}\right)
\end{equation}
and
\begin{equation} \label{eq:g.vol.S3}
g_{\mathcal{S}^3}=b_1^2+b_2^2+b_3^2\quad\text{and}\quad \vol_{\mathcal{S}^3}=b_1\wedge b_2\wedge b_3.
\end{equation}
We can pullback $b_1,b_2,b_3$ to the bundle $\mathbb{S}(\mathcal{S}^3)$ to be horizontal 1-forms using the natural projection $\pi:\mathbb{S}(\mathcal{S}^3)\to\mathcal{S}^3$. (We omit the pullback notation here since the bundle is trivial.) Using~\eqref{eq:b-S3} we can rewrite~\eqref{eq:d.rho-S3.1} as
\begin{equation}\label{eq:d.rho-S3.2}
\d\left(\begin{array}{c} \rho_1\\ \rho_2\\ \rho_3\end{array}\right)+2\left(\begin{array}{c} \rho_2\wedge\rho_3\\ \rho_3\w\rho_1\\ \rho_1\w\rho_2\end{array}\right)=-\frac{\kappa}{2}\left(\begin{array}{c} b_2\w b_3\\ b_3\w b_1\\ b_1\w b_2\end{array}\right).
\end{equation} 

We now let $(a_0,a_1,a_2,a_3)$ be linear coordinates on the $\R^4$ fibres of $\mathbb{S}(\mathcal{S}^3)$ and define the vertical 1-forms on the spinor bundle by
\begin{equation}\label{eq:zetas-S3}
\begin{gathered}
\zeta_0=\d a_0+a_1\rho_1+a_2\rho_2+a_3\rho_3,\quad \zeta_1=\d a_1-a_0\rho_1+a_3\rho_2-a_2\rho_3,\\
\zeta_2=\d a_2-a_3\rho_1-a_0\rho_2+a_1\rho_3,\quad \zeta_3=\d a_3+a_2\rho_1-a_1\rho_2-a_0\rho_3.
\end{gathered}
\end{equation}
For use below, one may compute using~\eqref{eq:d.rho-S3.2} and~\eqref{eq:zetas-S3} that
\begin{equation}\label{eq:d-zetas-S3}
\begin{aligned}
\d\zeta_0&=\quad\!\!\zeta_1\wedge\rho_1+\zeta_2\wedge\rho_2+\zeta_3\wedge\rho_3- \frac{\kappa}{2}(a_1 b_2\wedge b_3+a_2 b_3\wedge b_1+a_3 b_1\wedge b_2),\\
\d\zeta_1&=-\zeta_0\wedge\rho_1+\zeta_3\wedge\rho_2-\zeta_2\wedge\rho_3+ \frac{\kappa}{2}(a_0 b_2\wedge b_3-a_3 b_3\wedge b_1+a_2 b_1\wedge b_2),\\
\d\zeta_2&=-\zeta_3\wedge\rho_1-\zeta_0\wedge\rho_2+\zeta_1\wedge\rho_3+ \frac{\kappa}{2}(a_3 b_2\wedge b_3+a_0 b_3\wedge b_1-a_1 b_1\wedge b_2),\\
\d\zeta_3&=\quad\!\zeta_2\wedge\rho_1-\zeta_1\wedge\rho_2-\zeta_0\wedge\rho_3+ \frac{\kappa}{2}(-a_2 b_2\wedge b_3+a_1 b_3\wedge b_1+a_0 b_1\wedge b_2).
\end{aligned}
\end{equation}

\subsubsection{Torsion-free \texorpdfstring{G\textsubscript{2}}{G2}-structures}

Let $r^2=a_0^2+a_1^2+a_2^2+a_3^2$ be the squared distance from the zero section in the fibres of $M$. We define vertical 2-forms in terms of the vertical 1-forms in~\eqref{eq:zetas-S3} by
\begin{equation}\label{eq:Omegas-S3}
\Omega_1=\zeta_0\wedge\zeta_1-\zeta_2\wedge\zeta_3,\quad 
\Omega_2=\zeta_0\wedge\zeta_2-\zeta_3\wedge\zeta_1,\quad 
\Omega_3=\zeta_0\wedge\zeta_3-\zeta_3\wedge\zeta_1.
\end{equation}
Notice that these 2-forms are anti-self-dual on the $\R^4$ fibres with respect to any metric conformal to $\zeta_0^2+\zeta_1^2+\zeta_2^2+\zeta_3^2$ and the orientation $\zeta_0\wedge\zeta_1\wedge\zeta_2\wedge\zeta_3$, which is the vertical volume form on the fibres.

Let $c_0>0$ and $c_1>0$. Bryant--Salamon defined the following $\GG_2$-structure on $M$:
\begin{equation}\label{eq:phi.S3.general}
\varphi_{c_0,c_1,\kappa}= 3\kappa(c_0+c_1 r^2) \vol_{\mathcal{S}^3}+4c_1 (b_1\wedge\Omega_1+b_2\wedge\Omega_2+b_3\wedge\Omega_3).
\end{equation}
The metric induced by $\varphi_{c_0,c_1,\kappa}$ is given by
\begin{equation}\label{eq:g.S3.general}
g_{c_0,c_1,\kappa}= (3\kappa)^{\frac{2}{3}}(c_0+c_1 r^2)^{\frac{2}{3}}g_{\mathcal{S}^3}+4\left(\frac{c_1^3}{3\kappa}\right)^{\frac{1}{3}}(c_0+c_1r^2)^{-\frac{1}{3}}(\zeta_0^2+\zeta_1^2+\zeta_2^2+\zeta_3^2)
\end{equation}
and the Hodge dual of $\varphi_{c_0,c_1,\kappa}$ is given by:
\begin{align}
*_{\varphi_{c_0,c_1,\kappa}}\varphi_{c_0,c_1,\kappa}&= 16\left(\frac{c_1^3}{3\kappa}\right)^{\frac{2}{3}}(c_0+c_1r^2)^{-\frac{2}{3}}\zeta_0\wedge\zeta_1\wedge\zeta_2\wedge\zeta_3\nonumber\\
&\quad-4(3\kappa c_1^3)^{\frac{1}{3}}(c_0+c_1 r^2)^{\frac{1}{3}}(b_2\wedge b_3\wedge\Omega_1+b_3\wedge b_1\wedge\Omega_2+b_1\wedge b_2\wedge\Omega_3).\label{eq:starphi.S3.general}
\end{align}
Finally, the volume form is given by
\begin{equation*}
\vol_{c_0,c_1,\kappa}=16(3\kappa c_1^6)^{\frac{1}{3}}(c_0+c_1r^2)^{\frac{1}{3}}\vol_{\mathcal{S}^3}\wedge\zeta_0\wedge\zeta_1\wedge\zeta_2\wedge\zeta_3.
\end{equation*}

\begin{lem}\label{lem:torsion-free.S3}
The $\GG_2$-structure $\varphi_{c_0,c_1,\kappa}$ in~\eqref{eq:phi.S3.general} is torsion-free.
\end{lem}

\begin{proof}
We first observe from~\eqref{eq:zetas-S3} that
\begin{equation}\label{eq:dr-S3}
r\d r=a_0\d a_0+a_1\d a_1+a_2\d a_2+a_3\d a_3=a_0\zeta_0+a_1\zeta_1+a_2\zeta_2+a_3\zeta_3.
\end{equation}
Hence we have
\begin{equation}\label{eq:dr-z4}
r\d r \w \zeta_0 \w \zeta_1 \w \zeta_2 \w \zeta_3 = 0.
\end{equation}
We then compute from~\eqref{eq:d-zetas-S3},~\eqref{eq:Omegas-S3} and~\eqref{eq:dr-S3} that
\begin{align*}
b_1\w\d\Omega_1 &=b_1\w(2\rho_3\w\Omega_2-2\rho_2\w\Omega_3)+\frac{\kappa}{2}r\d r\w \vol_{\mathcal{S}^3},\\
b_2\w\d\Omega_2 &= b_2\w (2\rho_1\w\Omega_3-2\rho_3\w\Omega_1)+\frac{\kappa}{2}r\d r\w\vol_{\mathcal{S}^3},\\
b_3\w\d\Omega_3 &= b_3\w (2\rho_2\w\Omega_1-2\rho_1\w\Omega_2)+\frac{\kappa}{2}r\d r\w\vol_{\mathcal{S}^3}.
\end{align*}
Combining the above with~\eqref{eq:d-b-S3} gives that
\begin{equation*}
\d(b_1\w\Omega_1+b_2\w\Omega_2+b_3\w\Omega_3)=-\frac{3\kappa}{2}r\d r\w\vol_{\mathcal{S}^3}.
\end{equation*}
It then follows from~\eqref{eq:phi.S3.general} that
\begin{align*}
\d\varphi_{c_0,c_1,\kappa}=6\kappa c_1r\d r\w\vol_{\mathcal{S}^3}-4c_1\frac{3\kappa}{2}r\d r\w\vol_{\mathcal{S}^3}=0,
\end{align*}
so the $\GG_2$-structure is closed.

Turning to the coclosed condition, we obtain from~\eqref{eq:b-S3},~\eqref{eq:d-zetas-S3}, and~\eqref{eq:dr-S3} that
\begin{align*}
\d(\zeta_0\w\zeta_1\w\zeta_2\w\zeta_3)=\frac{\kappa}{2}r\d r\w (b_2\w b_3\w\Omega_1+b_3\w b_1\w\Omega_2+b_1\w b_2\w\Omega_3).
\end{align*}
We can also compute from~\eqref{eq:Omegas-S3},~\eqref{eq:d-zetas-S3}, and~\eqref{eq:b-S3} that
\begin{align*}
b_2\w b_3\w\d\Omega_1&= b_2\w b_3\w (2\rho_3\w\Omega_2-2\rho_2\w\Omega_3)=0,\\
b_3\w b_1\w\d\Omega_2&= b_3\w b_1\w (2\rho_1\w\Omega_3-2\rho_3\w\Omega_1)=0,\\
b_1\w b_2\w\d\Omega_3&= b_1\w b_2\w (2\rho_2\w\Omega_1-2\rho_1\w\Omega_2)=0.
\end{align*}
Therefore, since $b_i \w b_j$ is closed, we have
\[
\d(b_2\w b_3\w\Omega_1+b_3\w b_1\w\Omega_2+b_1\w b_2\w\Omega_3)=0.
\]
Combining these observations with~\eqref{eq:dr-z4}, taking the exterior derivative of~\eqref{eq:starphi.S3.general} yields
\begin{align*}
\d*_{\varphi_{c_0,c_1,\kappa}}\varphi_{c_0,c_1,\kappa}&=16\left(\frac{c_1^3}{3\kappa}\right)^{\frac{2}{3}}(c_0+c_1r^2)^{-\frac{2}{3}}\frac{\kappa}{2}r\d r\w (b_2\w b_3\w\Omega_1+b_3\w b_1\w\Omega_2+b_1\w b_2\w\Omega_3)\\
&-\frac{8}{3}(3\kappa c_1^3)^{\frac{1}{3}}c_1r\d r\w(b_2\w b_3\w\Omega_1+b_3\w b_1\w\Omega_2+b_1\w b_2\w\Omega_3)=0,
\end{align*}
so that $\varphi_{c_0,c_1,\kappa}$ is indeed torsion-free.
\end{proof}

Bryant--Salamon show that the metric $g_{c_0,c_1,\kappa}$ in~\eqref{eq:g.S3.general} in fact has holonomy equal to $\GG_2$. It is also clear from~\eqref{eq:phi.S3.general} that the zero section $\mathcal{S}^3$ is associative in $(M,\varphi_{c_0,c_1,\kappa})$.

\subsubsection{Asymptotically conical \texorpdfstring{G\textsubscript{2}}{G2} manifolds}

Bryant--Salamon show that, although there appear to be three positive parameters in the construction of torsion-free $\GG_2$-structures on $M$ (namely $c_0$, $c_1$, and $\kappa$), one can always rescale and reparametrize coordinates in the fibres of $M$ so that, in reality, there is only one  true parameter in this family of $\GG_2$-structures.

Concretely, we can assume that $\kappa=1$ and for $c>0$ we can let
$$c_0= \sqrt{3} c\quad\text{and}\quad c_1=\sqrt{3}$$
and define
$$\varphi_c=\varphi_{\sqrt{3} c,\sqrt{3},1}.$$
We see from~\eqref{eq:phi.S3.general} that
\begin{align}\label{eq:phic.S3}
\varphi_c&=3\sqrt{3}(c+r^2)\vol_{\mathcal{S}^3}+4\sqrt{3} (b_1\wedge\Omega_1+b_2\wedge\Omega_2+b_3\wedge\Omega_3)
\end{align}
and its induced metric, Hodge dual and volume form are:
\begin{align}
g_c&=3(c + r^2)^{\frac{2}{3}}g_{\mathcal{S}^3}+4(c + r^2)^{-\frac{1}{3}}(\zeta_0^2+\zeta_1^2+\zeta_2^2+\zeta_3^2),\label{eq:gc.S3}\\
*_{\varphi_c}\varphi_c&=16(c+r^2)^{-\frac{2}{3}}\zeta_0\wedge\zeta_1\wedge\zeta_2\wedge\zeta_3-12(c+ r^2)^{\frac{1}{3}}(b_2\wedge b_3\wedge\Omega_1+b_3\wedge b_1\wedge\Omega_2+b_1\wedge b_2\wedge\Omega_3),\label{eq:starphic.S3}\\
\vol_c&=48\sqrt{3}(c+r^2)^{\frac{1}{3}}\vol_{\mathcal{S}^3}\wedge\zeta_0\wedge\zeta_1\wedge\zeta_2\wedge\zeta_3.\label{eq:volc.S3}
\end{align}

Hence, we have a 1-parameter family $(M,\varphi_c)$ of complete holonomy $\GG_2$ manifolds. As mentioned above, these manifolds are asymptotically conical with rate $-3$, which is straightforward to verify from~\eqref{eq:gc.S3}.

Moreover, as $M\cong\R^4\times\mathcal{S}^3$, one sees that the asymptotic cone is $M_0=M\setminus\mathcal{S}^3=\R^+\times\mathcal{S}^3\times\mathcal{S}^3$. Setting $c=0$ in~\eqref{eq:phic.S3}--\eqref{eq:volc.S3} gives the conical torsion-free $\GG_2$-structure $\varphi_0$ on $M_0$ and its induced holonomy $\GG_2$ metric, Hodge dual 4-form, and volume form. We summarize these and our earlier observations.

\begin{thm}\label{thm:BS.S3} Let $\mathcal{S}^3$ be endowed with the constant curvature $1$ metric. There is a 1-parameter family $\varphi_c$, for $c>0$, of 
torsion-free $\GG_2$-structures on $M=\mathbb{S}(\mathcal{S}^3)$ which induce complete holonomy $\GG_2$ metrics $g_c$ on $M$ and which are asymptotically conical with rate $-3$. Moreover, the zero section $\mathcal{S}^3$ is associative in $(M,\varphi_c)$, and the asymptotic cone $\R^+\times\mathcal{S}^3\times\mathcal{S}^3=M\setminus\mathcal{S}^3$ is endowed with a torsion-free $\GG_2$-structure
$\varphi_0$ inducing a conical holonomy $\GG_2$ metric $g_0$ such that 
$\varphi_0$ is the limit of $\varphi_c$ on $M\setminus\mathcal{S}^3$ as $c\to 0$.
\end{thm}

\subsubsection{Flat limit}\label{subs:S3.flat}

We conclude this section by showing, in a formal manner, how we can take a limit of Bryant--Salamon holonomy $\GG_2$ metrics on $\mathbb{S}(\mathcal{S}^3)$ to obtain the flat metric on $\R^7$, which we view as $\mathbb{S}(\R^3)$.

Recall the family $\varphi_{c_0,c_1,\kappa}$ of torsion-free $\GG_2$-structures from Lemma~\ref{lem:torsion-free.S3} and consider the curvature $\kappa$ of the metric on $\mathcal{S}^3$ as a parameter. In the limit as $\kappa\to 0$ the metric becomes flat and $\mathcal{S}^3$ becomes $\R^3$. Therefore, the spin connection becomes trivial, which says that $\d b_j \to 0$ for $j=1, 2, 3$. Then by~\eqref{eq:d-b-S3} and the fact that the $b_i$ are orthonormal, we deduce that
$$\lim_{\kappa\to 0}\rho_j=0 \quad \text{for $j=1,2,3$.}$$
It then follows from~\eqref{eq:zetas-S3} that
$$\lim_{\kappa\to 0}\zeta_k=\d a_k \quad \text{for $k=0,1,2,3$.} $$
Consequently, by~\eqref{eq:Omegas-S3}, for $j=1,2,3$ we have that
$$\lim_{\kappa\to 0}\Omega_j=\omega_j,$$
where $\omega_1, \omega_2, \omega_3$ are the standard anti-self-dual 2-forms on $\R^4$ as in~\eqref{eq:omegas-R7}, where $y_k=a_k$ for $k=0,1,2,3$. If we then let
$$c_0=\frac{1}{3\kappa}\quad\text{and} \quad c_1=\frac{1}{4},$$
we find from~\eqref{eq:phi.S3.general} that
\begin{align*}
\lim_{\kappa\to 0}\varphi_{\frac{1}{3\kappa},\frac{1}{4},\kappa}
&=\lim_{\kappa\to 0}\big((1+\textstyle\frac{3\kappa}{4} r^2) b_1\wedge b_2\wedge b_3+ b_1\wedge\Omega_1+b_2\wedge\Omega_2+b_3\wedge\Omega_3\big)\\
&=b_1\wedge b_2\wedge b_3+b_1\wedge\omega_1+b_2\wedge\omega_2+b_3\wedge\omega_3.
\end{align*}
Since the $b_j$ form a parallel orthonormal frame on $\R^3$ in the limit as $\kappa\to 0$, we can set $b_j=\d x_j$ for $j=1,2,3$ and deduce that
$$\lim_{\kappa\to 0}\varphi_{\frac{1}{3\kappa},\frac{1}{4},\kappa}=\varphi_{\R^7}$$
as given in~\eqref{eq:varphi.R7}. Therefore, in this limit as $\kappa\to 0$, the Bryant--Salamon torsion-free $\GG_2$-structures on $\SS(\mathcal{S}^3)$ become the standard flat $\GG_2$-structure on $\R^7$, viewed as $\SS(\mathbb{R}^3)=\R^3\oplus\R^4$.

The discussion above, together with work in~\cite{BryantSalamon}, shows how, after reparametrization in the fibres of $\SS(\mathcal{S}^3)$, we can view the limit of $\varphi_c$ as $c\to\infty$ as the flat $\GG_2$-structure $\varphi_{\R^7}$ on $\R^7=\SS(\R^3)$.

\subsection{Self-dual Einstein 4-manifolds} \label{sec:N}

Let $N^4$ be a compact self-dual Einstein 4-manifold with positive scalar curvature $12\kappa$. That is, $N$ is endowed with a metric $g_N$ such that $\Ric(g_N)=3\kappa g_N$ and such that the anti-self-dual part of the Weyl tensor is zero. Then we know from~\cite{Hitchin} that either $N$ is $\mathcal{S}^4$ with the constant curvature $\kappa$ metric or $N$ is $\C\P^2$ with the Fubini--Study metric of the appropriate scale.

\begin{remark}
The Bryant--Salamon construction also works for self-dual Einstein 4-manifolds with non-positive scalar curvature and for self-dual Einstein 4-orbifolds, but one does not obtain smooth complete holonomy $\GG_2$ metrics in those cases.
\end{remark}

Let $M^7=\Lambda^2_-(T^*N)$ be the bundle of anti-self-dual 2-forms on $N$. This is the space on which Bryant--Salamon construct their torsion-free $\GG_2$-structures, as well as on the cone over the twistor space of $N$. We briefly describe this construction. The precise statement is summarized in Theorem~\ref{thm:BS.S4CP2} below.

\subsubsection{Coframe, anti-self-dual 2-forms, induced connection and vertical 1-forms}

Given a positively oriented local orthonormal coframe $\{b_0,b_1,b_2,b_3\}$ on $N$, we define a local orthogonal basis for the anti-self-dual 2-forms on $N$ by
\begin{equation}\label{eq:big-Omega-N}
\Omega_1=b_0\wedge b_1-b_2\wedge b_3,\quad \Omega_2=b_0\w b_2-b_3\w b_1,\quad \Omega_3=b_0\w b_3-b_1\w b_2.
\end{equation}
We note that the volume form on $N$ is
\begin{equation}\label{eq:volN}
\vol_N=b_0\wedge b_1\wedge b_2\wedge b_3.
\end{equation}
Using the canonical projection $\pi_N:M\to N$, we can pullback $b_i$ for $i=0,1,2,3$ to horizontal 1-forms on $M$, and thus the same is true of any forms constructed purely from the $b_i$. We omit the pullback from our notation throughout this paper. This should not cause any confusion.

One may see that
\begin{equation}\label{eq:d-big-Omega-N}
\d\left(\begin{array}{c}\Omega_1\\ \Omega_2\\ \Omega_3\end{array}\right) = -\left(\begin{array}{ccc}0 & -2\rho_3 & 2\rho_2 \\ 2\rho_3 & 0 & -2\rho_1 \\ -2\rho_2 & 2\rho_1 & 0\end{array}\right)\w \left(\begin{array}{c}\Omega_1\\ \Omega_2\\ \Omega_3\end{array}\right),
\end{equation}
where the forms $1$-forms $\rho_1,\rho_2,\rho_3$ describe the induced connection on $M$ from the Levi-Civita connection of the Einstein metric $g_N$ on $N$. (The factor of 2 is natural here because if we decompose the Levi-Civita connection 1-forms using the splitting $\mathfrak{so}(4)=\mathfrak{su}(2)_+\oplus\mathfrak{su}(2)_-$, then the connection forms on $\Lambda^2_-(T^*N$) are twice the $\mathfrak{su}(2)_-$ component.) Furthermore, we have that
\begin{equation}\label{eq:d-rho-N}
\d\left(\begin{array}{c} \rho_1\\ \rho_2\\ \rho_3\end{array}\right)+2\left(\begin{array}{c} \rho_2\wedge\rho_3\\ \rho_3\w\rho_1\\ \rho_1\w\rho_2\end{array}\right) = \frac{\kappa}{2}\left(\begin{array}{c} \Omega_1\\ \Omega_2\\ \Omega_3\end{array}\right),
\end{equation}
which is equivalent~\cite[p. 842]{BryantSalamon} to saying that the metric on $N$ is self-dual Einstein with scalar curvature equal to $12\kappa$.

Using the local basis $\{\Omega_1,\Omega_2,\Omega_3\}$ for the anti-self-dual 2-forms on $N$, we can define linear coordinates $(a_1,a_2,a_3)$ on the fibres of $M$ over $N$ with respect to this basis. Then 
\[a_1\Omega_1+a_2\Omega_2+a_3\Omega_3\]
is the tautological 2-form on $M=\Lambda^2_-(T^*N)$ and so is in fact globally defined. 

Using the coordinates $(a_1,a_2,a_3)$, the vertical 1-forms  on $M$ for the induced connection are:
\begin{equation} \label{eq:zetas-N}
\zeta_1=\d a_1-2a_2\rho_3+2a_3\rho_2,\quad
\zeta_2=\d a_2-2a_3\rho_1+2a_1\rho_3,\quad
\zeta_3=\d a_3-2a_1\rho_2+2a_2\rho_1.
\end{equation}
For use below, one may compute that
\begin{equation}\label{eq:d-zetas-N}
\d\left(\begin{array}{c} \zeta_1 \\ \zeta_2 \\ \zeta_3\end{array}\right)
=-\left(\begin{array}{ccc} 0 & -2\rho_3 & 2\rho_2 \\ 2\rho_3 & 0 & -2\rho_1 \\ -2\rho_2 & 2\rho_1 & 0\end{array}\right)\w\left(\begin{array}{c}\zeta_1\\ \zeta_2\\ \zeta_3\end{array}\right)
+\kappa\left(\begin{array}{c} a_3\Omega_2-a_2\Omega_3\\ a_1\Omega_3-a_3\Omega_1\\ a_2\Omega_1-a_1\Omega_2\end{array}\right).
\end{equation}

\subsubsection{Torsion-free \texorpdfstring{G\textsubscript{2}}{G2}-structures}

Let $r^2=a_1^2+a_2^2+a_3^2$ be the squared distance in the fibres of $M$ from the zero section, and let $c_0>0$ and $c_1>0$. 
Bryant--Salamon then define the following $\GG_2$-structure on $M$:
\begin{equation}\label{eq:phi.BS.general}
\varphi_{c_0,c_1,\kappa}=\left(\frac{c_1^2}{2\kappa}\right)^{\frac{3}{4}}(c_0+c_1r^2)^{-\frac{3}{4}}\zeta_1\w\zeta_2\w\zeta_3+(2\kappa c_1^2)^\frac{1}{4}(c_0+c_1r^2)^{\frac{1}{4}}(\zeta_1\w\Omega_1+\zeta_2\w\Omega_2+\zeta_3\w\Omega_3).
\end{equation}
The metric induced by $\varphi_{c_0,c_1,\kappa}$ is given by
\begin{equation}\label{eq:g.BS.general}
g_{c_0,c_1,\kappa}=\left(\frac{c_1^2}{2\kappa}\right)^{\frac{1}{2}}(c_0+c_1r^2)^{-\frac{1}{2}}(\zeta_1^2+\zeta_2^2+\zeta_3^2)+(2\kappa)^{\frac{1}{2}}(c_0+c_1r^2)^{\frac{1}{2}}g_N
\end{equation}
and the Hodge dual of $\varphi_{c_0,c_1,\kappa}$ is given by
\begin{equation}\label{eq:starphi.BS.general}
*_{\varphi_{c_0,c_1,\kappa}}\varphi_{c_0,c_1,\kappa}=2\kappa(c_0+c_1r^2)\vol_N-c_1(\zeta_2\w\zeta_3\w\Omega_1+\zeta_3\w\zeta_1\w\Omega_2+\zeta_1\w\zeta_2\w\Omega_3).
\end{equation}
Finally, the volume form is given by
\begin{equation}\label{eq:vol.BS.general}
\vol_{c_0,c_1,\kappa}=(2\kappa c_1^6)^{\frac{1}{4}}(c_0+c_1r^2)^{\frac{1}{4}}\zeta_1\w\zeta_2\w\zeta_3\w\vol_N.
\end{equation}

\begin{lem}\label{lem:torsion-free.N}
 The $\GG_2$-structure $\varphi_{c_0,c_1,\kappa}$ in~\eqref{eq:phi.BS.general} is torsion-free.
\end{lem}

\begin{proof}
One may check from~\eqref{eq:d-big-Omega-N} and~\eqref{eq:zetas-N} that
\[\zeta_1\w\Omega_1+\zeta_2\w\Omega_2+\zeta_3\w\Omega_3=\d(a_1\Omega_1+a_2\Omega_2+a_3\Omega_3)\]
and hence
\begin{equation} \label{eq:d(z.O)}
\d(\zeta_1\w\Omega_1+\zeta_2\w\Omega_2+\zeta_3\w\Omega_3)=0.
\end{equation}
Moreover, it is straightforward to verify from~\eqref{eq:zetas-N} that
\begin{equation} \label{eq:rdr}
r\d r=a_1\d a_1+a_2 \d a_2+ a_3\d a_3=a_1\zeta_1+a_2\zeta_2+a_3\zeta _3.
\end{equation}
It then follows from~\eqref{eq:rdr} and~\eqref{eq:d-zetas-N} that
\begin{equation} \label{eq:d(z^3)}
\d(\zeta_1\w\zeta_2\w\zeta_3)=-\kappa r\d r\w(\zeta_1\w\Omega_1+\zeta_2\w\Omega_2+\zeta_3\w\Omega_3).
\end{equation}
It is also straightforward to compute using~\eqref{eq:d-big-Omega-N} and~\eqref{eq:d-zetas-N} that
\begin{equation}\label{eq:d(z^2.O)}
\d (\zeta_2\w\zeta_3\w\Omega_1+\zeta_3\w\zeta_1\w\Omega_2+\zeta_1\w\zeta_2\w\Omega_3)=4\kappa r\d r\w \vol_N.
\end{equation}
Taking the exterior derivative of~\eqref{eq:phi.BS.general} and using the identities~\eqref{eq:d(z.O)},~\eqref{eq:rdr}, and~\eqref{eq:d(z^3)}, we find that
\begin{align*}
\d\varphi_{c_0,c_1,\kappa}&=\left(\frac{c_1^2}{2\kappa}\right)^{\frac{3}{4}}(c_0+c_1r^2)^{-\frac{3}{4}}\d(\zeta_1\w\zeta_2\w\zeta_3)\\
&\quad+(2\kappa c_1^2)^{\frac{1}{4}}\frac{c_1}{2}(c_0+c_1r^2)^{-\frac{3}{4}}r\d r\w(\zeta_1\w\Omega_1+\zeta_2\w\Omega_2+\zeta_3\w\Omega_3)=0.
\end{align*}
Similarly taking the exterior derivative of~\eqref{eq:starphi.BS.general} and using the identity~\eqref{eq:d(z^2.O)} and the fact that $\vol_N$ is closed, we find that
\begin{align*}
\d *_{\varphi_{c_0,c_1,\kappa}}\varphi_{c_0,c_1,\kappa}&=4\kappa c_1 r\d r\w\vol_N-c_1\d(\zeta_2\w\zeta_3\w\Omega_1+\zeta_3\w\zeta_1\w\Omega_2+\zeta_1\w\zeta_2\w\Omega_3)=0.
\end{align*}
Thus, we conclude that $\varphi_{c_0,c_1,\kappa}$ is torsion-free for all positive $c_0,c_1,\kappa$.
\end{proof}

Moreover, Bryant--Salamon show that $g_{c_0,c_1,\kappa}$ in~\eqref{eq:g.BS.general} has holonomy equal to $\GG_2$, and it is evident from~\eqref{eq:phi.BS.general} and~\eqref{eq:starphi.BS.general} that the zero section $N$ is coassociative in $(M,\varphi_{c_0,c_1,\kappa})$.

\subsubsection{Asymptotically conical \texorpdfstring{G\textsubscript{2}}{G2} manifolds}

Bryant--Salamon show that, although there appear to be three positive parameters in $\varphi_{c_0,c_1,\kappa}$, in fact we can eliminate two of these parameters by rescaling and reparametrisation of the coordinates in the fibres of $M$, to obtain a true 1-parameter family of torsion-free $\GG_2$-structures.

Explicitly, we can scale the metric $g_N$ on $N$ so that $\kappa=1$, and for $c>0$ we can set
$$c_0=2 c\quad\text{and}\quad c_1=2.$$
Then we define 
$$\varphi_c=\varphi_{2c,2,1}.$$
With these choices, we find from~\eqref{eq:phi.BS.general} that
\begin{equation} \label{eq:phic}
\varphi_c=(c+r^2)^{-\frac{3}{4}}\zeta_1\w\zeta_2\w\zeta_3+2(c+r^2)^{\frac{1}{4}}(\zeta_1\w\Omega_1+\zeta_2\w\Omega_2+\zeta_3\w\Omega_3).
\end{equation}
Moreover, the metric determined by $\varphi_c$ is given by~\eqref{eq:g.BS.general} as
\begin{equation} \label{eq:gc}
g_c=(c+r^2)^{-\frac{1}{2}}(\zeta_1^2+\zeta_2^2+\zeta_3^2)+2(c+r^2)^{\frac{1}{2}}(b_0^2+b_1^2+b_2^2+b_3^2),
\end{equation}
and the 4-form $*_{\varphi_c} \varphi_c$ can be seen from~\eqref{eq:starphi.BS.general} to be
\begin{equation} \label{eq:psic}
*_{\varphi_c} \varphi_c=4(c+r^2)\vol_{N}-2(\zeta_2\w\zeta_3\w\Omega_1+\zeta_3\w\zeta_1\w\Omega_2+\zeta_1\w\zeta_2\w\Omega_3).
\end{equation}
Finally, the volume form given by~\eqref{eq:vol.BS.general} is
\begin{equation} \label{eq:volc}
\vol_c = 4(c+r^2)^{\frac{1}{4}}\zeta_1\w\zeta_2\w\zeta_3\w\vol_{N}.
\end{equation}

We thus have a 1-parameter family $(M,\varphi_c)$ of complete holonomy $\GG_2$ manifolds. As mentioned above, these are all asymptotically conical with rate $-4$, which one can verify directly from~\eqref{eq:gc}. Moreover, the asymptotic cone is $M_0=M\setminus N=\R^+\times\Sigma$ where $\Sigma$ is the twistor space of $N$, since this may be identified with the unit sphere bundle in $M$. We also observe that setting $c=0$ in~\eqref{eq:phic}--\eqref{eq:volc} gives all of the data determined by the conical torsion-free $\GG_2$-structure $\varphi_0$, which induces the conical holonomy $\GG_2$ metric $g_0$. We collect all of these observations in the following result.

\begin{thm}\label{thm:BS.S4CP2}
Let $N$ be either $\mathcal{S}^4$ with the constant curvature $1$ metric or $\C\P^2$ with the Fubini--Study metric of scalar curvature $12$. There is a 1-parameter family $\varphi_c$, for $c>0$, of torsion-free $\GG_2$-structures on $M=\Lambda^2_-(T^*N)$ which induce complete holonomy $\GG_2$ metrics $g_c$ on $M$ and which are asymptotically conical with rate $-4$. Moreover, the zero section $N$ is coassociative in $(M,\varphi_c)$, and the asymptotic cone $\R^+\times\Sigma=M\setminus N$, where $\Sigma=\C\P^3$ if $N=\mathcal{S}^4$ and $\Sigma=\SU(3)/T^2$ if $N=\C\P^2$, is endowed with a torsion-free $\GG_2$-structure $\varphi_0$ inducing a conical holonomy $\GG_2$ metric $g_0$ such that $\varphi_0$ is the limit of $\varphi_c$ on $M\setminus N$ as $c\to 0$.
\end{thm} 

\subsubsection{Flat limit}\label{subs:flat}

We conclude this section by showing, in a formal manner, how we can take a limit of Bryant--Salamon holonomy $\GG_2$ metrics on $\Lambda^2_-(T^*N)$ to obtain the flat metric on $\R^7$, which we view as $\Lambda^2_-(T^*\R^4)$.

Recall the family $\varphi_{c_0,c_1,\kappa}$ of torsion-free $\GG_2$-structures from Lemma~\ref{lem:torsion-free.N} and consider the curvature $\kappa$ of the metric on $N$ as a parameter. In the limit as $\kappa\to 0$ the metric becomes flat and $N$ becomes $\R^4$. Hence we have
\begin{equation*}
\lim_{\kappa\to 0}\Omega_i=\omega_i,
\end{equation*}
are the standard anti-self-dual 2-forms on $\R^4$ in~\eqref{eq:omegas-R7}. As the $\omega_i$ are closed, we see from~\eqref{eq:d-big-Omega-N} that 
\begin{equation*}
\lim_{\kappa\to 0}\rho_i=0
\end{equation*}
and thus from~\eqref{eq:zetas-N} we have
\begin{equation*}
\lim_{\kappa\to 0}\zeta_i=\d a_i.
\end{equation*}
Moreover, if we let
\begin{equation*}
c_0= \frac{1}{2\kappa} \quad\text{and}\quad c_1=1
\end{equation*}
we see from~\eqref{eq:phi.BS.general} that
\begin{align*}
\lim_{\kappa\to 0}\varphi_{\frac{1}{2\kappa},1,\kappa}&=
\lim_{\kappa\to 0}\big((1+ 2\kappa r^2)^{-\frac{3}{4}}\zeta_1\w\zeta_2\w\zeta_3+(1+2\kappa r^2)^{\frac{1}{4}}(\zeta_1\w\Omega_1+\zeta_2\w\Omega_2+\zeta_3\w\Omega_3)\big)
\\
&=\d a_1\w\d a_2\w \d a_3+\d a_1\w\omega_1+\d a_2\w\omega_2+\d a_3\w\omega_3
=\varphi_{\mathbb{R}^7},
\end{align*}
as in~\eqref{eq:varphi.R7} (after setting $a_i=x_i$). We conclude that taking this particular limit as $\kappa\to 0$, we obtain the standard flat $\GG_2$-structure on $\mathbb{R}^7$, where $\R^7$ is viewed as $\Lambda^2_-(T^*\mathbb{R}^4)$.

This analysis, together with the discussion in~\cite{BryantSalamon}, also indicates how, after reparametrization in the fibres of $\Lambda^2_-(T^*N)$, one can take the limit as $c\to\infty$ in $\varphi_c$ to obtain the flat limit $\varphi_{\R^7}$ on $\R^7=\Lambda^2_-(T^*\R^4)$.

\section{Spinor bundle of \texorpdfstring{$\mathcal{S}$\textsuperscript{3}}{S3}} \label{sec:SS3}

In this short section we consider $M^7=\mathbb{S}(\mathcal{S}^3)$ with the Bryant--Salamon torsion-free $\GG_2$-structure $\varphi_c$ from~\eqref{eq:phic.S3}, and let $M_0=\mathbb{S}(\mathcal{S}^3)\setminus\mathcal{S}^3=\R^+\times\mathcal{S}^3\times\mathcal{S}^3$ be its asymptotic cone, with the conical $\GG_2$-structure $\varphi_0$. We describe both $M$ and $M_0$ as coassociative fibrations over $\mathcal{S}^3$. Specifically we prove the following result.

\begin{thm}\label{thm:S3.fib} Let $(M=\SS(\mathcal{S}^3),\varphi_c)$ and $(M_0=\SS(\mathcal{S}^3)\setminus\mathcal{S}^3,\varphi_0)$ be as in Theorem~\ref{thm:BS.S3}.
\begin{itemize}
\item[\emph{(a)}] The canonical projection $\pi:M\to\mathcal{S}^3$ is a coassociative fibration with respect to $\varphi_c$, where the fibres are all $\SO(4)$-invariant and diffeomorphic to $\R^4$. 
\item[\emph{(b)}] The canonical projection $\pi:M_0\to\mathcal{S}^3$ is a coassociative fibration with respect to $\varphi_0$, where the fibres are all $\SO(4)$-invariant and diffeomorphic to $\R^4\setminus\{0\}$.
\end{itemize}
\end{thm}

Note that Theorem~\ref{thm:S3.fib} is quite straightforward, but we nevertheless present it here with full details, for comparison with the analogous but significantly more non-trivial theorems for the other two Bryant--Salamon manifolds. These are Theorem~\ref{thm:S4.fib} for $\Lambda^2_-(T^* \mathcal{S}^4)$ and Theorem~\ref{thm:CP2.fib} for $\Lambda^2_-(T^* \C\P^2)$. In particular, we refer the reader to the discussion in the introduction about the marked difference between the $\SS(\mathcal{S}^3)$ case and the two $\Lambda^2_-(T^* N^4)$ cases.

\subsection{Group actions}

We observe that there is a natural action of $\SU(2)^3$ on $M$, which we can define in terms of a triple $(q_1,q_2,q_3)$ of unit quaternions acting on $\SS(\mathcal{S}^3)\cong\H\times\mathcal{S}^3\subseteq\H\oplus\H$ by
\begin{equation}\label{eq:SU(2)cubed}
(q_1,q_2,q_3):(\mathbf{a},\mathbf{x})\mapsto (q_1\mathbf{a} \overline{q_3},q_2\mathbf{x}\overline{q_3})
\end{equation}
for $\mathbf{a}\in\H$ and $\mathbf{x}\in\mathcal{S}^3\subseteq\H$. Note that if we take $q_2=\mathbf{x}q_3\overline{\mathbf{x}}$ for any $\mathbf{x},q_3\in\mathcal{S}^3$, then 
$$q_2\mathbf{x}\overline{q_3}=\mathbf{x}$$
and hence by~\eqref{eq:SU(2)cubed} we have an action of $\SU(2)^2$ on each fibre of $\SS(\mathcal{S}^3)$ given by 
$$(q_1,q_3):\mathbf{a}\mapsto q_1\mathbf{a}\overline{q_3}.$$
In fact, this is an $\SO(4)$ action on each fibre since the action of $(-1,-1)\in\SU(2)^2$ on $\mathbf{a}$ is trivial.

One can verify from the construction of $\varphi_c$ in~\eqref{eq:phic.S3} that it is invariant under the $\SU(2)^3$ action in~\eqref{eq:SU(2)cubed}. (See also~\cite{BryantSalamon}.) Moreover, the fibres of $\SS(\mathcal{S}^3)$ are $\SO(4)$-invariant and are manifestly coassociative with respect to $\varphi_c$ by Definition~\ref{dfn:coassociative}, as claimed in Theorem~\ref{thm:S3.fib}.

Notice that we have a global $\SU(2)$ action on $M$, given by taking $(q,1,1)$ for a unit quaternion $q$ in~\eqref{eq:SU(2)cubed}, which acts on every fibre of the $\SS(\mathcal{S}^3)$ but acts trivially on the base $\mathcal{S}^3$, i.e.~
\begin{equation}\label{eq:SU2.action.S3}
 q:(\mathbf{a},\mathbf{x})\mapsto (q\mathbf{a},\mathbf{x})
\end{equation}
for $\mathbf{a}\in\H$ and $\mathbf{x}\in\mathcal{S}^3\subseteq\H$.

\subsection{Relation to multi-moment maps}\label{subs:multimoment.S3}

Consider the $\SU(2)$ action in~\eqref{eq:SU2.action.S3}. In this section we find the multi-moment map $\nu$ for this $\SU(2)$ action for the $4$-form $*_{\varphi_c} \varphi_c$ as in~\eqref{eq:mmm4}.

Let $X_1$, $X_2$, $X_3$ denote vector fields generating the left $\SU(2)$ action~\eqref{eq:SU2.action.S3}, for example
$$X_1(\mathbf{a},\mathbf{x})=i\mathbf{a},\quad X_2(\mathbf{a},\mathbf{x})=j\mathbf{a},\quad
X_3(\mathbf{a},\mathbf{x})=k\mathbf{a}$$
for $\mathbf{a}\in\H$ and $\mathbf{x}\in\mathcal{S}^3\subseteq\H$. We may then compute directly from~\eqref{eq:starphic.S3} (for example by rotating to the point $\mathbf{a}=r\in\R_{\geq 0}$), that
\[
*_{\varphi_c} \varphi_c (X_1,X_2,X_3,\cdot)=-16(c+r^2)^{-\frac{2}{3}}r^3\d r=\d\big(6(3c-r^2)(c+r^2)^{\frac{1}{3}}\big).
\]
We deduce the following.

\begin{prop} The multi-moment map for the $\SU(2)$ action~\eqref{eq:SU2.action.S3} on $*_{\varphi_c} \varphi_c$ is 
$$\nu=6(3c-r^2)(c+r^2)^{\frac{1}{3}}-18c^{\frac{4}{3}},$$
which maps onto $(-\infty,0]$ for $c>0$ and $(-\infty,0)$ for $c=0$.
\end{prop} 

\begin{remark} \label{rmk:mmmS3}
One can compute from equation~\eqref{eq:phic.S3} for $\varphi_c$ that
\begin{align*}
\varphi_c(X_2,X_3,\cdot) & = 4\sqrt{3} \big( (-a_0^2 - a_1^2 + a_2^2 + a_3^2) b_1 + (2 a_0 a_3 - 2 a_1 a_2) b_2 + (-2 a_0 a_2 - 2 a_3 a_1) b_3 \big), \\
\varphi_c(X_3,X_1,\cdot) & = 4\sqrt{3} \big( (-2 a_0 a_3 - 2 a_1 a_2) b_1 + (-a_0^2 + a_1^2 - a_2^2 + a_3^2) b_2 + (2 a_0 a_1 - 2 a_2 a_3) b_3 \big), \\
\varphi_c(X_1,X_2,\cdot) & = 4\sqrt{3} \big( (2 a_0 a_2 - 2 a_3 a_1) b_1 + (-2 a_0 a_1 - 2 a_2 a_3) b_2 + (-a_0^2 + a_1^2 + a_2^2 - a_3^2) b_3 \big).
\end{align*}
As expected by Remark~\ref{rmk:mmm}, these are not closed forms. However, the forms $b_1,b_2,b_3$ are a left-invariant coframe on the base $\mathcal{S}^3$, and so are natural in this context.

The above expressions simplify considerably if we write them in terms of quaternion-valued forms. Define an $\H$-valued $0$-form $\mathbf{a} = a_0 \mathbf{1} + a_1 \mathbf{i} + a_2 \mathbf{j} + a_3 \mathbf{k}$ and an $\Imm\H$-valued $1$-form $\mathbf{b} = b_1 \mathbf{i} + b_2 \mathbf{j} + b_3 \mathbf{k}$. Then one can check that
\begin{align*}
\varphi_c(X_2,X_3,\cdot) & = -4\sqrt{3} \langle \mathbf{\bar{a}} \mathbf{i} \mathbf{a}, \mathbf{b} \rangle, \quad
\varphi_c(X_3,X_1,\cdot)  = -4\sqrt{3} \langle \mathbf{\bar{a}} \mathbf{j} \mathbf{a}, \mathbf{b} \rangle, \quad
\varphi_c(X_1,X_2,\cdot)  = -4\sqrt{3} \langle \mathbf{\bar{a}} \mathbf{k} \mathbf{a}, \mathbf{b} \rangle.
\end{align*}
The above way of writing these 1-forms exhibits a relation to the \emph{hyperk\"ahler moment map} for the action of $\SU(2) \cong \Sp(1)$ on $\H$ by right multiplication, which is given by $\frac{1}{2} ( \mathbf{\bar{a}} \mathbf{i} \mathbf{a}, \mathbf{\bar{a}} \mathbf{j} \mathbf{a}, \mathbf{\bar{a}} \mathbf{k} \mathbf{a})$. This is interesting as the ``multi-moment map'' becomes an $\Imm\H$-valued left-invariant 1-form built from the hyperk\"ahler moment map in a natural way.\end{remark}

\subsection{Riemannian and hypersymplectic geometry on the fibres} \label{sec:hypersymplectic.S3}

Let $N$ denote one of the coassociative fibres on $M$ or $M_0$. In this section we describe the geometry induced on $N$ from the ambient torsion-free $\GG_2$ structure. This includes the Riemannian geometry of the induced metric, and the induced \emph{hypersymplectic} structure in the sense of Definition~\ref{dfn:hypersymplectic}.

\begin{prop} Let $g_{\mathcal{S}^3}$ denote the constant curvature $1$ metric on $\mathcal{S}^3$.
\begin{itemize}
\item[(a)] On $M$, the induced metric on the coassociative $\R^4$ fibres is conformally flat and asymptotically conical with rate $-3$ to the metric
\begin{equation}\label{eq:cone.R4minus0}
g_{\R^+\times\mathcal{S}^3}=\d R^2+\frac{4}{9}R^2g_{\mathcal{S}^3},
\end{equation} 
where $R$ is the coordinate on $\R^+$.
\item[(b)] On $M_0$, the induced metric on the coassociative $\R^4\setminus\{0\}$ fibres is precisely the cone metric in~\eqref{eq:cone.R4minus0}, which is conformally flat, but not flat, and so is not smooth at the origin.
\end{itemize}
Moreover, in both $M$ and $M_0$, the induced hypersymplectic triple on the coassociative fibres is (up to a constant multiple) the standard hyperk\"ahler triple $(\omega_1,\omega_2,\omega_3)$ on $\R^4$ given in~\eqref{eq:omegas-R7}.
\end{prop}
\begin{proof}
We begin by describing the induced metric on the coassociative fibres $N$. The metric on the fibres is obtained by setting the $\sigma_i = 0$, which by~\eqref{eq:b-S3} is equivalent to setting the $\rho_i = 0$. Thus we deduce from~\eqref{eq:zetas-S3} and~\eqref{eq:gc.S3} that the induced metric on $N$ is
$$g_c|_N=4(c+r^2)^{-\frac{1}{3}}(\d a_0^2+\d a_1^2+\d a_2^2+\d a_3^2),$$
where $(a_0,a_1,a_2,a_3)$ are linear coordinates on $N$. Since $r$ is the distance in the fibres with respect to these coordinates on $N$, we have that
$$g_c|_N=4(c+r^2)^{-\frac{1}{3}}g_{\R^4}=4(c+r^2)^{-\frac{1}{3}}(\d r^2+r^2g_{\mathcal{S}^3}),$$
where $g_{\mathcal{S}^3}$ is the constant curvature 1 metric on $\mathcal{S}^3$. Thus, $g_c|_N$ is conformally flat, but \emph{not} flat.

If we define $R=3r^{\frac{2}{3}}$, then one can compute that the metric becomes
$$g_c|_N=\left(1+\frac{27 c}{R^3}\right)^{-\frac{1}{3}}\left(\d R^2+\frac{4}{9}R^2g_{\mathcal{S}^3}\right).$$
Thus we see that with its induced metric, $N$ is asymptotically conical with rate $-3$ to the cone metric~\eqref{eq:cone.R4minus0} on $\R^4\setminus\{0\}$ which is \emph{not} flat and \emph{not} smooth at $0$. Moreover, when $c=0$ we see that $g_0|_N$ is exactly the cone metric on $\R^4\setminus\{0\}$ in~\eqref{eq:cone.R4minus0}. 

Let $N \cong \R^4$ be a fibre of $\SS(\mathcal{S}^3)$, which is coassociative. Then $b_k|_N = 0$, and by~\eqref{eq:b-S3} and~\eqref{eq:d-b-S3} we have that $\d b_k|_N = 0$ as well. It follows from~\eqref{eq:d-zetas-S3} and~\eqref{eq:Omegas-S3} that
\begin{equation} \label{eq:dOmegas.S3}
\d \Omega_k|_N = 0 \quad \text{for $k=1, 2, 3$}.
\end{equation}
Now let $n_1, n_2, n_3$ be the linearly independent normal vector fields that are dual to $b_1, b_2, b_3$. From~\eqref{eq:phic.S3} and~\eqref{eq:g.vol.S3} and~\eqref{eq:dOmegas.S3} we thus deduce that $\d (n_k \lrcorner \varphi_c)|_N = 0$. We can therefore apply Lemma~\ref{lem:hypersymplectic} to conclude that the hypersymplectic triple induced on $N$ is 
$$4\sqrt{3}(\Omega_1,\Omega_2,\Omega_3)|_N.$$
Recalling that the fibre corresponds to setting the $\rho_i = 0$, from~\eqref{eq:zetas-S3} and~\eqref{eq:Omegas-S3} we deduce that
$$\Omega_j|_N=\omega_j,$$
the standard hyperk\"ahler forms on $\R^4$ given in~\eqref{eq:omegas-R7} (with $a_k=y_k$ for $k=0,1,2,3$). 

Therefore, the induced hypersymplectic triple on $N$ is just the hyperk\"ahler triple $4\sqrt{3}(\omega_1,\omega_2,\omega_3)$ on $\R^4$ as claimed in Theorem~\ref{thm:S3.fib}.
\end{proof}

\subsection{Flat limit} \label{sec:flat.limit.SS3}

We have seen throughout this section that the geometry on the fibres is independent of the projection to $\mathcal{S}^3$. Consider the flat limit of the torsion-free $\GG_2$-structure $\varphi_{\frac{1}{3\kappa},\frac{1}{4},\kappa}$ on $M$ as $\kappa\to 0$ as in~\S\ref{subs:S3.flat}. Since we are simply rescaling, the fibres of $M$ are coassociative for all values of $\kappa$ in this family. This can also be seen by examining~\eqref{eq:phi.S3.general}.

In the limit as $\kappa\to 0$, $\mathcal{S}^3$ becomes the flat $\R^3$, and we have a coassociative fibration of $\R^7$ over $\R^3$ which is independent of $\R^3$. Moreover, the $\SO(4)$-invariance of the fibres is preserved, and so the resulting fibration must be the trivial coassociative $\R^4$ fibration of $\R^7=\R^3\oplus\R^4$ over $\R^3$.

\section{Anti-self-dual 2-form bundle of \texorpdfstring{$\mathcal{S}$\textsuperscript{4}}{S4}} \label{sec:ASDS4}

In this section we consider $M^7=\Lambda^2_-(T^*\mathcal{S}^4)$ with the Bryant--Salamon torsion-free $\GG_2$-structure $\varphi_c$ from~\eqref{eq:phic}, and let $M_0=\R^+\times\C\P^3$ be its asymptotic cone, with the conical $\GG_2$-structure $\varphi_0$. We describe both $M$ and $M_0$ as coassociative fibrations over $\R^3$. Specifically we prove the following result.

\begin{thm}\label{thm:S4.fib}
Let $M=\Lambda^2_-(T^*\mathcal{S}^4)$, let $M_0=\R^+\times\C\P^3=M\setminus\mathcal{S}^4$, and recall the Bryant--Salamon $\GG_2$-structures $\varphi_c$ given in Theorem~\ref{thm:BS.S4CP2} for $c\geq 0$. 
\begin{itemize}
\item[\emph{(a)}] For every $c>0$, there is an $\SO(3)$-invariant projection $\pi_c:M\to\R^3$ such that each fibre $\pi_c^{-1}(x)$ is coassociative in $(M,\varphi_c)$. Moreover, there is a circle $\mathcal{S}^1_c\subseteq\{0\}\times\R^2\subseteq\R^3$ such that the fibres of $\pi_c$ are given by
$$\pi_c^{-1}(x)\cong \left\{\begin{array}{cl} T^*\mathcal{S}^2, & x\notin\mathcal{S}^1_c,\\[2pt]
(\R^+\times\R\P^3)\cup\{0\}, & x\in\mathcal{S}^1_c.\end{array} \right.$$
\item[\emph{(b)}] There is an $\SO(3)$-invariant projection $\pi_0:M_0=M\setminus\mathcal{S}^4\to\R^3$ such that $\pi_0^{-1}(x)$ is coassociative in $(M,\varphi_0)$. Moreover, the fibres of $\pi_0$ are given by
 $$\pi_0^{-1}(x)\cong \left\{\begin{array}{cl} T^*\mathcal{S}^2, & x\neq 0,\\[2pt]
\R^+\times\R\P^3, & x=0. \end{array}\right.$$
\end{itemize}
That is, both $(M,\varphi_c)$ and $(M_0,\varphi_0)$ can be realized as $\SO(3)$-invariant coassociative fibrations whose fibres are generically smooth and diffeomorphic to $T^*\mathcal{S}^2$, and whose singular fibres consist of a circle of $(\R^+\times\R\P^3)\cup\{0\}$ singular fibres in $M$, and a single $\R^+\times\R\P^3$ singular fibre in $M_0$. 
\end{thm}

We also study the induced Riemannian and hypersymplectic geometry on the coassociative fibres. The coassociative fibres that we call \emph{singular} in Theorem~\ref{thm:S4.fib} turn out to have conically singular induced Riemannian metrics, including Riemannian cones in some cases. The precise statement is given in Proposition~\ref{prop:AC.CS-S4}. 

\subsection{A coframe on \texorpdfstring{$\mathcal{S}^4$}{S4}}

As discussed in~\S\ref{sec:N}, the starting point for describing the Bryant--Salamon $\GG_2$-structure on $M$ is an orthonormal coframe on $\mathcal{S}^4$. We therefore start by constructing such a coframe which is \emph{adapted} to the symmetries we wish to impose.

Consider $\mathcal{S}^4$ as the unit sphere in $\R^5$ with the induced Riemannian metric. Choose a 3-dimensional linear subspace $P$ in $\R^5$. Our construction will depend on this choice, and thus will break the usual $\SO(5)$ symmetry of $\mathcal{S}^4$. We will identify $P\cong \R^3$ and $P^{\perp}\cong \R^2$.
 
With respect to the splitting $\R^5=P\oplus P^{\perp}$, write
\[\mathcal{S}^4=\{(\mathbf{x},\mathbf{y})\in P\oplus P^{\perp} : |\mathbf{x}|^2+|\mathbf{y}|^2=1\}.\]
For all $(\mathbf{x},\mathbf{y})\in\mathcal{S}^4$, there exists a unique $\alpha\in[0,\frac{\pi}{2}]$ such that
\[\mathbf{x}=\cos\alpha \mathbf{u},\quad \mathbf{y}=\sin\alpha\mathbf{v},\]
where $\mathbf{u}\in P$ with $|\mathbf{u}|=1$ and $\mathbf{v}\in P^{\perp}$ with $|\mathbf{v}|=1$. More precisely, for $\alpha \neq 0,\frac{\pi}{2}$, the vectors $\mathbf{u}, \mathbf{v}$ are uniquely determined. For $\alpha = 0$, we have $\mathbf{u} = \mathbf{x} \in \mathcal{S}^2$ and $\mathbf{y} = 0$ so $\mathbf{v}$ is undefined. For $\alpha = \frac{\pi}{2}$, we have $\mathbf{v} = \mathbf{y} \in \mathcal{S}^1$, and $\mathbf{x} = 0$, so $\mathbf{u}$ is undefined. Thus we write $\mathcal{S}^4$ as the disjoint union of an $\mathcal{S}^2$, an $\mathcal{S}^1$, and a $(0, \frac{\pi}{2})$ family of $\mathcal{S}^2 \times \mathcal{S}^1$ spaces.

Moreover, for each $\mathbf{u}\in P$ with $|\mathbf{u}|=1$ there is a unique $\theta\in [0,\pi]$ and some $\phi\in [0,2\pi)$ such that
\[\mathbf{u}=(\cos\theta, \sin\theta\cos\phi, \sin\theta\sin\phi).\]
(There are just the usual spherical coordinates on $P \cong \R^3$.) Note that $\phi$ is unique if $\theta\neq 0,\pi$. Finally, for each $\mathbf{v}\in P^{\perp}$ with $|\mathbf{v}|=1$ there exists a unique $\beta\in[0,2\pi)$ such that
\[\mathbf{v}=(\cos\beta,\sin\beta).\]

Therefore, we have local coordinates $(\alpha,\beta,\theta,\phi)$ on $\mathcal{S}^4$ which are well-defined when both $\theta\in(0,\pi)$ and $\alpha\in(0,\frac{\pi}{2})$. (Note that the restriction $\alpha \in (0, \frac{\pi}{2})$ corresponds to those $(\mathbf{x},\mathbf{y})\in\mathcal{S}^4$ where neither $\mathbf{x}$ nor $\mathbf{y}$ is zero.) This coordinate patch $U$ is geometrically given as the complement of two totally geodesic $\mathcal{S}^2$'s in $\mathcal{S}^4$, the first one being
\[ \mathcal{S}^2_{y_1 = y_2 = 0} = \{ (\mathbf{x}, \mathbf{0}) \in P \oplus P^{\perp} \, : \, |\mathbf{x}|^2 = 1 \}, \]
which is the unit sphere in the $3$-plane $y_1 = y_2 = 0$ (i.e.~$P$), corresponding to $\alpha = 0$, and the second one being 
\begin{align}
\mathcal{S}^2_{x_2 = x_3 = 0} & = \{ (\mathbf{x}, \mathbf{y}) = ( \pm \cos \alpha, 0, 0, \sin \alpha \cos \beta, \sin \alpha \sin \beta) \in P \oplus P^{\perp} \, : \, \alpha \in (0, \tfrac{\pi}{2}) \}, \label{eq:S2.x2x3}\\
& = \{ (\mathbf{x}, \mathbf{y}) = ( \cos \alpha, 0, 0, \sin \alpha \cos \beta, \sin \alpha \sin \beta) \in P \oplus P^{\perp} \, : \, \alpha \in (0, \pi) \},\nonumber
\end{align}
which is the unit sphere in the $3$-plane $x_2 = x_3 = 0$, corresponding to $\theta = 0$ or $\theta = \pi$. We note that the $\mathcal{S}^1$ where $\alpha=\frac{\pi}{2}$ is contained as the equator in $\mathcal{S}^2_{x_2=x_3=0}$ where $\theta=0$ or $\theta=\pi$. It is clear that these two totally geodesic $2$-spheres intersect at the two points $(\pm 1,0,0,0,0)$ in $\mathcal{S}^4$.

We have an orthonormal coframe on this coordinate patch on $\mathcal{S}^4$ given by
\begin{equation}\label{eq:coframe.S4}
b_0=\d\alpha,\quad b_1=\sin\alpha\d\beta, \quad b_2=\cos\alpha\d\theta,\quad b_3=\cos\alpha\sin\theta\d\phi.
\end{equation}
One may check that if $r\partial_r$ is the dilation vector field on $\R^5$ and $\vol_{\R^5}$ is the standard volume form on $\R^5$ then 
\[ (\partial_r\lrcorner\vol_{\R^5})|_{\mathcal{S}^4}=b_0\wedge b_1\wedge b_2\wedge b_3=\sin\alpha\cos^2\alpha\sin\theta \d\alpha\wedge \d\beta\wedge\d\theta\w\d\phi.\]
In other words, $\{b_0,b_1,b_2,b_3\}$ is a positively oriented orthonormal coframe on $\mathcal{S}^4$, so that by~\eqref{eq:volN} we have
\begin{equation} \label{eq:vols4}
\vol_{\mathcal{S}^4}=b_0\w b_1\w b_2\w b_3 = \sin\alpha\cos^2\alpha\sin\theta \d\alpha\wedge \d\beta\wedge\d\theta\w\d\phi.
\end{equation}

\subsection{Induced connection and vertical 1-forms}

We now define a basis $\{\Omega_1,\Omega_2,\Omega_3\}$ for the anti-self-dual 2-forms on $\mathcal{S}^4$ on the coordinate patch $U$ from the previous section as in~\eqref{eq:big-Omega-N} using~\eqref{eq:coframe.S4}. We obtain
\begin{equation} \label{eq:big-Omegas-S4}
\begin{aligned}
\Omega_1&
=\sin\alpha\d\alpha\w\d\beta-\cos^2\alpha\sin\theta\d\theta\w\d\phi,\\
\Omega_2&
=\cos\alpha\d\alpha\w\d\theta-\sin\alpha\cos\alpha\sin\theta\d\phi\w\d\beta,\\
\Omega_3&
=\cos\alpha\sin\theta\d\alpha\w\d\phi-\sin\alpha\cos\alpha\d\beta\w\d\theta.
\end{aligned}
\end{equation}
One can then compute that
\begin{align*}
\d\Omega_1&=2\sin\alpha\cos\alpha\sin\theta\d\alpha\w\d\theta\w\d\phi,\\
\d\Omega_2&=(\sin^2\alpha-\cos^2\alpha)\sin\theta\d\alpha\w\d\phi\w\d\beta-\sin\alpha\cos\alpha\cos\theta\d\theta\w\d\phi\w\d\beta,\\
\d\Omega_3&=\cos\alpha\cos\theta\d\theta\w\d\alpha\w\d\phi+(\sin^2\alpha-\cos^2\alpha)\d\alpha\w\d\beta\w\d\theta.
\end{align*}
Hence, we can check that~\eqref{eq:d-big-Omega-N} is satisfied with
\begin{align}\label{eq:rhos-S4}
2\rho_1=-\cos\alpha\d\beta+\cos\theta\d\phi,\quad
2\rho_2=\sin\alpha\d\theta,\quad 2\rho_3=\sin\alpha\sin\theta\d\phi.
\end{align}
One can then verify that~\eqref{eq:d-rho-N} is satisfied with $\kappa=1$, so the forms $\rho_1,\rho_2,\rho_3$ describe the induced connection on $M$ from the Levi-Civita connection of the round constant curvature 1 metric on $\mathcal{S}^4$.

Letting $(a_1,a_2,a_3)$ be linear coordinates on the fibres of $M$ with respect to the basis $\{\Omega_1,\Omega_2,\Omega_3\}$, we find the vertical 1-forms on $M$ over $N$ with respect to the induced connection from~\eqref{eq:zetas-N} and~\eqref{eq:rhos-S4}. They are
\begin{equation} \label{eq:zetas-S4}
\begin{aligned}
\zeta_1&
=\d a_1-a_2\sin\alpha\sin\theta\d\phi+a_3\sin\alpha\d\theta,\\
\zeta_2&
=\d a_2-a_3(-\cos\alpha\d\beta+\cos\theta\d\phi)+a_1\sin\alpha\sin\theta\d\phi,\\
\zeta_3&
=\d a_3-a_1\sin\alpha\d\theta+a_2(-\cos\alpha\d\beta+\cos\theta\d\phi).
\end{aligned}
\end{equation}

\subsection{SO(3) action} \label{sec:SO3.action}

Recall that we have split $\R^5=P\oplus P^{\perp}$ where $P$ is a 3-dimensional linear subspace. We can therefore define an $\SO(3)$ action on $\R^5$, contained in $\SO(5)$, by taking the standard $\SO(3)$ action on $P\cong\R^3$ and the trivial action on $P^{\perp}\cong\R^2$. This then induces an $\SO(3)$ action on $\mathcal{S}^4$.

Recall from~\eqref{eq:big-Omegas-S4} that
\[
\Omega_1=\sin\alpha\d\alpha\w\d\beta-\cos^2\alpha\sin\theta\d\theta\w\d\phi.
\]
We observe that $\alpha$ and $\beta$ are fixed by this particular $\SO(3)$ action. Moreover, $\sin\theta\d\theta\w\d\phi$ is the volume form of the unit sphere in $P$, which is invariant under this $\SO(3)$ action. We thus deduce that $\Omega_1$ is invariant under the $\SO(3)$ action.

One can see, for example by computing the Lie derivatives of $\Omega_2$ and $\Omega_3$ with respect to generators of the $\SO(3)$ action on $\mathcal{S}^4$, that the $\SO(3)$ action on $\mathcal{S}^4$ rotates $\Omega_2$ and $\Omega_3$. This is also clear from the fact that $\Omega_1, \Omega_2, \Omega_3$ are mutually orthogonal, so since $\Omega_1$ is fixed by the $\SO(3)$ action, the other two must be rotated amongst themselves because the action preserves the metric on $\mathcal{S}^4$ and hence preserves the induced metric on $M$.

\begin{remark} \label{remark:SO3-special}
When $\alpha = \frac{\pi}{2}$, we are no longer in the domain of our coordinate patch $U$. Thus we cannot use~\eqref{eq:big-Omegas-S4} to define the forms $\Omega_1$, $\Omega_2$, $\Omega_3$ . Such a point $p$ in $\mathcal{S}^4$ has coordinates $(0,0,0,\cos \beta, \sin \beta)$ and the cotangent space to $\mathcal{S}^4$ has orthonormal basis
\begin{equation*}
e^0 = - \sin \beta \d x^4 + \cos \beta \d x^5, \quad e^1 = \d x^1, \quad e^2 = \d x^2, \quad e^3 = \d x^3.
\end{equation*}
Thus at such a point $p$ we can take $\Omega_k = e^0 \w e^k - e^i \w e^j$ for $i,j,k$ a cyclic permutation of $1,2,3$ for our basis of anti-self-dual $2$-forms. Then one can straightforwardly compute that the $\SO(3)$ action defined above induces the standard action on the fibre $\Lambda^2_- (T_p^* \mathcal{S}^4)$. That is, if $A \in \SO(3)$, then
\begin{equation*}
A^* (a_1 \Omega_1 + a_2 \Omega_2 + a_3 \Omega_3) = \tilde a_1 \Omega_1 + \tilde a_2 \Omega_2 + \tilde a_3 \Omega_3,
\end{equation*}
where $\tilde a_i = \sum_{j=1}^3 A_{ij} a_j$.
\end{remark}

\paragraph{Orbits.} We now describe the orbits of this $\SO(3)$ action. Recall that $\alpha \neq \frac{\pi}{2}$ on our coordinate patch $U$ and that we defined $\mathcal{S}^2_{x_2=x_3=0}$ in~\eqref{eq:S2.x2x3}. Let
\[ r^2 = a_1^2 + a_2^2 + a_3^2. \]
The orbits are as follows:
\begin{itemize}
\item If $\alpha \neq \frac{\pi}{2}$ and $ a_2^2 + a_3^2 = 0$, then we are at a multiple of $\Omega_1$ in the fibre over a point in the base which does not lie in $P^{\perp}$. Since $\Omega_1$ is invariant, this orbit is an $\mathcal{S}^2$.
\item If $\alpha \neq \frac{\pi}{2}$ and $ a_2^2 + a_3^2 > 0$, then we are at a generic point in $M$, and its orbit will be $\SO(3)\cong\R\P^3$.
\item If $\alpha=\frac{\pi}{2}$ and $r=0$, then we are at the origin of the fibre over a point on the equator $\mathcal{S}^1 = \mathcal{S}^2_{x_2=x_3=0} \cap P^{\perp}$. Such a point is fixed by the $\SO(3)$ action, so its orbit is a point.
\item If $\alpha=\frac{\pi}{2}$ and $r> 0$, then we are at a non-zero vector in the fibre over a point on the equator $\mathcal{S}^1 = \mathcal{S}^2_{x_2=x_3=0} \cap P^{\perp}$. The point in the base is fixed by the $\SO(3)$ action, and by Remark~\ref{remark:SO3-special}, the vector in the fibre is acted on by $\SO(3)$ transitively on the sphere of radius $r$ in the fibre. Thus such an orbit is an $\mathcal{S}^2$.
\end{itemize}

We collect the above observations in a lemma.

\begin{lem} \label{lemma:SO3.orbits}
The orbits of the $\SO(3)$ action are given in Table~\ref{table:SO(3).orbits}.
\begin{table}[H]
\begin{center}
{\setlength{\extrarowheight}{4pt}
\begin{tabular}{|c|c|c|c|}
\hline
$\alpha$ & $r$ & $\sqrt{a_2^2+a_3^2}$ & Orbit\\[2pt]
\hline
$\neq\frac{\pi}{2}$ & $>0$ & $>0$ & $\R\P^3$\\[2pt]
\hline
$\neq\frac{\pi}{2}$ & $\geq 0$ & $0$ & $\mathcal{S}^2$\\[2pt]
\hline
$\frac{\pi}{2}$ & $>0$ & $\geq 0$ & $\mathcal{S}^2$\\[2pt]
\hline
$\frac{\pi}{2}$ & $0$ & $0$ & Point \\[2pt]
\hline
\end{tabular}
}
\end{center}
\caption{$\SO(3)$ orbits}\label{table:SO(3).orbits}
\end{table}
\end{lem}

\subsection{SO(3) adapted coordinates}

The discussion of the orbit structure in the previous section motivates us to introduce the following notation. Let
\begin{equation}\label{eq:a-S4}
a_1=t,\quad a_2=s\cos\gamma,\quad a_3=s\sin\gamma,
\end{equation}
where $s\geq 0$, $t\in\R$ and $\gamma\in [0,2\pi)$. The coordinates $s,\gamma,t$ are well-defined away from $s=0$, so we now make the further restriction going forward that $s > 0$. Observe that
\[
\d a_2=\cos\gamma\d s-s\sin\gamma \d\gamma\quad\text{and}\quad 
\d a_3=\sin\gamma\d s+s\cos\gamma\d\gamma.
\]
Hence we have
\[
\cos\gamma\d a_2+\sin\gamma\d a_3=\d s\quad\text{and}\quad 
-\sin\gamma\d a_2+\cos\gamma\d a_3=s\d\gamma.
\]

We also introduce the notation
\begin{equation}
\label{eq:sigmas-S4}
\sigma_1=\d\gamma+\cos\theta\d\phi,\quad
\sigma_2=\cos\gamma\d\theta+\sin\gamma\sin\theta\d\phi,
\quad \sigma_3=\sin\gamma\d\theta-\cos\gamma\sin\theta\d\phi,
\end{equation}
which define the standard left-invariant coframe on $\SO(3)$, with coordinates $\gamma,\theta,\phi$, and in fact are a coframe on the 3-dimensional orbits of the $\SO(3)$-action. We note that
\[
\d\left(\begin{array}{c}\sigma_1\\ \sigma_2\\ \sigma_3\end{array}\right)=\left(\begin{array}{c}\sigma_2\w\sigma_3\\ \sigma_3\w\sigma_1\\ \sigma_1\w\sigma_2\end{array}\right)
\]
and that
\begin{equation} \label{eq:s2s3-temp}
\sigma_2\w\sigma_3=-\sin\theta\d\theta\w\d\phi.
\end{equation}

We then see using~\eqref{eq:sigmas-S4} that in the coordinates~\eqref{eq:a-S4}, the 1-form $\zeta_1$ from~\eqref{eq:zetas-S4} is given by
\begin{align*}
\zeta_1 &=\d t+s\sin\alpha \sigma_3.
\end{align*}
We also compute from~\eqref{eq:zetas-S4} and~\eqref{eq:sigmas-S4} that
\begin{align*}
\cos\gamma\zeta_2+\sin\gamma\zeta_3 &=\d s-t\sin\alpha\sigma_3.
\end{align*}
Similarly, we can compute that
\begin{align*}
-\sin\gamma\zeta_2+\cos\gamma\zeta_3 &=s(\sigma_1-\cos\alpha\d\beta)-t\sin\alpha\sigma_2.
\end{align*}
We observe from~\eqref{eq:big-Omegas-S4} and~\eqref{eq:s2s3-temp} that
\[
\Omega_1=\sin\alpha\d\alpha\w\d\beta+\cos^2\alpha\sigma_2\w\sigma_3.
\]
We can further compute from~\eqref{eq:big-Omegas-S4} and~\eqref{eq:sigmas-S4} that:
\begin{align*}
\cos\gamma\Omega_2+\sin\gamma\Omega_3 &=\cos\alpha(\d\alpha\w \sigma_2-\sin\alpha\d\beta\w\sigma_3),\\
-\sin\gamma\Omega_2+\cos\gamma\Omega_3 &=\cos\alpha(-\d\alpha\w\sigma_3-\sin\alpha\d\beta\w\sigma_2).
\end{align*}
The above expressions give
\begin{align*}
\zeta_2 \w \zeta_3 & = (\cos \gamma \zeta_2 + \sin \gamma \zeta_3) \w (-\sin \gamma \zeta_2 + \cos \gamma \zeta_3) \\
& = (\d s-t\sin\alpha\sigma_3) \w (s(\sigma_1-\cos\alpha\d\beta)-t\sin\alpha\sigma_2)
\end{align*}
and
\begin{align*}
\zeta_2 \w \Omega_2 + \zeta_3 \w \Omega_3 & = ( \cos \gamma \zeta_2 + \sin \gamma \zeta_3) \w ( \cos \gamma \Omega_2 + \sin \gamma \Omega_3) \\
& \qquad{} + (-\sin \gamma \zeta_2 + \cos \gamma \zeta_3) \w (- \sin \gamma \Omega_2 + \cos \gamma \Omega_3) \\
& = \cos \alpha (\d s-t\sin\alpha\sigma_3) \w (\d\alpha\w \sigma_2-\sin\alpha\d\beta\w\sigma_3) + \\
& \qquad {} + \cos \alpha (s(\sigma_1-\cos\alpha\d\beta)-t\sin\alpha\sigma_2) \w (-\d\alpha\w\sigma_3-\sin\alpha\d\beta\w\sigma_2).
\end{align*}
Using all of the above formulae, we can write the $\GG_2$-structure $\varphi_c$ from~\eqref{eq:phic} in terms of $\sigma_1, \sigma_2, \sigma_3$ and the coordinates $s,t,\alpha,\beta$ as follows:
\begin{align}
\varphi_c & = (c+s^2+t^2)^{-\frac{3}{4}} (\d t+s\sin\alpha \sigma_3) \w (\d s-t\sin\alpha\sigma_3) \w (s(\sigma_1-\cos\alpha\d\beta)-t\sin\alpha\sigma_2) \nonumber \\
&\quad {}+2(c+s^2+t^2)^{\frac{1}{4}} (\d t+s\sin\alpha \sigma_3) \w (\sin\alpha\d\alpha\w\d\beta+\cos^2\alpha\sigma_2\w\sigma_3) \nonumber \\
& \quad{} +2(c+s^2+t^2)^{\frac{1}{4}} \cos \alpha (\d s-t\sin\alpha\sigma_3) \w (\d\alpha\w \sigma_2-\sin\alpha\d\beta\w\sigma_3) \nonumber \\
& \quad{} +2(c+s^2+t^2)^{\frac{1}{4}} \cos \alpha (s(\sigma_1-\cos\alpha\d\beta)-t\sin\alpha\sigma_2) \w (-\d\alpha\w\sigma_3-\sin\alpha\d\beta\w\sigma_2).
\label{eq:phi.c.SO3}
\end{align}
Using similar computations one can show that we can write the dual 4-form $\ast_{\varphi_c} \varphi_c$ from~\eqref{eq:psic} in terms of $\sigma_1, \sigma_2, \sigma_3$ and the coordinates $s,t,\alpha,\beta$ as follows:
\begin{align}
\ast_{\varphi_c} \varphi_c & = -4 (c+ s^2 + t^2) \sin \alpha \cos^2 \alpha \d \alpha \w \d \beta \w \sigma_2 \w \sigma_3 \nonumber \\
& \quad {}- 2 (\d s-t\sin\alpha\sigma_3) \w (s(\sigma_1-\cos\alpha\d\beta)-t\sin\alpha\sigma_2) \w (\sin\alpha\d\alpha\w\d\beta+\cos^2\alpha\sigma_2\w\sigma_3) \nonumber \\
& \quad {}- 2 \cos \alpha (\d t+s\sin\alpha \sigma_3) \w (\d s-t\sin\alpha\sigma_3) \w (-\d\alpha\w\sigma_3-\sin\alpha\d\beta\w\sigma_2) \nonumber \\
& \quad {} +2 \cos \alpha (\d t+s\sin\alpha \sigma_3) \w (s(\sigma_1-\cos\alpha\d\beta)-t\sin\alpha\sigma_2) \w (\d\alpha\w \sigma_2-\sin\alpha\d\beta\w\sigma_3). \label{eq:psi.c.SO3}
\end{align}

\subsection{SO(3)-invariant coassociative 4-folds}\label{sub:SO(3)-invariant}

We observe that 
\[
\sigma_1\w\sigma_2\w\sigma_3=-\sin\theta\d\gamma\w\d\theta\w\d\phi
\]
is a volume form on the 3-dimensional $\SO(3)$-orbits, since $\partial_\gamma,\partial_{\theta},\partial_{\phi}$ span the tangent space to the orbits. By inspection, this term does not appear in the expression~\eqref{eq:phi.c.SO3} for $\varphi_c$. Moreover, the rest of the coframe consists of invariant 1-forms. It follows that $\varphi_c$ vanishes on all 3-dimensional $\SO(3)$-orbits. Therefore, this fact certainly motivates the search for $\SO(3)$-invariant coassociative 4-folds: abstractly they exist by Harvey--Lawson's local existence theorem in~\cite{HarveyLawson}, but we can also find them explicitly and describe the fibration.

\begin{remark}
The coassociative 4-folds which are invariant under the $\SO(3)$ action studied here were already described in~\cite{Kawai}. However, the fibration and the induced structures on the coassociative 4-folds, which are the main focus here, were not examined in~\cite{Kawai}.
\end{remark}

We want to consider a 1-parameter family of 3-dimensional $\SO(3)$-orbits in $M$ defining a coassociative submanifold $N$. Thus the remaining coordinates $s,t,\alpha,\beta$ must be functions of a parameter $\tau$. If we then restrict the form $\varphi_c$ to this $4$-dimensional submanifold, from the expression~\eqref{eq:phi.c.SO3} we find that
\begin{align} \nonumber
\varphi_c|_N & = (c+s^2+t^2)^{-\frac{3}{4}} \big( -s(t \dot{t} + s \dot{s}) \sin \alpha \d \tau \w \sigma_3 \w \sigma_1 - t (t \dot{t} + s \dot{s}) \sin^2 \alpha \d \tau \w \sigma_2 \w \sigma_3 \big) \\
& \qquad {} + 2 (c+s^2+t^2)^{\frac{1}{4}} \big( s \sin \alpha \cos \alpha \dot{\beta} \d \tau \w \sigma_1 \w \sigma_2 - s \cos \alpha \dot{\alpha} \d \tau \w \sigma_3 \w \sigma_1 \big) \nonumber \\
& \qquad {} + 2 (c+s^2+t^2)^{\frac{1}{4}} \big( \dot{t} \cos^2 \alpha - 2 t \sin \alpha \cos \alpha \dot{\alpha} \big) \d \tau \w \sigma_2 \w \sigma_3.
\label{eq:phi.c.ODE.temp}
\end{align}
By considering the $\sigma_1 \w \sigma_2 \w d\tau$ term in~\eqref{eq:phi.c.ODE.temp}, we find that $\dot{\beta} = 0$, so $\beta$ must be independent of $\tau$. Thus such an $\SO(3)$-invariant coassociative submanifold must in fact be invariant under the larger symmetry group $\SO(3) \times \SO(2)$.

Hence we can write
\begin{equation}\label{eq:N.SO(3)} N=\{\big((\cos\alpha(\tau)\mathbf{u},\sin\alpha(\tau)\mathbf{v}),(t(\tau),s(\tau)\cos\gamma,s(\tau)\sin\gamma)\big) : |\mathbf{u}|=1,\,|\mathbf{v}|=1,\gamma\in[0,\pi),\,\tau\in(-\epsilon,\epsilon)\}.
\end{equation}
Because we have only three independent functions $\alpha,s,t$ in~\eqref{eq:N.SO(3)} defining the 4-dimensional submanifold $N$, we deduce that $N$ will be determined by two relations between them. We establish the following result, which also appeared in a slightly different form in~\cite{Kawai}.

\begin{prop} \label{prop:smooth.coass.S4}
Let $N$ be an $\SO(3)$-invariant coassociative 4-fold in $M$. Then we have that $\cos\alpha \not \equiv 0$ on $N$, and the following are constant on $N$:
\begin{equation}\label{eq:uv.S4}
u=t\cos\alpha \, \in\R \qquad\text{and}\qquad v=2(c+s^2+t^2)^{\frac{1}{4}}\sin\alpha \, \in[0,\infty).
\end{equation}
\end{prop}
\begin{proof}
We first observe from either Lemma~\ref{lemma:SO3.orbits} or equation~\eqref{eq:N.SO(3)} that the condition $\cos \alpha \equiv 0$ yields orbits that are less than 3-dimensional, and hence does not lead to coassociative 4-folds.

Using the formula~\eqref{eq:phi.c.ODE.temp} for $\varphi_c|_N$, we examine the coassociative condition that $\varphi_c|_N\equiv 0$. By considering the $\sigma_3 \w \sigma_1 \w d\tau$ and the $\sigma_2 \w \sigma_3 \w d\tau$ terms, the coassociative condition becomes the following pair of ordinary differential equations:
\begin{gather}
-(c+s^2+t^2)^{-\frac{3}{4}}s\sin\alpha(t\dot{t}+s\dot{s})-2(c+s^2+t^2)^{\frac{1}{4}}s\cos\alpha\dot{\alpha}=0,\label{eq:SO3.ODE.1} \\
-(c+s^2+t^2)^{-\frac{3}{4}}t\sin^2\alpha(t\dot{t}+s\dot{s})+2(c+s^2+t^2)^{\frac{1}{4}}\cos^2\alpha \dot{t}
-4(c+s^2+t^2)^{\frac{1}{4}}t\sin\alpha\cos\alpha \dot{\alpha}=0.\label{eq:SO3.ODE.2}
\end{gather}
Observe that
\[
\frac{\d}{\d\tau}\big(2(c+s^2+t^2)^{\frac{1}{4}}\sin\alpha\big)=
(c+s^2+t^2)^{-\frac{3}{4}}(s\dot{s}+t\dot{t})\sin\alpha+2(c+s^2+t^2)^{\frac{1}{4}}\cos\alpha\dot{\alpha},
\]
which is equivalent to equation~\eqref{eq:SO3.ODE.1}. Therefore, one of the conditions on $\alpha,s,t$ is that
\begin{equation*}
2(c+s^2+t^2)^{\frac{1}{4}}\sin\alpha =v\in[0,\infty)
\end{equation*}
is constant. Given that $v$ as defined above is constant, a computation yields that the remaining condition for $N$ to be coassociative from equation~\eqref{eq:SO3.ODE.2} becomes:
\begin{equation}\label{eq:SO3.ODE.3}
2(c+s^2+t^2)^{\frac{1}{4}}\cos \alpha (\cos \alpha \dot{t} - t\sin\alpha \dot{\alpha})=0.
\end{equation}
If $c>0$ then $(c+s^2+t^2)$ is never zero. If $c=0$ then $s^2+t^2=0$ corresponds to the zero section $\mathcal{S}^4$ which is excluded on $M_0 = M \setminus \mathcal{S}^4$. Thus in either case $(c+s^2+t^2)$ is never zero. Since $\cos \alpha \not \equiv 0$, we deduce that
\[
\frac{\d}{\d\tau} (t\cos\alpha)=\cos\alpha\dot{t} - t\sin\alpha\dot{\alpha} = 0
\]
and thus the condition given by~\eqref{eq:SO3.ODE.3} for $N$ to be coassociative is that 
$t\cos\alpha=u\in\R$
 is constant. 
\end{proof}

\begin{remark} \label{rmk:K-work}
The coassociative fibres where $\alpha=0$ appear in~\cite{KarigiannisMinOo} for $t = 0$ and in~\cite{KarigiannisNat} for $t \neq 0$.
\end{remark}

\subsection{The fibration} \label{sec:fibration.S4}

In this section we describe the coassociative fibrations of $M$ and $M_0$, and the topology of the fibres.

The $\mathcal{S}^1$ action given by $\beta$ degenerates when $\sin\alpha=0$, which by~\eqref{eq:uv.S4} corresponds precisely to $v=0$, both in the smooth case $M$ (where $c+s^2+t^2>0$ as $c>0$) and in the cone case $M_0$ (where $s^2+t^2>0$). We therefore see that the parameter space $(u,v,\beta)$ is naturally $\R\times\R^2=\R^3$, where $v$ is the radial coordinate in $\R^2$ and $\beta$ is the angular coordinate in $\R^2$. 

\begin{dfn} We define a projection map $\pi_{c}:M\to\R^3$ from~\eqref{eq:uv.S4} by
\begin{align}
\pi_{c}\big(&(\cos\alpha\cos\theta,\cos\alpha\sin\theta\cos\phi,\cos\alpha\sin\theta\sin\phi,\sin\alpha\cos\beta,\sin\alpha\sin\beta),(t,s\cos\gamma,s\sin\gamma)\big)\nonumber\\
&=(t\cos\alpha,2(c+s^2+t^2)^{\frac{1}{4}}\sin\alpha\cos\beta,2(c+s^2+t^2)^{\frac{1}{4}}\sin\alpha\sin\beta)=(u,v\cos\beta,v\sin\beta).\label{eq:pi-S4}
\end{align}
The map $\pi_c$ is well-defined even where $\alpha,\beta,\theta,\phi,t,s,\gamma$ do not provide coordinates on $M$. Moreover, this construction realises the smooth $\GG_2$ manifold $M$ as a coassociative fibration over $\R^3$ as described in Theorem~\ref{thm:S4.fib}. We observe that the image of the zero section $\mathcal{S}^4$ in $M$ is the 2-dimensional disc
$$\pi_{c}(\mathcal{S}^4)=\{(0,v\cos\beta,v\sin\beta):v\in [0,2c^{\frac{1}{4}}],\,\beta\in[0,2\pi)\}.$$
We can also use the same formula in~\eqref{eq:pi-S4} to define a coassociative fibration $\pi_{0}:M_0\to\R^3$ of the $\GG_2$ cone $M_0$ as in Theorem~\ref{thm:S4.fib}.
\end{dfn}

\begin{prop} \label{prop:fibration.S4}
As stated in Theorem~\ref{thm:S4.fib}, the coassociative fibres of $\pi_c$ and of $\pi_0$ are given as follows:
\begin{itemize}
\item The fibres of $\pi_0$ are all topologically $T^* \mathcal{S}^2$ except for a \emph{single} singular fibre, corresponding to the origin in $\R^3$ (where $u=v=0$). This singular fibre is topologically $\R^+ \times \R\P^3$.
\item The fibres of $\pi_c$ for $c>0$ are all topologically $T^* \mathcal{S}^2$ except for a \emph{circle} $\mathcal{S}^1_c$ of \emph{singular} fibres where $\mathcal{S}^1_c = \{ u = 0, v = 2 c^{\frac{1}{4}} , \beta \in [0,2\pi) \}$ is the boundary of the disc $\pi_c (\mathcal{S}^4)$. These singular fibres are all topologically $(\R^+ \times \R\P^3) \cup \{ 0 \}$.
\end{itemize}
\end{prop}
\begin{proof}
It is straightforward to verify that the fibres of $\pi_{c}$ are all either $T^*\mathcal{S}^2$ or $(\R^+\times\R\P^3) \cup \{ 0 \}$. (See Kawai~\cite{Kawai} for details, or in~\S\ref{sub:induced.S4} of the present paper, where we discuss the ``bolt sizes'' of each fibre, noting also the discussion on the topology of the smooth fibres at the end of \S\ref{introduction}.)

Explicitly, we show in~\S\ref{sub:induced.S4} that for $\pi_0$, the fibres are all $T^*\mathcal{S}^2$ except over the origin, where it is $\R^+\times\R\P^3$. When $c>0$, we show that there is instead a circle $\mathcal{S}^1_c$ given by the points where $(u,v)=(0,2c^{\frac{1}{4}})$ in $\R^3$ such that the fibres of $\pi_c$ over $\mathcal{S}^1_c$ are $(\R^+\times\R\P^3) \cup \{ 0 \}$, whereas the rest of the fibres are $T^*\mathcal{S}^2$.
\end{proof}

\subsection{Relation to multi-moment maps} \label{subs:multimoment.S4}

Let $X_1$, $X_2$, $X_3$ denote the vector fields which are dual to the 1-forms $\sigma_1$, $\sigma_2$, $\sigma_3$ in~\eqref{eq:sigmas-S4}, that is, such that $\sigma_i(X_j)=\delta_{ij}$. Then $X_1$, $X_2$, $X_3$ generate the $\SO(3)$ action and hence preserve both $\varphi_c$ and $*_{\varphi_c} \varphi_c$. In this section we find the multi-moment map $\nu$ for this $\SO(3)$ action for the $4$-form $*_{\varphi_c} \varphi_c$ as in~\eqref{eq:mmm4}.

Using the expression~\eqref{eq:psi.c.SO3} for the 4-form, a computation gives
\[
*_{\varphi_c} \varphi_c (X_1,X_2,X_3,\cdot)= 2 s \cos^2 \alpha \d s - 2 s^2 \sin \alpha \cos \alpha \d \alpha = \d (s^2 \cos^2 \alpha).
\]
The above computation motivates the following definition.

\begin{dfn}
We define the function $\rho$ by
\begin{equation}\label{eq:rho.S4}
\rho=s\cos\alpha
\end{equation}
so that $*_{\varphi_c} \varphi_c (X_1,X_2,X_3,\cdot) = \d (\rho^2)$. Notice that $\rho=0$ precisely when $s=0$ or $\alpha=\frac{\pi}{2}$. By Table~\ref{table:SO(3).orbits}, on a smooth fibre of the fibration this corresponds exactly to the zero section $\mathcal{S}^2$ (the ``bolt'') in $T^*\mathcal{S}^2$, and otherwise corresponds to the vertex of a singular fibre. (In fact, we will see that $\rho$ is effectively the distance from the bolts in the smooth fibres.) 
\end{dfn}

Thus we have established the following proposition.

\begin{prop} $\!\!$The multi-moment map for the $\SO(3)$ action on $*_{\varphi_c} \varphi_c$ is $\rho^2$, which maps onto $[0,\infty)$.
\end{prop} 

One can compute using the expression~\eqref{eq:phi.c.SO3} for $\varphi_c$ or by studying the proof of Proposition~\ref{prop:smooth.coass.S4} that
\begin{align*}
\varphi_c(X_2,X_3,\cdot)&=2(c+s^2+t^2)^{\frac{1}{4}}\cos\alpha\d u-t\sin\alpha \d v,\\
\varphi_c(X_3,X_1,\cdot)&=-s\d v,\\
\varphi_c(X_1,X_2,\cdot)&=\rho v\d\beta.
\end{align*}
As expected by Remark~\ref{rmk:mmm}, these are not closed forms.

\begin{remark} \label{rmk:mmmS4}
One might hope that a suitable modification or extension of the notion of multi-moment map (perhaps to the setting of bundle-valued maps) would mean that the data of $u$, $v$ and $\beta$ could be interpreted as a kind of multi-moment map for the $\SO(3)$ action on $\varphi_c$. Such extensions are known in the setting of multi-symplectic geometry. (See~\cite{Herman}, for example.)
\end{remark}

\subsection{Rewriting the package of the \texorpdfstring{G\textsubscript{2}}{G2}-structure} \label{sec:package.SO3}

We can now construct a $\GG_2$ adapted coframe that is compatible with the coassociative fibration structure as in Lemma~\ref{lemma:adapted-frame}. In this section, due to the complexity of the intermediate formulae, we do not give the step-by-step computations, but we present enough details that the reader will know how to reproduce the computation if desired.

From~\S\ref{sub:SO(3)-invariant} we know that $u, v, \beta$ are good coordinates for the $\R^3$ base of the fibration away from the line where $v=0$. Using the fundamental relation~\eqref{eq:metric-from-form}, we can obtain an explicit, albeit quite complicated, formula for $g_c$ in terms of the local coordinates $s, t, \alpha, \beta, \gamma, \theta, \phi$. We omit the particular expression here. However, one can use this expression to obtain the inverse metric $g_c^{-1}$ on 1-forms and verify that the three horizontal 1-forms $\d u, \d v, \d \beta$ and the 1-form $\d \rho$ are all mutually orthogonal. Moreover, one can also verify that $\d u, \d v, \d \beta$ are positively oriented in the sense that $(\d u)^{\sharp} \times (\d v)^{\sharp}$ is a positive multiple of $(\d \beta)^{\sharp}$.

We can therefore apply Lemma~\ref{lemma:adapted-frame} to obtain our $\GG_2$ adapted oriented orthonormal coframe
\[ \{ \hat h_1, \hat h_2, \hat h_3, \hat \varpi_0,  \hat \varpi_1,  \hat \varpi_2,  \hat \varpi_3 \} \]
where
\begin{equation*}
h_1 = \d u, \quad h_2 = \d v, \quad h_3 = \d \beta, \quad \varpi_0 = \d \rho.
\end{equation*}
It is useful to scale $\hat \varpi_1, \hat \varpi_2, \hat \varpi_3$ to unnormalized versions such that $\varpi_k = \pm \sigma_k + \text{\emph{other terms}}$, giving
\begin{align}
\varpi_1 &=\sigma_1-\frac{t\sin\alpha}{s}\sigma_2-\cos\alpha \d\beta, \nonumber \\
\varpi_2 &= -\sigma_2, \nonumber \\
\varpi_3 &=\sigma_3-\frac{t\sin\alpha}{\cos\alpha(2c\cos^2\alpha+(s^2+t^2)(1+\cos^2\alpha))}\d\rho \nonumber \\
&\qquad+\frac{s\sin\alpha}{\cos\alpha(2c\cos^2\alpha+(s^2+t^2)(1+\cos^2\alpha))}\d u, \label{eq:vertical.SO3}
\end{align}
  on the open dense subset where both $\rho$ and $v$ are positive. Then the formulae~\eqref{eq:canonical-metric},~\eqref{eq:canonical-3form}, and~\eqref{eq:canonical-4form} for $g_c$, $\varphi_c$, and $\ast_{\varphi_c} \varphi_c$, respectively become
\begin{align} \nonumber
g_c&=\frac{2(c+s^2+t^2)^{\frac{1}{2}}}{2c\cos^2\alpha+(s^2+t^2)(1+\cos^2\alpha)}\d u^2\\ \nonumber
&\qquad+\frac{(c+s^2+t^2)}{2c\cos^2\alpha+(s^2+t^2)(1+\cos^2\alpha)}\d v^2+2 (c+s^2+t^2)^{\frac{1}{2}} \sin^2 \alpha \d\beta^2\\ \nonumber 
&\qquad+\frac{2(c+s^2+t^2)^{\frac{1}{2}}}{2c\cos^2\alpha+(s^2+t^2)(1+\cos^2\alpha)}\d \rho^2+\frac{s^2}{(c+s^2+t^2)^{\frac{1}{2}}}\varpi_1^2\\ \label{eq:g.c.SO3}
&\qquad+2(c+s^2+t^2)^{\frac{1}{2}}\cos^2\alpha \varpi_2^2+\frac{2c\cos^2\alpha+(s^2+t^2)(1+\cos^2\alpha)}{(c+s^2+t^2)^{\frac{1}{2}}}\varpi_3^2
\end{align}
and
\begin{align}
\varphi_c&=\frac{2 (c+s^2+t^2) \sin \alpha}{2c\cos^2\alpha+(s^2+t^2)(1+\cos^2\alpha)}\d u\w\d v\w \d\beta\nonumber\\
& \qquad {} +\d u\wedge\Big(\frac{2s(c+s^2+t^2)^{\frac{1}{4}}}{2c\cos^2\alpha+(s^2+t^2)(1+\cos^2\alpha)}\d \rho\w \varpi_1-2(c+s^2+t^2)^{\frac{1}{4}}\cos\alpha \varpi_2\w \varpi_3 \Big) \nonumber\\
& \qquad {} +\d v\wedge\Big(\frac{2(c+s^2+t^2)\cos\alpha}{2c\cos^2\alpha+(s^2+t^2)(1+\cos^2\alpha)}\d \rho\w \varpi_2 - s \varpi_3\w \varpi_1\Big) \nonumber\\
& \qquad {} +2 (c+s^2+t^2)^{\frac{1}{4}} \sin \alpha \d\beta\wedge\left(\d \rho\w \varpi_3 - s\cos\alpha \varpi_1\w \varpi_2\right)\label{eq:phi.c.SO3.hv}
\end{align}
and
\begin{align}
*_{\varphi_c} \varphi_c&=2s\cos\alpha \d \rho\w \varpi_1\w \varpi_2\w \varpi_3\nonumber\\
& \qquad {} -\d v\w \d\beta\w \Big(\frac{2 s(c+s^2+t^2)^{\frac{3}{4}} \sin \alpha}{2c\cos^2\alpha+(s^2+t^2)(1+\cos^2\alpha)}\d \rho\w \varpi_1-2 (c+s^2+t^2)^{\frac{3}{4}}\sin \alpha \cos\alpha \varpi_2\w \varpi_3\Big) \nonumber\\
& \qquad {} -\d\beta\w \d u\w \Big(\frac{4(c+s^2+t^2) \sin \alpha \cos\alpha}{2c\cos^2\alpha+(s^2+t^2)(1+\cos^2\alpha)}\d \rho\w \varpi_2 - 2 s \sin \alpha \varpi_3\w \varpi_1\Big) \nonumber\\
& \qquad {} -\d u\w \d v\w \frac{2(c+s^2+t^2)^{\frac{3}{4}}}{2c\cos^2\alpha+(s^2+t^2)(1+\cos^2\alpha)}( \d \rho\w \varpi_3-s\cos\alpha \varpi_1\w \varpi_2 ).\label{eq:psi.c.SO3.hv}
\end{align}
Moreover, the volume form is
\begin{equation} \label{eq:vol.c.SO3}
\vol_c =-\frac{4s (c+s^2+t^2) \sin \alpha \cos\alpha}{2c\cos^2\alpha+(s^2+t^2)(1+\cos^2\alpha)}\d u\w \d v\wedge \d\beta\w \d\rho\w \sigma_1\w\sigma_2\w \sigma_3.
\end{equation}

\subsection{Riemannian and hypersymplectic geometry on the fibres} \label{sec:hypersymplectic.S4}

We now discuss the induced geometric structure on the coassociative fibres coming from the torsion-free $\GG_2$-structure. Recall that by Lemma~\ref{lem:hypersymplectic}, the ambient $\GG_2$ structure induces a hypersymplectic triple on the fibres in the sense of Definiton~\ref{dfn:hypersymplectic}.

In this section we establish the following result. (Recall that topologically $\SO(3) \cong \R\P^3$.)

\begin{prop}\label{prop:AC.CS-S4} Let $\sigma_1,\sigma_2,\sigma_3$ be as in~\eqref{eq:sigmas-S4}.
\begin{itemize}
\item[\emph{(a)}] All the coassociative fibres (both smooth and singular) in $M$ or $M_0$ are asymptotically conical with rate at least $-2$, and their asymptotic cone is topologically $\R^+\times\R\P^3$, equipped with the particular cone metric
\begin{equation}\label{eq:gAC-S4}
g_{\text{AC}}=\d R^2+\frac{R^2}{4}\sigma_1^2+\frac{R^2}{2}(\sigma_2^2+\sigma_3^2).
\end{equation}
\item[\emph{(b)}] All the \emph{singular} coassociative fibres in $M$ are conically singular with asymptotic cone which is topologically $\R^+\times\R\P^3$, equipped with the particular cone metric
\begin{equation}\label{eq:gCS-S4}
g_{\text{CS}}=\d R^2+\frac{R^2}{2}(\sigma_1^2+\sigma_2^2)+R^2\sigma_3^2.
\end{equation}
\item[\emph{(c)}] The unique singular coassociative fibre in $M_0$ is exactly a Riemannian cone, equipped with the cone metric in~\eqref{eq:gAC-S4}.
\end{itemize}
\end{prop}

\begin{remark}
We emphasize that the \emph{singular} coassociative fibres in $M = \Lambda^2_- (T^* \mathcal{S}^4)$ are \emph{not} Riemannian cones. They are both asymptotically conical and conically singular, but with different cone metrics at infinity and at the singular point.
\end{remark}

Throughout this section, we use $N$ to denote \emph{any} coassociative fibre and we use $N_0$ to denote a singular coassociative fibre.

\subsubsection{Induced metric} \label{sub:induced.S4}

We begin by describing the induced metric on the coassociative fibres $N$. The metric on the fibres is obtained by setting the horizontal $1$-forms $\d u, \d v, v \d \beta$ equal to zero. From~\eqref{eq:g.c.SO3} and the expressions~\eqref{eq:vertical.SO3} for the 1-forms $\varpi_i$, one can straightforwardly compute that the induced metric on $N$ is
\begin{align}
g_c|_N&=\frac{2(c+s^2+t^2)\cos^2\alpha+t^2\sin^2\alpha}{(c+s^2+t^2)^{\frac{1}{2}}\cos^2\alpha(2c\cos^2\alpha+(s^2+t^2)(1+\cos^2\alpha))}\d\rho^2-\frac{2t\sin\alpha}{(c+s^2+t^2)^{\frac{1}{2}}\cos\alpha}\d\rho\sigma_3\nonumber\\
& \qquad {} +\frac{2c\cos^2\alpha+(s^2+t^2)(1+\cos^2\alpha)}{(c+s^2+t^2)^{\frac{1}{2}}}\sigma_3^2
+\frac{s^2}{(c+s^2+t^2)^{\frac{1}{2}}}\sigma_1^2-\frac{2st\sin\alpha}{(c+s^2+t^2)^{\frac{1}{2}}}\sigma_1\sigma_2\nonumber\\ 
& \qquad {} +\frac{2(c+s^2+t^2)\cos^2\alpha+t^2\sin^2\alpha}{(c+s^2+t^2)^{\frac{1}{2}}}\sigma_2^2.
\label{eq:gc.N}
\end{align}
Recall that $\rho$ is a coordinate in the fibre of $N=T^*\mathcal{S}^2$ and that $\sigma_2,\sigma_3$ form a coframe on the bolt when $\rho=0$.

\paragraph{Bolt size.} We observe from Table~\ref{table:SO(3).orbits} that we have two classes of ``bolts''. These are the \emph{generic} bolts, corresponding to $s = 0$ when $\alpha \in (0, \frac{\pi}{2})$, and the \emph{non-generic} bolts, corresponding to $\alpha = \frac{\pi}{2}$.

Consider first the generic bolts. From the expression~\eqref{eq:gc.N} for the metric on the smooth fibres $N = T^* \mathcal{S}^2$, we can deduce the size of the ``bolt'', which is the zero section $\mathcal{S}^2$. The zero section corresponds to $\rho = 0$ (equivalently by~\eqref{eq:rho.S4} to $s=0$) and the size of the bolt is given (up to a factor of $4\pi$) by the coefficient of $\sigma_2^2 + \sigma_3^2$, as that is the round metric on $\mathcal{S}^2$. Thus we find from~\eqref{eq:gc.N} that the size of the bolt in $N$ is given (up to a multiple of $4\pi$) by 
\begin{equation}\label{eq:bolt.size}
a_c=\frac{2c\cos^2\alpha+t^2(1+\cos^2\alpha)}{(c+t^2)^{\frac{1}{2}}},
\end{equation}
which can actually be expressed explicitly in terms of $u,v$ (although it does not seem useful).

It is a reassuring consistency check to observe from~\eqref{eq:bolt.size} that, in the $c>0$ case, the generic bolt size vanishes if and only if $t=\cos\alpha=0$, which by~\eqref{eq:uv.S4} is equivalent to $u=0$ and $v=2c^{\frac{1}{4}}$. This corresponds to the circle of singular coassociative fibres which are $(\R^+\times\mathbb{RP}^3)\cup\{0\}$, as discussed in~\S\ref{sec:fibration.S4}. By contrast, in the $c=0$ case, the generic bolt size (again, up to a multiple of $4\pi$) is
 $$a_0=|t|(1+\cos^2\alpha),$$
which vanishes if and only if $t=0$, which by~\eqref{eq:uv.S4} corresponds to $u=v=0$, and hence to a single singular coassociative fibre over the origin, which is $\R^+\times\R\P^3$.

Now consider the non-generic bolts. From~\eqref{eq:uv.S4} we have that $u=0$ for any coassociative fibre containing a point where $\alpha = \frac{\pi}{2}$. By the discussion on $\SO(3)$ orbits preceding Lemma~\ref{lemma:SO3.orbits}, a non-generic bolt is precisely a sphere $r^2 = a_1^2 + a_2^2 + a_3^2 = \text{constant}$ in a fibre of $M = \Lambda^2_-( T^* \mathcal{S}^4)$ or $M_0 = M \setminus \mathcal{S}^4$. Hence, with respect to the Euclidean fibre metric, the bolt size (up to a factor of $4 \pi$) is just $r^2$. Therefore, by~\eqref{eq:gc}, with respect to the Bryant--Salamon metric the bolt size is $r^2(c+r^2)^{-\frac{1}{2}}$. Since $v=2(c+r^2)^{\frac{1}{4}}\sin \alpha$ is constant, taking $\alpha=\frac{\pi}{2}$ gives $c+r^2=\frac{1}{16} v^4$.  Therefore, the bolt size for the non-generic bolts is
\begin{equation*}
a_c = \Big( \frac{v^4}{16}-c \Big) \frac{4}{v^2} = \frac{v^2}{4} \Big( 1 - \frac{16 c}{v^4} \Big).
\end{equation*}
We see that when $c>0$, this is positive for $v > 2c^{\frac{1}{4}}$ and zero for $v = 2 c^{\frac{1}{4}}$. So in the $(u,v)$ half-space these non-generic bolts correspond to the coassociative fibres lying over the ray with $u=0$ starting at $(0,2c^{1/4})$. Notice also that when $c=0$, the non-generic bolt size reduces to $\frac{1}{4} v^2$, which is positive for all $v>0$ and zero for $v = 0$. These observations are again in agreement with the discussion in \S\ref{sec:fibration.S4} on the topology of the coassociative fibres.

We remark that the bolt size for the non-generic bolts is not given by the na\"ive limit obtained by putting $\alpha = \frac{\pi}{2}$ in the formula~\eqref{eq:bolt.size} for the bolt size of the generic bolts. This is not surprising as the $t$ coordinate is not valid at $\alpha = \frac{\pi}{2}$.

All of the above observations are displayed visually in plots for the bolt size, given in Figures~\ref{figure:bolt.0} and~\ref{figure:bolt.c}. These plots in the $(u,v)$ half-space show the level sets for the bolt size, with concentric curves indicating the shrinking bolt size.

\begin{figure}[H]
\begin{center}
\includegraphics[width=0.8\textwidth]{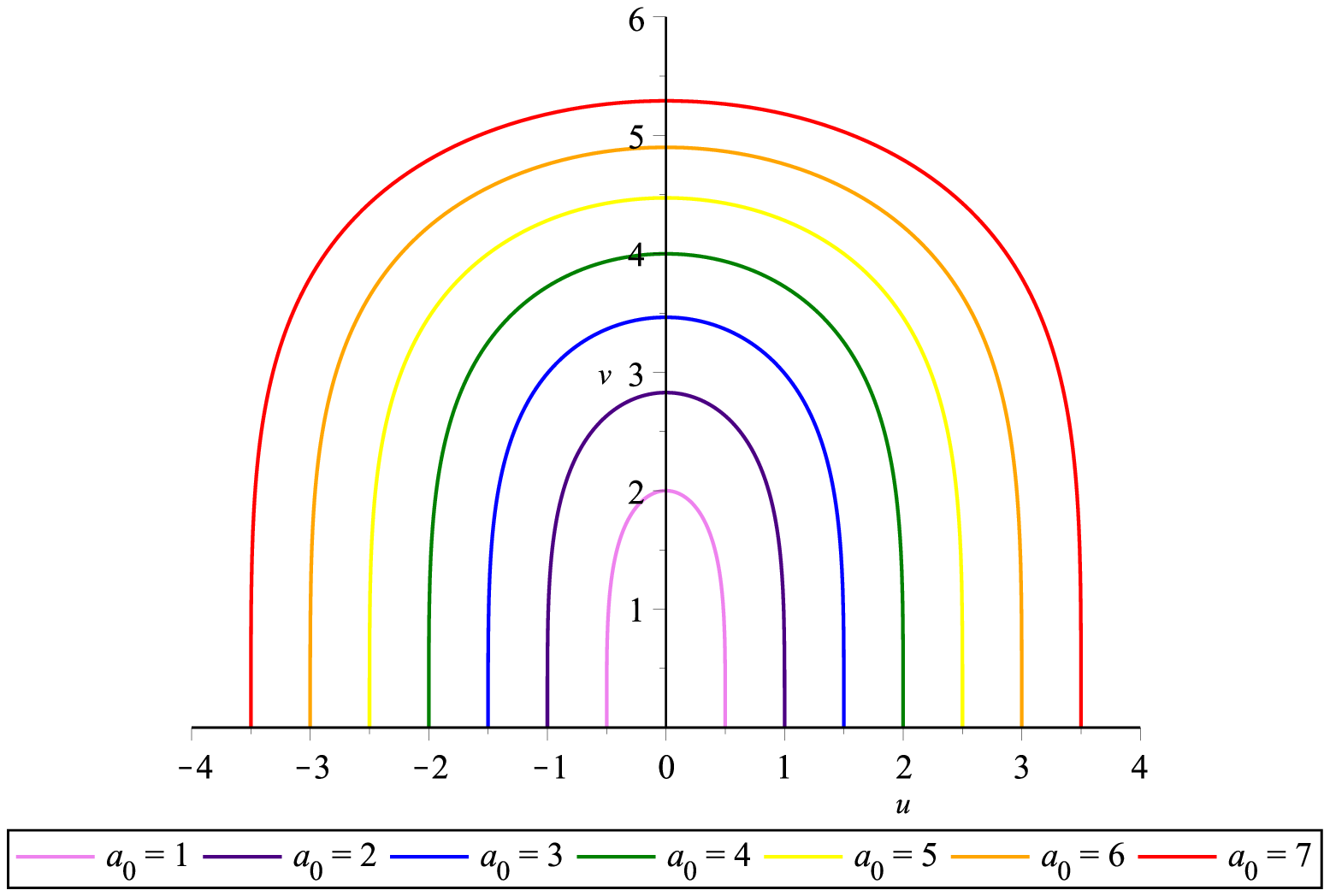}
\end{center}
\caption{Bolt size for $c=0$} \label{figure:bolt.0}
\end{figure}

\begin{figure}[H]
\begin{center}
\includegraphics[width=0.8\textwidth]{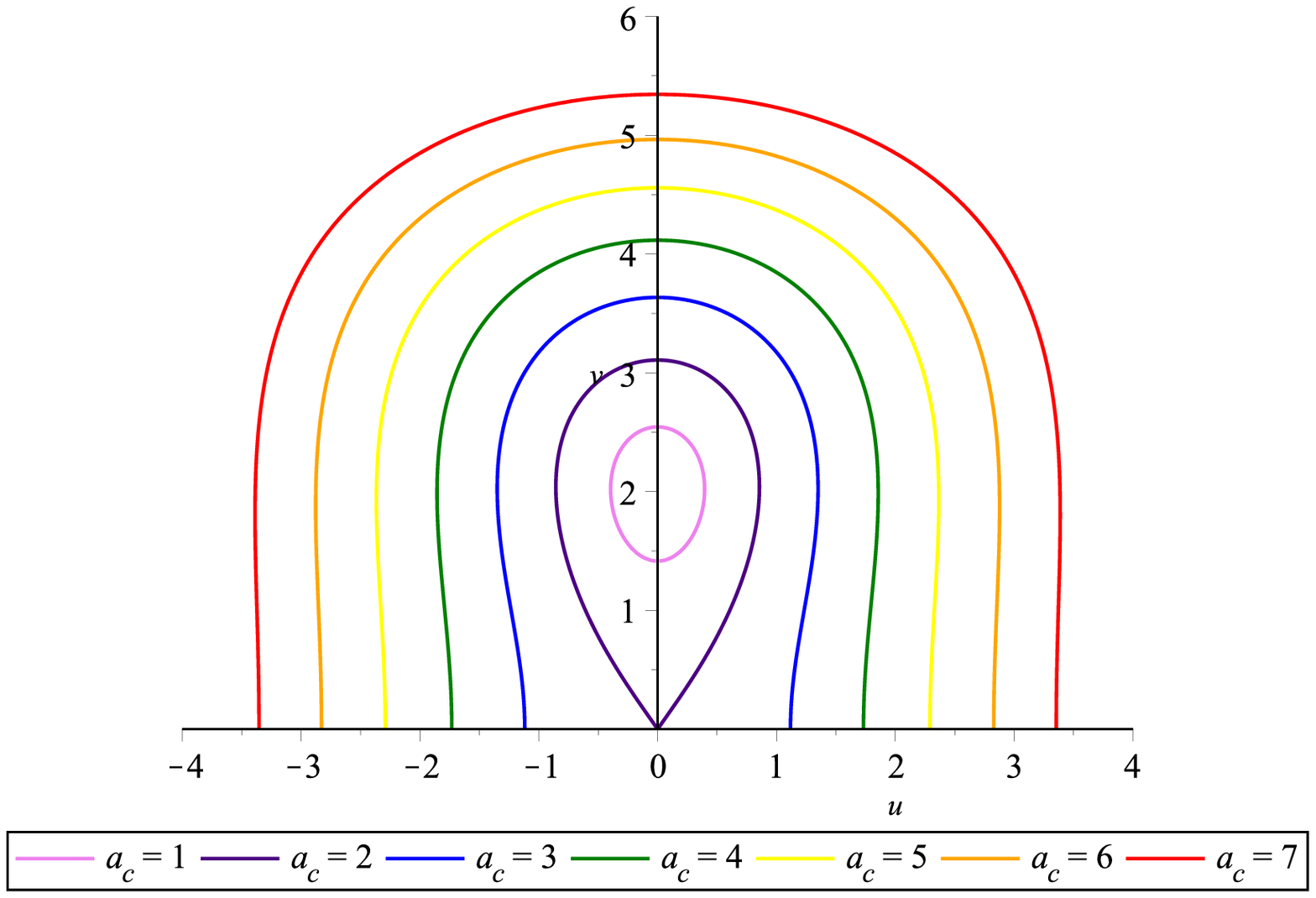}
\end{center}
\caption{Bolt size for $c=1$} \label{figure:bolt.c}
\end{figure}

\begin{remark} It is worth observing the following, by considering the rotation of Figures~\ref{figure:bolt.0} and~\ref{figure:bolt.c} about the horizontal $u$-axis. For the cone ($c=0$), the smooth fibres with a given bolt size come in a $2$-sphere family. By contrast, for the smooth case ($c>0$) there are three distinct cases: 2-spheres of smooth fibres with large bolt size, 2-tori of smooth fibres with small bolt size, and an immersed 2-sphere with a critical bolt size, which corresponds to $a_c = 2$ when $c=1$ in Figure~\ref{figure:bolt.c}. This transition clearly realizes the way in which a 2-sphere can become immersed with a double point, then become a 2-torus before collapsing to a circle.
\end{remark}

\subsubsection{Asymptotic geometry}

From~\eqref{eq:uv.S4} and~\eqref{eq:rho.S4} we obtain
\begin{equation}\label{eq:v.asymptotic}
v^4=16(c+s^2+t^2)\sin^4\alpha=16 \Big(c+\frac{u^2+\rho^2}{\cos^2\alpha} \Big) \sin^4\alpha.
\end{equation}
Hence if $\rho=s\cos\alpha
\to\infty$ on $N$, then as $u,v$ are constant, we must have that $\sin\alpha\to 0$, which means that $\cos\alpha\to 1$. Therefore, asymptotically we have that $s\approx \rho$, $\alpha\approx 0$, and $u\approx t$ must stay bounded. We can therefore compute from~\eqref{eq:gc.N} that asymptotically (that is for $\rho$ large) the metric $g_c|_N$ becomes
\begin{align} 
g_c|_N&\approx \frac{1}{(c+\rho^2+u^2)^{\frac{1}{2}}}\d\rho^2+\frac{\rho^2}{(c+\rho^2+u^2)^{\frac{1}{2}}}\sigma_1^2+2(c+\rho^2+u^2)^{\frac{1}{2}}\sigma_2^2
+2(c+\rho^2+u^2)^{\frac{1}{2}}\sigma_3^2\nonumber\\
&\approx \frac{1}{\rho}\d\rho^2+\rho\sigma_1^2+2\rho\sigma_2^2+2\rho\sigma_3^2\nonumber\\
&=\d R^2+\frac{R^2}{4}\sigma_1^2+\frac{R^2}{2}(\sigma_2^2+\sigma_3^2),\label{eq:gc.N.tgtcone}
\end{align}
where $\rho= \frac{1}{4}R^2$. This asymptotic metric is \emph{not} the flat metric on $\R^+\times \mathbb{RP}^3$, but rather differs from it by a dilation of a factor of $\sqrt{2}$ on the base $\mathcal{S}^2$ of the fibration of $\mathbb{RP}^3$ over $\mathcal{S}^2$.

In particular, we see that for all of the smooth coassociative fibres:
 \begin{itemize}
\item they have complete metrics with Euclidean volume growth;
\item they have the same asymptotic cone; 
\item they are \emph{not} Ricci flat.
\end{itemize}
(In particular the induced metric on the smooth coassociative fibres is \textbf{not} the Eguchi--Hanson metric on $T^*\mathcal{S}^2$).

Consider now the cone setting ($c=0$). When $v=0$, by~\eqref{eq:uv.S4} and~\eqref{eq:rho.S4} we \emph{exactly} have $\alpha=0$, $t=u$, $s=\rho$. Thus we see that
\begin{align}
g_0|_N&=\frac{1}{(\rho^2+u^2)^{\frac{1}{2}}}\d\rho^2+\frac{\rho^2}{(\rho^2+u^2)^{\frac{1}{2}}}\sigma_1^2+2(\rho^2+u^2)^{\frac{1}{2}}(\sigma_2^2+\sigma_3^2)\nonumber\\
&=\Big(1-\frac{16u^2}{R^4} \Big)^{-1}\d R^2+ \frac{R^2}{4}\Big(1-\frac{16u^2}{R^4} \Big)\sigma_1^2+\frac{R^2}{2}(\sigma_2^2+\sigma_3^2),\label{eq:approx.EH}
\end{align}
where $(\rho^2+u^2)^{\frac{1}{2}}= \frac{1}{4} R^2$. The above metric is very similar to the Eguchi--Hanson metric (with bolt at $R^2=4|u|=2a_0$), except that there is an additional factor of $2$ multiplying the 2-sphere metric $\sigma_2^2+\sigma_3^2$. Though this may seem like a minor variation, it destroys the hyperk\"ahler structure and gives a very different asymptotic cone, although the metric is still asymptotically conical with rate $-4$. That is, the metric differs from the cone metric by terms that are $O(R^{-4})$ for $R$ large.

We can also perform our asymptotic analysis in a more refined manner to see the rate of decay of the metric to the asymptotic cone in general. We obtain from \eq{eq:v.asymptotic} that, for large $\rho$, we have
$$\sin^2\alpha=\frac{1}{4}v^2\rho^{-1}+o(\rho^{-1}) \quad\text{and}\quad \cos^2\alpha=1-\frac{1}{4}v^2\rho^{-1}+o(\rho^{-1}).$$
Since $\rho=s\cos\alpha$ and $u=t\cos\alpha$, it follows that, for large $\rho$, we have
$$s=\rho+\frac{1}{8}v^2+o(1)\quad\text{and}\quad t=u+\frac{1}{8}uv^2\rho^{-1}+o(\rho^{-1}).$$
We thus deduce from~\eqref{eq:gc.N} that, for large $\rho$, the metric becomes
\begin{align*}
g_c|_N&=\rho^{-1}\left(1+O(\rho^{-2})\right)\d\rho^2-uv\rho^{-\frac{3}{2}}\left(1+O(\rho^{-1})\right)\d\rho\sigma_3+2\rho\left(1+O(\rho^{-2})\right)\sigma_3^2\\
&\quad {} +\rho\big(1+\tfrac{1}{8}v^2 \rho^{-1}+O(\rho^{-2})\big)\sigma_1^2+uv\rho^{2}\big(\rho^{-\frac{5}{2}}+O(\rho^{-\frac{7}{2}})\big)\sigma_1\sigma_2 +2\rho\big(1-\frac{v^2}{8}\rho^{-1}+O(\rho^{-2})\big)\sigma_2^2.
\end{align*}
Therefore, because $\rho$ is approximately $\frac{1}{4} R^2$ for large $\rho$, we conclude that $g_c|_N$ is asymptotically conical with rate $-2$, except when $v=0$, when it is asymptotically conical with rate $-4$, in agreement with the calculation in~\eqref{eq:approx.EH}.

\subsubsection{Hypersymplectic triple}

The three horizontal 1-forms $h_1 = \d u$, $h_2 = \d v$, and $h_3 = \d \beta$ are closed on $M$ and vanish when restricted to each coassociative fibre. Thus we can apply Corollary~\ref{cor:hypersymplectic} to obtain the induced hypersymplectic triple on each fibre. Using~\eqref{eq:phi.c.SO3.hv}, we deduce that the hypersymplectic triple induced on the coassociative 4-fold $N$ with $u,v,\beta$ all constant is given by
\begin{align}
\omega_1&=\Big(\frac{2s(c+s^2+t^2)^{\frac{1}{4}}}{2c\cos^2\alpha+(s^2+t^2)(1+\cos^2\alpha)}\d \rho\w \varpi_1-2(c+s^2+t^2)^{\frac{1}{4}}\cos\alpha \varpi_2\w \varpi_3\Big)\Big|_N\nonumber\\
&=\frac{2s(c+s^2+t^2)^{\frac{1}{4}}}{2c\cos^2\alpha+(s^2+t^2)(1+\cos^2\alpha)}\d\rho\wedge\sigma_1+2(c+s^2+t^2)^{\frac{1}{4}}\cos\alpha\sigma_2\wedge\sigma_3\label{eq:omega1.c}\\
\omega_2&=\Big(\frac{2(c+s^2+t^2)\cos\alpha}{2c\cos^2\alpha+(s^2+t^2)(1+\cos^2\alpha)}\d \rho\w \varpi_2-s\varpi_3\w \varpi_1\Big)\Big|_N\nonumber\\
&=-\frac{2(c+s^2+t^2)\cos^2\alpha+t^2\sin^2\alpha}{\cos\alpha(2c\cos^2\alpha+(s^2+t^2)(1+\cos^2\alpha))}\d\rho\wedge\sigma_2-s\sigma_3\wedge\sigma_1\nonumber\\
&\quad+\frac{st\sin\alpha}{\cos\alpha(2c\cos^2\alpha+(s^2+t^2)(1+\cos^2\alpha))}\d\rho\wedge\sigma_1
-t\sin\alpha\sigma_2\wedge\sigma_3\label{eq:omega2.c}\\
\omega_3&=\left( 2 (c+s^2+t^2)^{\frac{1}{4}} \sin \alpha (\d \rho\wedge \varpi_3-s\cos\alpha \varpi_1\wedge \varpi_2 ) \right)|_N\nonumber\\
&=2 (c+s^2+t^2)^{\frac{1}{4}} \sin \alpha (\d\rho\wedge\sigma_3+s\cos\alpha \sigma_1\wedge\sigma_2). \label{eq:omega3.c}
\end{align}

By construction, the $\omega_i$ are self-dual and one can verify explicitly that these 2-forms are indeed all closed on the fibres. We can also check that the matrix $\omega_i\wedge\omega_j$ of 4-forms is diagonal. One can compute:
\begin{align*}
\omega_1\w\omega_1&=\frac{8s \cos \alpha (c+s^2+t^2)^\frac{1}{2}}{2c\cos^2\alpha+(s^2+t^2)(1+\cos^2\alpha)}\d\rho\w\sigma_1\w\sigma_2\w\sigma_3,\\
\omega_2\w\omega_2 &=\frac{4s \cos \alpha (c+s^2+t^2)}{2c\cos^2\alpha+(s^2+t^2)(1+\cos^2\alpha)}\d\rho\w\sigma_1\w\sigma_2\w\sigma_3\\
\omega_3\w\omega_3&=8s(c+s^2+t^2)^{\frac{1}{2}} \sin^2 \alpha \cos \alpha \d\rho\w\sigma_1\w\sigma_2\w\sigma_3.
\end{align*}
It is straightforward to check using the formula~\eqref{eq:gc.N} that the volume form induced on $N$ is
\begin{equation} \label{eq:volN.S4}
\vol_N=2\rho\d\rho\w\sigma_1\w\sigma_2\w\sigma_3
\end{equation}
for any $c$. From this, one deduces that $\omega_i\w\omega_j=2 Q_{ij}\vol_N$ where the matrix $Q$ is given by
\begin{align}\label{eq:Q.matrix}
Q=\text{diag}\Big(\frac{2(c+s^2+t^2)^\frac{1}{2}}{2c\cos^2\alpha+(s^2+t^2)(1+\cos^2\alpha)},\frac{c+s^2+t^2}{2c\cos^2\alpha+(s^2+t^2)(1+\cos^2\alpha)}, 2 (c+s^2+t^2)^{\frac{1}{2}} \sin^2 \alpha \Big).
\end{align}
Note that the matrix $Q$ is not constant (nor is it a functional multiple of the identity matrix), and thus the hypersymplectic structure is \emph{not} hyperk\"ahler.

\subsubsection{Singular fibres}

We saw in~\S\ref{sub:induced.S4} that in $M_0$ that there is precisely one singular fibre $N_0$, corresponding to $u=v=0$, which by~\eqref{eq:uv.S4} is equivalent to $t=\alpha=0$. In this setting, $\rho=s$ by~\eqref{eq:rho.S4} and hence from~\eqref{eq:gc.N} the induced metric on the singular fibre is
\begin{align*}
g_0|_{N_0}&= \frac{1}{\rho}\d\rho^2 + \rho\sigma_1^2+2\rho\sigma_2^2+2\rho\sigma_3^2.
\end{align*}
This is \emph{exactly} the cone metric~\eqref{eq:gAC-S4}, which is the same as the asymptotic cone of all of the smooth fibres by~\eqref{eq:gc.N.tgtcone}.

By contrast, in~\S\ref{sub:induced.S4} we saw that in $M$ that there is instead a circle $\mathcal{S}^1_c$ of singular fibres $N_0$, corresponding to $u=0$ and $v=2c^{\frac{1}{4}}$. By~\eqref{eq:uv.S4}, this forces $t=0$ and
\begin{equation} \label{eq:sing.S4}
\sin\alpha=\left(\frac{c}{c+s^2}\right)^{\frac{1}{4}}, \quad \cos\alpha=\sqrt{1-\sqrt{\frac{c}{c+s^2}}}.
\end{equation}
Hence, using~\eqref{eq:gc.N}, the induced metric becomes
\begin{align*}
g_c|_{N_0}&=\frac{2(c+s^2 ) }{ 2c((c+s^2)^{\frac{1}{2}}-c^{\frac{1}{2}})+s^2(2(c+s^2)^{\frac{1}{2}}-c^{\frac{1}{2}})} \d\rho^2 +\frac{s^2}{(c+s^2)^{\frac{1}{2}}}\sigma_1^2 \\
& \quad +2((c+s^2)^{\frac{1}{2}}-c^{\frac{1}{2}})\sigma_2^2+\frac{2c((c+s^2)^{\frac{1}{2}}-c^{\frac{1}{2}})+s^2(2(c+s^2)^{\frac{1}{2}}-c^{\frac{1}{2}})}{c+s^2}\sigma_3^2. 
\end{align*}
We can see immediately that the above is \emph{not} a conical metric. However, for $s$ small, we can deduce from~\eqref{eq:sing.S4} that
$$\cos\alpha\approx \sqrt{\frac{s^2}{2c}} \quad \Rightarrow \quad \rho = s \cos \alpha \approx \frac{s^2}{\sqrt{2c}},$$
and from this we can deduce that for small $\rho$, we have
\begin{align}
g_c|_{N_0}&\approx \frac{1}{\rho\sqrt{2}}\d\rho^2+\rho\sqrt{2}\sigma_1^2+\rho\sqrt{2}\sigma_2^2+2\rho\sqrt{2}\sigma_3^2\nonumber\\
&= \d R^2+ \frac{R^2}{2}(\sigma_1^2+\sigma_2^2)+R^2\sigma_3^2,\nonumber
\end{align}
where $\rho= \frac{\sqrt{2}}{4} R^2$. The above metric \emph{is} the conical metric of~\eqref{eq:gCS-S4}, and is the same~\cite[Example 4.3]{LotayStab} as the cone metric induced on the complex cone
\begin{equation}\label{eq:cx.cone.singularity}
\{(z_1,z_2,z_3)\in\mathbb{C}^3 : z_1^2+z_2^2+z_3^2=0\}.
\end{equation}
This can clearly be realized as a coassociative cone in $\mathbb{R}^7$ and is the ordinary double point singularity for a complex surface.

\paragraph{Stability.} From the results in~\cite{LotayStab} on stability of coassociative conical singularities (or just explicitly), we can deduce that $N_0$ has an isolated conical singularity at the point where $\rho=0$. Notice that the cone metric at the singularity is over a \emph{squashed} $\mathbb{RP}^3$ with a different ratio of the Hopf fibres to the base than for the tangent cone at infinity. This ratio is 2:1 in the \emph{opposite} direction (that is, the Hopf fibres now have been dilated by a factor of $\sqrt{2}$ versus the Hopf fibration).

This is one of the coassociative conical singularities that was studied in detail in~\cite[$\S$5]{LotayStab} and we see that (as predicted by~\cite{LotayStab} and by the results in~\cite{LotayCS} on deformations of conically singular coassociative 4-folds) the singular fibres have a smooth 1-dimensional moduli space of deformations.

\subsection{Harmonic 1-form and metric on the base} \label{sec:harmonicS4}

\paragraph{Motivation.} The recent construction~\cite{JoyceKarigiannis}, by Joyce and the second author, of compact $\GG_2$ manifolds begins with an associative 3-fold $L$, which is given as the fixed point set of a non-trivial $\GG_2$ involution on a $\GG_2$ orbifold $X_0$, together with a nowhere vanishing harmonic 1-form $\lambda$ on $L$. Using $\lambda$ allows the authors of~\cite{JoyceKarigiannis} to construct a resolution of $X_0$ by gluing in a family of Eguchi--Hanson spaces $T^*\mathcal{S}^2$ along $L$ to obtain a smooth 7-manifold $X$, which then admits a torsion-free $\GG_2$-structure $\varphi$. 

The role of the 1-form $\lambda$ is such that at each point of $L$ we have an Eguchi--Hanson space attached at that point where $\frac{1}{|\lambda|} \lambda$ determines which complex structure in the 2-sphere of hyperk\"ahler complex structures is distinguished, and $|\lambda|$ determines the size of the bolt. Given $\varphi$ on $X$ and the 5-dimensional submanifold $B$ of $X$ which is the $\mathcal{S}^2$ bundle over $L$ given by the bolts in the Eguchi--Hanson spaces, one can recover $\lambda$ on $L$ by the pushforward of $\varphi_c$ along the 2-spheres in $B$. 

Since we have exhibited dense open subsets of $M$ and $M_0$ as a family of $T^*\mathcal{S}^2$ fibres  over a 3-dimensional base, we are motivated to compute the pushforward of $\varphi_c$ along the $\mathcal{S}^2$ subbundle to obtain a 1-form $\lambda_c$ which will extend naturally by zero to $\R^3$. In this section we will establish the following result.

\begin{prop}\label{prop:harmonic.1form}
Let $B$ be the $\mathcal{S}^2$-bundle over the $(u,v,\beta)$-space $\R^3$ given by the union of the bolts in the coassociative fibration of $M$ or $M_0$. Let $k_c$ and $\lambda_c$ be the pushforward of $g_c|_B$ and $\varphi_c|_B$ along the $\mathcal{S}^2$ fibres of $B$. Then $\lambda_c$ is an $\mathcal{S}^1$-invariant harmonic 1-form on $\R^3$ with respect to the metric $k_c$. That is, we have
$$\d\lambda_c=\d\!*_{k_c}\!\lambda_c=0.$$
Moreover, (up to a constant multiplicative factor) we have $|\lambda_c|_{k_c}=a_c$, where $a_c$ is the bolt size given in~\eqref{eq:bolt.size}. Consequently,
\begin{itemize}
\item[\emph{(a)}] if $c>0$, then $\lambda_c$ vanishes precisely on the circle $\mathcal{S}^1_c$ in Theorem~\ref{thm:S4.fib};
\item[\emph{(b)}] if $c=0$, then $\lambda_0$ only vanishes at the origin.
\end{itemize}
\end{prop} 

\begin{remark}
By the Poincar\'e Lemma, $\lambda_c=\d h_c$ for an $\mathcal{S}^1$-invariant function $h_c$. The behaviour of the zeros of $\lambda_c$, and thus the critical points of $h_c$, in Proposition~\ref{prop:harmonic.1form}, recalls the behaviour of homogeneous harmonic cubic polynomials on $\R^3$. This suggests that $h_0$ has a \emph{degenerate} critical point, and thus $\lambda_0$ has a \emph{degenerate} zero, at the origin. In fact, the authors verified explicitly using the expression~\eqref{eq:kc} for the metric $k_c$ which we derive below, that the Hessian of $h_0$ with respect to $k_0$ has determinant
\begin{equation} \label{eq:Hessian}
t^{\frac{3}{2}} \sin^2 \alpha (3 - \cos^2 \alpha).
\end{equation}
The origin in $(u,v,\beta)$-space corresponds to $u = 0$, $v=0$. From~\eqref{eq:uv.S4}, this implies that either $\alpha = 0, t=0$ or $\alpha = \frac{\pi}{2}, s^2 + t^2 =0$, which in either case requires $t=0$. Thus we deduce from~\eqref{eq:Hessian} that the critical point of $h_0$, which is the origin, is indeed \emph{degenerate}.
\end{remark}

\subsubsection{The 1-form \texorpdfstring{$\lambda_c$}{lambdac}}

Because $\sigma_2$, $\sigma_3$ define a coframe on the bolt, which corresponds to $\rho=0$, we simply need to examine the term in $\varphi_c$ of the form $\lambda\w\sigma_2\w\sigma_3$ in $\varphi_c$, where $\lambda$ is a horizontal 1-form. We see from the formulae~\eqref{eq:phi.c.SO3.hv} for $\varphi_c$ and~\eqref{eq:omega1.c}--\eqref{eq:omega3.c} for the hypersymplectic triple that $\lambda$ is given by
\begin{equation} \label{eq:lambda-temp}
\lambda=2(c+s^2+t^2)^{\frac{1}{4}}\cos\alpha\d u-t\sin\alpha \d v.
\end{equation}
Using~\eqref{eq:uv.S4} we write this as
\begin{align*}
\lambda&=(c+s^2+t^2)^{-\frac{3}{4}}\big(-st\sin^2\alpha\d s+\big(2(c+s^2+t^2)\cos^2\alpha-t^2\sin^2\alpha\big)\d t\big)\\
&\qquad{} -4t(c+s^2+t^2)^{\frac{1}{4}}\sin\alpha\cos\alpha\d\alpha.
\end{align*}
Setting $\rho=0$ in $\varphi_c$ is equivalent to restricting to the bundle $B$ of bolts from Proposition~\ref{prop:harmonic.1form}, and thus we obtain the pushforward $\lambda_c$ of $\varphi_c|_B$ (up to a factor of $4\pi$) as follows:
\begin{align}
\lambda_c&=\big(2(c+t^2)^{\frac{1}{4}}\cos^2\alpha-(c+t^2)^{-\frac{3}{4}}t^2\sin^2\alpha\big)\d t-4t(c+t^2)^{\frac{1}{4}}\sin\alpha\cos\alpha\d\alpha\nonumber\\
&=\frac{2}{3}\d\big(t(c+t^2)^{\frac{1}{4}}(3\cos^2\alpha-1)+l_c(t)\big)\label{eq:lambda.c}
\end{align}
where 
\begin{equation}\label{eq:l.c}
l_c'(t)=c(c+t^2)^{-\frac{3}{4}}\quad\text{and}\quad l_c(0)=0.
\end{equation}
(One can in fact give an ``explicit'' expression for $l_c(t)$ using hypergeometric functions, and write it in terms of $u$ and $v$, but we do not do this.) Notice that $\lambda_c$ is \emph{exact} and that it is invariant under the $\mathcal{S}^1$ action given by $\beta$. 

We can deduce from the first line of~\eqref{eq:lambda.c} that $\lambda_c$ is zero if and only if $t=\cos\alpha=0$, which, by~\eqref{eq:uv.S4}, corresponds to $(u,v)=(0,2c^{\frac{1}{4}})$. Hence, $\lambda_c$ vanishes precisely on $\mathcal{S}^1_c$ for $c>0$ and $\lambda_0$ only vanishes at the origin as claimed.

\begin{remark}
Notice that when $c=0$, which means $l_c=0$ by~\eqref{eq:l.c}, then $\lambda_c$ in~\eqref{eq:lambda.c} becomes
\begin{equation}\label{eq:lambda.0}
\lambda_0=\textstyle\frac{2}{3}\d\big(t|t|^{\frac{1}{2}}(3\cos^2\alpha-1)\big).
\end{equation}
Thus $\lambda_0=\d h_0$ where $h_0$ has a branch point at the origin.
\end{remark}

\subsubsection{Base metric}

To compute the pushforward $k_c$ of the metric $g_c$ on $B$ given in Proposition~\ref{prop:harmonic.1form}, it suffices to compute the induced metric on the horizontal space for the fibration when $\rho=0$. Recall that we are working in our coordinate patch $U$, where $\alpha \in (0, \frac{\pi}{2})$. Thus $\rho = 0$ corresponds to $s = 0$.

We see from~\eqref{eq:g.c.SO3} that when $\rho=0$ (so that $s=0$) the metric on the horizontal space becomes
\begin{align}\label{eq:kc}
k_c&=\frac{2(c+t^2)^{\frac{1}{2}}}{2c\cos^2\alpha+t^2(1+\cos^2\alpha)}\d u^2+\frac{c+t^2}{2c\cos^2\alpha+t^2(1+\cos^2\alpha)}\d v^2+\frac{1}{2}v^2\d\beta^2.
\end{align}
In fact, some miraculous cancellations occur if we express the base metric $k_c$ in terms of $(t,\alpha, \beta)$ where $u = t \cos \alpha$ and $v = 2 (c+t^2)^{\frac{1}{4}} \sin \alpha$. A straightforward computation reveals that
\begin{align}\label{eq:kc-2}
k_c&= \frac{1}{(c+t^2)^{\frac{1}{2}}} \d t^2 + 2 (c+t^2)^{\frac{1}{2}} \d \alpha^2 + 2(c+t^2)^{\frac{1}{2}} \sin^2 \alpha \d \beta^2.
\end{align}
In the cone case $(c=0)$, if we define $\varrho$ for $t>0$ by $\varrho^2=4t$ then~\eqref{eq:kc-2} becomes
\begin{equation*}
k_0 = \d \varrho^2+\frac{\varrho^2}{2}(\d\alpha^2+\sin^2\alpha\d\beta^2).
\end{equation*}
We see that $k_0$ is a cone metric on a half-space in $\mathbb{R}^3$. It is the cone metric on $\mathbb{R}^+\times \mathcal{S}_+^2(\frac{1}{\sqrt{2}})$. Thus overall, $k_0$ is a conical metric on two half-spaces in $\R^3$, with a common vertex.

For $c>0$, if we define $\varrho$ as above, an asymptotic expansion for large $\rho$ gives
\begin{equation} \label{eq:kc-AC}
k_c = \d \varrho^2+\frac{\varrho^2}{2}(\d\alpha^2+\sin^2\alpha\d\beta^2) + c \, O(\varrho^{-2}) = k_0 + c \, O(\varrho^{-2}).
\end{equation}
From~\eqref{eq:kc-2} and~\eqref{eq:kc-AC}, we deduce that $k_c$ is a smooth asymptotically conical metric on $\mathbb{R}^3$, which converges at infinity to its asymptotic cone $k_0$ with rate $-2$.

Finally, the induced orientation on the horizontal space can be seen from~\eqref{eq:vol.c.SO3} to be given by the 3-form $- \d u \w \d v \w \d \beta$. We can then use~\eqref{eq:kc} or~\eqref{eq:kc-2} to deduce that the volume form on the horizontal space at $\rho=0$, associated to the metric $k_c$ and the induced orientation, is 
\begin{align}\label{eq:vol.kc}
\vol_{k_c} & = -\frac{(c +t^2)^{\frac{3}{4}}v}{2c\cos^2\alpha+t^2(1+\cos^2\alpha)}\d u\w\d v\w \d\beta \\
& = 2 (c+t^2)^{\frac{1}{4}} \sin \alpha \, \d t \w \d \alpha \w \d \beta. \nonumber
\end{align}

\subsubsection{Properties of \texorpdfstring{$\lambda_c$}{lambdac}}

We have already seen that the 1-form $\lambda_c$ is closed, and in fact even exact. We now show that it is also \emph{coclosed}. One can compute using~\eqref{eq:kc} and~\eqref{eq:vol.kc} that
\begin{equation*}
*_{k_c} \d u = - \frac{(c + t^2)^{\frac{1}{4}} v}{2} \d v \w \d \beta, \qquad *_{k_c} \d v = \frac{v}{(c + t^2)^{\frac{1}{4}}} \d u \w \d \beta.
\end{equation*}
It then follows from~\eqref{eq:lambda-temp} (with $s=0$) that
\begin{align*}
*_{k_c}\lambda_c&=2(c+t^2)^{\frac{1}{4}}\cos\alpha *_{k_c}\d u -t\sin\alpha *_{k_c}\d v\\
&=-v(c+t^2)^{\frac{1}{2}}\cos\alpha\d v\w \d\beta-\frac{tv\sin\alpha}{(c+t^2)^{\frac{1}{4}}}\d u\w\d \beta.
\end{align*}
Using~\eqref{eq:uv.S4} the above can then be expressed as
\begin{equation*}
*_{k_c}\lambda_c =-4t\sin^2\alpha\cos\alpha\d t\w\d\beta+\big(-4c\sin\alpha\cos^2\alpha-2t^2\sin\alpha(2\cos^2\alpha-\sin^2\alpha)\big)\d\alpha\w\d\beta.
\end{equation*}
From the above it follows easily that
$$\d*_{k_c}\lambda_c=0.$$
Since we already know that $\d \lambda_c = 0$, we conclude that $\lambda_c$ is harmonic, as claimed. Hence, by~\eqref{eq:lambda.c}, we have
$$\Delta_{k_c}\big(t(c+t^2)^{\frac{1}{4}}(3\cos^2\alpha-1)+l_c(t)\big)=0.$$

\begin{remark}
We deduce from~\eqref{eq:lambda.0} that the function $h_0$ defined by $\lambda_0=\d h_0$ is a branched harmonic function with respect to $k_0$.
\end{remark}

We can also compute directly from~\eqref{eq:lambda-temp} (with $s=0$) and~\eqref{eq:kc} that
\begin{align*}
|\lambda_c|^2_{k_c} &=4(c+t^2)^{\frac{1}{2}}\cos^2\alpha |\d u|^2_{k_c}+t^2\sin^2\alpha|\d v|^2_{k_c}\\
&=2\cos^2\alpha(2c\cos^2\alpha+t^2(1+\cos^2\alpha))+\frac{t^2\sin^2\alpha(2c\cos^2\alpha+t^2(1+\cos^2\alpha))}{c+t^2}\\
&=\frac{(2c\cos^2\alpha+t^2(1+\cos^2\alpha))^2}{c+t^2},
\end{align*}
and hence $|\lambda_c|=a_c$ is the size of the bolt from~\eqref{eq:bolt.size}.

\subsection{Flat limit} \label{sec:flat.limit.SO3}

In this section we describe what happens to the coassociative fibration of $M$ as we take the flat limit as in $\S$\ref{subs:flat}, including the limiting harmonic 1-form.

\subsubsection{The coassociative fibration}

Recall that in taking the flat limit as in $\S$\ref{subs:flat} we simply undertake a rescaling. Thus the coassociative fibration and the $\SO(3)$-invariance is preserved along the rescaling, and in the limit we obtain an $\SO(3)$-invariant coassociative fibration of $\R^7=\C^3\oplus\R$ which is also translation invariant, because we have an additional \emph{commuting} circle action which becomes a translation action in the limit.

Recall that the coassociative fibres all have \emph{constant} $\beta$. So in the flat limit each fibre lies in a translate of $\C^3$ in $\C^3 \oplus \R$. It is well-known that a coassociative 4-fold in $\R^7$ which is contained in $\C^3$ is   a complex surface. Thus we deduce that the coassociative fibration over $\R^3$ becomes the product of $\R$ with an $\SO(3)$-invariant complex surface fibration of $\C^3$ over $\R^2=\C$. The $\SO(3)$-invariant complex surface fibration of $\C^3$ is known to be the standard Lefschetz fibration given by 
$$\pi:\C^3\to\C, \qquad \pi(z_1,z_2,z_3)=z_1^2+z_2^2+z_3^2.$$
This has to be (after a coordinate change) the flat limit of the coassociative fibration on $\Lambda^2_-(T^*\mathcal{S}^4)$.

The Lefschetz fibration has a unique singular fibre at $0$, which is a cone over $\R\P^3$, that is precisely describing the conical singularity of the singular fibres in the coassociative fibration for $\varphi_c$ as we saw in~\eqref{eq:cx.cone.singularity}. The smooth fibres are also well-known to be diffeomorphic to $T^*\mathcal{S}^2$ but (as we show below) they are \emph{not} endowed with the Eguchi--Hanson metric.

Moreover, we see explicitly that the circle of singular fibres in $M$ becomes a line in the flat limit and that the 2-tori of smooth fibres become cylinders of smooth fibres in the limit (namely, the product of a circle in $\C$ with the $\R$ factor in $\C^3\oplus\R$).

\subsubsection{Hypersymplectic geometry}

The singular fibre of the Lefschetz fibration is a cone endowed with the conical metric in~\eqref{eq:gCS-S4}, which one should note is the same as the asymptotic cone metric at the conical singularity of the singular fibres in the smooth ($c>0$) Bryant--Salamon coassociative fibration, as we noted above in the discussion surrounding equation~\eqref{eq:cx.cone.singularity}.

Let $(z_1,z_2,z_3)\in\C^3$ be a point on a smooth fibre $\pi^{-1}(w)$. Thus $w = \pi(z_1, z_2, z_3) = z_1^2 + z_2^2 + z_3^2$. We describe the fibre as a normal graph over its asymptotic cone. That is, we write $z_j=a_j+b_j$ where $a=(a_1,a_2,a_3)\in\C^3$ satisfies $a_1^2+a_2^2+a_3^2=0$ and $b=(b_1,b_2,b_3)$ is \emph{orthogonal} to $a$. We deduce that $|b|=O(|w||a|^{-1})$ when $|a|$ is large, so that, as a submanifold, a smooth fibre converges with $O(r^{-1})$ to the asymptotic cone~\eqref{eq:cx.cone.singularity}. Therefore, the induced metric on a smooth fibre is asymptotically conical with rate $-2$. This matches well with the analysis of the asymptotic behaviour for the induced metrics on the coassociative fibres in the Bryant--Salamon setting as given in Proposition~\ref{prop:AC.CS-S4}.

We can also verify the claims about the induced metric on the smooth fibres, as well as the induced hypersymplectic triple, by introducing coordinates as follows. Let
\[ \pi(z_1, z_2, z_3) = z_1^2+z_2^2+z_3^2=r^2 e^{i 2 \eta}.\]
Then we can define coordinates $\xi, \theta, \phi, \psi$ by
\begin{align*}
(z_1,z_2,z_3)&=re^{i\eta}\cosh \xi(\sin\theta\cos\phi,\sin\theta\sin\phi,\cos\theta)\\
&\quad +ire^{i\eta}\sinh\xi\big(-\sin\psi(\cos\theta\cos\phi,\cos\theta\sin\phi,-\sin\theta)+\cos\psi(-\sin\phi,\cos\phi,0)\big).
\end{align*}

If we let $\sigma_1$, $\sigma_2$, $\sigma_3$ be as in~\eqref{eq:sigmas-S4} then we can explicitly compute that the flat metric $g_{\C^3}$ on $\C^3$ is
\begin{align}
g_{\C^3}&=\cosh 2\xi \d r^2+r^2\cosh 2\xi\d\eta^2+2r\sinh2\xi \d r \d\xi+2r^2\sinh 2\xi\d\eta\sigma_3\nonumber\\
&\qquad+r^2\cosh 2\xi\d \xi^2+r^2\sinh^2\xi\sigma_1^2+r^2\cosh^2\xi\sigma_2^2+r^2\cosh 2\xi\sigma_3^2\nonumber\\
&=\frac{\d r^2+r^2\d\eta^2}{\cosh 2\xi}+\cosh 2\xi(r\d\xi+\tanh 2\xi\d r)^2
+r^2\sinh^2\xi\sigma_1^2+r^2\cosh^2\xi\sigma_2^2\nonumber\\
&\qquad +r^2\cosh 2\xi(\sigma_3+\tanh 2\xi \d\eta)^2.\label{eq:metric.C3}
\end{align}
Thus, the induced metric on the fibre $N = \pi^{-1} (w)$, which corresponds to constant $r$ and $\eta$, is 
\begin{align} \label{eq:metric.C3.N.0}
g_{\C^3}|_N&=r^2\cosh 2\xi\d \xi^2+r^2\sinh^2\xi\sigma_1^2+r^2\cosh^2\xi\sigma_2^2+r^2\cosh 2\xi\sigma_3^2.
\end{align}
Writing $\rho=\sqrt{2}r\sinh\xi$, this becomes
\begin{align} \label{eq:metric.C3.N}
g_{\C^3}|_N&=\Big(1-\frac{r^2}{2r^2+\rho^2}\Big)\d\rho^2+\frac{\rho^2}{2}\sigma_1^2+\frac{\rho^2}{2}\Big(1+\frac{2r^2}{\rho^2}\Big)\sigma_2^2+\rho^2\Big(1+\frac{r^2}{\rho^2}\Big)\sigma_3^2.
\end{align}

\begin{remark} For fixed $r>0$ we have $\rho = \sqrt{2} r \sinh \xi = 0$ if and only if $\xi = 0$. Moreover from~\eqref{eq:metric.C3.N.0} we see that $\xi = 0$ corresponds to an $\mathcal{S}^2$. Hence $\rho$ can be interpreted as a measure of the distance to the $\mathcal{S}^2$ ``bolt'' in $N\cong T^*\mathcal{S}^2$.
\end{remark}

We can see explicitly from~\eqref{eq:metric.C3.N} that the metric is asymptotically conical with rate $-2$ to the conical metric in~\eqref{eq:gCS-S4}. Moreover, it also follows from~\eqref{eq:metric.C3.N} that the volume form is
\begin{equation} \label{eq:volume.C3.N}
\vol_N=\frac{1}{4}r^4\sinh 4\xi\d\xi\w\sigma_1\w\sigma_2\w\sigma_3
=\frac{1}{2}\rho(r^2+\rho^2)\d\rho\w\sigma_1\w\sigma_2\w\sigma_3.
\end{equation}

The standard $\GG_2$-structure $\varphi_{\R^7}$ in this setting can be written as the product structure
\begin{equation}\label{eq:varphi.R7.product}
\varphi_{\R^7}=\Ree\Omega - \d x \w \omega,
\end{equation}
where $\omega$ and $\Omega$ are the standard K\"ahler form and holomorphic volume from on $\C^3$, respectively. Hence, to obtain the hypersymplectic triple on $N$, where $x$, $r$, and $\eta$ are constant, we must compute
\begin{equation} \label{eq:hyper.lefschetz}
\begin{aligned}
\omega_1&= \Big( \frac{\partial}{\partial r}\lrcorner\varphi_{\R^7}\Big) \Big|_N=
\Big( \frac{\partial}{\partial r}\lrcorner\Ree\Omega\Big) \Big|_N,\\
\omega_2&= \Big(\frac{1}{r}\frac{\partial}{\partial \eta}\lrcorner\varphi_{\R^7}\Big) \Big|_N=
\Big( \frac{1}{r}\frac{\partial}{\partial \eta}\lrcorner\Ree\Omega\Big) \Big|_N,\\
\omega_3&= \Big(\frac{\partial}{\partial x}\lrcorner\varphi_{\R^7}\Big) \Big|_N=-\omega|_N.
\end{aligned}
\end{equation}
In these coordinates, one can compute that
\begin{align*}
\omega&=\frac{r^2}{2}\left(2\cosh 2\xi\d\xi\wedge\sigma_3+\sinh 2\xi\sigma_1\w\sigma_2\right) +r\sinh 2\xi \d r\w\sigma_3\\
& \qquad {}-r^2\sinh 2\xi\d\eta\w\d\xi +r\cosh 2\xi\d r\w\d\eta
\end{align*}
and
\begin{align*}
\Omega&=-e^{3i\eta}(r^2\d r+ir^3\d\eta)\wedge\big((\sinh \xi\d\xi\w\sigma_1+\cosh\xi\sigma_2\w\sigma_3)+i(\cosh\xi\d\xi\w\sigma_2+\sinh\xi\sigma_3\w\sigma_1)\big).
\end{align*}
It follows that
\begin{align*}
\Ree\Omega&= -r^2 \cos 3\eta \d r \w (\sinh \xi\d\xi\w\sigma_1+\cosh\xi\sigma_2\w\sigma_3) + r^3 \cos 3\eta \d \eta \w (\cosh\xi\d\xi\w\sigma_2+\sinh\xi\sigma_3\w\sigma_1) \\
& \qquad {}+ r^2 \sin 3 \eta \d r \w  (\cosh\xi\d\xi\w\sigma_2+\sinh\xi\sigma_3\w\sigma_1) + r^3 \sin 3\eta \d \eta \w (\sinh \xi\d\xi\w\sigma_1+\cosh\xi\sigma_2\w\sigma_3).
\end{align*}
From the above expressions for $\omega$ and $\Ree\Omega$ and the formulas in~\eqref{eq:hyper.lefschetz} for the hypersymplectic triple, we can compute that
\begin{align*}
\omega_1&=-r^2\cos 3\eta (\sinh\xi \d\xi\w\sigma_1+\cosh\xi\sigma_2\w\sigma_3)
+r^2\sin 3\eta(\cosh\xi\d\xi\w\sigma_2+\sinh\xi\sigma_3\w\sigma_1)\\
&=-\frac{r\cos 3\eta}{\sqrt{2}}\bigg(\frac{\rho}{\sqrt{2r^2+\rho^2}}\d\rho\w\sigma_1+\sqrt{2r^2+\rho^2}\sigma_2\w\sigma_3\bigg)+\frac{r\sin 3\eta}{\sqrt{2}}\left(\d\rho\w\sigma_2+\rho\sigma_3\w\sigma_1\right),\displaybreak[0]\\ 
\omega_2&=r^2\sin 3\eta(\sinh\xi\d\xi\w\sigma_1+\cosh\xi\sigma_2\w\sigma_3)
+r^2\cos 3\eta (\cosh\xi\d\xi\w\sigma_2+\sinh\xi\sigma_3\w\sigma_1)\\
&=\frac{r\sin 3\eta}{\sqrt{2}}\bigg(\frac{\rho}{\sqrt{2r^2+\rho^2}}\d\rho\w\sigma_1+\sqrt{2r^2+\rho^2}\sigma_2\w\sigma_3\bigg)+\frac{r\cos 3\eta}{\sqrt{2}}\left(\d\rho\w\sigma_2+\rho\sigma_3\w\sigma_1\right),\displaybreak[0]\\
\omega_3&=-\frac{r^2}{2}\left(2\cosh 2\xi\d\xi\w\sigma_3+\sinh 2\xi\sigma_1\w\sigma_2\right)\\
&=-\frac{r^2+\rho^2}{\sqrt{2r^2+\rho^2}}\d\rho\w\sigma_3-\frac{\rho}{2}\sqrt{2r^2+\rho^2}\sigma_1\w\sigma_2,
\end{align*}
where $\rho=\sqrt{2} r\sinh\xi$ as before. It is easy to check that the 2-forms $\omega_i$ are closed and, using~\eqref{eq:metric.C3.N} and~\eqref{eq:volume.C3.N}, that they are also self-dual. If we now define the matrix $Q$ by $\omega_i\w\omega_j=2 Q_{ij}\vol_N$ then
$$Q=\text{diag}\Big(\frac{1}{\cosh 2\xi},\frac{1}{\cosh 2\xi},1\Big)=\text{diag}\Big(\frac{r^2}{r^2+\rho^2},\frac{r^2}{r^2+\rho^2},1\Big).$$
In particular, we see that this matrix is not constant, nor a functional multiple of the identity, and so the induced structure on the fibres is \emph{not} hyperk\"ahler.

\subsubsection{Harmonic 1-form}

In this section we compute the 1-form $\lambda$ which one obtains on $\C\oplus\R$ by restricting the $\GG_2$-structure $\varphi_{\R^7}$ in~\eqref{eq:varphi.R7.product} to the $\mathcal{S}^2$-bundle over $\C\oplus\R$, which is the ``bundle of bolts'', and then taking the pushforward. Notice from the expression~\eqref{eq:metric.C3} that the metric on the horizontal space on the $\mathcal{S}^2$-bundle (which corresponds to $\xi=0$ in the notation there, and $(r,\eta)$ are the coordinates for the horizontal directions) is just the flat metric, and so we would expect to obtain a harmonic 1-form with respect to the flat metric. (There is a subtlety here in that the coordinates in~\eqref{eq:metric.C3} for the horizontal space provide a double cover for the coordinates on the base of the fibration, and so we should actually expect a 1-form which lifts to be harmonic for the flat metric on a double cover of $\C$.)

Suppose that $\pi(z_1,z_2,z_3)=r^2 e^{i2 \eta}=w$. Then the 2-sphere $\Sigma$ in the fibre $N=\pi^{-1}(r^2 e^{i2 \eta})$ is given by the intersection of $N$ with the plane
$$P_\eta=\{(e^{i\eta}x_1,e^{i\eta}x_2,e^{i\eta}x_3):x_1,x_2,x_3\in\R\}.$$
Since $P_{\eta}$ is a Lagrangian plane we see that $\omega|_{\Sigma}=0$. Hence, $\lambda$ has no $\d x$ component. Moreover,  
\begin{align*}
\Omega|_{P_{\eta}}&=e^{3i\eta}\d x_1\w \d x_2\w\d x_3+ i e^{3i\eta}\d\eta\w(x_1\d x_2\w\d x_3 +x_2\d x_3\w\d x_1+x_3\d x_1\w\d x_2)\\
&=e^{3i\eta}\big(R^2\d R\w\vol_{\mathcal{S}^2}+ i R^3\d\eta\w\vol_{\mathcal{S}^2}\big),
\end{align*}
where $R$ is the radial coordinate in $\R^3$ and $\mathcal{S}^2$ is the unit sphere. Since $R^2=r$, we see that
$$R^2\d R=\frac{1}{3}\d R^3=\frac{1}{3}\d r^{\frac{3}{2}}.$$
Therefore, integrating $\Omega$ over $\Sigma$ will give:
$$4\pi e^{3i\eta}\Big(\frac{1}{3}\d r^\frac{3}{2}+ i r^{\frac{3}{2}}\d\eta\Big)=\frac{4\pi}{3}\d (r^{\frac{3}{2}}e^{3i\eta})=\frac{4\pi}{3}\d w^{\frac{3}{2}}.$$
Hence, up to a multiplicative constant, the 1-form $\lambda$ which is the pushforward to $\C\oplus\R$ of $\varphi_{\mathbb{R}^7}$ restricted to the 2-sphere bundle is:
$$\lambda=\d \Ree w^{\frac{3}{2}}.$$
Notice that this 1-form has a line of zeros given by $w=0$. Moreover, the function $w^{\frac{3}{2}}$ has a branch point at $w=0$. 

\begin{remarks} Branched harmonic functions such as these appear in recent work of Donaldson~\cite{Donaldson-KovalevLefschetz} when studying deformations of adiabatic Kovalev--Lefschetz fibrations, which we do not believe is coincidental, but rather suggests a link to Donaldson's work. Moreover, the fact that $\lambda$ is naturally well-defined and a regular harmonic 1-form on a double cover of $\C$ fits well with our earlier comments. See also the related recent preprint~\cite{Donaldson-new} of Donaldson.
\end{remarks}

\subsection{Circle quotient} \label{sec:circlequotientS4}

There is a particular $\mathcal{S}^1$ action on $M$ and $M_0$, that commutes with the $\SO(3)$ action considered here, which is considered in the case of $M_0$ in~\cites{AcharyaBryantSalamon, AtiyahWitten}. Our analysis in Theorem~\ref{thm:S4.fib} allows us to describe the quotient by this particular $\mathcal{S}^1$ action, which was shown in~\cite{AcharyaBryantSalamon} to be topologically $\C^3$ but with a \emph{non-flat} and \emph{non-torsion-free} $\SU(3)$-structure, as a fibration by $T^*\mathcal{S}^2$ over the upper half-plane in $\R^2$ (defined by the parameters $u \in \R$ and $v \geq 0$), with a single singular fibre which is $(\R^+\times\R\P^3)\cup\{0\}$ in $M$ and $\R^+\times\R\P^3$ in $M_0$.

In the case of $M_0$, the singular fibre is over the origin in the half-plane $v \geq 0$ and the image of the fixed points of the $\mathcal{S}^1$ action under the projection to the half-plane is the boundary $v=0$. Moreover, because this particular $\mathcal{S}^1$ action induces a rotation in each fibre of a smooth $T^* \mathcal{S}^2$ fibre, the fixed points of the $\mathcal{S}^1$ action are given by the union of the bolts (the zero section in $T^*\mathcal{S}^2$) over the boundary of the half-plane (minus the origin). If we include the vertex of the cone $M_0$, the fixed point set then becomes two copies of $\R^3$ which intersect at a point, which is the origin in each $\R^3$. 

In $M$, the singular fibres lie over the point $(u,v) = (0,2c^{\frac{1}{4}})$, and the image of the fixed points of the $\mathcal{S}^1$ action is still the line $v=0$. As above, the fixed point set  is still the union of bolts over $v=0$, but this is now topologically $\R\times\mathcal{S}^2$. 

From this perspective, it is natural to try to relate the Bryant--Salamon cone $M_0$ and its smoothings $M$ to the union of two transverse copies of $\R^3$ in $\C^3$ and a ``smoothing'' $\R\times\mathcal{S}^2$ which is analogous to how the special Lagrangian Lawlor necks smooth a pair of transverse special Lagrangian planes in $\C^3$ with the \emph{flat} $\SU(3)$-structure. (See, for example,~\cite[Chapter 7]{Harvey} for more details on Lawlor necks.)

It is also worth observing that in the flat limit the circle action becomes a translation action, and thus the circle quotient limits to the flat $\C^3$, and the coassociative fibration is the standard Lefschetz fibration over $\C$ by complex surfaces.

\subsection{Vanishing cycles and thimbles} \label{sec:VCTS4}

As we have seen in~\S\ref{sec:flat.limit.SO3}, the flat limit of the coassociative fibration we have constructed is the product of a real line and the standard Lefschetz fibration of $\mathbb{C}^3$. In fact, our coassociative fibration of $M$ is a non-compact example of what Donaldson calls a \emph{Kovalev--Lefschetz fibration}~\cite{Donaldson-KovalevLefschetz}. In the study of Lefschetz fibrations, particularly from the symplectic viewpoint, two of the central objects are so-called \emph{vanishing cycles} and \emph{thimbles}. In this section we show that we have analogues of these objects for our coassociative fibrations which, moreover, can be represented by distinguished submanifolds.

\subsubsection{Special Lagrangian vanishing cycles}

Suppose we have a path in the base of the fibration connecting a smooth fibre to a singular fibre. Following this path one finds that a certain cycle in the smooth fibre collapse to zero, hence the term ``vanishing cycle''. In our setting, just as in the standard Lefschetz fibration of $\mathbb{C}^3$, these cycles are the 2-sphere ``bolts'' in $T^*\mathcal{S}^2$. In this section we show that these 2-spheres can naturally be viewed as distinguished submanifolds in $T^*\mathcal{S}^2$ from the symplectic viewpoint.

\begin{prop}
There is a natural $\SU(2)$-structure on each smooth coassociative $T^*\mathcal{S}^2$ fibre in $M$ or $M_0$ such that the zero section $\mathcal{S}^2$ is special Lagrangian. Hence, we have special Lagrangian representatives for the ``vanishing cycles'' for the coassociative fibration.
\end{prop}

\begin{proof}
Recall the restriction of the metric $g_c$ to a coassociative fibre $N$ in~\eqref{eq:gc.N} and the hypersymplectic triple on $N$ given in~\eqref{eq:omega1.c}--\eqref{eq:omega3.c}. Let $\Sigma$ be the central 2-sphere in $N$ if $N$ is smooth. Then, by setting $\rho = s = 0$, one sees that
\begin{align*}
 g_c|_{\Sigma}&= \frac{2c\cos^2\alpha+t^2(1+\cos^2\alpha)}{(c+t^2)^\frac{1}{2}}(\sigma_2^2+\sigma_3^2)
 \end{align*}
 and hence
 \begin{align}\label{eq:vol.Sigma}
\vol_{\Sigma}&=\frac{2c\cos^2\alpha+t^2(1+\cos^2\alpha)}{(c+t^2)^\frac{1}{2}}\sigma_2\wedge\sigma_3.
\end{align}
Moreover, we have
\begin{align*}
 \omega_1|_{\Sigma}= 2(c+t^2)^{\frac{1}{4}}\cos\alpha\sigma_2\w\sigma_3,\quad
 \omega_2|_{\Sigma}= -t\sin\alpha\sigma_2\w\sigma_3,\quad
 \omega_3|_{\Sigma}=0.
\end{align*}

Thus, viewing $(N,\omega_3)$ as a symplectic manifold we see that $\Sigma$ is a \emph{Lagrangian} vanishing cycle. 

Recall from~\eqref{eq:omega3.c} that
\begin{equation} \label{eq:omega3-temp}
\omega_3=v(\d\rho\w\sigma_3+\rho\sigma_1\w\sigma_2).
\end{equation}
(Note that $\omega_3$ degenerates at the origin.) Using~\eqref{eq:omega3-temp} and~\eqref{eq:volN.S4} we have
\begin{equation*}
\omega_3\w\omega_3=2 v^2 \rho\d\rho\w\sigma_1\w\sigma_2\w\sigma_3= v^2 \vol_N.
\end{equation*}
Since $\omega_3$ is self-dual, we deduce that $\omega_3$ has length $v$, rather than  2 as one might expect. Hence, we set
$$\omega=\frac{\sqrt{2}}{v} \omega_3=\sqrt{2}(\d\rho\w\sigma_3+\rho\sigma_1\w\sigma_2).$$
Of course, $\Sigma$ is still Lagrangian with respect to $\omega$.

Using the metric $g=g_c|_N$ and $\omega$ we may now define an almost complex structure $J$ in the usual way by
$$g(\cdot , \cdot)=\omega(\cdot ,J\cdot).$$
If we let $X_1,X_2,X_3$ denote the dual vector fields to $\sigma_1,\sigma_2,\sigma_3$, so that $\sigma_i(X_j)=\delta_{ij}$, one can compute that
\begin{align*}
J\frac{\partial}{\partial\rho}&=\frac{t\sin\alpha}{\sqrt{2}(c+s^2+t^2)^{\frac{1}{2}}\cos\alpha}\frac{\partial}{\partial\rho}+\frac{2(c+s^2+t^2)\cos^2\alpha+t^2\sin^2\alpha}{\sqrt{2}(c+s^2+t^2)^{\frac{1}{2}}\cos^2\alpha(2c\cos^2\alpha+(s^2+t^2)(1+\cos^2\alpha))}X_3,\\
JX_3&=-\frac{2c\cos^2\alpha+(s^2+t^2)(1+\cos^2\alpha)}{\sqrt{2}(c+s^2+t^2)^{\frac{1}{2}}}\frac{\partial}{\partial\rho}-\frac{t\sin\alpha}{\sqrt{2}(c+s^2+t^2)^\frac{1}{2}\cos\alpha}X_3,\\
JX_1&=\frac{t\sin\alpha}{\sqrt{2}(c+s^2+t^2)^{\frac{1}{2}}\cos\alpha}X_1+\frac{s}{\sqrt{2}(c+s^2+t^2)^{\frac{1}{2}}\cos\alpha}X_2,\\
JX_2&=-\frac{2(c+s^2+t^2)\cos^2\alpha+t^2\sin^2\alpha}{\sqrt{2}s(c+s^2+t^2)^{\frac{1}{2}}\cos\alpha}X_1-\frac{t\sin\alpha}{\sqrt{2}(c+s^2+t^2)^\frac{1}{2}\cos\alpha}X_2.
\end{align*}

We observe that if we define
\begin{equation} \label{eq:ReeImm}
\begin{aligned}
\Ree\Omega&=(c+s^2+t^2)^{\frac{1}{4}}\cos\alpha\omega_1-\frac{t\sin\alpha}{(c+s^2+t^2)^{\frac{1}{2}}} \omega_2,\\
\Imm\Omega&=-\frac{1}{\sqrt{2}}\frac{t\sin\alpha}{(c+s^2+t^2)^{\frac{1}{4}}}\omega_1-\sqrt{2}\cos\alpha\omega_2,
\end{aligned}
\end{equation}
then it is clear from~\eqref{eq:Q.matrix} that
\begin{align*}
\Ree\Omega\w\omega=0=\Imm\Omega\w\omega. 
\end{align*}

Recall that we are working in our coordinate patch $U$, where $\alpha \in (0, \frac{\pi}{2})$, and the bolts are given by $s=0$. Noting that
\begin{equation*}
\frac{2(c+s^2+t^2)\cos^2\alpha+t^2\sin^2\alpha}{2c\cos^2\alpha+(s^2+t^2)(1+\cos^2\alpha)} = 1 - \frac{s^2 \sin^2 \alpha}{2c\cos^2\alpha+(s^2+t^2)(1+\cos^2\alpha)} \leq 1 \quad \text{with equality iff $s = 0$}, 
\end{equation*}
we can explicitly compute, again from~\eqref{eq:Q.matrix}, that
\begin{align*}
\Ree\Omega\w\Ree\Omega&=2\left(\frac{2(c+s^2+t^2)\cos^2\alpha+t^2\sin^2\alpha}{2c\cos^2\alpha+(s^2+t^2)(1+\cos^2\alpha)}\right)\vol_N\leq 2\vol_N,\\
\Imm\Omega\w\Imm\Omega&=2\left(\frac{2(c+s^2+t^2)\cos^2\alpha+t^2\sin^2\alpha}{2c\cos^2\alpha+(s^2+t^2)(1+\cos^2\alpha)}\right)\vol_N\leq 2\vol_N,\\
\Ree\Omega\w\Imm\Omega&=2\sin\alpha\cos\alpha\bigg(\frac{\sqrt{2}t(c+s^2+t^2)^{\frac{1}{2}}-\sqrt{2}t(c+s^2+t^2)^{\frac{1}{2}}}{2c\cos^2\alpha+(s^2+t^2)(1+\cos^2\alpha)}\bigg)\vol_N=0,
\end{align*}
with equality for the first two inequalities if and only if we are on the bolt $s=0$.

Because $\Ree\Omega$ and $\Imm\Omega$ are self-dual and orthogonal to $\omega$, they span the forms of type $(2,0)+(0,2)$ with respect to $J$ (and the same was true of $\omega_1$ and $\omega_2$). Moreover, we have shown that the norms of $\Ree\Omega$ and $\Imm\Omega$ are less than that of $\omega$.

We claim that $\Omega=\Ree\Omega+i\Imm\Omega$ is of type $(2,0)$. To see this, we first compute that
\begin{align*}
X_3 \lrcorner\omega_1&=-2(c+s^2+t^2)^{\frac{1}{4}}\cos\alpha\sigma_2,\\
(JX_3) \lrcorner\omega_1&=-\frac{\sqrt{2} s}{(c+s^2+t^2)^{\frac{1}{4}}}\sigma_1+\frac{\sqrt{2} t\sin\alpha}{(c+s^2+t^2)^{\frac{1}{4}}}\sigma_2,\\
X_3 \lrcorner\omega_2&=-s\sigma_1+t\sin\alpha\sigma_2,\\
(JX_3) \lrcorner\omega_2&=\sqrt{2}(c+s^2+t^2)^{\frac{1}{2}}\cos\alpha\sigma_2.
\end{align*}
Using the above we obtain
\begin{align*}
X_3\lrcorner \big((c+s^2+t^2)^{\frac{1}{4}}\omega_1-i\sqrt{2}\omega_2\big)&=
-2(c+s^2+t^2)^{\frac{1}{2}}\cos\alpha\sigma_2-i(-\sqrt{2}s\sigma_1+\sqrt{2}t\sin\alpha\sigma_2),\\
JX_3\lrcorner \big((c+s^2+t^2)^{\frac{1}{4}}\omega_1-i\sqrt{2}\omega_2\big)&=-\sqrt{2} s\sigma_1+\sqrt{2} t\sin\alpha\sigma_2-i 2(c+s^2+t^2)^{\frac{1}{2}}\cos\alpha\sigma_2\\
&=iX_3\lrcorner \big((c+s^2+t^2)^{\frac{1}{4}}\omega_1-i\sqrt{2}\omega_2\big).
\end{align*}
From this and similar calculations we deduce that $(c+s^2+t^2)^{\frac{1}{4}}\omega_1-i\sqrt{2}\omega_2$ is of type $(2,0)$ and hence we conclude that
$$\Omega=\Ree\Omega+i\Imm\Omega=\bigg(\cos\alpha-i\frac{t\sin\alpha}{\sqrt{2} (c+s^2+t^2)^{\frac{1}{2}}} \bigg)\big((c+s^2+t^2)^{\frac{1}{4}}\omega_1-i\sqrt{2}\omega_2\big)$$
is also of type $(2,0)$ as claimed.

Now we can easily compute from~\eqref{eq:ReeImm} that when $s=0$, the term involving $\sigma_2\w\sigma_3$ in $\Ree\Omega$ is
$$\bigg(2(c+t^2)^{\frac{1}{2}}\cos^2\alpha+\frac{t^2\sin^2\alpha}{(c+t^2)^{\frac{1}{2}}}\bigg)\sigma_2\w\sigma_3=\frac{2 c\cos^2\alpha+t^2(1+\cos^2\alpha)}{(c+t^2)^{\frac{1}{2}}}\sigma_2\w\sigma_3,$$
and that there is no term involving $\sigma_2\w\sigma_3$ in $\Imm\Omega$ when $s=0$. We deduce from~\eqref{eq:vol.Sigma} that $\Sigma$ is calibrated by $\Ree\Omega$ and hence can be thought of as ``\emph{special Lagrangian}''. Although we emphasize that $\Omega$ is \emph{not} closed and so $N$ with this $\SU(2)$-structure $(\omega,\Omega)$ is \emph{not} Calabi--Yau.

In conclusion, the vanishing cycles can be naturally viewed as \emph{special Lagrangian}.
\end{proof}

\subsubsection{Associative thimbles}

In the standard Lefschetz fibration of $\mathbb{C}^3$, one has paths of vanishing cycles terminating in a singularity. Such paths are called \emph{thimbles}. For example, if we take a path in the base of the fibration consisting of a straight line from a real point (say $1$) in $\mathbb{C}$ to the origin, then the corresponding vanishing cycles are simply given by the intersection of the real $\mathbb{R}^3$ in $\mathbb{C}^3$ with the smooth fibres, and the thimble is just a ball in a real $\mathbb{R}^3$. Hence, this thimble is special Lagrangian.

In a similar way, suppose we are in the cone $M_0$ and we take a path in the base of the fibration where $v=0$, which is when $\alpha=0$. Then $u=t$ and, when $t>0$, recall that the term in $\varphi_0$ of the form $\lambda_0\wedge\sigma_2\wedge\sigma_3$ at $s=0$ (since we consider paths of vanishing cycles) has $\lambda_0$ given by~\eqref{eq:lambda.c} as:
$$\lambda_0=2|t|^\frac{1}{2}\d t.$$
If we therefore take the straight line path of vanishing cycles from say $(u,v)=(1,0)$ to the origin in the base of the fibration (recall that the origin corresponds to the singular fibre), then we obtain an \emph{associative thimble} in the Bryant--Salamon cone $M_0$.

More generally, the term in $\varphi_c$ of the form $\lambda_c\wedge\sigma_2\wedge\sigma_3$ at $s=0$ is given by~\eqref{eq:lambda.c} as
\[
\lambda_c=\frac{2}{3}\d\big(t(c+t^2)^{\frac{1}{4}}(3\cos^2\alpha-1)+l_c(t)\big)=\d h_c,
\]
where $l_c$ is determined by~\eqref{eq:l.c}. Thus, the critical points of   $h_c$ correspond to the singular fibres in the fibration. Hence, if one considers gradient flow lines of $h_c$ starting from a point $(t,\alpha)$ in the base of the fibration (equivalently, $(u,v)$ when $s=0$) corresponding to a smooth fibre and ending at a critical point of $h_c$, the union of vanishing cycles over this gradient flow line will be an associative thimble. 

\section{Anti-self-dual 2-form bundle of \texorpdfstring{$\mathbb{CP}$\textsuperscript{2}}{CP2}} \label{sec:ASDCP2}

In this section we consider $M^7=\Lambda^2_-(T^*\mathbb{CP}^2)$ with the Bryant--Salamon torsion-free $\GG_2$-structure $\varphi_c$ from~\eqref{eq:phic}, and let $M_0=\R^+\times (\SU(3)/T^2)$ be its asymptotic cone, with the conical $\GG_2$-structure $\varphi_0$. We describe both $M$ and $M_0$ as coassociative fibrations over a 3-dimensional base. (Unlike in Theorem~\ref{thm:S4.fib}, in this case we cannot describe the base more explicitly.) Specifically we prove the following result.

\begin{thm}\label{thm:CP2.fib}
Let $M=\Lambda^2_-(T^*\C\P^2)$, let $M_0=\R^+\times(\SU(3)/T^2)=M\setminus\C\P^2$ and recall the Bryant--Salamon $\GG_2$-structures $\varphi_c$ given in Theorem~\ref{thm:BS.S4CP2} for $c\geq 0$.
 
For every $c\geq 0$, there is a 3-parameter family of $\SU(2)$-invariant coassociative 4-folds which, together with the zero section $\C\P^2$ if $c>0$, forms a fibration of $(M,\varphi_c)$ if $c>0$ or $(M_0,\varphi_0)$ if $c=0$, in the sense of Definition~\ref{dfn:fibration}.
 
\begin{itemize}
\item[\emph{(a)}] The generic fibre in the fibration is smooth and diffeomorphic to $\mathcal{O}_{\C\P^1}(-1)$. Each $\mathcal{O}_{\C\P^1}(-1)$ fibre generically intersects a 2-parameter family of other $\mathcal{O}_{\C\P^1}(-1)$ fibres in the $\C\P^1$ zero section.
\item[\emph{(b)}] The other fibres in the fibration (other than the zero section $\C\P^2$ if $c>0$) form a codimension 1 subfamily and are each diffeomorphic to $\R^+\times\mathcal{S}^3$. Moreover, these $\R^+\times\mathcal{S}^3$ fibres do not intersect any other fibres.
\end{itemize}
\end{thm}

\begin{remark} \label{rmk:CP2.fib}
There are \emph{two marked differences} between the $\C\P^2$ case in Theorem~\ref{thm:CP2.fib} and the $\mathcal{S}^4$ case in Theorem~\ref{thm:S4.fib}. First, we do not obtain a ``genuine'' fibration in Theorem~\ref{thm:CP2.fib}, unlike in Theorem~\ref{thm:S4.fib}. Rather it is a ``fibration'' in the looser sense of Definition~\ref{dfn:fibration}. Second, the singular fibres (the $\R^+\times\mathcal{S}^3$ fibres) form a codimension 1 family in the fibration of Theorem~\ref{thm:CP2.fib}, whereas the singular fibres form a codimension 2 (or 3 in the cone case) family in the fibration of Theorem~\ref{thm:S4.fib}. Thus, informally speaking, in the $\C\P^2$ case the singular fibres form a ``wall'' in the fibration; in the flat limit, this wall consists of cones and marks a topological transition of the fibres of the fibration. (See also~\S\ref{sec:flat.limit.SU2}.)
\end{remark}

A more detailed description of the fibration is given in Table~\ref{table:fibration-M} for $M$ and Table~\ref{table:fibration-M0} for $M_0$, in \S\ref{subs:summary}. In that section we also include a discussion of the structure of the subsets $M \setminus M'$ and $M_0 \setminus M_0'$ which contain the points that lie on multiple coassociative ``fibres''.

We also study the induced Riemannian geometry on the coassociative fibres. The fibres in Theorem~\ref{thm:CP2.fib}(b), which we refer to as \emph{singular} fibres, turn out to have conically singular induced Riemannian metrics, including Riemannian cones in some cases. The precise statement is in Proposition~\ref{prop:AC.CS-CP2}. 

\subsection{A coframe on \texorpdfstring{$\mathbb{CP}^2$}{CP2}}

We begin by constructing a suitable coframe on $\C\P^2$.

\subsubsection{Local coordinates}

Denote an element of $\C^*$ by $s e^{i \eta}$ where $s>0$. Consider $\mathbb{CP}^2$ as the quotient of $\mathbb{C}^3\setminus\{0\}$ by the standard diagonal $\C^*$-action given by
\[
s e^{i\eta}\cdot(z_1,z_2,z_3)=(se^{i\eta}z_1,se^{i\eta}z_2,se^{i\eta}z_3).
\]
We endow $\mathbb{CP}^2$ with the quotient metric, i.e.~the Fubini--Study metric. Choose a complex 2-dimensional linear subspace $P$ in $\C^3$. Our construction depends on this choice, which breaks the $\SU(3)$ symmetry of $\C\P^2$. Choose coordinates on 
 $\C^3$ so that $P\cong\C^2=\{z_3=0\}$ and $P^{\perp}\cong\C=\{z_1=z_2=0\}$. 

With respect to the splitting $\C^3=P\oplus P^{\perp}$ we can write points $(z_1,z_2,z_3)\in\C^3\setminus\{0\}$ uniquely as
\[
(z_1,z_2,z_3)=(sw_1,sw_2,sw_3)
\]
for $s>0$ and $(w_1,w_2,w_3)\in\mathcal{S}^5\subseteq\C^3$. Since $|w_1|^2+|w_2|^2+|w_3|^2=1$, there exists a unique $\alpha\in [0,\frac{\pi}{2}]$ such that
\[
(w_1,w_2)=\sin\alpha(u_1,u_2)\quad\text{and}\quad w_3=\cos\alpha v_3
\]
where $|u_1|^2+|u_2|^2=1$ and $|v_3|=1$. Similarly, because $|u_1|^2+|u_2|^2=1$, there exists a unique $\theta\in [0,\pi]$ such that 
\[
u_1=\cos(\textstyle\frac{\theta}{2})v_1\quad\text{and}\quad u_2=\sin(\frac{\theta}{2})v_2
\]
where $|v_1|=|v_2|=1$. Overall, we have that 
\begin{equation} \label{eq:CP2-coords}
(z_1,z_2,z_3)=(s\sin\alpha\cos(\textstyle\frac{\theta}{2})v_1,s\sin\alpha\sin(\frac{\theta}{2})v_2,s\cos\alpha v_3)
\end{equation}
for unique $s>0$, $\alpha\in [0,\frac{\pi}{2}]$, $\theta\in [0,\pi]$ and $|v_1|=|v_2|=|v_3|=1$.

We can also write
\[
v_1=e^{i(\eta+\frac{1}{2}(\psi+\phi))},\quad v_2=e^{i(\eta+\frac{1}{2}(\psi-\phi))},\quad v_3=e^{i\eta},
\]
for some $\eta,\phi,\psi\in [0,2\pi)$. 

Taking the quotient by the $s e^{i \eta}$ action, we conclude that we have local coordinates $(\alpha,\theta,\phi,\psi)$ on $\C\P^2$ which are well-defined for $\alpha\in(0,\frac{\pi}{2})$ and $\theta\in (0,\pi)$, so that we can determine the unit vectors $v_1,v_2,v_3$. From~\eqref{eq:CP2-coords} we can see that this coordinate patch $U$ is given by the complement of a point (corresponding to $\alpha=0$) and the union of three $\C\P^1$'s (corresponding to $\alpha=\frac{\pi}{2}$, $\theta=0$, and $\theta=\pi$, respectively).

\subsubsection{Fubini--Study metric}

In the coordinates defined above, we compute the Euclidean metric on $\C^3$ as follows:
\begin{align*}
\d z_1\d \overline{z}_1+\d z_2\d\overline{z}_2+\d z_3\d\overline{z}_3&=
\left(\sin\alpha\cos(\textstyle\frac{\theta}{2}) \d s + s\cos\alpha \cos(\frac{\theta}{2}) \d \alpha
-\frac{s}{2}\sin\alpha\sin(\frac{\theta}{2})\d\theta\right)^2\\
&\quad {}+s^2\sin^2\alpha\cos^2(\textstyle\frac{\theta}{2})(\d\eta+\frac{1}{2}(\d\psi+\d\phi))^2\\
&\quad {}+\left(\sin\alpha\sin(\textstyle\frac{\theta}{2}) \d s +s\cos\alpha \sin(\frac{\theta}{2}) \d \alpha 
+\frac{s}{2}\sin\alpha\cos(\frac{\theta}{2})\d\theta\right)^2\\
&\quad {}+s^2\sin^2\alpha\sin^2(\textstyle\frac{\theta}{2})(\d\eta+\frac{1}{2}(\d\psi-\d\phi))^2\\
&\quad {}+ \left(\cos\alpha \d s - s\sin\alpha\d\alpha\right)^2+s^2\cos^2\alpha\d\eta^2.
\end{align*}
The above simplifies as
\begin{align*}
\d z_1\d \overline{z}_1+\d z_2\d\overline{z}_2+\d z_3\d\overline{z}_3&=
\d s^2+s^2\d\eta^2+s^2\sin^2\alpha\d\eta(\d\psi+\cos\theta\d\phi)\\
&\quad{}+s^2\d\alpha^2+\textstyle\frac{s^2}{4}\sin^2\alpha(\d\psi +\cos\theta\d\phi)^2+\frac{s^2}{4}\sin^2\alpha(\d\theta^2+\sin^2\theta\d\phi^2)\\
&=\d s^2+s^2\big(\d\eta+\textstyle\frac{1}{2}\sin^2\alpha(\d\psi+\cos\theta\d\phi)\big)^2+s^2\d\alpha^2\\
&\quad{}+\textstyle\frac{s^2}{4}\sin^2\alpha\cos^2\alpha(\d\psi+\cos\theta\d\phi)^2+\frac{s^2}{4}\sin^2\alpha(\d\theta^2+\sin^2\theta\d\phi^2).
\end{align*}
From the above formula, we see that the horizontal space for the fibration of $\C^3\setminus\{0\}$ over $\C\P^2$ is spanned by $\d\alpha$, $\d\psi+\cos\theta\d\phi$, $\d\theta$, and $\d\phi$. Thus, we obtain the Fubini--Study metric by setting $s$ to be a constant and the connection 1-form $\d\eta + \frac{1}{2} \sin^2 \alpha (\d\psi + \cos \theta \d \phi)$ to zero. In order to obtain scalar curvature 12, which we verify below, we must take $s = \sqrt{2}$. Therefore this particular scale Fubini--Study metric on $\C\P^2$ is expressed in these coordinates as
\begin{equation} \label{eq:g.FS}
g_{FS} = 2\d\alpha^2+\frac{1}{2}\sin^2\alpha\cos^2\alpha(\d\psi+\cos\theta\d\phi)^2+\frac{1}{2}\sin^2\alpha(\d\theta^2+\sin^2\theta\d\phi^2).
\end{equation}
From the above expression for $g_{FS}$ one can verify that at $\alpha=0$ we get a point and at each of $\alpha=\frac{\pi}{2}$, $\theta = 0$, and $\theta = \pi$ we get a $\C\P^1$ with radius $\frac{1}{\sqrt{2}}$.

\subsubsection{SU(2)-invariant coframe}

If we now define
\begin{equation}\label{eq:sigmas-CP2}
\sigma_1=\d\psi+\cos\theta\d\phi,\quad \sigma_2=\cos\psi\d\theta+\sin\psi\sin\theta\d\phi,\quad \sigma_3=\sin\psi\d\theta-\cos\psi\sin\theta\d\phi,
\end{equation}
we see from~\eqref{eq:g.FS} that on $U\subseteq\C\P^2$, the following form an orthonormal coframe:
\begin{equation}\label{eq:coframe.CP2}
b_0=\sqrt{2}\d\alpha,\quad b_1=\tfrac{1}{\sqrt{2}}\sin\alpha\cos\alpha\sigma_1,\quad b_2=\tfrac{1}{\sqrt{2}}\sin\alpha\sigma_2,\quad b_3=\tfrac{1}{\sqrt{2}}\sin\alpha\sigma_3.
\end{equation}
Note that $\sigma_1,\sigma_2,\sigma_3$ define the standard left-invariant coframe on $\SU(2)\cong\mathcal{S}^3$ with coordinates $\theta,\phi,\psi$, and that we have
\[
\d\left(\begin{array}{c}\sigma_1\\ \sigma_2\\ \sigma_3\end{array}\right)
=\left(\begin{array}{c}\sigma_2\w\sigma_3\\ \sigma_3\w\sigma_1\\ \sigma_1\w\sigma_2\end{array}\right).
\]

We observe also from~\eqref{eq:coframe.CP2} that 
\begin{align*}
\d(b_0\w b_1+b_2\w b_3)&=\d(\sin\alpha\cos\alpha \d\alpha\w\sigma_1+\tfrac{1}{2}\sin^2\alpha\sigma_2\w\sigma_3)\\
&=-\sin\alpha\cos\alpha\d\alpha\w\sigma_2\w\sigma_3+\sin\alpha\cos\alpha\d\alpha\w\sigma_2\w\sigma_3=0,
\end{align*}
which says that $b_0\w b_1+b_2\w b_3$ is a closed 2-form and thus must be cohomologous to a constant multiple of the K\"ahler form. Since the K\"ahler form is self-dual with the standard orientation on $\C\P^2$, we see that $\{b_0,b_1,b_2,b_3\}$ is a positively oriented basis, and so $b_0\w b_1+b_2\w b_3$ is a multiple of the K\"ahler form.

We therefore see by~\eqref{eq:volN} and~\eqref{eq:coframe.CP2} that the volume form on $\C\P^2$ is:
\begin{align*}
\vol_{\C\P^2}&=b_0\w b_1\w b_2\w b_3=\frac{1}{2}\sin^3\alpha\cos\alpha\d\alpha\w\sigma_1\w\sigma_2\w\sigma_3 = -\frac{1}{2}\sin^3\alpha\cos\alpha\sin\theta\d\alpha\w\d\psi\w\d\theta\w\d\phi.
\end{align*}

\subsection{Induced connection and vertical 1-forms}

We now define a basis $\{\Omega_1,\Omega_2,\Omega_3\}$ for the anti-self-dual 2-forms on $\C\P^2$ on our coordinate patch $U$ from the previous section as in~\eqref{eq:big-Omega-N} using~\eqref{eq:coframe.CP2}:
\begin{equation}\label{eq:big-Omegas-CP2}
\begin{aligned}
\Omega_1&
=\sin\alpha\cos\alpha\d\alpha\w\sigma_1-\textstyle\frac{1}{2}\sin^2\alpha\sigma_2\w\sigma_3,\\
\Omega_2&
=\sin\alpha\d\alpha\w\sigma_2-\textstyle\frac{1}{2}\sin^2\alpha\cos\alpha\sigma_3\w\sigma_1,\\
\Omega_3&
=\sin\alpha\d\alpha\w\sigma_3-\textstyle\frac{1}{2}\sin^2\alpha\cos\alpha\sigma_1\w\sigma_2.
\end{aligned}
\end{equation}
It is easy to compute that
\begin{align*}
\d\Omega_1 &=-2\sin\alpha\cos\alpha\d\alpha\w\sigma_2\w\sigma_3,\\
\d\Omega_2 &=-\textstyle\frac{1}{2}(1+3\cos^2\alpha)\sin\alpha\d\alpha\w\sigma_3\w\sigma_1,\\
\d\Omega_3&=-\textstyle\frac{1}{2}(1+3\cos^2\alpha)\sin\alpha\d\alpha\w\sigma_1\w\sigma_2.
\end{align*}
It follows that~\eqref{eq:d-big-Omega-N} is satisfied where the induced connection 1-forms are
\begin{align}\label{eq:rhos-CP2}
 2\rho_1= -\tfrac{1}{2}(1+\cos^2\alpha)\sigma_1,\quad
 2\rho_2= -\cos\alpha\sigma_2,\quad
 2\rho_3=-\cos\alpha\sigma_3 .
\end{align}
One can then check that~\eqref{eq:d-rho-N} is satisfied with $\kappa=1$, so that the Fubini--Study metric on $\C\P^2$ we have chosen is indeed self-dual Einstein with scalar curvature $12$.

Let $(a_1,a_2,a_3)$ be linear coordinates on the fibres of $M$ with respect to the local basis $\{\Omega_1,\Omega_2,\Omega_3\}$. Then by~\eqref{eq:zetas-N} and~\eqref{eq:rhos-CP2}, the vertical 1-forms $\zeta_1, \zeta_2, \zeta_3$ for the induced connection over $U$ are given by
\begin{equation} \label{eq:zetas-CP2}
\begin{aligned}
\zeta_1&
=\d a_1+a_2\cos\alpha\sigma_3-a_3\cos\alpha\sigma_2,\\
\zeta_2&
=\d a_2+\textstyle\frac{1}{2}a_3(1+\cos^2\alpha)\sigma_1-a_1\cos\alpha\sigma_3,\\
\zeta_3&
=\d a_3+a_1\cos\alpha\sigma_2-\textstyle\frac{1}{2}a_2(1+\cos^2\alpha)\sigma_1.
\end{aligned}
\end{equation}

\subsection{SU(2) action}

Recall that we have split $\C^3=P\oplus P^{\perp}$ for a complex 2-dimensional linear subspace $P$ of $\C^3$. We may therefore define an action of $\SU(2)$ on $\C^3$, contained in $\SU(3)$, by having $\SU(2)$ act in the usual way on $P\cong \C^2$ and trivially on $P^{\perp}$. This action descends to $\C\P^2$ and, in fact, the 1-forms $\sigma_1,\sigma_2,\sigma_3$ given in~\eqref{eq:sigmas-CP2} are the left-invariant 1-forms for this $\SU(2)$ action on $\C\P^2$.

This $\SU(2)$ action lifts to $M$, and we see that $\Omega_1,\Omega_2,\Omega_3$ in~\eqref{eq:big-Omegas-CP2} are left-invariant, since these are defined using the left-invariant forms $\sigma_1, \sigma_2, \sigma_3$ and the invariant form $\d\alpha$. Thus, we see that $a_1$, $a_2$, and $a_3$ are all $\SU(2)$-invariant functions on the fibres of $M$ over $U\subseteq\C\P^2$. Define $r \geq 0$ by
$$r^2 = a_1^2 + a_2^2 + a_3^2.$$

\paragraph{Orbits.} Note that this $\SU(2)$ action is \emph{free} on the chart $U\subseteq \C\P^2$, as it corresponds to the usual left multiplication of $\SU(2) \cong \Sp(1)$ on non-zero quaternions. We can therefore describe the orbits of the $\SU(2)$-action on $M$ as follows.
\begin{itemize}
\item If $\alpha\in(0,\frac{\pi}{2})$ then since the $\SU(2)$ action is free on $U$, the orbit is $\mathcal{S}^3\cong\SU(2)$.
\item If $\alpha=0$ and $r=0$, then we are at a fixed point of the $\SU(2)$ action, so the orbit is a point.
\item If $\alpha=0$ and $r>0$, then we have the full $\SU(2)$ acting on the fibre, so the orbit is an $\mathcal{S}^2 \cong \C\P^1$. 
\item If $\alpha=\frac{\pi}{2}$ and $r=0$ then we are on a $\C\P^1$ in $\C\P^2$ which is acted on transitively by the $\SU(2)$ action, so the orbit is $\C\P^1\cong\mathcal{S}^2$.
\item If $\alpha=\frac{\pi}{2}$, $r>0$, and $a_2^2+a_3^2=0$, then the orbit is just $\C\P^1=\mathcal{S}^2$ as in the previous item.
\item If $\alpha=\frac{\pi}{2}$, $r>0$, and $a_2^2+a_3^2>0$, then the orbit is an $\mathcal{S}^1$-bundle over $\C\P^1$, which one sees is just $\mathcal{S}^3\cong\SU(2)$ because there is no stabilizer for the $\SU(2)$ action in this case.
\end{itemize}

We summarise all the above observations in the following lemma.

\begin{lem}
The orbits of the $\SU(2)$ action are given in Table~\ref{table:SU2.orbits}.
\begin{table}[H]
\begin{center}
{\setlength{\extrarowheight}{4pt}
\begin{tabular}{|c|c|c|c|}
\hline
$\alpha$ & $r$ & $\sqrt{a_2^2+a_3^2}$ & Orbit \\[2pt]
\hline
$\in(0,\frac{\pi}{2})$ & $\geq 0$ & $\geq 0$ & $\mathcal{S}^3$\\[2pt]
\hline
$0$ & $>0$ & $\geq 0$ & $\C\P^1$ \\[2pt]
\hline
$0$ & $0$ & $0$ & Point \\[2pt]
\hline
$\frac{\pi}{2}$ & $>0$ & $>0$ & $\mathcal{S}^3$\\[2pt]
\hline
$\frac{\pi}{2}$ & $>0$ & $0$ & $\C\P^1$\\[2pt]
\hline
$\frac{\pi}{2}$ & $0$ & $0$ & $\C\P^1$ \\[2pt]
\hline
\end{tabular}}
\end{center}
\caption{$\SU(2)$ orbits}\label{table:SU2.orbits}
\end{table}
\end{lem}

\subsection{SU(2) adapted coordinates}

Even though the functions $a_1$, $a_2$, $a_3$ are all $\SU(2)$-invariant, we notice that there is an additional $\U(1)$ symmetry in $\Omega_2,\Omega_3$, and hence in $a_2,a_3$. This motivates us to write
\begin{equation}\label{eq:r.gamma.beta}
a_1=r\cos\gamma,\quad a_2=r\sin\gamma\cos\beta,\quad a_3=r\sin\gamma\sin\beta
\end{equation}
for $\gamma\in [0,\pi]$ and $\beta\in [0,2\pi)$. We then have
\begin{equation}\label{eq:dak.CP2}
\begin{aligned}
\d a_1&=\cos\gamma\d r-r\sin\gamma\d\gamma,\\
\d a_2&=\sin\gamma\cos\beta\d r +r\cos\gamma \cos\beta\d\gamma-r\sin\gamma\sin\beta\d\beta,\\
\d a_3&=\sin\gamma\sin\beta\d r+r\cos\gamma\sin\beta\d\gamma+r\sin\gamma\cos\beta\d\beta.
\end{aligned}
\end{equation}

Using~\eqref{eq:zetas-CP2},~\eqref{eq:r.gamma.beta}, and~\eqref{eq:dak.CP2}, a lengthy computation gives
\begin{align*}
\zeta_1\w\zeta_2\w\zeta_3 &= r^2 \sin \gamma \d r \w \d \gamma \w \d \beta - \tfrac{1}{2} r^2 \sin \gamma (1 + \cos^2 \alpha) \d r \w \d \gamma \w \sigma_1 \\
& \quad {} + r^2 \cos \gamma \cos \beta \cos \alpha \d r\w \d \gamma \w \sigma_2 - r^2 \sin \gamma \sin \beta \cos \alpha \d r \w \d \beta \w \sigma_2 \\
& \quad {} + r^2 \cos \gamma \sin \beta \cos \alpha \d r\w \d \gamma \w \sigma_3 + r^2 \sin \gamma \cos \beta \cos \alpha \d r\w \d \beta \w \sigma_3 \\
& \quad {} + r^2 \cos \gamma \cos^2 \alpha \d r \w \sigma_2 \w \sigma_3 + \tfrac{1}{2} r^2 \sin \gamma \cos \beta \cos \alpha (1 + \cos^2 \alpha) \d r \w \sigma_3 \w \sigma_1 \\
& \quad {}+ \tfrac{1}{2} r^2 \sin \gamma \sin \beta \cos \alpha (1 + \cos^2 \alpha) \d r \w \sigma_1 \w \sigma_2.
\end{align*}
Similarly from~\eqref{eq:big-Omegas-CP2} and~\eqref{eq:zetas-CP2} more lengthy computations give
\begin{align*}
\zeta_1\w\Omega_1&=\cos \gamma \sin \alpha \cos \alpha \d r\w \d \alpha \w \sigma_1 - r \sin \gamma \sin \alpha \cos \alpha \d \gamma \w \d \alpha \w \sigma_1 \\
& \quad {} - \tfrac{1}{2} \cos \gamma \sin^2 \alpha \d r \w \sigma_2 \w \sigma_3 + \tfrac{1}{2} r \sin \gamma \sin^2 \alpha \d \gamma \w \sigma_2 \w \sigma_3 \\
& \quad {} - r \sin \gamma \cos \beta \sin \alpha \cos^2 \alpha \d \alpha \w \sigma_3 \w \sigma_1 - r \sin \gamma \sin \beta \sin \alpha \cos^2 \alpha \d \alpha \w \sigma_1 \w \sigma_2,
\\
\zeta_2\w\Omega_2&=\sin \gamma \cos \beta \sin \alpha \d r \w \d \alpha \w \sigma_2 + r \cos \gamma \cos \beta \sin \alpha \d \gamma \w \d \alpha \w \sigma_2 \\
& \quad {} + r \sin \gamma \sin \beta \sin \alpha \d \alpha \w \d \beta \w \sigma_2 - \tfrac{1}{2} \sin \gamma \cos \beta \sin^2 \alpha \cos \alpha \d r \w \sigma_3 \w \sigma_1 \\
& \quad {} - \tfrac{1}{2} r \cos \gamma \cos \beta \sin^2 \alpha \cos \alpha \d \gamma \w \sigma_3 \w \sigma_1 + \tfrac{1}{2} r \sin \gamma \sin \beta \sin^2 \alpha \cos \alpha \d \beta \w \sigma_3 \w \sigma_1 \\
& \quad {} - \tfrac{1}{2} r \sin \gamma \sin \beta \sin \alpha (1 + \cos^2 \alpha) \d \alpha \w \sigma_1 \w \sigma_2 - r \cos \gamma \sin \alpha \cos \alpha \d \alpha \w \sigma_2 \w \sigma_3,
\\
\zeta_3\w\Omega_3&= \sin \gamma \sin \beta \sin \alpha \d r\w \d \alpha \w \sigma_3 + r \cos \gamma \sin \beta \sin \alpha \d \gamma \w \d \alpha \w \sigma_3 \\
& \quad {} - r \sin \gamma \cos \beta \sin \alpha \d \alpha \w \d \beta \w \sigma_3 - \tfrac{1}{2} \sin \gamma \sin \beta \sin^2 \alpha \cos \alpha \d r \w \sigma_1 \w \sigma_2 \\
& \quad {} - \tfrac{1}{2} r \cos \gamma \sin \beta \sin^2 \alpha \cos \alpha \d \gamma \w \sigma_1 \w \sigma _2 - \tfrac{1}{2} r \sin \gamma \cos \beta \sin^2 \alpha \cos \alpha \d \beta \w \sigma_1 \w \sigma_2 \\
& \quad {} - r \cos \gamma \sin \alpha \cos \alpha \d \alpha \w \sigma_2 \w \sigma_3 - \tfrac{1}{2} r \sin \gamma \cos \beta \sin \alpha (1 + \cos^2 \alpha) \d \alpha \w \sigma_3 \w \sigma_1.
\end{align*}
The above expressions can be substituted into~\eqref{eq:phic} to obtain the 3-form
\begin{equation} \label{eq:phic.CP2}
\varphi_c=(c+r^2)^{-\frac{3}{4}}\zeta_1\w\zeta_2\w\zeta_3+2(c+r^2)^{\frac{1}{4}}(\zeta_1\w\Omega_1+\zeta_2\w\Omega_2+\zeta_3\w\Omega_3).
\end{equation}
In fact the expression for $\varphi_c$ simplifies considerably if we write it in terms of a $\beta$-rotated coframe
\begin{equation} \label{eq:beta-rotated}
p_1 = \sigma_1, \qquad p_2 = \cos \beta \sigma_2 + \sin \beta \sigma_3, \qquad p_3 = - \sin \beta \sigma_2 + \cos \beta \sigma_3,
\end{equation}
rather than $\sigma_1, \sigma_2, \sigma_3$. With respect to the coframe $\d r, \d \gamma, \d \alpha, \d \beta, p_1, p_2, p_3$ one can compute that
\begin{align}
\zeta_1 \w \zeta_2 \w \zeta_3 & = r^2 \sin \gamma \d r\w \d \gamma \w \d \beta - \tfrac{1}{2} r^2 \sin \gamma (1+\cos^2 \alpha) \d r \w \d \gamma \w p_1 \nonumber \\
& \quad {} + r^2 \cos \gamma \cos \alpha \d r \w \d \gamma \w p_2 + r^2 \sin \gamma \cos \alpha \d r \w \d \beta \w p_3 \nonumber \\
& \quad {} + r^2 \cos \gamma \cos^2 \alpha \d r \w p_2 \w p_3 + \tfrac{1}{2} r^2 \sin \gamma \cos \alpha (1+\cos^2 \alpha) \d r \w p_3 \w p_1 \label{eq:z1z2z3.CP2}
\end{align}
and
\begin{align}
& \zeta_1 \w \Omega_1 + \zeta_2 \w \Omega_2 + \zeta_3 \w \Omega_3 = \nonumber \\
& \qquad \cos \gamma \sin \alpha \cos \alpha \d r \w \d \alpha \w p_1 - r \sin \gamma \sin \alpha \cos \alpha \d \gamma \w \d \alpha \w p_1 \nonumber \\
& \qquad {} + \sin \gamma \sin \alpha \d r \w \d \alpha \w p_2 + r \cos \gamma \sin \alpha \d \gamma \w \d \alpha \w p_2 \nonumber \\
& \qquad {} - r \sin \gamma \sin \alpha \d \alpha \w \d \beta \w p_3 - \tfrac{1}{2} \cos \gamma \sin^2 \alpha \d r \w p_2 \w p_3 \nonumber \\
& \qquad {} + \tfrac{1}{2} r \sin \gamma \sin^2 \alpha \d \gamma \w p_2 \w p_3 - 2 r \cos \gamma \sin \alpha \cos \alpha \d \alpha \w p_2 \w p_3 \nonumber \\
& \qquad {} - \tfrac{1}{2} r \sin \gamma \sin \alpha (1 + 3 \cos^2 \alpha) \d \alpha \w p_3 \w p_1 - \tfrac{1}{2} r \cos \gamma \sin^2 \alpha \cos \alpha \d \gamma \w p_3 \w p_1 \nonumber \\
& \qquad {} - \tfrac{1}{2} \sin \gamma \sin^2 \alpha \cos \alpha \d r \w p_3 \w p_1 - \tfrac{1}{2} r \sin \gamma \sin^2 \alpha \cos \alpha \d \beta \w p_1 \w p_2. \label{eq:zkOk.CP2}
\end{align}

Similarly one can compute that with respect to the coframe $\d r, \d \gamma, \d \alpha, \d \beta, p_1, p_2, p_3$ we have
\begin{equation} \label{eq:volCP2.CP2}
\vol_{\C\P^2} = -\tfrac{1}{2} \Omega_k^2 = \tfrac{1}{2} \sin^3 \alpha \cos \alpha \d \alpha \w p_1 \w p_2 \w p_3
\end{equation}
and
\begin{align}
& \zeta_2 \w \zeta_3 \w \Omega_1 + \zeta_3 \w \zeta_1 \w \Omega_2 + \zeta_1 \w \zeta_2 \w \Omega_3 = \nonumber \\
& \qquad -r \sin^2 \gamma \sin \alpha \cos \alpha \d r \w \d \alpha \w \d \beta \w p_1 - r^2 \sin \gamma \cos \gamma \sin \alpha \cos \alpha \d \gamma \w \d \alpha \w \d \beta \w p_1 \nonumber \\
& \qquad {} + r \sin \gamma \cos \gamma \sin \alpha \d r \w \d \alpha \w \d \beta \w p_2 - r^2 \sin^2 \gamma \sin \alpha \d \gamma \w \d \alpha \w \d \beta \w p_2 \nonumber \\
& \qquad {} + r \sin \alpha \d r \w \d \gamma \w \d \alpha \w p_3 - \tfrac{1}{2} r \sin^2 \gamma \sin^2 \alpha \d r \w \d \beta \w p_2 \w p_3 \nonumber \\
& \qquad {} - \tfrac{1}{2} r^2 \sin \gamma \cos \gamma \sin^2 \alpha \d \gamma \w \d \beta \w p_2 \w p_3 - r^2 \sin^2 \gamma \sin \alpha \cos \alpha \d \alpha \w \d \beta \w p_2 \w p_3 \nonumber \\
& \qquad {} + \tfrac{1}{2} r \sin \gamma \cos \gamma \sin^2 \alpha \cos \alpha \d r \w \d \beta \w p_3 \w p_1 - \tfrac{1}{2} r^2 \sin^2 \gamma \sin^2 \alpha \cos \alpha \d \gamma \w \d \beta \w p_3 \w p_1 \nonumber \\
& \qquad {} + r^2 \sin \gamma \cos \gamma \sin \alpha \cos^2 \alpha \d \alpha \w \d \beta \w p_3 \w p_1 - \tfrac{1}{2} r \sin^2 \alpha \cos \alpha \d r \w \d \gamma \w p_1 \w p_2 \nonumber \\
& \qquad {} - \tfrac{1}{2} r \sin \gamma \cos \gamma \sin^3 \alpha \d r \w \d \alpha \w p_1 \w p_2 + \tfrac{1}{2} r^2 (2 \cos^2 \alpha + \sin^2 \gamma \sin^2 \alpha) \sin \alpha \d \gamma \w \d \alpha \w p_1 \w p_2 \nonumber \\
& \qquad {} + \tfrac{1}{4} r (4 \cos^2 \alpha + \sin^2 \gamma \sin^2 \alpha) \sin^2 \alpha \d r \w p_1 \w p_2 \w p_3 + \tfrac{1}{4} r^2 \sin \gamma \cos \gamma \sin^4 \alpha \d \gamma \w p_1 \w p_2 \w p_3 \nonumber \\
& \qquad {} + \tfrac{1}{2} r^2 (2 \cos^2 \alpha + \sin^2 \gamma \sin^2 \alpha) \sin \alpha \cos \alpha \d \alpha \w p_1 \w p_2 \w p_3. \label{eq:zizjOk.CP2}
\end{align}
The above expressions can be substituted into~\eqref{eq:psic} to obtain the 4-form
\begin{equation} \label{eq:psic.CP2}
\ast_{\varphi_c}\varphi_c=4(c+r^2)\vol_{N}-2(\zeta_2\w\zeta_3\w\Omega_1+\zeta_3\w\zeta_1\w\Omega_2+\zeta_1\w\zeta_2\w\Omega_3).
\end{equation}

\subsection{SU(2)-invariant coassociative 4-folds}\label{sub:SU2-invariant}

As $\sigma_1$, $\sigma_2$, $\sigma_3$ in~\eqref{eq:sigmas-CP2} are left-invariant 1-forms for the $\SU(2)$ action, $\sigma_1\w\sigma_2\w\sigma_3$ is a volume form on the 3-dimensional $\SU(2)$ orbits. Since $\sigma_1 \w \sigma_2 \w \sigma_3 = p_1 \w p_2 \w p_3$, we can see by inspection of~\eqref{eq:phic.CP2},~\eqref{eq:z1z2z3.CP2}, and~\eqref{eq:zkOk.CP2} that such a term does not appear in $\varphi_c$. Moreover, the rest of the coframe consists of invariant 1-forms. It follows that $\varphi_c$ vanishes on all the $\SU(2)$ orbits and hence, every 3-dimensional $\SU(2)$ orbit is contained in a unique coassociative 4-fold by the Harvey--Lawson existence theorem in~\cite{HarveyLawson}. We now describe these coassociative 4-folds. However, in contrast to the case of $\mathcal{S}^4$, we are unable to describe the generic coassociative 4-folds in this fibration explicitly in terms of conserved quantities, except when we are in the cone setting ($c=0$) of $M_0=\R^+\times(\SU(3)/T^2)$. However, we can still describe the fibration structure and the topology of the fibres.

The free coordinates (that is, the coordinates which are invariant under the $\SU(2)$ action) on our coordinate patch are $r, \gamma, \alpha, \beta$. Therefore any $\SU(2)$-invariant coassociative 4-fold $N$ is defined by a path 
\[
\big(r(\tau), \gamma(\tau), \alpha(\tau), \beta(\tau)\big)
\]
for some real parameter $\tau$. We want to find this path passing through any particular point. The results are stated below in Propositions~\ref{prop:smooth.coass.CP2} and~\ref{prop:cone.coass.CP2}. The proofs follow the two statements.

First consider the smooth case $M = \Lambda^2_- (T^* \C\P^2)$, where $c>0$. In this setting, we have the following partial result.
\begin{prop} \label{prop:smooth.coass.CP2}
Let $N$ be an $\SU(2)$-invariant coassociative 4-fold in $M$ which is \emph{not} the zero section $\C\P^2$. 
\begin{itemize}
\item[\emph{(a)}] If $\cos\alpha\sin\gamma\cos\gamma \not \equiv 0$ on $N$, then the following are constant on $N$:
\begin{equation*}
\beta \, \in [0,2\pi), \qquad v=2(c+r^2)^{\frac{1}{4}}\cos\alpha\cot\gamma \, \in\R.
\end{equation*}
\item[\emph{(b)}] If $\cos\alpha \equiv 0$ on $N$, then the following is constant on $N$:
\begin{equation*}
w_{\alpha=\frac{\pi}{2}}=r\cos\gamma \, \in\R.
\end{equation*}
\end{itemize}
\end{prop}

Next consider the cone case $M_0 = \R^+\times(\SU(3)/T^2)$, corresponding to $c=0$. Here, we can completely integrate the equations and thus explicitly describe all of the $\SU(2)$-invariant coassociative 4-folds.
\begin{prop} \label{prop:cone.coass.CP2}
Let $N$ be an $\SU(2)$-invariant coassociative 4-fold in $M_0$. 
\begin{itemize}
\item[\emph{(a)}] If $\cos\alpha\sin\gamma\cos\gamma \not \equiv 0$ on $N$, then 
the following are constant on $N$:
\begin{equation*}
\beta \, \in [0,2\pi), \qquad v=2r^{\frac{1}{2}}\cos\alpha\cot\gamma \, \in\R, \qquad u=\frac{2\cos^2\alpha-\sin^2\alpha\sin^2\gamma}{\cos^2\alpha\cos\gamma} \, \in\R.
\end{equation*}
\item[\emph{(b)}] If $\cos\alpha \equiv 0$ on $N$, then the following is constant on $N$:
\begin{equation*}
w_{\alpha=\frac{\pi}{2}}=r\cos\gamma \, \in\R.
\end{equation*}
\item[\emph{(c)}] If $\sin\gamma \equiv 0$ on $N$, then the following is constant on $N$:
\begin{equation*}
w_{\sin\gamma=0}=r\cos 2\alpha \, \in\R.
\end{equation*}
\item[\emph{(d)}] If $\cos\gamma \equiv 0$ on $N$, then $w_{\cos\gamma=0}$ defined by the equation below is constant on $N$: 
\begin{equation*}
w_{\cos\gamma=0}^2=r\frac{(3\cos2\alpha+1)^2}{8(\cos2\alpha+1)} \, \in [0, \infty).
\end{equation*}
\end{itemize}
\end{prop}

We examine the coassociative condition that $\varphi_c|_N\equiv 0$. (Recall that this   says that $N$ is coassociative up to a choice of orientation by Definition~\ref{dfn:coassociative}.) Let $\dot{}$ denote differentiation with respect to $\tau$. On $N$,
$$\d r = \dot{r} \d \tau, \quad \d \gamma = \dot{\gamma} \d \tau, \quad \d \alpha = \dot{\alpha} \d \tau, \quad \d \beta = \dot{\beta} \d \tau.
$$
Substituting the expressions into~\eqref{eq:z1z2z3.CP2} and~\eqref{eq:zkOk.CP2}, we obtain
\begin{align}
(\zeta_1 \w \zeta_2 \w \zeta_3)|_N & = r^2 \cos \gamma \cos^2 \alpha \dot{r} \d \tau \w p_2 \w p_3 + \tfrac{1}{2} r^2 \sin \gamma \cos \alpha (1+\cos^2 \alpha) \dot{r} \d \tau \w p_3 \w p_1, \label{eq:z1z2z3.N}
\end{align}
and
\begin{align}
& (\zeta_1 \w \Omega_1 + \zeta_2 \w \Omega_2 + \zeta_3 \w \Omega_3)|_N \\
& \qquad = - \tfrac{1}{2} \cos \gamma \sin^2 \alpha \dot{r} \d \tau \w p_2 \w p_3 + \tfrac{1}{2} r \sin \gamma \sin^2 \alpha \dot{\gamma} \d \tau \w p_2 \w p_3 \nonumber \\
& \qquad \quad{} - 2 r \cos \gamma \sin \alpha \cos \alpha \dot{\alpha} \d \tau \w p_2 \w p_3 - \tfrac{1}{2} r \sin \gamma \sin \alpha (1 + 3 \cos^2 \alpha) \dot{\alpha} \d \tau \w p_3 \w p_1 \nonumber \\
& \qquad \quad {} - \tfrac{1}{2} r \cos \gamma \sin^2 \alpha \cos \alpha \dot{\gamma} \d \tau \w p_3 \w p_1 - \tfrac{1}{2} \sin \gamma \sin^2 \alpha \cos \alpha \dot{r} \d \tau \w p_3 \w p_1 \nonumber \\
& \qquad \quad {} - \tfrac{1}{2} r \sin \gamma \sin^2 \alpha \cos \alpha \dot{\beta} \d \tau \w p_1 \w p_2. \label{eq:zkOk.N}
\end{align}
Using the formula~\eqref{eq:phic.CP2} for $\varphi_c$ we have
\begin{equation*}
\varphi_c|_N = (c+r^2)^{-\frac{3}{4}} (\zeta_1\w\zeta_2\w\zeta_3)|_N + 2(c+r^2)^{\frac{1}{4}}(\zeta_1\w\Omega_1+\zeta_2\w\Omega_2+\zeta_3\w\Omega_3)|_N.
\end{equation*}
Substituting~\eqref{eq:z1z2z3.N} and~\eqref{eq:zkOk.N} into the above, the coassociative condition $\varphi_c|_N\equiv 0$ becomes three independent ordinary differential equations, obtained from the $\d \tau \w p_i \w p_j$ terms.

The $\d \tau \w p_1 \w p_2$ term gives
\begin{equation} \label{eq:CP2.coass0}
r \sin \gamma \sin^2 \alpha \cos \alpha \dot{\beta} = 0.
\end{equation}
The $\d \tau \w p_2 \w p_3$ term gives
\begin{equation} \label{eq:CP2.coass1}
\begin{aligned}
& (c + r^2)^{-\frac{3}{4}} r^2 \cos \gamma \cos^2 \alpha \dot{r} \\
& \qquad {} + (c + r^2)^{\frac{1}{4}} (- \cos \gamma \sin^2 \alpha \dot{r} + r \sin \gamma \sin^2 \alpha \dot{\gamma} - 4 r \cos \gamma \sin \alpha \cos \alpha \dot{\alpha}) = 0,
\end{aligned}
\end{equation}
and the $\d \tau \w p_3 \w p_1$ term gives
\begin{equation} \label{eq:CP2.coass2} 
\begin{aligned}
& (c + r^2)^{-\frac{3}{4}} ( r^2 \sin \gamma \cos \alpha (1+\cos^2 \alpha) \dot{r} ) \\
& \qquad {} + 2(c+r^2)^{\frac{1}{4}} (- r \sin \gamma \sin \alpha (1 + 3 \cos^2 \alpha) \dot{\alpha} - r \cos \gamma \sin^2 \alpha \cos \alpha \dot{\gamma} - \sin \gamma \sin^2 \alpha \cos \alpha \dot{r} ) = 0. 
\end{aligned}
\end{equation}

Observe that in the smooth case $(c>0)$, the submanifold $r=0$ which corresponds to the zero section $\C\P^2$ in $\Lambda^2_- (T^* \C\P^2)$ solves the three equations~\eqref{eq:CP2.coass0}--\eqref{eq:CP2.coass2}, and thus is an $\SU(2)$-invariant coassociative submanifold. From now on we assume that $r$ is \emph{not} identically zero on $N$.

\paragraph{The generic case $\mathbf{\cos \alpha \sin \gamma \cos \gamma \not \equiv 0}$.} By~\eqref{eq:CP2.coass0} we have that $\beta\in [0,2\pi)$ is constant. Let us eliminate the $\dot{r} (c+r^2)^{\frac{1}{4}}$ terms, by multiplying~\eqref{eq:CP2.coass1} by $2 \sin \gamma \cos \alpha$, multiplying~\eqref{eq:CP2.coass2} by $\cos \gamma$, and taking the difference. After some manipulation, the result is
\begin{equation*}
0 = (c+r^2)^{-\frac{3}{4}} ( - r^2 \sin \gamma \cos \gamma \cos \alpha \sin^2 \alpha \dot{r} ) + (c+r^2)^{\frac{1}{4}} ( 2 r \sin^2 \alpha \cos \alpha \dot{\gamma} + 2 r \sin \gamma \cos \gamma \sin^3 \alpha \dot \alpha).
\end{equation*}
Dividing the above expression by $- r \sin^2 \gamma \sin^2 \alpha$ gives
\begin{equation*}
(c+r^2)^{-\frac{3}{4}} ( r \cos \alpha \cot \gamma \dot r) + 2 (c+r^2)^{\frac{1}{4}} (- \cos \alpha \csc^2 \gamma \dot{\gamma} - \sin \alpha \cot \gamma \dot{\alpha} ) = 0.
\end{equation*}
We can write the above as
\begin{equation*}
\frac{\d}{\d\tau}\big(2(c+r^2)^{\frac{1}{4}}\cos\alpha\cot\gamma\big)=0.\label{eq:dot.v}
\end{equation*}
We deduce in this case that
\begin{equation}\label{eq:v}
v=2(c+r^2)^{\frac{1}{4}}\cos\alpha\cot\gamma\in\R
\end{equation}
is constant on $N$.

We can now use $v$ in~\eqref{eq:v} to eliminate $r$ from the two differential equations~\eqref{eq:CP2.coass1}--\eqref{eq:CP2.coass2}. Specifically, from~\eqref{eq:v} and its derivative with respect to $\tau$, we obtain
\begin{equation} \label{eq:v.r.temp}
\begin{aligned}
r^2 & = \tfrac{1}{16} v^4 \sec^4 \alpha \tan^4 \gamma - c, \\
r \dot{r} & = 2(c+r^2) (\tan \alpha \dot{\alpha} + \sec \gamma \csc \gamma \dot{\gamma}).
\end{aligned}
\end{equation}
If we substitute the two expressions in~\eqref{eq:v.r.temp} into either~\eqref{eq:CP2.coass1} or~\eqref{eq:CP2.coass2}, straightforward algebraic manipulation yields a $v$-dependent ordinary differential equation relating $\alpha$ and $\gamma$. This equation is
\begin{equation} \label{eq:flow.smooth}
\begin{aligned}
\cos\alpha\big((v^4\sin^4\gamma&-16c\cos^4\alpha\cos^4\gamma)(2\cos^2\alpha+\sin^2\alpha\sin^2\gamma)-2v^4\sin^2\alpha\sin^4\gamma\big)\dot{\gamma} \\ &-2\sin\alpha\sin\gamma\cos\gamma\big(v^4\sin^2\alpha\sin^4\gamma+(v^4\sin^4\gamma-16c\cos^4\alpha\cos^4\gamma)\cos^2\alpha\big)\dot{\alpha}=0.
\end{aligned}
\end{equation}
When $c=0$, the equation simplifies to
\begin{equation} \label{eq:u.flow.c=0}
\cos\alpha\big(2 \cos^2\alpha- 2 \sin^2\alpha+\sin^2\alpha\sin^2\gamma\big)\dot{\gamma}
-2\sin\alpha\sin\gamma\cos\gamma\dot{\alpha}=0.
\end{equation}
We can then compute that
\begin{align*}
&\frac{\d}{\d\tau}\left(\frac{2\cos^2\alpha-\sin^2\alpha\sin^2\gamma}{\cos^2\alpha\cos\gamma}\right)\\
&\qquad=
\sin\gamma\left(\frac{2\cos^2\alpha-\sin^2\alpha(\sin^2\gamma+2\cos^2\gamma)}{\cos^2\alpha\cos^2\gamma}\right)\dot{\gamma}-\frac{2\sin\alpha\sin^2\gamma}{\cos^3\alpha\cos\gamma}\dot\alpha\\
&\qquad=\frac{\sin\gamma}{\cos^3\alpha\cos^2\gamma}
\left(\cos\alpha\big(2 \cos^2\alpha- 2\sin^2\alpha+\sin^2\alpha\sin^2\gamma\big)\dot{\gamma}-2\sin\alpha\sin\gamma\cos\gamma\dot\alpha\right)=0.
\end{align*}
Hence, for $c=0$, in the generic case where $\cos \alpha \sin \gamma \cos \gamma \not \equiv 0$, we have the following constant:
\begin{equation}\label{eq:u.c=0}
u=\frac{2\cos^2\alpha-\sin^2\alpha\sin^2\gamma}{\cos^2\alpha\cos\gamma}\in\R.
\end{equation}

\begin{remark}
We can explicitly integrate the ODE~\eqref{eq:flow.smooth} in the cone case where $c=0$, but we were unfortunately unable to do so in the smooth case where $c>0$. We can rewrite~\eqref{eq:flow.smooth} as $\frac{\d y}{\d x}=f(x,y)$ where $f(x,y)$ is a rational function in $x$ and $y$. One can then, in principle, check  lists of ODEs which are known to be integrable to see if this ODE is in that class (perhaps after a change of variables). This is a non-trivial task and it currently remains open whether~\eqref{eq:flow.smooth} can be integrated in closed  form.
\end{remark}

Next we consider the three special cases in Propositions~\ref{prop:smooth.coass.CP2} and~\ref{prop:cone.coass.CP2}. In all three of these special cases, the equation~\eqref{eq:CP2.coass0} is identically satisfied, so we need only consider the equations~\eqref{eq:CP2.coass1} and~\eqref{eq:CP2.coass2}.

\paragraph{The special case $\mathbf{\cos \alpha \equiv 0}$.} Assume that $\alpha$ is identically $\frac{\pi}{2}$ on $N$. In this case equations~\eqref{eq:CP2.coass1} and~\eqref{eq:CP2.coass2} become $\dot{\alpha} = 0$ which is automatic and
\[ \cos \gamma \dot{r} - r \sin \gamma \dot{\gamma} = 0, \]
which gives that $w_{\alpha = \frac{\pi}{2}} = r \cos \gamma$ is constant on $N$. Note that in this case the solution is independent of whether $c=0$ or $c>0$.

\paragraph{The special case $\mathbf{\sin \gamma \equiv 0}$.} Assume that $\gamma$ is identically $0$ or $\pi$ on $N$. In this case equations~\eqref{eq:CP2.coass1} and~\eqref{eq:CP2.coass2} become $\dot{\gamma} = 0$ which is automatic and
\begin{equation} \label{eq:sin.gamma.zero}
\big(r^2\cos^2\alpha-(c+r^2)\sin^2\alpha\big)\dot{r}-4r(c+r^2)\sin\alpha\cos\alpha\dot{\alpha}=0.
\end{equation}
This equation can be integrated \emph{implicitly} to give
\begin{equation} \label{eq:sin.gamma.zero.int}
w_{\sin \gamma = 0} = r^{\frac{1}{2}} (c+r^2)^{\frac{1}{4}} \cos 2\alpha - \frac{c}{2} \int_0^r \frac{1}{s^{\frac{1}{2}} (c+s^2)^{\frac{3}{4}}} \d s \qquad \text{is constant on $N$}.
\end{equation}
When $c=0$, we obtain the exact solution $w_{\sin \gamma = 0} = r \cos 2 \alpha$.

\paragraph{The special case $\mathbf{\cos \gamma \equiv 0}$.} Assume that $\gamma$ is identically $\frac{\pi}{2}$ on $N$. In this case equations~\eqref{eq:CP2.coass1} and~\eqref{eq:CP2.coass2} become $\dot{\gamma} = 0$ which is automatic and
\begin{equation} \label{eq:cos.gamma.zero}
(r^2(2-\sin^2\alpha)-2(c+r^2)\sin^2\alpha)\cos\alpha\dot{r}-2r(c+r^2)(1+3\cos^2\alpha)\sin\alpha\dot{\alpha}=0.
\end{equation}
We were unable to integrate~\eqref{eq:cos.gamma.zero}, even implicitly, when $c>0$. However, when $c=0$, one can compute that~\eqref{eq:cos.gamma.zero} is equivalent to
\begin{align}
(2-3\sin^2\alpha)\cos\alpha\dot{r}&-2r(1+3\cos^2\alpha)\sin\alpha\dot{\alpha}\nonumber\\
&=\frac{\cos^3\alpha}{2(2-3\sin^2\alpha)}\frac{\d}{\d\tau}\left(r\frac{(3\cos 2\alpha+1)^2}{\cos 2\alpha +1}\right)=0.\label{eq:gamma0.4}
\end{align}
We deduce that when $\alpha\neq\frac{\pi}{2}$, we have a real constant $w_{\cos\gamma=0}$ defined by
\begin{equation}\label{eq:w.cosgamma}
w_{\cos\gamma=0}^2=r\frac{(3\cos2\alpha+1)^2}{8(\cos2\alpha+1)}.
\end{equation}
Note from~\eqref{eq:gamma0.4} that when $c=0$, we \emph{appear to have} an additional possibility for $\dot\alpha=0$ if $\sin^2\alpha \equiv \frac{2}{3}$. However, since $2(2-3\sin^2\alpha)=1+3\cos 2\alpha$, this is just equivalent to $w_{\cos\gamma=0}=0$, and in this case $r$ can take any value. In the smooth case $c>0$, there exists a constant  $w_{\cos\gamma=0}$ along the flow lines where $\cos\gamma \equiv 0$, which reduces to~\eqref{eq:w.cosgamma} in the cone case $c=0$.

\subsection{The fibration} \label{sec:fibration.CP2}

Given the ordinary differential equations and constants determined in $\S$\ref{sub:SU2-invariant}, we now describe the \emph{topology of the fibres} and the structure of the coassociative fibration which we have constructed. This involves considering various cases. The treatment is somewhat lengthy and involved, so the reader interested only in the final results may wish to simply consult the summary in $\S$\ref{subs:summary}. 

We note from the outset that the differential equations~\eqref{eq:CP2.coass0}--\eqref{eq:CP2.coass2} that we are studying are invariant under the transformation $\gamma\mapsto\pi-\gamma$. We can therefore restrict ourselves throughout to $\gamma\in[0,\frac{\pi}{2}]$. Also, we mention here that the \emph{arrows} in the various vector field plots that we present in this section do not have any real meaning. There is no preferred ``forward direction'' to the parameters in these ordinary differential equations. Finally, we use wording such as ``flow lines passing through a critical point'' to mean flow lines which approach the critical point in the limit as the flow parameter tends to one of its limiting allowed values.

\subsubsection{Case 1: $r \equiv 0$}

The solution $r\equiv0$ gives the coassociative zero section $\C\P^2$ in $\Lambda^2_-(T^*\C\P^2)$. We emphasize that this is in marked contrast to the $\Lambda^2_- (T^* \mathcal{S}^4)$ case, where the zero section $\mathcal{S}^4$ was \emph{not} a fibre of the coassociative fibration, despite being coassociative. This is due to the fact that the action of $\SU(2)$ on $\Lambda^2_-(T^*\C\P^2)$ is very different from the action of $\SO(3)$ on $\Lambda^2_-(T^* \mathcal{S}^4)$.

Note that, in the generic setting where $\cos \alpha \sin \gamma \cos \gamma \not \equiv 0$, we have from~\eqref{eq:v.r.temp} that
$$r^2=\frac{v^4\tan^4\gamma}{16\cos^4\alpha}-c.$$
Therefore, for each fixed $v$, the curve given by
\begin{equation}\label{eq:r=0}
v^4\tan^4\gamma=16c\cos^4\alpha
\end{equation}
describes the points in the $(\alpha,\gamma)$-plane where $r=0$ on the coassociative fibration. We show the curve in Figure~\ref{fig:r-zero-plot} below for $c=1$ and various values of $v$.
\begin{figure}[H]
\caption{$r=0$ when $c=1$}\label{fig:r-zero-plot}
\begin{center}
\includegraphics[width=0.5\textwidth]{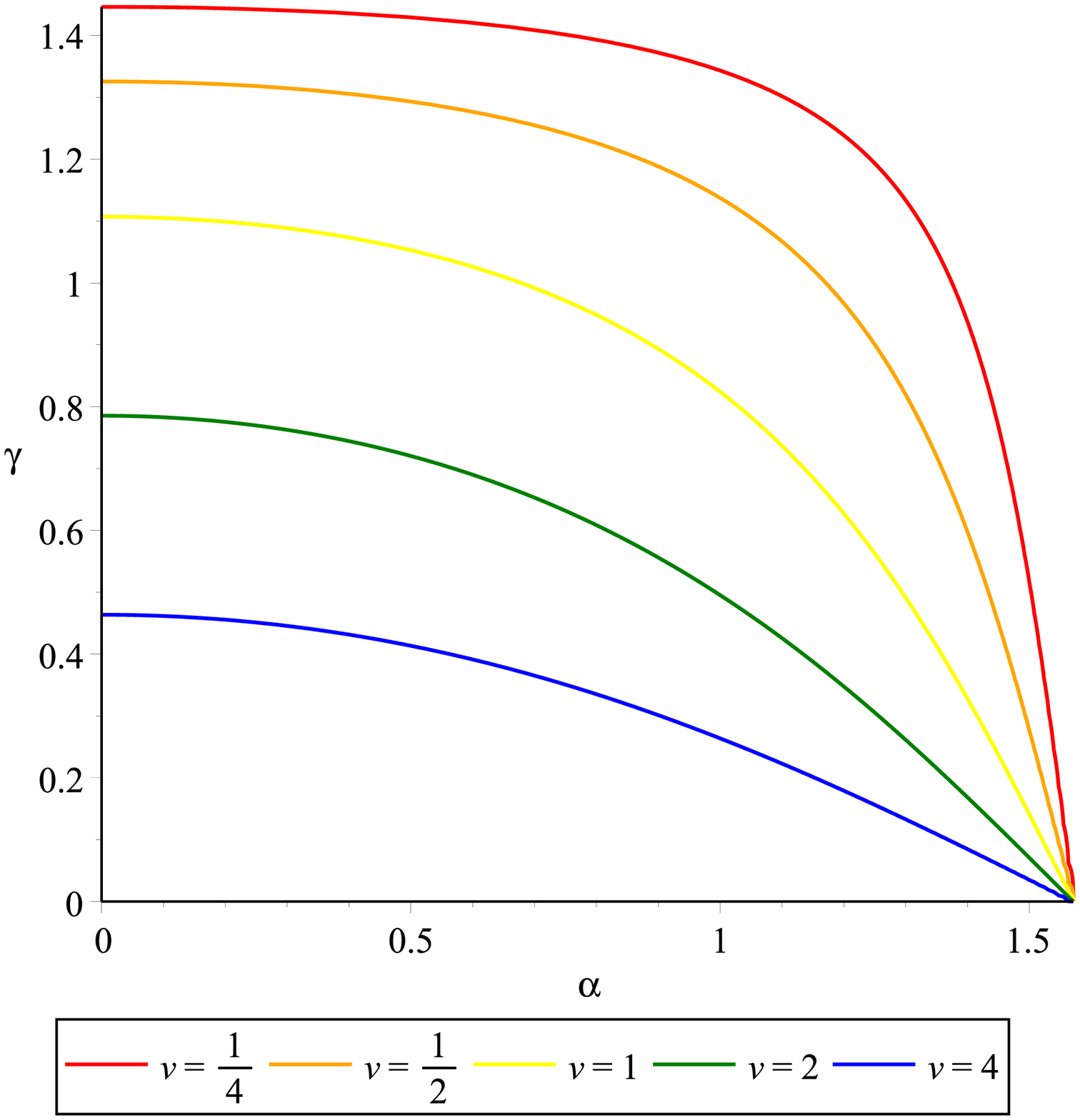}
\end{center}
\end{figure}

By the Harvey--Lawson local existence theorem~\cite[Theorem 4.6, page 139]{HarveyLawson}, two coassociatives cannot intersect in a 3-dimensional submanifold. Thus the orbit structure in Table~\ref{table:SU2.orbits} implies that any coassociative fibre cannot meet the $r=0$ fibre except when $\alpha\in\{0,\frac{\pi}{2}\}$, as they cannot both contain the same 3-dimensional orbit. However, there do exist coassociative fibres (see Case 3 below) which intersect the zero section $\C\P^2$ of $\Lambda^2_-(T^*\C\P^2)$ in the $\C\P^1$ given by $\alpha = \frac{\pi}{2}$ and $r=0$ in the last row of Table~\ref{table:SU2.orbits}. These thus are examples of the special type of calibrated submanifolds studied in~\cite{KarigiannisMinOo}.

\subsubsection{Case 2: $\sin\gamma\equiv0$}

In this case, by~\eqref{eq:r.gamma.beta} the circle action by $\beta$ is trivial, and by Table~\ref{table:SU2.orbits} we know that the orbits for $\alpha\in(0,\frac{\pi}{2})$ are $\mathcal{S}^3$, but for $\alpha=0$ or $\alpha=\frac{\pi}{2}$ and $r>0$, we have a $\C\P^1$. 

We reproduce here the defining equation~\eqref{eq:sin.gamma.zero} for the $\sin \gamma = 0$ case,
\begin{align}\label{eq:sin.gamma.zero.flow.2}
\big(r^2\cos^2\alpha-(c+r^2)\sin^2\alpha\big)\dot{r}-4r(c+r^2)\sin\alpha\cos\alpha\dot{\alpha}=0.
\end{align}

\paragraph{The smooth case $c>0$.} From~\eqref{eq:sin.gamma.zero.flow.2}, we see that if $r\to 0$ and $\alpha\to\alpha_0\neq 0$, then we must have that $\dot{r}\to 0$ at that point. However, the unique flow line in the $(\alpha,r)$-plane through $(\alpha_0,0)$ with $\dot{r}=0$ is just $r=0$, which we are excluding here. Hence, flow lines can only reach $r=0$ for $c>0$ when $\alpha=0$. However, from~\eqref{eq:sin.gamma.zero.flow.2} we see that at points $(0,r_0)$ with $r_0>0$, we must also have $\dot{r}=0$, hence by continuity the only flow line passing through $(0,0)$ must have $\dot{r}=0$ and therefore must be $r\equiv 0$. Thus, all flow lines must have $r>0$. 

Observe that~\eqref{eq:sin.gamma.zero.flow.2} has critical points only when 
$$r^2\cos^2\alpha=(c+r^2)\sin^2\alpha\quad\text{and}\quad r\sin\alpha\cos\alpha=0.$$
We see immediately that if $r>0$ then there is no solution to this pair of equations, so the unique critical point is at $r=0$ and $\alpha=0$, which matches well with our discussion above.

We rewrite~\eqref{eq:sin.gamma.zero.flow.2} as
\begin{equation} \label{eq:sin.gamma.zero.drdalpha}
\frac{\d r}{\d\alpha}=\frac{4r(c+r^2)\sin\alpha\cos\alpha} {r^2\cos^2\alpha-(c+r^2)\sin^2\alpha}.
\end{equation}
Since $r\sin\alpha\cos\alpha>0$ whenever $r>0$ and $\alpha\in(0,\frac{\pi}{2})$, we deduce that the sign of $\frac{\d r}{\d\alpha}$ is determined by the sign of 
$$r^2\cos^2\alpha-(c+r^2)\sin^2\alpha.$$
The set of points where this function vanishes, that is where
$$\cot^2\alpha=1+\frac{c}{r^2}>1$$
defines a curve $\Gamma$ in the $(\alpha,r)$-plane which is only defined for $\alpha<\frac{\pi}{4}$. Moreover, the curve $\Gamma$ has $r\to\infty$ as $\alpha\to \frac{\pi}{4}$, and $r\to 0$ as $\alpha\to 0$. See Figure~\ref{fig:sin-gamma-zero-smooth} below for the flow lines for~\eqref{eq:sin.gamma.zero.flow.2}, the curve $\Gamma$, and the asymptote $\alpha=\frac{\pi}{4}$. 

\begin{figure}[H]
\begin{center}
\caption{$\sin\gamma=0$ when $c=1$}\label{fig:sin-gamma-zero-smooth}
\includegraphics[width=0.5\textwidth]{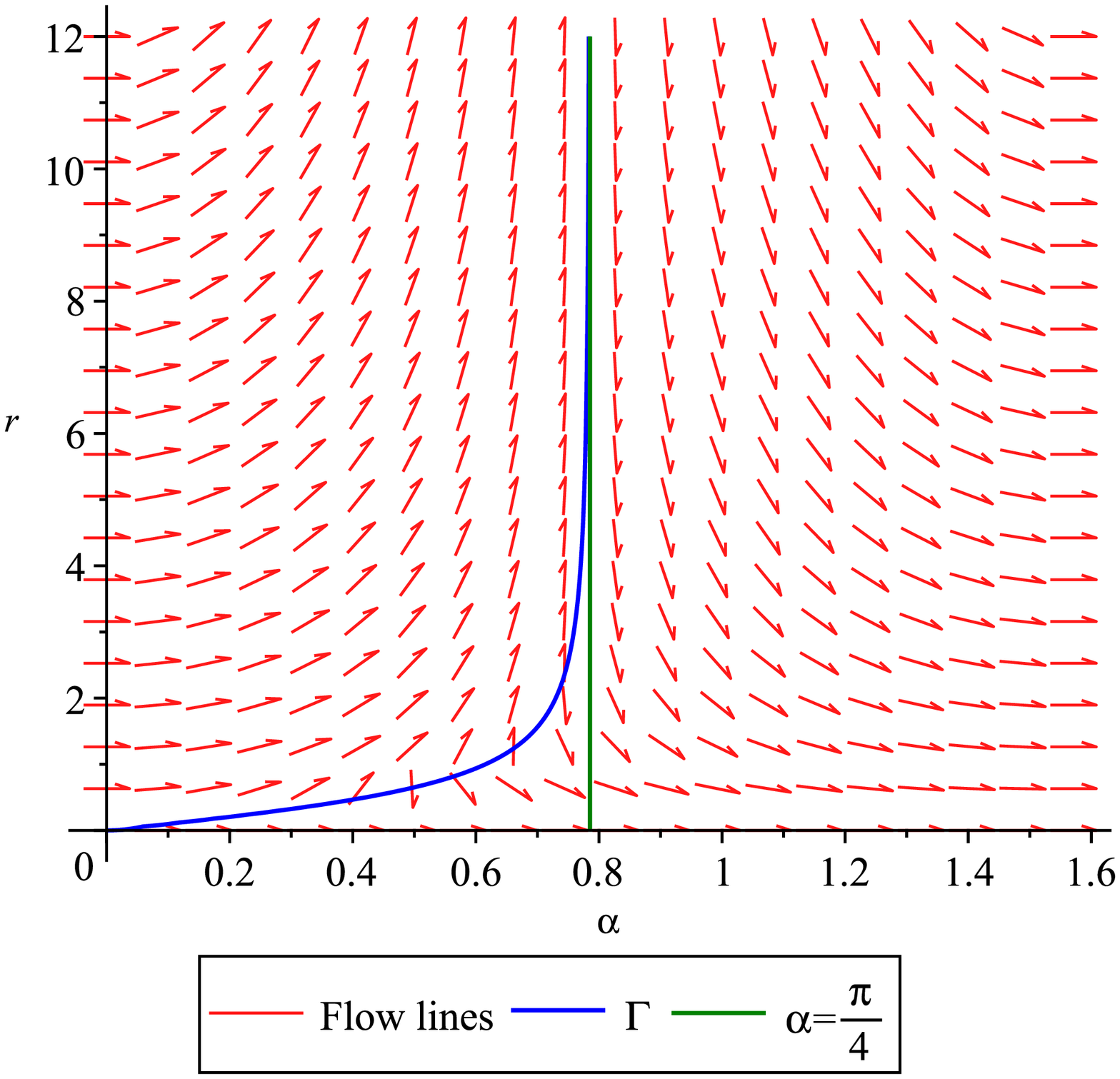}
\end{center}
\end{figure}

To the left of the curve $\Gamma$ we have that $\frac{\d r}{\d\alpha}>0$ and to the right of $\Gamma$ we have $\frac{\d r}{\d\alpha}<0$. Moreover, any flow line must cross $\Gamma$ vertically, and so any curve starting to the left of $\Gamma$ which crosses $\Gamma$ must stay to the right of $\Gamma$. Therefore, since $\frac{\d r}{\d\alpha}$ has a fixed sign on either side of $\Gamma$, we must have that $\dot{\alpha}\to 0$ and consequently $r \to \infty$ along any flow line. By equation~\eqref{eq:sin.gamma.zero.drdalpha}, this forces $\alpha\to\frac{\pi}{4}$ as $r \to \infty$. 

We have shown that all flow lines have $r\to\infty$ as $\alpha\to\frac{\pi}{4}$. Moreover, from~\eqref{eq:sin.gamma.zero.drdalpha} we see that for $r>0$ we have that $\frac{\d r}{\d\alpha}\to 0$ as $\alpha\to 0$ or $\alpha\to\frac{\pi}{2}$, so we have flow lines passing through $(0,r_0)$ and flow lines passing through $(\frac{\pi}{2},r_0)$ for all $r_0>0$. 

Since the flow lines are determined by a single constant, which is the angle or slope at $(0, r_0)$ or $(\frac{\pi}{2}, r_0)$, we have found all of the flow lines for the case $c>0$. We conclude that there exists a parameter $w$ such that for $w>0$ we have $\alpha\in[0,\frac{\pi}{4})$ and $r\to w>0$ as $\alpha\to 0$, and such that for $w<0$ we have $\alpha\in (0,\frac{\pi}{2}]$ and $r\to -w>0$ as $\alpha\to\frac{\pi}{2}$.

\paragraph{The cone case $c=0$.} In this case, we explicitly have $w=w_{\sin\gamma=0} = r \cos 2\alpha$ is constant as given in Proposition~\ref{prop:cone.coass.CP2}, so if $w=0$ then we have $\alpha\equiv\frac{\pi}{4}$. Otherwise, $\cos 2\alpha$ has a fixed sign (given by the sign of $w$) and so either $\alpha\in [0,\frac{\pi}{4})$ or $\alpha\in (\frac{\pi}{4},\frac{\pi}{2}]$. In each case, $r$ has a minimum at $\alpha=0$ or $\alpha=\frac{\pi}{2}$, given by $|w|$ and otherwise takes every greater value, tending to infinity as $\alpha\to\frac{\pi}{4}$. See Figure~\ref{fig:sin-gamma-zero-cone}.

Summarizing this case: for a real constant $w\neq 0$ we obtain $\mathcal{O}_{\C\P^1}(-1)$ for our coassociative fibre, but for $w=0$ and $c=0$ our coassociative fibre is $\R^+\times\mathcal{S}^3$.

\begin{figure}[H]
\begin{center}
\caption{$\sin\gamma=0$ when $c=0$}\label{fig:sin-gamma-zero-cone}
\includegraphics[width=0.5\textwidth]{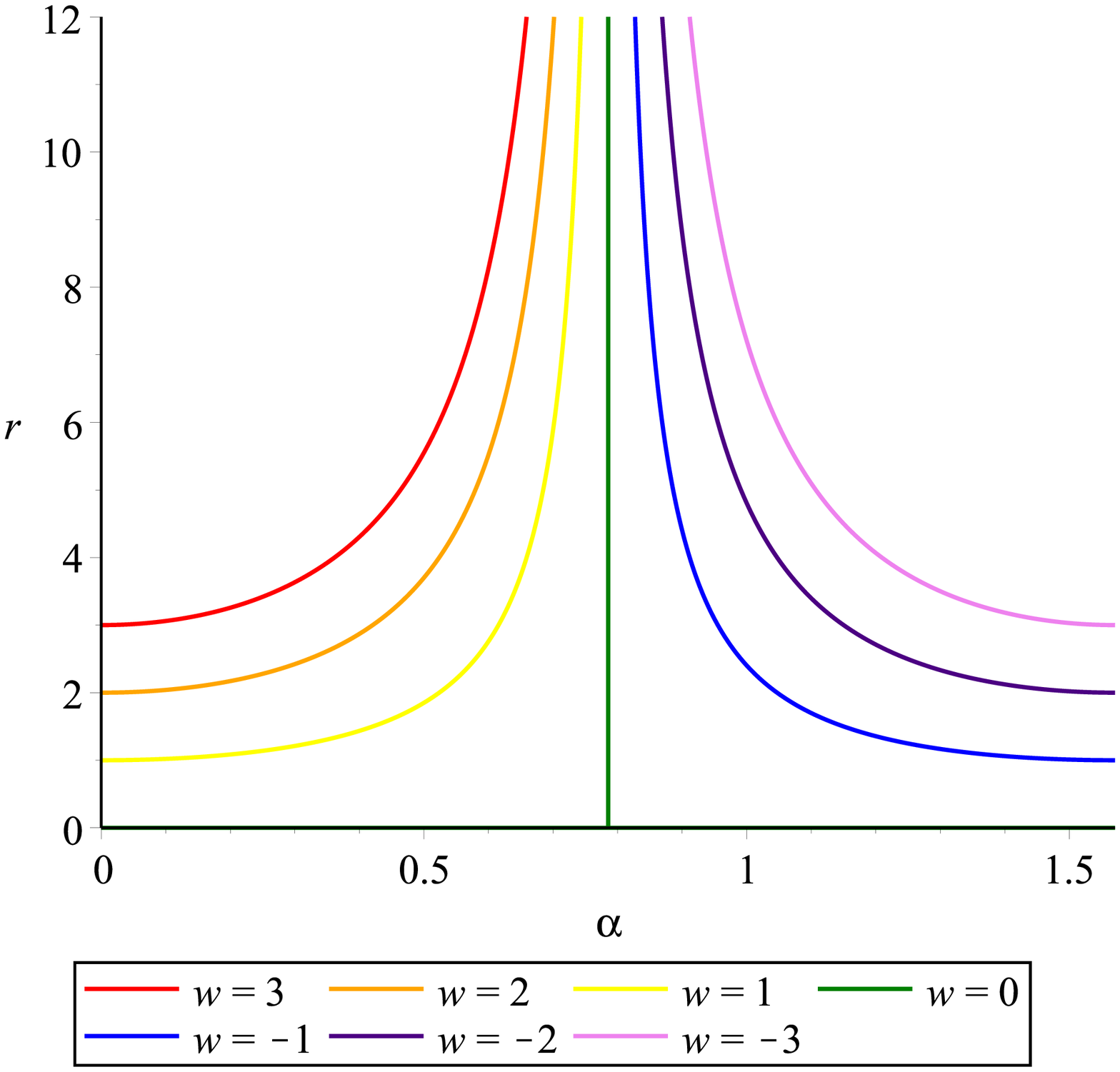}
\end{center}
\end{figure}

\subsubsection{Case 3: $\cos\gamma\equiv0$} \label{sec:cos.gamma.zero}

In this case, by~\eqref{eq:r.gamma.beta} we have a non-trivial $\mathcal{S}^1$-action parametrized by $\beta$, and by Table~\ref{table:SU2.orbits} we know that the orbits for $\alpha\in(0,\frac{\pi}{2})$ are $\mathcal{S}^3$, whereas the orbits for $\alpha=0$, $r>0$, and for $\alpha=\frac{\pi}{2}$, $r=0$, are both $\C\P^1$. (Note that $r=\sqrt{a_2^2+a_3^2}$ in this case since $\cos \gamma = 0$ implies $a_1=0$. Also, we show below that these coassociatives do not include the point orbit $\alpha = 0$, $r=0$.)

We reproduce here the defining equation~\eqref{eq:cos.gamma.zero} for the $\cos \gamma = 0$ case,
\begin{align}
(r^2(2-\sin^2\alpha)&-2(c+r^2)\sin^2\alpha)\cos\alpha\dot{r}-2r(c+r^2)(1+3\cos^2\alpha)\sin\alpha\dot{\alpha}=0.\label{eq:cos.gamma.zero.flow.2}
\end{align}

\paragraph{The smooth case $c>0$.} From~\eqref{eq:cos.gamma.zero.flow.2}, we see that if $r\to 0$ and $\alpha\to\alpha_0\in(0,\frac{\pi}{2})$, then we must have $\dot{r}\to 0$ at that point. Just as in the $\sin \gamma = 0$ case, we get a contradiction to the fact that we are assuming that $r$ is not identically $0$. Moreover, we see from~\eqref{eq:cos.gamma.zero.flow.2} that at $(0,r_0)$ for $r_0>0$ we must have $\dot{r}=0$, so by continuity just as in Case 2, the only flow line passing through $(0,0)$ is $r\equiv 0$. Therefore, the only possible way that $r\to 0$ is if $\alpha\to\frac{\pi}{2}$. 

Observe that~\eqref{eq:cos.gamma.zero.flow.2} has critical points only when 
$$r^2\cos\alpha(2-\sin^2\alpha)=2(c+r^2)\sin^2\alpha\cos\alpha\quad\text{and}\quad
r(1+3\cos^2\alpha)\sin\alpha=0.$$
We see immediately that if $r>0$ then there is no solution to this pair of equations, so the only critical points in the $(\alpha,r)$ plane are 
$$(0,0)\quad\text{and}\quad (\textstyle\frac{\pi}{2},0),$$
which fits well with the discussion above.

We rewrite~\eqref{eq:cos.gamma.zero.flow.2} as
\begin{equation} \label{eq:cos.gamma.zero.drdalpha}
\frac{\d r}{\d\alpha}=\frac{2r(c+r^2)(1+3\cos^2\alpha)\sin\alpha}{\big(r^2(2-\sin^2\alpha)-2(c+r^2)\sin^2\alpha\big)\cos\alpha}.
\end{equation}
Since $r(c+r^2)(1+3\cos^2\alpha)\sin\alpha>0$ whenever $r>0$ and $\alpha\in(0,\frac{\pi}{2})$, we deduce that the sign of $\frac{\d r}{\d\alpha}$ is determined by the sign of
$$r^2(2-\sin^2\alpha)-2(c+r^2)\sin^2\alpha.$$
The set of points where this function vanishes, that is where
$$\text{cosec}^2\alpha=\frac{3}{2}+\frac{c}{r^2}>\frac{3}{2}$$
defines a curve $\Gamma$ in the $(\alpha,r)$-plane which is only defined for $\cos 2 \alpha<-\frac{1}{3}$. We also see that $\Gamma$ passes through the critical point $(0,0)$. Moreover, the curve $\Gamma$ has $r\to\infty$ as $\alpha\to\frac{1}{2}\cos^{-1}(-\frac{1}{3})$. See Figure~\ref{fig:cos-gamma-zero-smooth} below for the flow lines for~\eqref{eq:cos.gamma.zero.flow.2}, the curve $\Gamma$, and the asymptote $\alpha=\frac{1}{2}\cos^{-1}(-\frac{1}{3})$.

\begin{figure}[H]
\begin{center}
\caption{$\cos\gamma=0$ when $c=1$}\label{fig:cos-gamma-zero-smooth}
\includegraphics[width=0.5\textwidth]{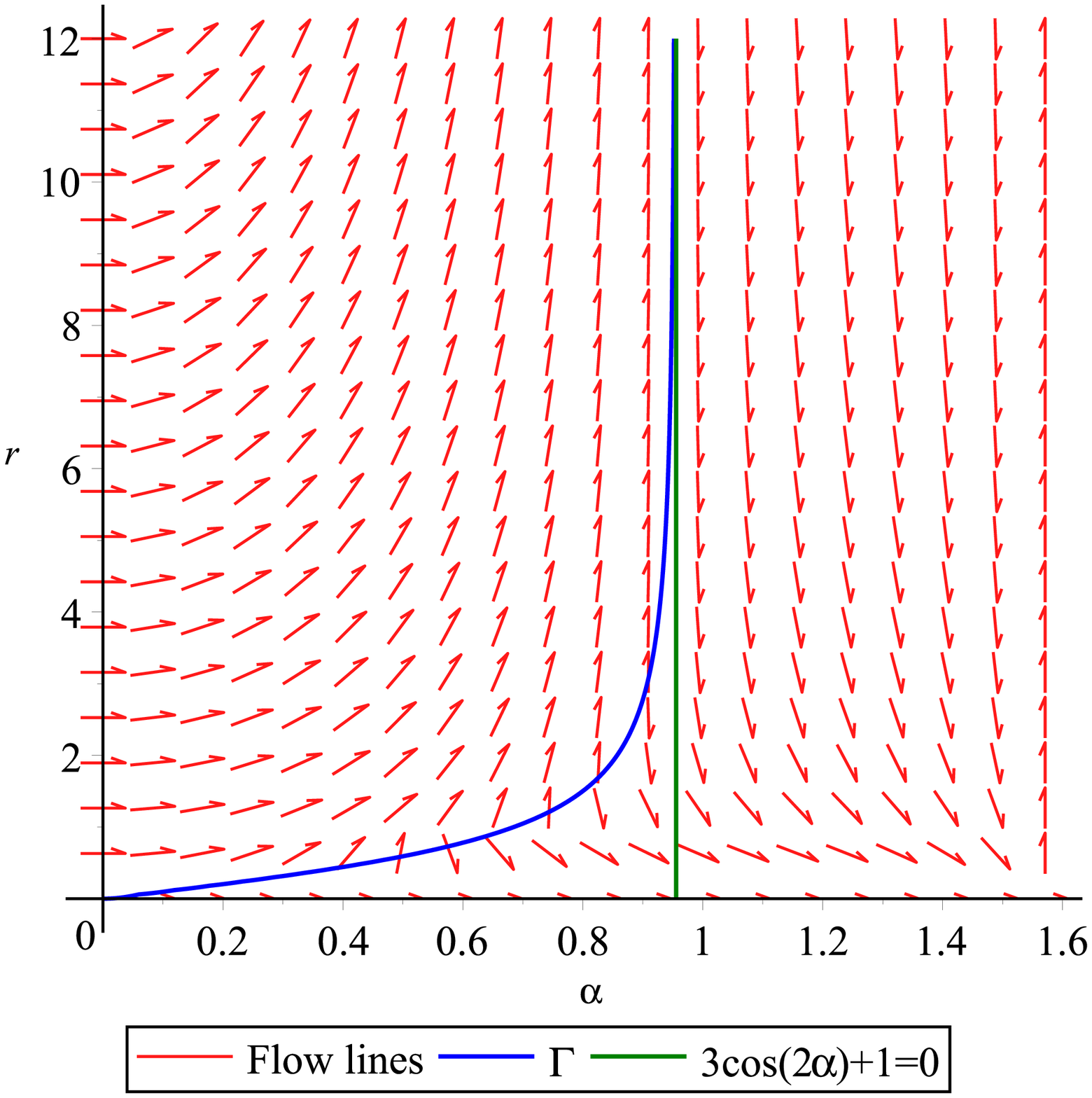}
\end{center}
\end{figure}

Since $\frac{\d r}{\d\alpha}>0$ to the left of $\Gamma$ and $\frac{\d r}{\d\alpha}<0$ to the right of $\Gamma$, as in the previous case we deduce that any flow line passing through $(\frac{\pi}{2},0)$ must either stay to the right of $\Gamma$ and have $r\to\infty$ as $\alpha\to\frac{1}{2}\cos^{-1}(-\frac{1}{3})$, or cross $\Gamma$ vertically then stay to the left of $\Gamma$ and have the same asymptotic behaviour.

Since $\alpha\equiv \frac{\pi}{2}$ is a solution to~\eqref{eq:cos.gamma.zero.flow.2}, a flow line can only reach $\alpha=\frac{\pi}{2}$ when $r=0$. Moreover, any flow line which tends to $(\frac{\pi}{2},r_0)$ for some $r_0>0$ would have to do so vertically, which is not possible, and so all flow lines with $\alpha\to\frac{\pi}{2}$ must pass through $(\frac{\pi}{2},0)$. In contrast, since $\frac{\d r}{\d\alpha}=0$ for any $(0,r_0)$, we have flow lines passing through $(0,r_0)$ for all $r_0>0$.

Overall, we have two disconnected (open) 1-parameter families of coassociative fibres diffeomorphic to $\mathcal{O}_{\C\P^1}(-1)$, which we can parameterize by a non-zero constant $w$ such that $w>0$ corresponds to flow lines meeting   $\alpha=0$ where $r$ has minimum $w$, and such that $w<0$ corresponds to flow lines passing through $(\frac{\pi}{2},0)$, where the correspondence is given by the angle at which they enter the critical point.

\paragraph{The cone case $c=0$.} In this case, we explicitly have $w=w_{\cos\gamma=0} = \sqrt{r} \frac{(3\cos2\alpha+1)}{\sqrt{8(\cos2\alpha+1)}}$ is constant as given in Proposition~\ref{prop:cone.coass.CP2}. See Figure~\ref{fig:cos-gamma-zero-cone} for some curves where $w$ is constant. 

If $w=0$ we know that $3\cos2\alpha+1=0$, so we obtain an $\mathcal{S}^1$-family of coassociative fibres diffeomorphic to $\R^+\times\mathcal{S}^3$. Otherwise, we know that there are two components corresponding to $3\cos 2\alpha+1>0$ and $3\cos 2\alpha+1<0$, both of which come in 2-parameter families (described by $\beta$ and $w$). The first case contains $\alpha=0$ and $r$ attains its minimum value $w$ (which is positive) as $\alpha\to 0$. Moreover, $r$ can take value above this minimum, so we have a coassociative fibre diffeomorphic to $\mathcal{O}_{\mathbb{CP}^1}(-1)$. The second case contains $\alpha=\frac{\pi}{2}$, but we see from~\eqref{eq:w.cosgamma} that as $\alpha\to \frac{\pi}{2}$, we must have $r\to 0$ for $w$ to remain finite. Therefore, since $r$ can take any value, these coassociative fibres are diffeomorphic to $\R^+\times\mathcal{S}^3$. 

\begin{figure}[H]
\caption{$\cos\gamma=0$ when $c=0$}\label{fig:cos-gamma-zero-cone}
\begin{center}
\includegraphics[width=0.5\textwidth]{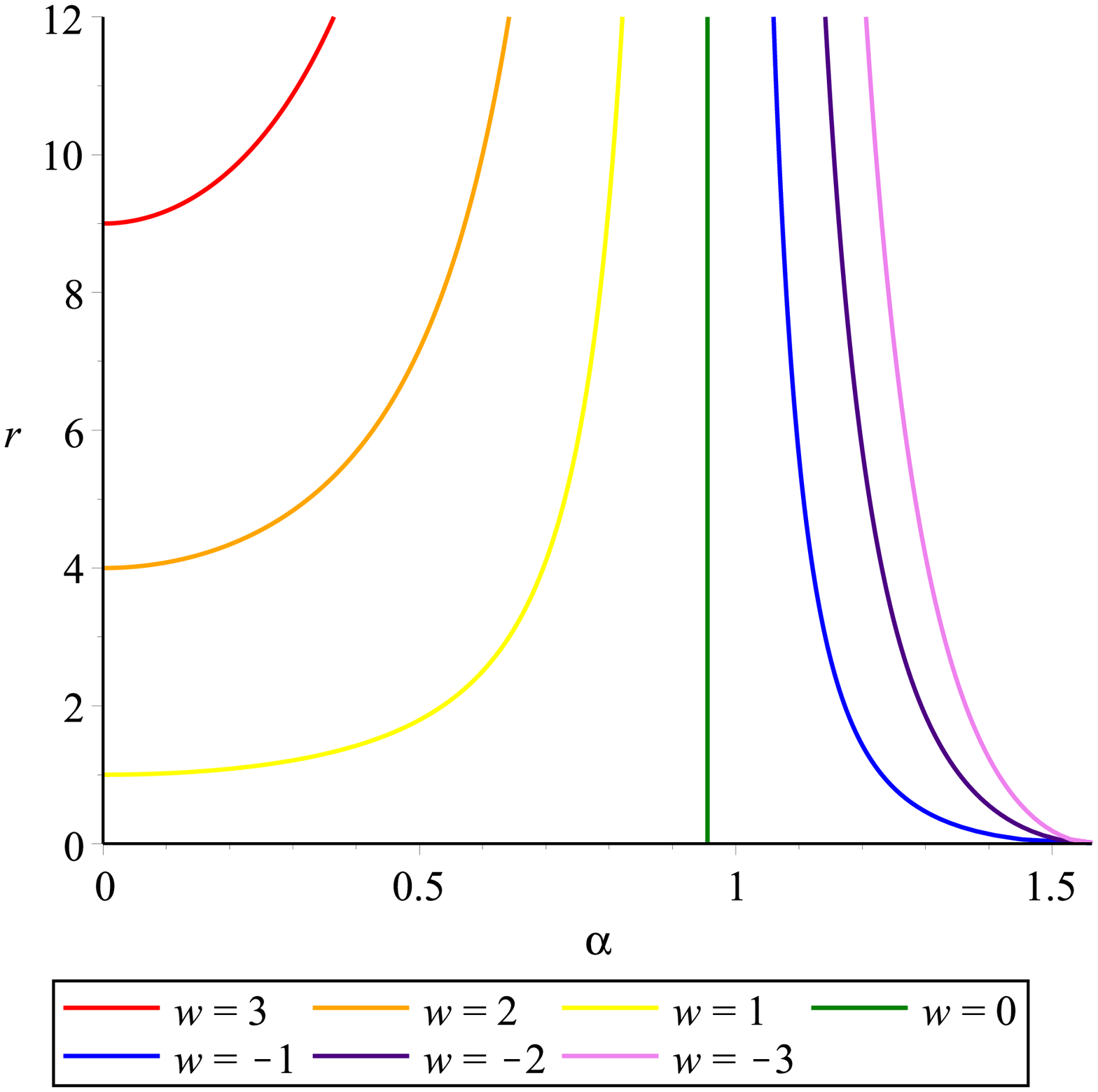}
\end{center}
\end{figure}

\subsubsection{Case 4: $\alpha \equiv \frac{\pi}{2}$}

When $\alpha \equiv \frac{\pi}{2}$, then by Table~\ref{table:SU2.orbits} the orbits are $\mathcal{S}^3$ if $\sin\gamma>0$, but are $\C\P^1$ if $\sin\gamma=0$ or $r=0$. Moreover, when $\sin\gamma>0$ the $\mathcal{S}^1$ action of $\beta$ is non-trivial. 

In this case, regardless of whether $c>0$ or $c=0$, we have a constant $w=w_{\alpha=\frac{\pi}{2}}=r\cos\gamma$ as given in Proposition~\ref{prop:cone.coass.CP2}. See Figure~\ref{fig:alpha-pi2-plot} for plots of some flow lines in the $(\gamma,r)$-plane given by constant $w$.

\begin{figure}[H]
\caption{$\alpha=\frac{\pi}{2}$}\label{fig:alpha-pi2-plot}
\begin{center}
\includegraphics[width=0.5\textwidth]{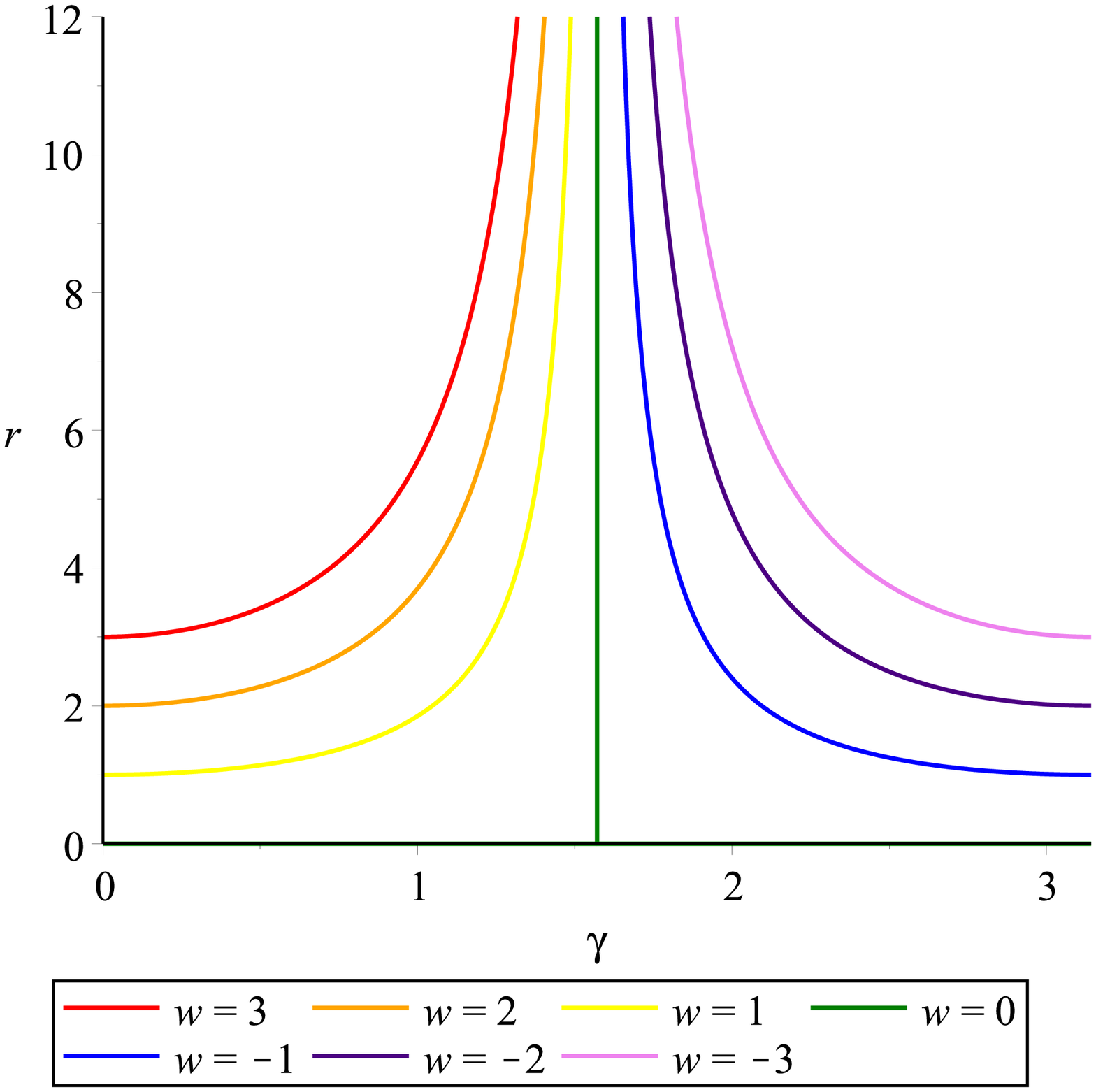}
\end{center}
\end{figure}

When $w=0$, we see that $r$ can take any positive value. If $c>0$, then we can also take $r = 0$, and by Table~\ref{table:SU2.orbits} we know that $\alpha=\frac{\pi}{2}$, $r = 0$ corresponds to a $\C\P^1\subseteq\C\P^2$. Hence, when $w=0$, we obtain coassociative fibres diffeomorphic to $\mathcal{O}_{\C\P^1}(-1)$ if $c>0$ and to $\R^+\times\mathcal{S}^3$ if $c=0$. When $w\neq 0$, its sign determines whether $\gamma\in [0,\frac{\pi}{2})$ or $\gamma\in (\frac{\pi}{2},\pi]$. In either case, we have that $r$ attains its minimum positive value when $\sin\gamma=0$ and takes every greater value, so each such coassociative fibre is diffeomorphic to $\mathcal{O}_{\C\P^1}(-1)$ and comes in a 2-parameter family.

\subsubsection{Case 5: generic setting} \label{sec:fibration-generic}

We now finally turn to the generic setting where $r>0$ and $\cos \alpha \sin \gamma \cos \gamma \not \equiv 0$, keeping in mind that we can restrict to $\gamma\in [0,\frac{\pi}{2}]$. Note that in this generic case, the expression $v = 2(c+r^2)^{\frac{1}{4}}\cos\alpha\cot\gamma$ given in~\eqref{eq:v} is a (necessarily positive) constant. 

To describe the coassociative fibres here it is equivalent to understand the flow lines for the vector field in the $(\alpha,\gamma)$-plane determined by~\eqref{eq:flow.smooth}. See Figures~\ref{fig:generic-smooth-plot} and~\ref{fig:generic-cone-plot} below.

\paragraph{Fixed points.} We observe from~\eqref{eq:flow.smooth} that critical points of the flow are the points where
\begin{align}\label{eq:smooth.crit.1}
\cos\alpha=0\quad\text{or}\quad (v^4\sin^4\gamma-16c\cos^4\alpha\cos^4\gamma)(2\cos^2\alpha+\sin^2\alpha\sin^2\gamma)=2v^4\sin^2\alpha\sin^4\gamma
\end{align}
and
\begin{align}\label{eq:smooth.crit.2-temp}
\sin\alpha\sin\gamma\cos\gamma=0\quad\text{or}\quad v^4\sin^2\alpha\sin^4\gamma+(v^4\sin^4\gamma-16c\cos^4\alpha\cos^4\gamma)\cos^2\alpha=0.
\end{align}
Note that from the first equation in~\eqref{eq:v.r.temp}  we have
\begin{equation*}
16r^2\cos^4\alpha\cos^4\gamma=v^4\sin^4\gamma-16c\cos^4\alpha\cos^4\gamma \qquad \text{ is non-negative}.
\end{equation*}
Thus~\eqref{eq:smooth.crit.2-temp} can be rewritten as
\begin{align}\label{eq:smooth.crit.2}
\sin\alpha\sin\gamma\cos\gamma=0\quad\text{or}\quad \sin\alpha\sin\gamma=(v^4\sin^4\gamma-16c\cos^4\alpha\cos^4\gamma)\cos^2\alpha=0.
\end{align}
From~\eqref{eq:smooth.crit.2} we see that at a fixed point we must have $\sin\alpha\sin\gamma\cos\gamma=0$, so $\alpha=0$ or $\gamma\in\{0,\frac{\pi}{2}\}$.

\emph{Fixed Point Case 1: $\alpha=0$.} In this case~\eqref{eq:smooth.crit.1} gives
$$\tan^4\gamma=\frac{16 c}{v^4}.$$
Notice by~\eqref{eq:r=0} that at such a point we would have $r=0$. Indeed, it is precisely the point where the curve in~\eqref{eq:r=0} meets the line $\alpha= 0$. (See Figure~\ref{fig:r-zero-plot}.) When $c=0$, this just gives the point $(0,0)$.

\emph{Fixed Point Case 2: $\gamma=0$.} In this case we see easily from~\eqref{eq:smooth.crit.2} that if $c>0$ we must have $\alpha=\frac{\pi}{2}$, which again lies on the curve in~\eqref{eq:r=0} corresponding to $r=0$. If $c=0$, we see from the simplified flow equation~\eqref{eq:u.flow.c=0} that in addition to the possibility that $\alpha = \frac{\pi}{2}$, we can also have
$$\sin^2\alpha-\cos^2\alpha=0$$
at the critical point, which gives $\alpha=\frac{\pi}{4}$.

\emph{Fixed Point Case 3: $\gamma=\frac{\pi}{2}$.} In this case if $\alpha<\frac{\pi}{2}$ then~\eqref{eq:smooth.crit.1} yields 
$$v^4(2\cos^2\alpha+\sin^2\alpha)=2v^4\sin^2\alpha.$$
Since $v > 0$ in this generic setting, this forces 
$$\tan^2\alpha=2.$$
We also have the critical point where $(\alpha,\gamma)=(\frac{\pi}{2},\frac{\pi}{2})$.

\emph{Fixed Points Summary.} We conclude that we have the following fixed points in the $(\alpha,\gamma)$-plane:
\begin{equation} \label{eq:fixed-points-c}
\big(0,\tan^{-1}(2c^{\frac{1}{4}}v^{-1})\big),\quad\big(\tfrac{\pi}{2},0\big),\quad \big(\tan^{-1}(\sqrt{2}),\tfrac{\pi}{2}\big),\quad \big(\tfrac{\pi}{2},\tfrac{\pi}{2}\big) \qquad \text{when $c>0$},
\end{equation}
and
\begin{equation} \label{eq:fixed-points-0}
(0,0), \quad \big(\tfrac{\pi}{4},0\big), \quad \big(\tfrac{\pi}{2},0\big), \quad \big(\tan^{-1}(\sqrt{2}),\tfrac{\pi}{2}\big),\quad \big(\tfrac{\pi}{2},\tfrac{\pi}{2}\big) \qquad \text{when $c=0$}.
\end{equation}

\paragraph{Boundary curve.} Now that we have the fixed points for our vector field in the $(\alpha,\gamma)$-plane, we seek to understand the dynamics of the vector field. 

When $c>0$, we are only interested in flow lines which lie above the curve defining $r=0$ in~\eqref{eq:r=0}, and here the coefficient in front of $\dot{\alpha}$ in~\eqref{eq:flow.smooth} is positive. Thus, the dynamics of the vector field are completely controlled by the coefficient of $\dot{\gamma}$ in~\eqref{eq:flow.smooth}, which vanishes on the curve $\Gamma$ where
\begin{equation}\label{eq:boundary.curve.smooth}
(v^4\sin^4\gamma-16c\cos^4\alpha\cos^4\gamma)(2\cos^2\alpha+\sin^2\alpha\sin^2\gamma)=2v^4\sin^2\alpha\sin^4\gamma.
\end{equation}
Away from this boundary curve $\Gamma$, we have that
\begin{equation}\label{eq:dgda.smooth}
\frac{\d\gamma}{\d\alpha}=\frac{2\sin\alpha\sin\gamma\cos\gamma\big(v^4\sin^2\alpha\sin^4\gamma+(v^4\sin^4\gamma-16 c\cos^4\alpha\cos^4\gamma)\cos^2\alpha\big)}{\cos\alpha\big((v^4\sin^4\gamma-16 c\cos^4\alpha\cos^4\gamma)(2\cos^2\alpha+\sin^2\alpha\sin^2\gamma)-2v^4\sin^2\alpha\sin^4\gamma\big)}.
\end{equation}
In Figure~\ref{fig:generic-smooth-plot} we show the flow lines for $c=1$ and $v=4$, as well as the boundary curve $\Gamma$, the curve where $r=0$, and the asymptote to $\Gamma$ where $\tan^2\alpha=2$.

\begin{figure}[H]
\caption{Flow lines when $c=1$}\label{fig:generic-smooth-plot}
\begin{center}
\includegraphics[width=0.5\textwidth]{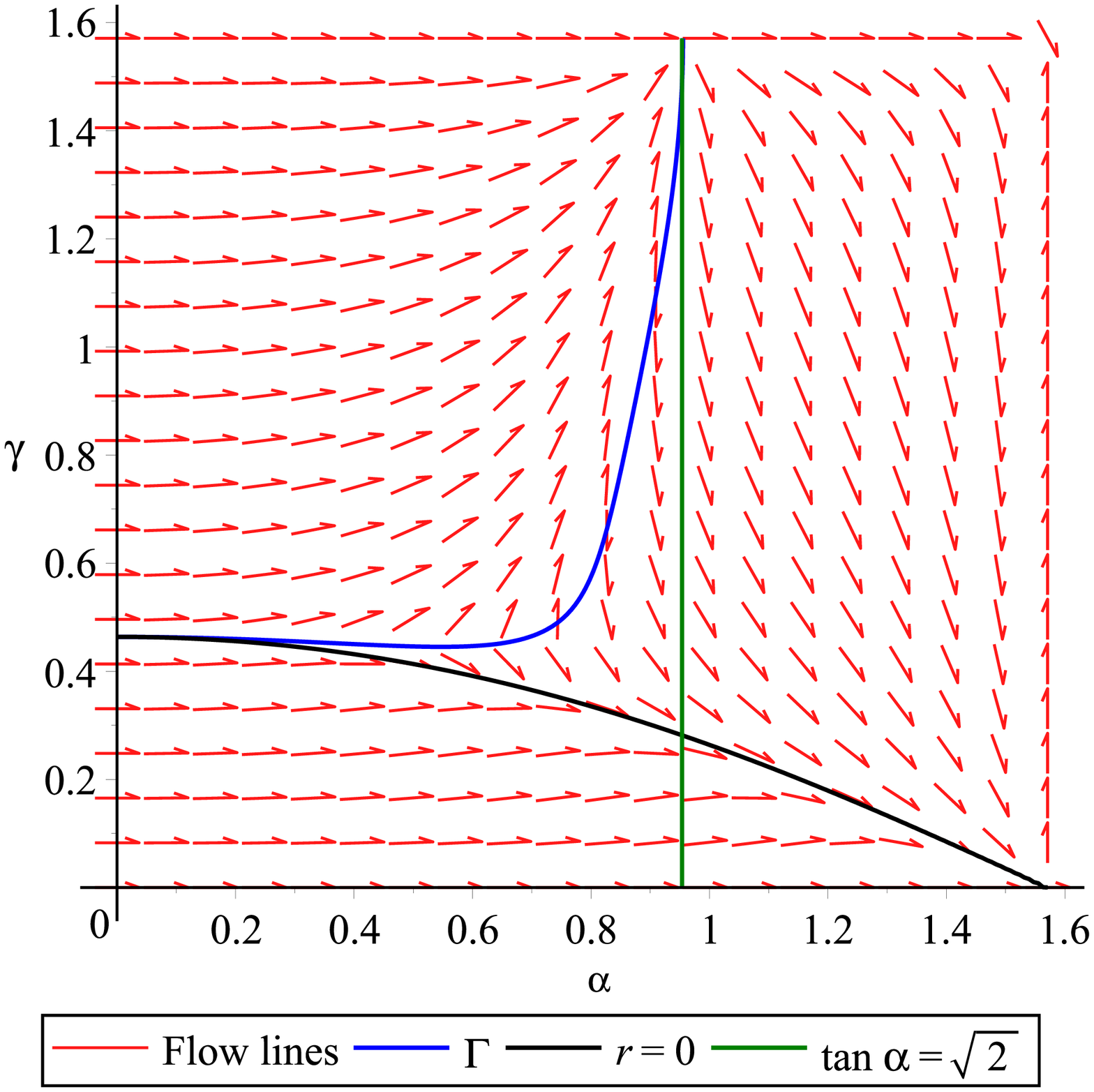}
\end{center}
\end{figure}

When $c=0$, equation~\eqref{eq:dgda.smooth} simplifies greatly, as one can also see directly from~\eqref{eq:u.flow.c=0}, to give
$$\frac{\d\gamma}{\d\alpha}=\frac{2 \sin\alpha\sin\gamma\cos\gamma}{\cos\alpha(2\cos^2\alpha-2\sin^2\alpha+\sin^2\alpha\sin^2\gamma)}$$
away from a boundary curve $\Gamma$ given by
$$2\cos^2\alpha-2\sin^2\alpha+\sin^2\alpha\sin^2\gamma=0.$$
In Figure~\ref{fig:generic-cone-plot} we again show the flow lines and the boundary curve $\Gamma$, but now we have two asymptotes to $\Gamma$ when $\tan\alpha=1$ and $\tan\alpha=\sqrt{2}$.

\begin{figure}[H]
\caption{Flow lines when $c=0$}\label{fig:generic-cone-plot}
\begin{center}
\includegraphics[width=0.6\textwidth]{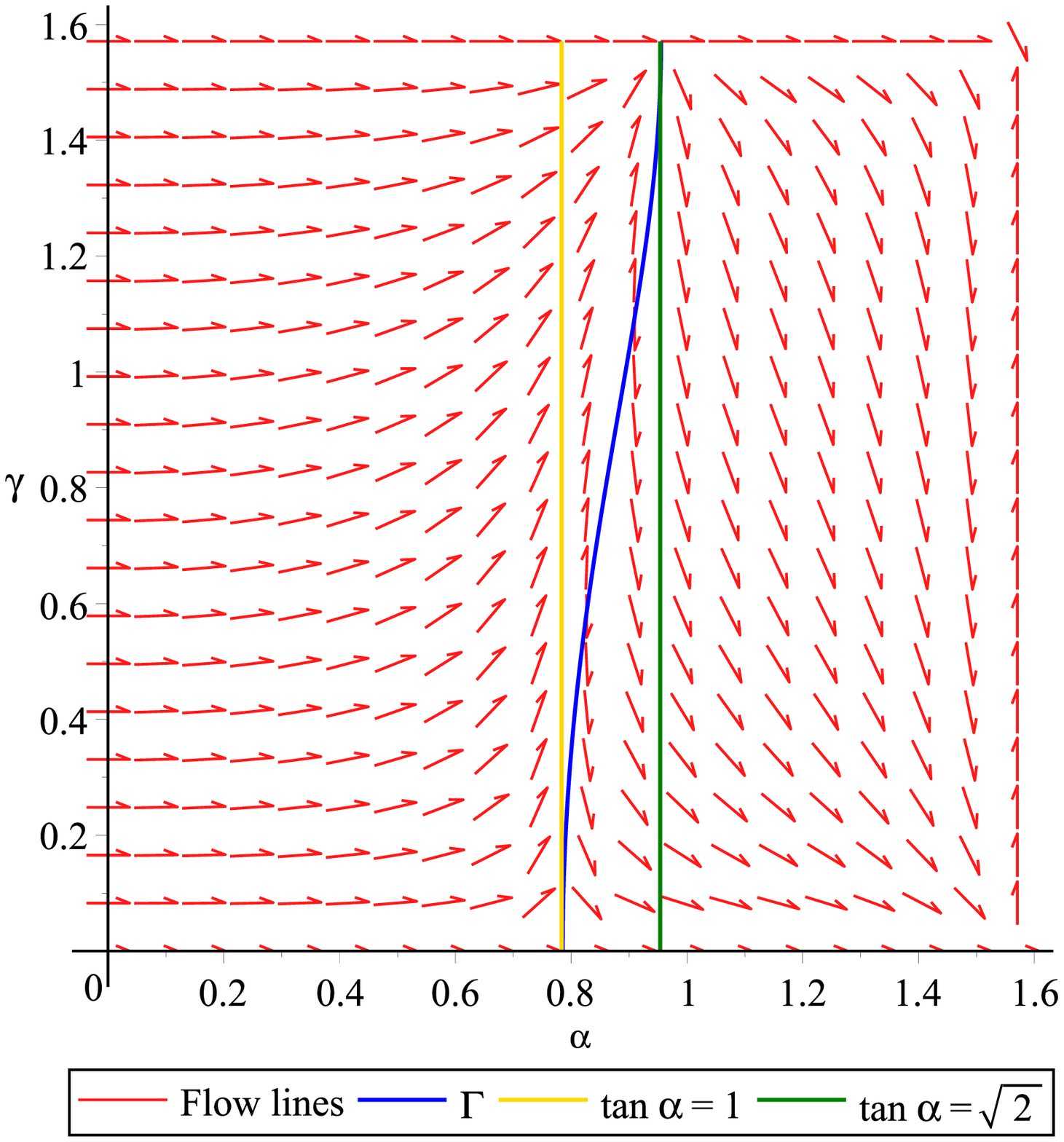}
\end{center}
\end{figure}

Of course, when $c=0$ we know that we can in fact integrate~\eqref{eq:u.flow.c=0} explicitly to give the constant~\eqref{eq:u.c=0} and thus we can plot the integral curves of the vector field below in Figure~\ref{fig:u-cone-plot}.

\begin{figure}[H]
\caption{Curves with constant $u$ when $c=0$}\label{fig:u-cone-plot}
\begin{center}
\includegraphics[width=0.7\textwidth]{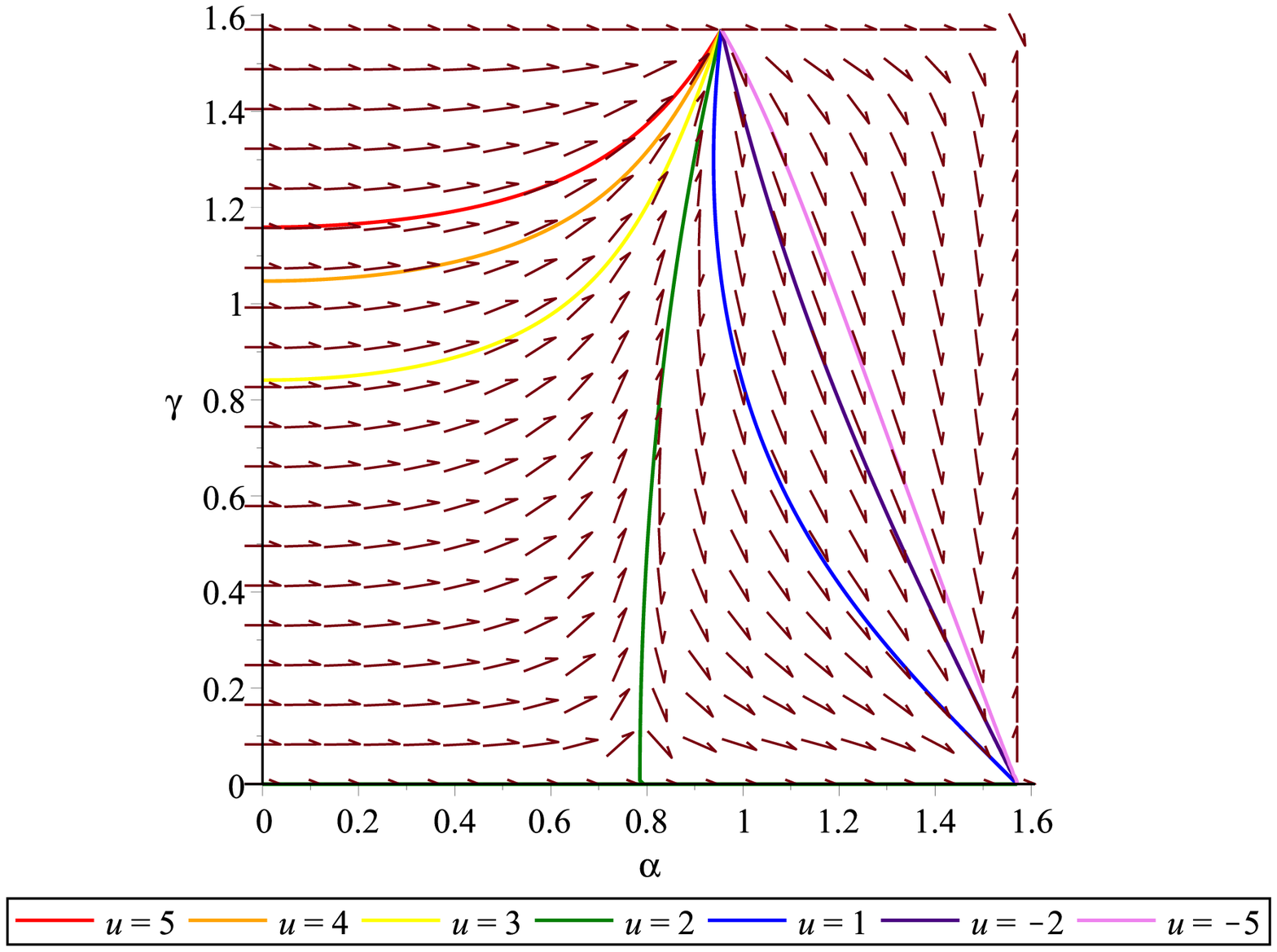}
\end{center}
\end{figure}

In both cases, $c>0$ and $c=0$, we again see that $\frac{\d\gamma}{\d\alpha}>0$ to the left of the boundary curve $\Gamma$, and $\frac{\d\gamma}{\d\alpha}<0$ to the right of $\Gamma$. Moreover, as we have seen before, flow lines must cross $\Gamma$ vertically.

\paragraph{Region 1: right of $\Gamma$.} We see that no flow line in this region can meet $\alpha=\frac{\pi}{2}$ except at $\gamma=0$, since $\alpha\equiv \frac{\pi}{2}$ solves~\eqref{eq:flow.smooth} and $\frac{\d\gamma}{\d\alpha}\to\infty$ as $(\alpha,\gamma)\to(\frac{\pi}{2},\gamma_0)$ for any $\gamma_0\in(0,\frac{\pi}{2})$. 

Moreover, as we already remarked earlier, no flow line can meet the curve where $r=0$ except when $\alpha\in\{0,\frac{\pi}{2}\}$. Therefore, any flow line in the region right of $\Gamma$ must flow into $(\frac{\pi}{2},0)$.

If the flow line is always to the right of $\Gamma$ then it must emanate from the critical point $(\tan^{-1}(\sqrt{2}),\frac{\pi}{2})$, since $\gamma\equiv\frac{\pi}{2}$ solves~\eqref{eq:flow.smooth} and we see that no flow line can meet this one except at a critical point.

If the flow line has crossed $\Gamma$, then for all earlier times it must have been to the left of $\Gamma$. However, $\gamma$ is decreasing as the flow line crosses the boundary curve and so $\alpha$ must have been increasing before then. Since the flow line cannot cross $\Gamma$ again it must have emanated from a point where $\gamma=\frac{\pi}{2}$ and, as before, must have come from a critical point, which is the same one as above.

We now want to understand the topology of these coassociatives.

As $\gamma\to\frac{\pi}{2}$ and $\alpha\to\tan^{-1}(\sqrt{2})$, we see that since $v$ in~\eqref{eq:v} is constant we must have that $r\to\infty$. When $c>0$, the point $(\frac{\pi}{2},0)$ corresponds to a $\C\P^1$, by Table~\ref{table:SU2.orbits} and the fact that $r=0$ there from~\eqref{eq:r=0}. Thus, when $c>0$, we see that all of the flow lines terminating at $(\frac{\pi}{2},0)$, of which there is an open 1-parameter family determined by the angle at which they enter the critical point, give coassociative fibres diffeomorphic to $\mathcal{O}_{\C\P^1}(-1)$.

When $c=0$ instead, we see again from~\eqref{eq:v} that as $(\alpha,\gamma)\to (\frac{\pi}{2},0)$ we must have that $\frac{r^{\frac{1}{2}}\cos\alpha}{\sin\gamma}$ tends to a positive constant, because $v>0$. Moreover, we must have that $\frac{\d\gamma}{\d\alpha}$ tends to a negative constant as $(\alpha,\gamma)\to (\frac{\pi}{2},0)$, because the unique curve passing through $(\frac{\pi}{2},0)$ with $\frac{\d\gamma}{\d\alpha}\to 0$ at $(\frac{\pi}{2},0)$ is $\gamma\equiv 0$, and the unique curve with $\frac{\d\alpha}{\d\gamma}\to 0$ at $(\frac{\pi}{2},0)$ is $\alpha\equiv \frac{\pi}{2}$. Therefore $\frac{\cos\alpha}{\sin\gamma}$ must tend to a positive constant as $(\alpha,\gamma)\to (\frac{\pi}{2},0)$ and hence $r$ tends to a positive constant, which by Table~\ref{table:SU2.orbits} corresponds to a $\C\P^1$. We note here for future use, that this argument shows that when $c=0$, as $(\alpha,\gamma)\to (\frac{\pi}{2},0)$, using~\eqref{eq:u.c=0} we have that
$$u\to 2-\lim_{(\alpha,\gamma)\to (\frac{\pi}{2},0)}\frac{\sin^2\gamma}{\cos^2\alpha}<2.$$

Altogether we see that these flow lines in Region 1 define coassociatives diffeomorphic to $\mathcal{O}_{\C\P^1}(-1)$, and that they come in a 3-parameter family, determined by $v$, $\beta$, and a constant $u<2$ (which is given by~\eqref{eq:u.c=0} when $c=0$).

\paragraph{Region 2: left of $\Gamma$.} Here, since $\frac{\d\gamma}{\d\alpha}>0$ for $\alpha>0$, we see that flow lines can only meet $\gamma=\frac{\pi}{2}$ at the critical point $(\tan^{-1}(\sqrt{2}),\frac{\pi}{2})$, so all flow lines pass through this point.

The flow lines must either cross $\Gamma$ or stay to the left of $\Gamma$. If the former occurs then we are back in Region 1 above, so we assume that our flow line stays to the left of $\Gamma$. Since $\frac{\d\gamma}{\d\alpha}\to 0$ as $(\alpha,\gamma)\to (0,\gamma_0)$ for $\gamma_0\in(0,\frac{\pi}{2})$, we see that we have flow lines emanating from $(0,\gamma_0)$ for all $\gamma_0\in(0,\frac{\pi}{2})$, which terminate in $(\tan^{-1}(\sqrt{2}),\frac{\pi}{2})$. The point $(0,\gamma_0)$ gives a $\C\P^1$, by Table~\ref{table:SU2.orbits} and the fact that $r$ has a minimum positive value there. Thus we obtain coassociative fibres diffeomorphic to $\mathcal{O}_{\C\P^1}(-1)$ from these flow lines.

Combining these observations with our analysis of Region 1, we therefore see that by varying the angle in $(0,\pi)$ at which the flow line enters (or leaves) the critical point $(\tan^{-1}(\sqrt{2}),\frac{\pi}{2})$, we have an open interval (with endpoint $0$) of angles giving flow lines which flow out of $(0,\gamma_0)$ for all $\gamma_0\in(0,\frac{\pi}{2})$ and another open interval with endpoint $\pi$ which flow into $(\frac{\pi}{2},0)$.

We deduce that, when $c>0$, there must be a closed connected set of flow lines which flow into the unique critical point that has $\alpha=0$, as given in~\eqref{eq:fixed-points-c} and~\eqref{eq:fixed-points-0}. Further, we see from~\eqref{eq:dgda.smooth} that $\frac{\d\gamma}{\d\alpha}\to 0$ as any flow line tends to the critical point where $\alpha=0$ (because $\gamma>0$ there). Hence, there must in fact be a unique flow line passing through that critical point. Recall that by the discussion on fixed points at the start of this section, this critical point where $\alpha = 0$ also has $r=0$. Hence, we obtain a coassociative fibre which is $\R^+\times\mathcal{S}^3$ (because we exclude the point where $r=0$).

When $c=0$, we cannot have any flow line emanating from $(0,0)$ because $\frac{\d\gamma}{\d\alpha}\to 0$ there, and the unique solution is $\gamma\equiv 0$. Therefore, by~\eqref{eq:fixed-points-0}, we instead have a closed connected set of flow lines emanating from the critical point $(\frac{\pi}{4},0)$. As $(\alpha,\gamma)\to (\frac{\pi}{4},0)$, because $v>0$ is constant, by~\eqref{eq:v} we must have that $r\to 0$. Furthermore we observe in this case from~\eqref{eq:u.c=0} that as $(\alpha,\gamma)\to (\frac{\pi}{4},0)$ we have
$$u=\frac{2}{\cos \gamma}-\frac{\cos^2\alpha\sin^2\gamma}{\sin^2\alpha\cos\gamma}\to 2.$$ 
Thus $u=2$, and this flow line again defines a coassociative diffeomorphic to $\R^+\times\mathcal{S}^3$.

In conclusion, we may define a parameter $u$ (which is given precisely by~\eqref{eq:u.c=0} in the case when $c=0$) so that the flow lines in Region 2 define a 3-parameter family of coassociatives diffeomorphic to $\mathcal{O}_{\C\P^1}(-1)$ for any $v$ and $\beta$ and any $u>2$, but we get a 2-parameter family of coassociatives fibres diffeomorphic to $\R^+\times\mathcal{S}^3$ when $u=2$, in which case the free parameters are $v$ and $\beta$.

\subsubsection{Summary and discussion}\label{subs:summary}

We summarise our findings in the smooth case $M$ in Table~\ref{table:fibration-M}. In some cases we do not know the parameters explicitly because we have been unable to integrate equations, but the argument gives natural parameters which we identify with perturbations of the parameters $u,v,w$ we found in the cone case. Here, $\alpha_0(w)$, $\alpha_0(u,v)$, $\gamma_0(u,v)$, and $r_0(u,v)$ are positive constants simply used to denote the minimum value $\alpha$, $\gamma$, and $r$ can attain in those cases (which is determined either by $w$, or by $u$ and $v$, respectively), and we let $\gamma_0(v)=\tan^{-1}(2c^\frac{1}{4}v^{-1})$.

\begin{table}[H]
\begin{center}
{\setlength{\extrarowheight}{4pt}
\begin{tabular}{|c|c|c|c|c|}
\hline
$r$ & $\gamma$ & $\alpha$ & Topology & Parameters\\[2pt]
\hline
$0$ & N/A & $\in[0,\frac{\pi}{2}]$ & $\C\P^2$ & N/A \\[2pt]
\hline
$\geq w>0$ & $0,\pi$ & $\in[0,\frac{\pi}{4})$ & $\mathcal{O}_{\C\P^1}(-1)$ & $w>0$ \\[2pt]
 \hline
$\geq -w> 0$ & $0,\pi$ & $\in[\alpha_0(w),\frac{\pi}{2}]$ & $\mathcal{O}_{\C\P^1}(-1)$ & $w<0$ \\[2pt]
\hline
$\geq w>0$ & $\frac{\pi}{2}$ & $<\frac{1}{2}\cos^{-1}(-\frac{1}{3})$ & $\mathcal{O}_{\C\P^1}(-1)$ & $w>0$, $\beta\in[0,2\pi)$\\[2pt]
\hline 
$\geq 0$ & $\frac{\pi}{2}$ & $\in [\alpha_0(w),\frac{\pi}{2}]$ & $\mathcal{O}_{\C\P^1}(-1)$ & $w<0$, $\beta\in[0,2\pi)$\\[2pt]
 \hline
$\geq w>0$ & $\in [0,\frac{\pi}{2})$ & $\frac{\pi}{2}$ & $\mathcal{O}_{\C\P^1}(-1)$ & $w>0$, $\beta\in [0,2\pi)$\\[2pt]
\hline 
$\geq 0$ & $\frac{\pi}{2}$ & $\frac{\pi}{2}$ & $\mathcal{O}_{\C\P^1}(-1)$ & $w=0$, $\beta\in [0,2\pi)$\\[2pt]
\hline 
$\geq -w>0$ & $\in (\frac{\pi}{2},\pi]$ & $\frac{\pi}{2}$ & $\mathcal{O}_{\C\P^1}(-1)$ & $w<0$, $\beta\in [0,2\pi)$ \\[2pt]
\hline
$\geq r_0(u,v)>0$ & $\in [\gamma_0(u,v),\frac{\pi}{2})$ & $\in [0,\tan^{-1}(\sqrt{2}))$ & $\mathcal{O}_{\C\P^1}(-1)$ & $u>2$, $v>0$, $\beta\in [0,2\pi)$\\[2pt]
\hline
$> 0$ & $\in (\gamma_0(v),\frac{\pi}{2})$ & $\in [0,\tan^{-1}(\sqrt{2}))$ & $\R^+\times\mathcal{S}^3$ & $u=2$, $v>0$, $\beta\in [0,2\pi)$\\[2pt]
\hline
$\geq 0$ & $\in [0,\frac{\pi}{2})$ & $\in [\alpha_0(u,v),\frac{\pi}{2}]$ & $\mathcal{O}_{\C\P^1}(-1)$ & $u<2$, $v>0$, $\beta\in [0,2\pi)$\\[2pt]
\hline
$\geq r_0(u,v)>0$ & $\in (\frac{\pi}{2},\pi-\gamma_0(u,v))$ & $\in [0,\tan^{-1}(\sqrt{2}))$ & $\mathcal{O}_{\C\P^1}(-1)$ & $u<-2$, $v>0$, $\beta\in [0,2\pi)$\\[2pt]
\hline
$> 0$ & $\in (\frac{\pi}{2},\pi-\gamma_0(v))$ & $\in [0,\tan^{-1}(\sqrt{2}))$ & $\R^+\times\mathcal{S}^3$ & $u=-2$, $v>0$, $\beta\in [0,2\pi)$\\[2pt]
\hline
$\geq 0$ & $\in (\frac{\pi}{2},\pi]$ & $\in [\alpha_0(u,v),\frac{\pi}{2}]$ & $\mathcal{O}_{\C\P^1}(-1)$ & $u>-2$, $v>0$, $\beta\in [0,2\pi)$\\[2pt]
\hline
\end{tabular}}
\end{center}
\caption{Fibration for $M=\Lambda^2_-(T^*\C\P^2)$} \label{table:fibration-M}
\end{table}

We summarize our findings in the cone case $M_0$ in Table~\ref{table:fibration-M0}. Here we do not give the values $\alpha_0(u,v)>\frac{\pi}{4}$, $\gamma_0(u,v)>0$, and $r_0(u,v)>0$ explicitly, but one can determine them from the parameters $u$ and $v$.

\begin{table}[H]
\begin{center}
{\setlength{\extrarowheight}{4pt}
\begin{tabular}{|c|c|c|c|c|}
\hline
$r$ & $\gamma$ & $\alpha$ & Topology & Parameters\\[2pt]
\hline
$\geq w>0$ & $0,\pi$ & $\in[0,\frac{\pi}{4})$ & $\mathcal{O}_{\C\P^1}(-1)$ & $w>0$ \\[2pt]
\hline
$>0$ & $0,\pi$ & $\frac{\pi}{4}$ & $\R^+\times\mathcal{S}^3$ & $w=0$ \\[2pt]
 \hline
$\geq -w>0$ & $0,\pi$ & $\in(\frac{\pi}{4},\frac{\pi}{2}]$ & $\mathcal{O}_{\C\P^1}(-1)$ & $w<0$\\[2pt]
\hline 
$\geq w>0$ & $\frac{\pi}{2}$ & $<\frac{1}{2}\cos^{-1}(-\frac{1}{3})$ & $\mathcal{O}_{\C\P^1}(-1)$ & $w>0$, $\beta\in [0,2\pi)$\\[2pt]
 \hline
$>0$ & $\frac{\pi}{2}$ & $\frac{1}{2}\cos^{-1}(-\frac{1}{3})$ & $\R^+\times\mathcal{S}^3$ & $w=0$, $\beta\in [0,2\pi)$\\[2pt]
\hline
$>0$ & $\frac{\pi}{2}$ & $>\frac{1}{2}\cos^{-1}(-\frac{1}{3})$ & $\R^+\times\mathcal{S}^3$ & $w<0$, $\beta\in [0,2\pi)$\\[2pt]
\hline 
$\geq w>0$ & $\in [0,\frac{\pi}{2})$ & $\frac{\pi}{2}$ & $\mathcal{O}_{\C\P^1}(-1)$ & $w>0$, $\beta\in [0,2\pi)$\\[2pt]
\hline 
$>0$ & $\frac{\pi}{2}$ & $\frac{\pi}{2}$ & $\R^+\times\mathcal{S}^3$ & $w=0$, $\beta\in [0,2\pi)$\\[2pt]
\hline 
$\geq -w>0$ & $\in (\frac{\pi}{2},\pi]$ & $\frac{\pi}{2}$ & $\mathcal{O}_{\C\P^1}(-1)$ & $w<0$, $\beta\in [0,2\pi)$ \\[2pt]
\hline
$\geq r_0(u,v)$ & $\in [\gamma_0(u,v),\frac{\pi}{2})$ & $\in [0,\tan^{-1}(\sqrt{2}))$ & $\mathcal{O}_{\C\P^1}(-1)$ & $u>2$, $v>0$, $\beta\in [0,2\pi)$\\[2pt]
\hline
$>0$ & $\in (0,\frac{\pi}{2})$ & $\in (\frac{\pi}{4},\tan^{-1}(\sqrt{2}))$ & $\R^+\times\mathcal{S}^3$ & $u=2$, $v>0$, $\beta\in [0,2\pi)$\\[2pt]
\hline
$\geq r_0(u,v)$ & $\in (0,\frac{\pi}{2})$ & $\in [\alpha_0(u,v),\frac{\pi}{2})$ & $\mathcal{O}_{\C\P^1}(-1)$ & $u<2$, $v>0$, $\beta\in [0,2\pi)$\\[2pt]
\hline
$\geq r_0(u,v)$ & $\in (\frac{\pi}{2},\pi-\gamma_0(u,v)]$ & $\in [0,\tan^{-1}(\sqrt{2}))$ & $\mathcal{O}_{\C\P^1}(-1)$ & $u<-2$, $v<0$, $\beta\in [0,2\pi)$\\[2pt]
\hline
$>0$ & $\in (\frac{\pi}{2},\pi)$ & $\in (\frac{\pi}{4},\tan^{-1}(\sqrt{2}))$ & $\R^+\times\mathcal{S}^3$ & $u=-2$, $v<0$, $\beta\in [0,2\pi)$\\[2pt]
\hline
$\geq r_0(u,v)$ & $\in (\frac{\pi}{2},\pi)$ & $\in [\alpha_0(u,v),\frac{\pi}{2})$ & $\mathcal{O}_{\C\P^1}(-1)$ & $u>-2$, $v<0$, $\beta\in [0,2\pi)$\\[2pt]
\hline
\end{tabular}}
\end{center}
\caption{Fibration for $M_0=\R^+\times(\SU(3)/T^2$)} \label{table:fibration-M0}
\end{table}

We note the following important observations about these two coassociative fibrations.
\begin{itemize}
\item In both $M$ and $M_0$, the \emph{generic} coassociative fibre in the fibration is diffeomorphic to $\mathcal{O}_{\C\P^1}(-1)$, but there is a \emph{codimension 1 subfamily} of the fibration consisting of coassociatives which are diffeomorphic to $\R^+\times\mathcal{S}^3$. (See also Remark~\ref{rmk:CP2.fib}.)
\item Comparing the two tables above with Table~\ref{table:SU2.orbits} on the $\SU(2)$ orbits, we deduce the following. Given a fixed $\C\P^1$ orbit of $\SU(2)$ in the $\GG_2$ manifold, there is generically a 2-parameter family of smooth $\mathcal{O}_{\C\P^1}(-1)$ fibres all of whose zero sections are the same fixed $\C\P^1$. (That is, each smooth fibre intersects in a $\C\P^1$ with a 2-parameter family of smooth fibres.) The two parameters are determined by $\beta \in [0, 2\pi)$ and the condition that the minimum value $r_0(u,v)$ of $r$ be the same for all smooth fibres in this family, as this minimum value corresponds to the bolt size.
\item The coassociative fibres only ever intersect in the bolts of smooth $\mathcal{O}_{\C\P^1}(-1)$ fibres. From Table~\ref{table:SU2.orbits} we see that there is a 1-parameter family of $\C\P^1$ orbits (the bolts). Hence the subset $M \setminus M'$ or $M_0 \setminus M'_0$ containing the points that lie in multiple coassociative ``fibres'' is a 1-parameter family of $\mathcal{S}^2$'s, and thus is 3-dimensional. It would be interesting to study the geometric structure of this set. For example, is it the case that at the smooth points this set is an associative submanifold?
\end{itemize}

\subsection{Relation to multi-moment maps} \label{subs:multimoment.CP2}

Let $X_1$, $X_2$, $X_3$ denote the vector fields which are dual to the 1-forms $\sigma_1$, $\sigma_2$, $\sigma_3$ in~\eqref{eq:sigmas-CP2}, that is such that $\sigma_i(X_j)=\delta_{ij}$. Then $X_1$, $X_2$, $X_3$ generate the $\SU(2)$ action and hence preserve both $\varphi_c$ and $*_{\varphi_c} \varphi_c$. In this section we find the multi-moment map $\nu$ for this $\SU(2)$ action for the $4$-form $*_{\varphi_c} \varphi_c$ as in~\eqref{eq:mmm4}.

Using the expression~\eqref{eq:psic.CP2} for the 4-form, the expressions~\eqref{eq:volCP2.CP2} and~\eqref{eq:zizjOk.CP2}, and the fact that $\sigma_1 \w \sigma_2 \w \sigma_3 = p_1 \w p_2 \w p_3$, a computation gives
\begin{align*}
*_{\varphi_c} \varphi_c (X_1,X_2,X_3,\cdot) & = \tfrac{1}{2} r (4 \cos^2 \alpha + \sin^2 \gamma \sin^2 \alpha) \sin^2 \alpha \d r + \tfrac{1}{2} r^2 \sin \gamma \cos \gamma \sin^4 \alpha \d \gamma \\
& \qquad{} -\sin \alpha \cos \alpha (2 c \sin^2 \alpha + 2 r^2 \sin^2 \alpha - 2 r^2 \cos^2 \alpha - r^2 \sin^2 \alpha \sin^2 \gamma) \d \alpha \\
& = \d \big( \tfrac{1}{4} r^2 \sin^2 \alpha (4 \cos^2 \alpha + \sin^2 \gamma \sin^2 \alpha) - \tfrac{1}{2} c \sin^4 \alpha \big).
\end{align*}
The above calculation motivates the following definition.

\begin{dfn}
We define the function $\rho$ by
\begin{equation}\label{eq:rho.CP2}
\rho= \tfrac{1}{4} r^2 \sin^2 \alpha (4 \cos^2 \alpha + \sin^2 \gamma \sin^2 \alpha) + \tfrac{1}{2} c (1 - \sin^4 \alpha).
\end{equation}
so that $*_{\varphi_c} \varphi_c (X_1,X_2,X_3,\cdot) = \d \rho$. We have added the constant $\frac{1}{2} c$ to ensure that $\rho \geq 0$ always. (Note that this $\rho$ is the analogue of what we called $\rho^2$ in the $\mathcal{S}^4$ case.)
\end{dfn}

\begin{remark}
The relationship between the function $\rho$ and the bolts of the smooth fibres is less satisfactory in this $\C\P^2$  case, as compared to the $\mathcal{S}^4$ case.  In the cone case, where $c=0$, we have that $\rho=0$ precisely on the bolts.  If $c>0$ and $\rho = 0$, then we must have $\alpha = \frac{\pi}{2}$ and $\sin \gamma = 0$, and thus by Table~\ref{table:SU2.orbits} necessarily be on a bolt. However, there are other bolts, corresponding to $\alpha = 0$ in Table~\ref{table:SU2.orbits}, that are not detected by $\rho = 0$ in the smooth case $(c>0)$. 
\end{remark}

We have thus established the following proposition.

\begin{prop} The multi-moment map for the $\SU(2)$ action on $*_{\varphi_c} \varphi_c$ is $\rho$, which maps onto $[0,\infty)$.
\end{prop} 

Recall from Proposition~\ref{prop:smooth.coass.CP2} that in the smooth case ($c>0)$ we were unable to integrate the final ordinary differential equation to obtain the third independent constant on the coassociative fibres. However, in the cone case ($c=0$) we were able in Proposition~\ref{prop:cone.coass.CP2} to obtain three linearly independent \emph{exact} horizontal 1-forms $\d u$, $\d v$, and $\d \beta$ on the generic coassociative fibres.

Consider the following rotated frame
\begin{equation*}
Y_1 = X_1, \qquad Y_2 = \cos \beta X_2 + \sin \beta X_3, \qquad Y_3 = - \sin \beta X_2 + \cos \beta X_3,
\end{equation*}
which is dual to the rotated coframe $p_1, p_2, p_3$ introduced in~\eqref{eq:beta-rotated}. In the $c=0$ case, one can compute using the expression~\eqref{eq:phic.CP2} for $\varphi_c$ and the expressions~\eqref{eq:z1z2z3.CP2} and~\eqref{eq:zkOk.CP2} that
\begin{align*}
\varphi_c(Y_2,Y_3,\cdot)&= \frac{r^{\frac{3}{2}} \cos^2 \alpha \cos^2 \gamma}{\sin^2 \gamma} \d u - \frac{r(1-2\cos^2 \alpha)\sin \gamma}{\cos \alpha} \d v, \\
\varphi_c(Y_3,Y_1,\cdot)&=  \frac{r^{\frac{3}{2}} \cos^3 \alpha \cos \gamma}{\sin \gamma} \d u - \frac{r(1-3\cos^2 \alpha)\sin^2 \gamma}{2 \cos \gamma} \d v, \\
\varphi_c(Y_1,Y_2,\cdot)&=- r^{\frac{3}{2}} \sin^2 \alpha \cos \alpha \sin \gamma \d\beta.
\end{align*}
As expected by Remark~\ref{rmk:mmm}, these are not closed forms. Remark~\ref{rmk:mmmS4} also applies here.

\subsection{Rewriting the package of the \texorpdfstring{G\textsubscript{2}}{G2}-structure} \label{sec:package.SU2}

We can now construct a $\GG_2$ adapted coframe that is compatible with the coassociative fibration structure as in Lemma~\ref{lemma:adapted-frame}. In this section, due to the complexity of the intermediate formulae, we do not give the step-by-step computations, but we present enough details that the reader will know how to reproduce the computation if desired.

From~\S\ref{sub:SU2-invariant} we know that generically $\d v$ and $\d \beta$ are well-defined horizontal 1-forms pulled back from the base of the fibration. Moreover, when $c=0$, we have a third independent horizontal 1-form which is $\d u$. As we did in the $\mathcal{S}^4$ case in~\S\ref{sec:package.SO3}, we can apply the fundamental relation~\eqref{eq:metric-from-form}, yielding an explicit, albeit extremely complicated, formula for $g_c$ in terms of the local coordinates $r, \gamma, \alpha, \beta, \gamma, \theta, \phi$. As before, we omit the particular expression here. Again, one can use this expression to obtain the inverse metric $g_c^{-1}$ on 1-forms. In this case it is still true that $\d v$, $\d \beta$, and $\d \rho$ are mutually orthogonal. However, even when we can define $\d u$ in the $c=0$ case, it is \emph{not} orthogonal to $\d v$. Nevertheless, in either case we have enough data to apply Lemma~\ref{lemma:adapted-frame} to obtain our $\GG_2$ adapted oriented orthonormal coframe
\[ \{ \hat h_1, \hat h_2, \hat h_3, \hat \varpi_0,  \hat \varpi_1,  \hat \varpi_2,  \hat \varpi_3 \} \]
where
\begin{equation*}
h_2 = \d v, \quad h_3 = \d \beta, \quad \varpi_0 = \d \rho.
\end{equation*}
We can express the three 1-forms $\hat \varpi_1, \hat \varpi_2, \hat \varpi_3$ in terms of $\d v$, $\d \beta$, $\d \rho$, the rotated coframe $p_1, p_2, p_3$ of~\eqref{eq:beta-rotated}, and $\hat h_1 = \hat h_3^{\sharp} \lrcorner ( \hat h_2^{\sharp} \lrcorner \varphi)$. We do not make use of the expressions for $\varphi_c$ and $\ast_{\varphi_c} \varphi_c$, so we present only the one for $g_c$. In fact, all we   need is the restriction $g_c|_N$ of the formula for $g_c$ to a coassociative fibre $N$. For convenience, we define the following four non-negative quantities:
\begin{align*}
F^2 & = c + r^2, \\
J^2 & = 4 F^2 \cos^2 \alpha + r^2 \sin^2 \alpha \sin^2 \gamma \cos^2 \gamma, \\
K^2 & = 4 c \sin^2 \alpha \cos^2 \alpha + 4 r^2 \cos^2 \alpha + 3 r^2 \sin^2 \alpha \cos^2 \alpha \sin^2 \gamma + r^2 \sin^2 \alpha \sin^2 \gamma, \\
L^2 & = r^2 + c \sin^2 \alpha.
\end{align*}
The result one obtains is
\begin{align}
g_c|_N & = \frac{4 (F^2 K^2 + r^4 \sin^2 \alpha \cos^2 \alpha \sin^2 \gamma \cos^2 \gamma)}{F K^4 L^2 \sin^2 \alpha} \d \rho^2 + \frac{K^2 - 4 r^2 \cos^4 \alpha \cos^2 \gamma}{4 F} p_1^2 \nonumber \\
& \qquad {} + \frac{L^2 - r^2 \cos^2 \alpha \sin^2 \gamma}{F} p_2^2 + \frac{L^2}{F} p_3^2 \nonumber \\
& \qquad {} + \frac{4 r^2 \cos \alpha \sin \gamma \cos^2 \gamma}{J K L} \d \rho \, p_3 - \frac{r^2 \cos \alpha (2 - \sin^2 \alpha) \sin \gamma \cos \gamma}{F} p_1 p_2. \label{eq:gc.N.SU2}
\end{align}

\begin{remark}
When $c>0$, because we are unable to integrate~\eqref{eq:flow.smooth} to obtain a third independent exact horizontal 1-form, we cannot use Corollary~\ref{cor:hypersymplectic} to obtain the induced hypersymplectic structure on $N$. However, in the cone case $c=0$ we \emph{were} able to find the third conserved quantity $u$ in~\eqref{eq:u.c=0}. However, unlike in the $\Lambda^2_-(T^* \mathcal{S}^4)$ case, this time $\d u$ is \emph{not} orthogonal to $\d v$. Nevertheless, one can explicitly compute $\hat h_1$ in terms of $\d u$ and $\d v$. The result is
\begin{equation*}
\hat h_1 = \frac{r^{\frac{1}{2}} J \cos^2 \alpha \cos \gamma}{K \sin^2 \gamma} \d u + \frac{r^2 \sin \gamma \big(4 \cos^4 \alpha - \sin^4 \alpha \sin^2 \gamma \cos^2 \gamma -2 \sin^2 \alpha \cos^2 \alpha (1 + \cos^2 \gamma) \big)}{J K \cos \alpha \cos \gamma} \d v.
\end{equation*}
From this, we can  rewrite $\varphi_c$ so that we can invoke Corollary~\ref{cor:hypersymplectic} to determine the induced hypersymplectic triple $( \omega_1, \omega_2, \omega_3 )$ on $N$. In this case the matrix $Q_{ij}$ from Definition~\ref{dfn:hypersymplectic} is \emph{not} diagonal, as there is a non-zero $Q_{12} = Q_{21}$ term. The precise expressions are extremely complicated so we omit them.
\end{remark}

\subsection{Riemannian geometry on the fibres} \label{sec:hypersymplectic.CP2}

We now discuss the induced Riemannian geometric structure on the coassociative fibres coming from the torsion-free $\GG_2$-structure. As mentioned in the remark at the end of~\S\ref{sec:package.SU2}, the induced hypersymplectic structure is only tractable when $c=0$ and even then, it is so complicated that we choose to omit it.

Recall the expression~\eqref{eq:gc.N.SU2} for the metric $g_c|_N$ on a coassociative fibre $N$. According to the classification in Tables~\ref{table:fibration-M} and~\ref{table:fibration-M0} the coassociative fibres (except the zero section $r=0$ in the $c>0$ case) are noncompact so we can study   the asymptotics of the metric as $r \to \infty$. In this section we show that in all cases these fibres are \emph{asymptotically conical} (AC). We describe the explicit metric on the link of the cone at infinity, and in almost all cases we can also give the rate of convergence to the asymptotic cone. For some special fibres, which are topologically $\R \times \mathcal{S}^3$, we can also consider the limit $r \to 0$ and in those cases, we find that near the vertex, they are \emph{conically singular} (CS). We describe the explicit metric on the link of the cone at the vertex.

We first introduce some notation. Recall that $\SU(2) \cong \mathcal{S}^3$. With respect to the $\beta$-rotated coframe $p_1, p_2, p_3$ of~\eqref{eq:beta-rotated}, define metric cones $g_A$ and $g_B$ on $\R^+ \times \mathcal{S}^3$ by
\begin{equation} \label{eq:gAgB}
\begin{aligned}
g_A & = \d R^2 + \frac{R^2}{6}(p_1^2+p_2^2)+\frac{R^2}{4}p_3^2, \\
g_B & = \d R^2 + \frac{R^2}{16}p_1^2+\frac{R^2}{4}(p_2^2 + p_3^2).
\end{aligned}
\end{equation}
Note that since $p_2^2 + p_3^2 = \sigma_2^2 + \sigma_3^2$, the metric $g_B$ can be expressed nicely in terms of $\sigma_1, \sigma_2, \sigma_3$ and is independent of $\beta$, but this is not true for $g_A$. However, here $\beta$ is constant so the metrics $g_A$ for different values of $\beta$ are all isometric to each other.

In this section we establish the following result.

\begin{prop}\label{prop:AC.CS-CP2} Let $g_A$, $g_B$ be the cone metrics on $\R^+ \times \mathcal{S}^3$ given in~\eqref{eq:gAgB}. Recall from Tables~\ref{table:fibration-M} and~\ref{table:fibration-M0} that all the noncompact fibres are topologically either $\R^+ \times \mathcal{S}^3$ or $\mathcal{O}_{\C\P^1}(-1)$. The geometry of the induced Riemannian metric on the noncompact fibres is as given in Tables~\ref{table:CP2.ACCS.cone} and~\ref{table:CP2.ACCS.smooth}. (Recall that AC denotes asymptotically conical and CS denotes conically singular.)
\end{prop}

\begin{table}[H]
\begin{center}
{\setlength{\extrarowheight}{4pt}
\begin{tabular}{|c|c|c|c|c|}
\hline
Defining parameters & Topology & Induced Riemannian metric \\[2pt]
\hline
$\sin \gamma = 0$ and $w = 0$ & $\R^+\times\mathcal{S}^3$ & exactly cone $g_B$ \\[2pt]
\hline
$\sin \gamma = 0$ and $|w| > 0$ & $\mathcal{O}_{\C\P^1}(-1)$ & AC to $g_B$ with rate $-4$ \\[2pt]
\hline
$\cos \gamma = 0$ and $w = 0$ & $\R^+\times\mathcal{S}^3$ & exactly cone $g_A$ \\[2pt]
\hline
$\cos \gamma = 0$ and $w > 0$ & $\mathcal{O}_{\C\P^1}(-1)$ & AC to $g_A$ with rate $-1$ \\[2pt]
\hline
$\cos \gamma = 0$ and $w < 0$ & $\R^+\times\mathcal{S}^3$ & AC to $g_A$ with rate $-1$, and CS to $g_B$ \\[2pt]
\hline
$\alpha = \frac{\pi}{2}$ and $w = 0$ & $\R^+\times\mathcal{S}^3$ & exactly cone $g_B$ \\[2pt]
\hline
$\alpha = \frac{\pi}{2}$ and $|w| > 0$ & $\mathcal{O}_{\C\P^1}(-1)$ & AC to $g_B$ with rate $-4$ \\[2pt]
\hline
generic and $|u| \neq 2$ & $\mathcal{O}_{\C\P^1}(-1)$ & AC to $g_A$ with rate $-1$ \\[2pt]
\hline
generic and $|u| = 2$ & $\R^+\times\mathcal{S}^3$ & AC to $g_A$ with rate $-1$, and CS to $g_B$ \\[2pt]
\hline
\end{tabular}}
\end{center}
\caption{Induced Riemannian geometry on fibres: cone case ($c=0$)}\label{table:CP2.ACCS.cone}
\end{table}

\begin{table}[H]
\begin{center}
{\setlength{\extrarowheight}{4pt}
\begin{tabular}{|c|c|c|c|c|}
\hline
Defining parameters & Topology & Induced Riemannian metric \\[2pt]
\hline
$\sin \gamma = 0$ & $\mathcal{O}_{\C\P^1}(-1)$ & AC to $g_B$ with rate $-4$ \\[2pt]
\hline
$\cos \gamma = 0$ & $\mathcal{O}_{\C\P^1}(-1)$ & AC to $g_A$ with rate $-1$ \\[2pt]
\hline
$\alpha = \frac{\pi}{2}$ & $\mathcal{O}_{\C\P^1}(-1)$ & AC to $g_B$ with rate $-4$ \\[2pt]
\hline
generic and $|u| \neq 2$ & $\mathcal{O}_{\C\P^1}(-1)$ & AC to $g_A$ at unknown rate \\[2pt]
\hline
generic and $|u| = 2$ & $\R^+\times\mathcal{S}^3$ & AC to $g_A$ and unknown as $r\to 0$ (but likely CS) 
\\[2pt]
\hline
\end{tabular}}
\end{center}
\caption{Induced Riemannian geometry on fibres: smooth case ($c>0$)}\label{table:CP2.ACCS.smooth}
\end{table}

\begin{remark}
We emphasize that some, but not all, of the \emph{singular} coassociative fibres in the Bryant--Salamon cone $M_0$ are Riemannian cones. The remaining singular fibres are both asymptotically conical and conically singular, but with different cone metrics at infinity and at the singular point. This is also almost certainly true of the singular fibres in $M = \Lambda^2_- (T^* \C\P^2)$.
\end{remark}

In the remainder of this section we give the details of the proof of Proposition~\ref{prop:AC.CS-CP2}.

\subsubsection{The cone case $(c=0)$}

We consider the cone $(c=0)$ first, where the results are more complete, and then consider the smooth case $(c>0)$. We number the cases as in~\S\ref{sec:fibration.CP2}. Note that when $c=0$ then~\eqref{eq:rho.CP2} gives
\begin{equation} \label{eq:rho.cone}
\rho = \tfrac{1}{4} r^2 \sin^2 \alpha (4 \cos^2 \alpha + \sin^2 \gamma \sin^2 \alpha).
\end{equation}

\textbf{Case 2: $\sin \gamma = 0$.} In this case, in addition to $\beta, \gamma$, we have the conserved quantity $w = r \cos 2\alpha$.

Consider first $w = 0$. Then we have $\alpha = \frac{\pi}{4}$, and thus by~\eqref{eq:rho.cone} we have $\rho = \frac{1}{4} r^2$ exactly. The metric~\eqref{eq:gc.N.SU2} in this case is
\begin{equation*}
\frac{4}{r^3} \d\rho^2 + \frac{r}{4} p_1^2 + r p_2^2 + r p_3^2.
\end{equation*}
Writing $\d\rho = \frac{1}{2} r \d r$, this becomes
\begin{equation*}
\frac{1}{r} \d r^2 + \frac{r}{4} p_1^2 + r p_2^2 + r p_3^2.
\end{equation*}
Let $\d R = \frac{1}{r^{\frac{1}{2}}} \d r$, so $R = 2 r^{\frac{1}{2}}$ and $r = \frac{1}{4} R^2$. Then the metric becomes
\begin{equation*}
\d R^2 + \frac{R^2}{16} p_1^2 + \frac{R^2}{4} p_2^2 + \frac{R^2}{4} p_3^2,
\end{equation*}
which is \emph{exactly} $g_B$.

Now consider $|w| > 0$. From $w = r \cos 2\alpha$, we get $\cos 2\alpha = \frac{w}{r}$, so $\cos^2 \alpha = \frac{1}{2} (1 + \cos 2\alpha) = \frac{1}{2} (1 + \frac{w}{r})$ and $\sin^2 \alpha = \frac{1}{2} (1 - \cos 2\alpha) = \frac{1}{2}(1 - \frac{w}{r})$. Substituting these into~\eqref{eq:rho.cone} gives $\rho = \frac{1}{4} (r^2 - w^2)$. The metric~\eqref{eq:gc.N.SU2} in this case becomes
\begin{equation*}
\frac{1}{r^3 \cos^2 \alpha \sin^2 \alpha} \d\rho^2 + r \cos^2 \alpha \sin^2 \alpha p_1^2 + r p_2^2 + r p_3^2.
\end{equation*}
Writing $\d\rho = \frac{1}{2} r \d r$ and substituting the expression $\sin^2 \alpha \cos^2 \alpha = \frac{1}{4} (1 - \frac{w^2}{r^2})$, this becomes
\begin{equation*}
\frac{1}{r(1 - \frac{w^2}{r^2})} \d r^2 + \frac{r(1 - \frac{w^2}{r^2})}{4} p_1^2 + r p_2^2 + r p_3^2.
\end{equation*}
Now define $R = 2 r^{\frac{1}{2}}$ as in the $w=0$ case above, so that $r = \frac{1}{4} R^2$. Then one can compute that for large $R$, the metric becomes
\begin{equation*}
\d R^2 + \frac{R^2}{16} p_1^2 + \frac{R^2}{4} p_2^2 + \frac{R^2}{4} p_3^2 + O(R^{-4}) \qquad \text{as $R \to \infty$}.
\end{equation*}
Thus in this case the metric is asymptotically conical to the cone metric $g_B$ with rate $-4$.

\textbf{Case 3: $\cos \gamma = 0$.} Here, in addition to $\beta, \gamma$, we have the conserved quantity $w^2 = \frac{r(1+3\cos 2\alpha)^2}{8(1+\cos 2\alpha)}$.

Consider first $w = 0$. Then we have $\cos 2\alpha = -\frac{1}{3}$, and hence $\cos^2 \alpha = \frac{1}{3}$ and $\sin^2 \alpha = \frac{2}{3}$. It follows from~\eqref{eq:rho.cone} that $\rho = \frac{1}{3} r^2$ exactly. The metric~\eqref{eq:gc.N.SU2} in this case is
\begin{equation*}
\frac{4}{r^3( 3 \cos^6 \alpha - 9 \cos^4 \alpha + 5 \cos^2 \alpha + 1)} \d\rho^2 + \frac{r(1+ 6 \cos^2 \alpha - 3 \cos^4 \alpha)}{4} p_1^2 + r \sin^2 \alpha p_2^2 + r p_3^2,
\end{equation*}
which simplifies to
\begin{equation*}
\frac{9}{4 r^3} \d\rho^2 + \frac{2r}{3} p_1^2 + \frac{2r}{3} p_2^2 + r p_3^2.
\end{equation*}
Writing $\d\rho = \frac{2}{3} r \d r$, this becomes
\begin{equation*}
\frac{1}{r} \d r^2 + \frac{2r}{3} p_1^2 + \frac{2r}{3} p_2^2 + r p_3^2.
\end{equation*}
Let $\d R = \frac{1}{r^{\frac{1}{2}}} \d r$, so $R = 2 r^{\frac{1}{2}}$ and $r = \frac{1}{4} R^2$. Then the metric becomes
\begin{equation*}
\d R^2 + \frac{R^2}{6} p_1^2 + \frac{R^2}{6} p_2^2 + \frac{R^2}{4} p_3^2,
\end{equation*}
which is \emph{exactly} $g_A$.

Now consider $|w| > 0$. From $w^2 = \frac{r(1+3\cos(2\alpha))^2}{8(1+\cos(2\alpha))}$, we get $\cos 2\alpha = -\frac{1}{3} + \frac{4}{9r} ( w^2 - \sqrt{w^4 + 3 w^2 r} )$, from which it follows that $\cos^2 \alpha = \frac{1}{3} + \frac{2}{9r} ( w^2 - \sqrt{w^4 + 3 w^2 r} )$ and $\sin^2 \alpha = \frac{2}{3} - \frac{2}{9r} ( w^2 - \sqrt{w^4 + 3 w^2 r} )$. Substituting these into~\eqref{eq:rho.cone} and simplifying, we obtain
\begin{equation*}
\rho = \frac{1}{3} r^2 - \frac{w^2}{9} r - \frac{2 w^4}{27} + \frac{2 w^2 \sqrt{w^4 + 3 w^2 r}}{27}.
\end{equation*}
Defining $R$ by $r = \frac{1}{4} R^2$ as in the $w=0$ case above, a computation yields that for large $R$, the metric~\eqref{eq:gc.N.SU2} in this case becomes
\begin{equation*}
\d R^2 + \frac{R^2}{6} p_1^2 + \frac{R^2}{6} p_2^2 + \frac{R^2}{4} p_3^2 + O(R^{-1}) \qquad \text{as $R \to \infty$}.
\end{equation*}
Thus in this case the metric is asymptotically conical to the cone metric $g_A$ with rate $-1$.

When $w < 0$, we can take $r \to 0$. In this case we again let $r = \frac{1}{4} R^2$ and perform a careful asymptotic expansion as $R \to 0$. The result is that the metric can be written as
\begin{equation*}
\d R^2 + \frac{R^2}{16} p_1^2 + \frac{R^2}{4} p_2^2 + \frac{R^2}{4} p_3^2 + O(R) \qquad \text{as $R \to 0$}.
\end{equation*}
Thus in this case the metric is conically singular as $R \to 0$ to the cone metric $g_B$.

\textbf{Case 4: $\alpha = \frac{\pi}{2}$.} In this case, in addition to $\beta, \alpha$, we have the conserved quantity $w = r \cos \gamma$.

Consider first $w = 0$. Then we have $\gamma = \frac{\pi}{2}$, and~\eqref{eq:rho.cone} gives $\rho = \frac{1}{4} r^2$. Now the metric~\eqref{eq:gc.N.SU2} is
\begin{equation*}
\frac{4}{r^3} \d\rho^2 + \frac{r}{4} p_1^2 + r p_2^2 + r p_3^2.
\end{equation*}
This is identical to the $w=0$ subcase of Case 2 above, and thus is \emph{exactly} $g_B$.

Now consider $|w| > 0$. From $w = r \cos\gamma$, we get $\cos\gamma = \frac{w}{r}$, so $\sin\gamma = (1 - \frac{w^2}{r^2})^{\frac{1}{2}}$. Substituting these into~\eqref{eq:rho.cone} gives $\rho = \frac{1}{4} (r^2 - w^2)$. Now the metric~\eqref{eq:gc.N.SU2} is
\begin{equation} \label{eq:gc.N.alpha.pi/2}
\frac{1}{r(1 - \frac{w^2}{r^2})} \d r^2 + \frac{r(1 - \frac{w^2}{r^2})}{4} p_1^2 + r p_2^2 + r p_3^2.
\end{equation}
This is identical to the $|w|>0$ subcase of Case 2 above, and thus this case is asymptotically conical to the cone metric $g_B$ with rate $-4$.

\textbf{Case 5: generic setting.} In this case, in addition to $\beta$, we have the conserved quantities
$$v = 2 r^{\frac{1}{2}} \cos \alpha \cot \gamma \quad \text{ and } \quad u = \frac{2 \cos^2 \alpha - \sin^2 \alpha \sin^2 \gamma}{\cos^2 \alpha \cos \gamma}.$$
We can solve the first equation for $\cos^2 \alpha$, obtaining
\begin{equation*}
\cos^2 \alpha = \frac{v^2 \sin^2 \gamma}{4 r \cos^2 \gamma}.
\end{equation*}
Substituting the above into the expression for $u$, some manipulation gives
\begin{equation*}
u = \frac{3 v^2 - (v^2 + 4r) \cos^2 \gamma}{v^2 \cos \gamma}.
\end{equation*}
This can then be solved for $\cos \gamma$. (Recall that $\gamma \in [0, \pi]$ so $\cos \gamma \geq 0$.) The solution is
\begin{equation*}
\cos \gamma = \frac{v(-uv + \sqrt{u^2 v^2 + 12 v^2 + 48 r})}{2(v^2 + 4r)}.
\end{equation*}
From the above expressions for $\cos \gamma$ and $\cos \alpha$, further computation yields a complicated expression for $\rho$ in terms of   $r$ and the constants $u,v$. We proceed as before, setting $r = \frac{1}{4} R^2$ and studying the asymptotics of the metric~\eqref{eq:gc.N.SU2}. The authors did these computations on Maple, obtaining the following.
\begin{itemize}
\item As $R \to \infty$, the metric $g_c|_N$ converges to the cone metric $g_A$ at rate $-1$, just like the $|w|>0$ subcase of Case 3 above.
\item When $u=2$ or $u=-2$, we can take $r \to 0$ according to Table~\ref{table:fibration-M0}. As $R \to 0$, the metric $g_c|_N$ converges to the cone metric $g_B$, just like the $w<0$ subcase of Case 3 above.
\end{itemize}

\subsubsection{The smooth case ($c>0$)}

In the smooth case, when we can at least \emph{approximately} find the conserved quantities, we can proceed as in the cone case to compute the asymptotic cones \emph{as well as} the rates of convergence to those cones. We can do this when $\sin \gamma = 0$ or $\alpha = \frac{\pi}{2}$. The results are the same as in the cone case (we give explicit statements below.) However, when $\cos \gamma = 0$ or in the generic setting, we cannot even approximately find the third conserved quantity. Thus in these two cases we can only find the asymptotic cones but not the rates of convergence. A sketch of the details follows. Again, we number the cases as in~\S\ref{sec:fibration.CP2}.

\textbf{Case 2: $\sin \gamma = 0$.} In this case, in addition to $\beta, \gamma$, one can show that there is a conserved quantity
\begin{equation*}
w = r^{\frac{1}{2}} (r^2 + c)^{\frac{1}{4}} \cos 2\alpha - \frac{c}{2} \int \frac{1}{r^{\frac{1}{2}} (r^2 + c)^{\frac{3}{4}}} \d r. 
\end{equation*}
This cannot be integrated in closed form. For large $r$, we expand in powers of $\frac{1}{r}$ as before. We get
\begin{equation*}
w \approx r^{\frac{1}{2}} (r^2 + c)^{\frac{1}{4}} \cos 2\alpha + \frac{c}{2r} - \frac{c^2}{8 r^3}.
\end{equation*}
This can be solved (again, approximating for large $r$) for $\cos 2\alpha$, which in turn leads to the formula
\begin{equation*}
\rho = \frac{1}{4} r^2 - \frac{w^2}{4} + \frac{3c}{8} + O(r^{-1}).
\end{equation*}
Proceeding as before, the metric~\eqref{eq:gc.N.SU2} becomes
\begin{equation*}
\d R^2 + \frac{R^2}{16} p_1^2 + \frac{R^2}{4} p_2^2 + \frac{R^2}{4} p_3^2 + O(R^{-4}) \qquad \text{as $R \to \infty$}.
\end{equation*}
Thus in this case the metric is asymptotically conical to the cone metric $g_B$ with rate $-4$, identical to the $|w| > 0$ subcases of Cases 2 and 4 when $c=0$.

\textbf{Case 3: $\cos \gamma = 0$.} In this case, in addition to $\beta$ and $\gamma$, there is a third conserved quantity that we cannot find, even approximately. Nevertheless, because we have an explicit formula~\eqref{eq:cos.gamma.zero.drdalpha} for $\frac{\d r}{\d \alpha}$ in this case, we can still compute the asymptotic geometry. From the analysis in~\S\ref{sec:cos.gamma.zero} and~\eqref{eq:cos.gamma.zero.drdalpha} in particular, one can show that as $r \to \infty$, we have $\cos 2\alpha = -\frac{1}{3} + O(r^{-\frac{1}{2}})$ and hence $\cos^2 \alpha = \frac{1}{3} + O(r^{-\frac{1}{2}})$ and $\sin^2 \alpha = \frac{2}{3} + O(r^{-\frac{1}{2}})$. Now expanding in powers of $\frac{1}{r}$ for large $r$ and proceeding as before, one obtains
\begin{equation*}
\d R^2 + \frac{R^2}{6} p_1^2 + \frac{R^2}{6} p_2^2 + \frac{R^2}{4} p_3^2 + O(R^{-1}) \qquad \text{as $R \to \infty$}.
\end{equation*}
Thus this case is asymptotically conical to the metric cone $g_A$ with rate $-1$, just like the $|w|>0$ subcase of Case 3 and the generic Case 4 when $c=0$.

\textbf{Case 4: $\alpha = \frac{\pi}{2}$.} In this case, in addition to $\beta, \alpha$, we have the conserved quantity $w = r \cos \gamma$. Just as in the $|w|>0$ subcase of Case 4 for the cone, we get $\rho = \frac{1}{4} (r^2 - w^2)$. Here, the metric~\eqref{eq:gc.N.SU2} is
\begin{equation*}
\frac{1}{(r^2 + c)^{\frac{1}{2}} (1 - \frac{w^2}{r^2})} \d r^2 + \frac{r^2(1 - \frac{w^2}{r^2})}{4 (r^2 + c)^{\frac{1}{2}}} p_1^2 + (r^2 + c)^{\frac{1}{2}} p_2^2 + (r^2 + c)^{\frac{1}{2}} p_3^2.
\end{equation*}
(Compare the above expression with~\eqref{eq:gc.N.alpha.pi/2}.) Now expanding in powers of $\frac{1}{r}$ for large $r$ and proceeding as before, one obtains
\begin{equation*}
\d R^2 + \frac{R^2}{16} p_1^2 + \frac{R^2}{4} p_2^2 + \frac{R^2}{4} p_3^2 + O(R^{-4}) \qquad \text{as $R \to \infty$}.
\end{equation*}
Thus in this case the metric is asymptotically conical to the cone metric $g_B$ with rate $-4$, identical to the $|w| > 0$ subcases of Cases 2 and 4 when $c=0$ and also identical with Case 2 above when $c>0$.

\textbf{Case 5: generic setting.} In this case, in addition to $\beta, \gamma$, there is a third conserved quantity that we cannot find, even approximately. So we cannot compute the asymptotic rates of convergence. But we can compute the asymptotic cones.

From the analysis in~\S\ref{sec:fibration-generic} we know that as $r \to \infty$, we have $\gamma \to \frac{\pi}{2}$ and $\tan^2 \alpha \to 2$. This implies that $\cos^2 \alpha \to \frac{1}{3}$ and $\sin^2 \alpha \to \frac{2}{3}$. Now expanding in powers of $\frac{1}{r}$ for large $r$ and proceeding as before, one obtains that
\begin{equation*}
\d R^2 + \frac{R^2}{6} p_1^2 + \frac{R^2}{6} p_2^2 + \frac{R^2}{4} p_3^2 \quad \text{plus smaller terms as $R \to \infty$}.
\end{equation*}
Thus this case is asymptotically conical to the metric cone $g_A$, as in the $|w|>0$ subcase of Case 3 and the generic Case 4 when $c=0$ and also as in Case 3 above when $c>0$.

\begin{remark} \label{remark:inconclusive}
We do not know the rate of convergence. In fact, this naive analysis, where we replace $\cos^2 \alpha$ and $\sin^2 \alpha$ by constants and ignore their next terms in powers of $\frac{1}{r}$, actually results in an expression that equals $g_A$ plus terms of $O(R^{-4})$. This is clearly not correct, because in the cone case we get $O(R^{-1})$. This shows that we have thrown away too much information to be able to get the correct rate of convergence to the asymptotic cone. It is likely  that the rate is again $-1$ in this case, since the rates in Cases 2 and 4 for $c>0$ are the same as they are for $c=0$, but our analysis is inconclusive.
\end{remark}

Finally, from the analysis in~\S\ref{sec:fibration-generic} and Table~\ref{table:fibration-M} we know that for two special values of the unknown third conserved quantity $u$ we can actually have $r \to 0$. However, since we do not have an explicit (nor indeed even an approximate) expression for $u$ when $c>0$, it is not possible to correctly determine the asymptotic behaviour of the induced metric $g_c|_N$ as $r \to 0$. However, in analogy with the generic setting for the cone ($c=0$) case when $u = \pm 2$, we expect that these fibres should be conically singular as $r \to 0$. Moreover, because this fibration ``limits'' to the Harvey--Lawson coassociative fibration (as we discuss in \S\ref{sec:flat.limit.SU2}), we expect the limiting coassociative cone at the vertex to be the Lawson--Osserman cone. (See~\cite[Section 9.1]{LotayDef} for a detailed description of this cone, where it is denoted $M_0^{\pm}$.)

\subsection{Flat limit} \label{sec:flat.limit.SU2}

In this section we describe what happens to the coassociative fibration on $M$ as we take the flat limit as in $\S$\ref{subs:flat}.

The $\SU(2)$ action on $\C\P^2$ we have chosen becomes, in the flat limit, the standard $\SU(2)$ action on $\C^2=\mathbb{R}^4$. The induced action on $\Lambda^2_-(T^*\C^2)=\R^7$ is then the action of $\SU(2)$ on $\R^7=\R^3\oplus\C^2$ which acts as $\SO(3)$ on $\R^3$ and $\SU(2)$ on $\C^2$.

Since we have simply undertaken a rescaling in taking the flat limit, the coassociative fibration and the $\SU(2)$-invariance is preserved along the rescaling, and in the limit we obtain an $\SU(2)$-invariant coassociative fibration of $\R^7=\R^3\oplus\C^2$. By uniqueness this must be the $\SU(2)$-invariant coassociative fibration produced by Harvey--Lawson~\cite[Section IV.3]{HarveyLawson}, which we describe in this section.

\begin{remark}
Note that although we again have a $\U(1)$ action on the total space given by translations of $\beta$ as we did in the $\Lambda^2_-(T^* \mathcal{S}^4)$ case, the angle coordinate $\beta$ plays a much different role in $\Lambda^2_-(T^* \C\P^2)$, and in fact this $\U(1)$ action \emph{does not commute} with the $\SU(2)$ action. It is for this reason that in the $\C\P^2$ case, in the flat limit our coassocative fibration of $\R^7$ does not reduce to a fibration of $\C^3$ by complex surfaces as it did for the $\mathcal{S}^4$ case in~\S\ref{sec:flat.limit.SO3}.
\end{remark}

Identifying $\R^7=\Imm \mathbb{O}=\Imm \H\oplus\H e$ for $e \in \H^{\perp}$ with $|e|=1$, the coassociative fibration is given by
\begin{align*}
N_{\text{HL}}(\varepsilon,\tau)
=\{rq\varepsilon\bar{q}+(s\bar{q})e : q\in\mathcal{S}^3\subseteq\H,\,r\in\R,
\,s\geq 0,
\,r(4r^2-5s^2)^2=\tau\},
\end{align*}
for $\varepsilon \in \Imm \H$, $|\varepsilon|=1$, and $\tau\in\R$. Here the $\SU(2)$-action is given through multiplication by unit quaternions $q\in\mathcal{S}^3\subseteq\H$.

Notice that $\tau$ and $-\tau$ both give the same family of coassociative 4-folds (since we can just change $\varepsilon$ to $-\varepsilon$) and so we can restrict to $\tau\geq 0$. Moreover, when $\tau>0$ then $r>0$, and when $\tau=0$ (again by changing $\varepsilon$ to $-\varepsilon$ if necessary) we can assume $r\geq 0$. In all cases, we can view $r$ as the distance to $\H e=\C^2$ defined by $r=0$. This matches our earlier notation where $r$ was the distance in the fibres from $\C\P^2$, which has become $\C^2$ here in the flat limit. 

We observe that $N_{\text{HL}}(\varepsilon,0)$ is a union of two cones on $\mathcal{S}^3$, one which is the flat $\H e=\C^2$ and one which is a non-flat cone (the Lawson--Osserman cone~\cite{LawsonOsserman}) which has link given by the graph of a Hopf map from $\mathcal{S}^3$ to $\mathcal{S}^2$. 

For fixed $\varepsilon$ and $\tau\neq 0$, $N_{\text{HL}}(\varepsilon,\tau)$ has two components $N_{\text{HL}}^+(\varepsilon,\tau)$ and $N_{\text{HL}}^-(\varepsilon,\tau)$ determined by the sign of $4r^2-5s^2$. The component $N_{\text{HL}}^+(\varepsilon,\tau)$ has a single end asymptotic to the Lawson--Osserman cone and is diffeomorphic to $\mathcal{O}_{\C\P^1}(-1)$. Each of these $N_{\text{HL}}^+(\varepsilon,\tau)$ components intersects an $\mathcal{S}^2$-family of coassociatives which have the same $\tau$ (and varying $\varepsilon$). The other component $N_{\text{HL}}^-(\varepsilon,\tau)$ has two ends, one end asymptotic to the Lawson--Osserman cone and the other end asymptotic to the flat $\C^2$, and is diffeomorphic to $\R\times\mathcal{S}^3$. These $N_{\text{HL}}^-(\varepsilon,\tau)$ components never intersect for distinct values of $\varepsilon$ and $\tau$.

By~\cite[Section 9.1]{LotayDef} the $\mathcal{O}_{\C\P^1}(-1)$ components have asymptotically conical metrics with rate $-\frac{5}{2}$, whereas the (smooth) $\R\times\mathcal{S}^3$ components have asymptotically conical metrics with rate $-\frac{5}{2}$ at the Lawson--Osserman cone end, but have rate $-5$ at the flat $\C^2$ end.

\end{document}